\documentclass[reqno,10pt]{amsart}

\usepackage[margin=0.9in]{geometry}

\usepackage{amssymb,amsmath,amsthm}
\usepackage{amscd}
\usepackage{amsfonts}
\usepackage{amssymb}
\usepackage{latexsym}
\usepackage{xcolor}
\usepackage{esint}

\usepackage{hyperref}

\usepackage{graphicx}
\graphicspath{ {images/} }

\usepackage[makeroom]{cancel}

\newtheorem{theorem}{Theorem}

\newtheorem{definition}{Definition}
\newtheorem{lemma}{Lemma}
\newtheorem{proposition}[theorem]{Proposition}
\newtheorem{remark}{Remark}
 \newtheorem*{theorem*}{Rough version of the Main theorem}

\let\e=\varepsilon

\let\p=\partial

\let\O=\Omega

\let\pll = \parallel

\numberwithin{equation}{section}

\let\hide\iffalse
\let\unhide\fi

\DeclareMathAlphabet{\mathpzc}{OT1}{pzc}{m}{it}

\newcommand{\R}{\mathbb{R}}

\newcommand{\be}{\begin{equation}}
\newcommand{\bm}{\begin{multline}}
\newcommand{\ee}{\end{equation}}
\newcommand{\dd}{\mathrm{d}}

\newcommand{\fe}{f^{\epsilon}}

\newcommand{\xb}{x_{\mathbf{b}}}
\newcommand{\tb}{t_{\mathbf{b}}}
\newcommand{\vb}{v_{\mathbf{b}}}

\newcommand{\xf}{x_{\mathbf{f}}}
\newcommand{\tf}{t_{\mathbf{f}}}
\newcommand{\vf}{v_{\mathbf{f}}}

\newcommand{\vbn}{v_{\mathbf{b}, 3}}
\newcommand{\tbi}{t^i_{\mathbf{b}}}
\newcommand{\tbj}{t^j_{\mathbf{b}}}

\newcommand{\vib}{v^i_{\mathbf{b}}}
\newcommand{\xib}{x^i_{\mathbf{b}}}
\newcommand{\tib}{t^i_{\mathbf{b}}}

\newcommand{\Bes}{\begin{eqnarray*}}
\newcommand{\Ees}{\end{eqnarray*}}
\newcommand{\Be}{\begin{equation} }
\newcommand{\Ee}{\end{equation}}
\newcommand{\Bs}{\begin{split}}   

\newcommand{\N}{\mathbb{N}}
\newcommand{\Z}{\mathbb{Z}}
\newcommand{\J}{\mathfrak{J}}

\newcommand{\vertiii}[1]{{\left\vert\kern-0.25ex\left\vert\kern-0.25ex\left\vert #1 
    \right\vert\kern-0.25ex\right\vert\kern-0.25ex\right\vert}}

 \makeatletter
\def\munderbar#1{\underline{\sbox\tw@{$#1$}\dp\tw@\z@\box\tw@}}
\makeatother

\def\p{\partial}

\def\O{\Omega}
\def\R{\mathbb{R}}

\def\B{\begin{equation}}
\def\E{\end{equation}}
\def\BN{\begin{eqnarray*}}
\def\EN{\end{eqnarray*}}

\def\bcb{\begin{color}{blue}}
\def\ec{\end{color}}

\def\bcr{\begin{color}{red}}
\def\ec{\end{color}}

\begin{document}

\title{Exponential Mixing of Vlasov equations under the effect of Gravity and Boundary}

\author{Jiaxin Jin \& Chanwoo Kim}



 \address{Department of Mathematics, University of Wisconsin-Madison, Madison, WI, 53706, USA}\email{chanwookim.math@gmail.com}
 
 \address{Department of Mathematics, Ohio State University, Columbus, OH, 43210, USA}
\email{jin.1307@osu.edu}

\date{}

%

\begin{abstract}
In this paper, we study exponentially fast mixing induced/enhanced by gravity and stochastic boundary in the kinetic theory of Vlasov equations. We consider the Vlasov equations with and without a vertical magnetic field inside a horizontally-periodic 3D half-space equipped with a non-isothermal diffusive reflection boundary condition of bounded continuous boundary temperature at the bottom. We construct both stationary solutions and global-in-time dynamical solutions in $L^\infty$. We prove that moments of a dynamical fluctuation around the steady solutions decay exponentially fast in $L^\infty$. As a key of this proof, we establish a uniform bound of so-called residual measures independently of the bouncing number of stochastic characteristics, by constructing a continuous stationary outgoing boundary flux which is strictly positive almost everywhere. \end{abstract}

\maketitle



 \bigskip
\begin{center}
	{{I\footnotesize{NTRODUCTION}}}
\end{center}
\smallskip

Dynamics of molecules under gravity has been an important object in various mathematical studies from microscopic to macroscopic description \cite{EGKM2, LY,SSS, KTW}. In this paper, we consider a many body problem of molecules without intermoleculer interaction which are contained in a horizontally-periodic three dimensional half-space $\O = \mathbb{T}^2 \times \R_{+}$ and subjected to the gravity (positive gravitational constant $g>0$) with/without a vertical magnetic field. A governing kinetic model of the system is the Vlasov equations:
\Be \label{equation for F mag} 
\partial_t F + v \cdot \nabla_{x} F + (v \times B - \nabla \phi - g \mathbf{e}_3 ) \cdot \nabla_{v} F = 0   \ \  \text{in}  \  \ (t,x,v) \in \R_{+} \times \Omega \times \R^3.
\Ee
Such equations are used in the kinetic model of solar wind (see \cite{CK2} and the references therein).

At the bottom of domain, the molecules interact with the boundary thermodynamically via a \textit{diffusive reflection} boundary condition
\Be \label{diff_F}
F(\cdot, x, v)  
= \mu_{\Theta} (x, v) \int_{n(x) \cdot v^1 >0} F(\cdot, x, v^1) \{ n(x) \cdot v^1 \} \dd v^1 \ \ \text{for} \ \  (x,v) \in \gamma_- : = \{ x \in \p\O \ \text{and} \    v_3 > 0\},
\Ee   
such that an outgoing distribution is proportional to the thermal equilibrium of boundary temperature $\Theta(x) > 0$:
\Be\label{wall_M}
\mu_{\Theta} (x, v) = \frac{1}{2 \pi \Theta(x)^2}e^{-  \frac{|v|^2}{2 \Theta(x)}}
\ \ \ \ \ \ \ \ \ \ \ \   \text{(wall Maxwellian)}.
\Ee  
Due to the normalization  
$\int_{n(x) \cdot v^1 >0} \mu_{\Theta} (x, v^1) 
\{n(x) \cdot v^1\} \dd v^1 = 1$,
we have a null flux at the boundary and enjoy the conservation of total mass. Throughout this paper we assume the boundary temperature $\Theta(x)$ is 
\textit{continuous} on $\p\O$, so that 
\be \label{def:Theta}
0 < a \leq \Theta(x) \leq b < \infty, \ \text{for all} \ x \in \p\O.
\ee

\bigskip

\noindent\textbf{Main Theorem. }The main interest in this paper is to study a long-time behavior of solutions to the Vlasov equations for both a case without magnetic field
\Be \label{field property}
\begin{split}
B = \vec{0}, \    \ \text{and} \  \ \|\nabla_{x} \phi\|_{L^\infty(\O)} \leq g/2 , 	 \ \| e^{\varrho_1 x_3}   \nabla_{x} ^2\phi\|_{L^\infty(\O)}   \leq  \varrho_2 \ll 1 , \ 	\phi (x)|_{x_3 = 0} \equiv 0;
\end{split}
\Ee
and a case with vertical magnetic field
\be \label{field property mag}
\begin{split}
  B = (0, 0, B_3) \ \text{with a constant} \ B_3 \geq 1, 
\ \ \text{and} \ \ \nabla \phi \equiv 0.
\end{split}
\ee
First of all, it is well-known that stationary solutions to \eqref{equation for F mag}-\eqref{diff_F} are neither given by explicit formulas nor are equilibria (local Maxwellian) in general, if they exist. We refer to \cite{Sone, KLT, LY} for special cases in which a simple form of stationary solutions are available. In this paper, we construct a unique stationary solution by establishing $L^\infty$-bound using stochastic cycles and a natural $L^1$-bound from mass. This method is direct/robust and in particular does not rely on the dynamical stability. Second, we prove exponentially fast asymptotic stability of such stationary solutions. Gravity plays an important role to enhance the convergence speed as only polynomial decay \cite{Yu, Bernou,JinKim1,Lods} is possible without the gravity due to possible slow particles. For the other mixing mechanisms, we refer to \cite{CKRZ,ZY} in the active transport equations, \cite{BMM, MV, HV} for the Landau damping, and \cite{K_CPDE,CKL, KL, KL1} for the Boltzmann equation. Precise statements and remarks follow.

\begin{theorem} \label{existence on F}We consider both \eqref{field property} and \eqref{field property mag}. We assume a continuous boundary temperature $\Theta(x)$ satisfies \eqref{def:Theta}. Given any $\mathfrak{m} > 0$, there exists a unique solution $F_s (x, v)$ to the stationary Vlasov equation  
	\Be \label{equation for F_infty} 
	v \cdot \nabla_{x} F_s + (v \times B - \nabla \phi - g \mathbf{e}_3 )\cdot \nabla_{v} F_s = 0,  \ \  \text{in} \    \Omega \times \R^3,
	\Ee
	with either \eqref{field property} or \eqref{field property mag}; and the diffusive reflection boundary condition \eqref{diff_F}  
	such that
	\Be \label{initial mass on F}
	\iint_{\O \times \R^3} F_s (x, v) \dd x \dd v = \mathfrak{m}.
	\Ee
	Moreover, we have, for $b$ in \eqref{def:Theta}, 
	\Be \label{L_infty estimate on F_infty}
	\left\| e^{\frac{1}{2b}\big( \frac{|v|^2}{2} + \Phi(x)\big) } F_s(x,v) \right\|_{L^{\infty}_{x,v}} \lesssim  \mathfrak{m}.
	\Ee
\end{theorem}

\hide 
 
 \begin{theorem} \label{theorem_1}
 	Assume $\iint_{\O \times \R^3} f_0 (x,v) \dd x \dd v = 0$, 
 \end{theorem}
 
 \begin{theorem} \label{theorem_1}
 	Assume $\iint_{\O \times \R^3} f_0 (x,v) \dd x \dd v = 0$, $\| e^{\theta^\prime (|v|^2+ 2\Phi (x))} f_0\|_{L^\infty_{x,v}}< \infty$ for $0 < \theta^\prime < \frac{1}{2b}$ with $b$ defined in \eqref{def:Theta},
 	and $\| e^{\delta (|v|^2 + 2 \Phi(x))^{1/2} } f_0 \|_{L^1_{x,v}} < \infty$ with $0 < \delta < \infty$. There exists a unique solution $f(t,x,v)$ to \eqref{eqtn_f mag}-\eqref{diff_f mag},
 	such that
 	\Be \label{cons_mass_f} 
 	\iint_{\Omega \times \R^3} f (t, x, v) \dd x \dd v 
 	= \iint_{\Omega \times \R^3} f_0 ( x, v) \dd x \dd v = 0, \ \ \text{for all } t\geq 0. 
 	\Ee
 	\Be \label{exp decay f} 
 	\| f(t) \|_{L^1_{x,v}} 
 	\lesssim  e^{- \Lambda t}
 	\{ \| e^{\theta^\prime (|v|^2+ 2\Phi (x))} f_0\|_{L^\infty_{x,v}}
 	+ \| \varphi \big( (|v|^2 + 2 \Phi(x))^{1/2} \big) f_0 \|_{L^1_{x,v}} \},
 	\Ee
 	where $\Lambda$ is the same as in Theorem \ref{theorem_1}.
 \end{theorem}\unhide
 
 \begin{theorem} \label{theorem}We consider both \eqref{field property} and \eqref{field property mag}. Assume all conditions in Theorem \ref{existence on F}. Let $F_s(x,v)$ be a stationary solution in Theorem \ref{existence on F}. Consider the initial data $F_0 (x, v) = F_s (x,v) + f_0 (x,v)$ such that $\iint_{\O \times \R^3} f_0 (x,v) \dd x \dd v = 0$. $\| e^{\theta^\prime (|v|^2+ 2\Phi (x))} f_0\|_{L^\infty_{x,v}}< \infty$ for $0 < \theta^\prime < \frac{1}{2b}$ with $b$ defined in \eqref{def:Theta}, and $\| e^{\delta (|v|^2 + 2 \Phi(x))^{1/2} } f_0 \|_{L^1_{x,v}} < \infty$ with $0 < \delta < \infty$. There exists a unique global-in-time solution 
 \Be	\label{def:f}
 	F(t,x,v) = F_s (x,v) + f(t,x,v)\geq 0
 	\Ee
 	 to \eqref{equation for F mag} with either \eqref{field property} or \eqref{field property mag} and the boundary condition \eqref{diff_F} with the initial condition $F(t,x,v)|_{t=0} = F_0(x,v)$ in $\O \times \R^3$,
 	such that
 	\Be \label{cons_mass_f} 
 	\iint_{\Omega \times \R^3} f (t, x, v) \dd x \dd v = 0, \ \ \text{for all } t\geq 0,
 	\Ee
 		\Be \label{theorem_infty_1}
 	\sup_{t\geq0}\| e^{\theta^\prime (|v|^2+ 2\Phi (x))} f (t)\|_{L^\infty_{x,v}} \lesssim \| e^{\theta^\prime (|v|^2+ 2\Phi (x))} f_0\|_{L^\infty_{x,v}}.
 	\Ee
 	Moreover, for all $t \geq 0$ and $0 \leq \theta< \theta^\prime$,
 	\Be \label{theorem_infty}
 	\begin{split}
 		\sup_{x \in \bar{\O}}\int_{\R^3} e^{\theta  (|v|^2+ 2\Phi (x))} |f(t,x,v) |\dd v  
 		\lesssim_{\theta}  e^{- \Lambda t},
 	\end{split} 
 	\Ee
 	where $\Lambda  $ is defined in \eqref{def:M}.
 	\hide

 	\Be \label{est:theorem_1} 
 	\| f(t) \|_{L^1_{x,v}} 
 	\lesssim  e^{- \Lambda t}
 	\{ \| e^{\theta^\prime (|v|^2+ 2\Phi (x))} f_0\|_{L^\infty_{x,v}}+ \| e^{\delta (|v|^2 + 2 \Phi(x))^{1/2} } f_0 \|_{L^1_{x,v}}
 	\},
 	\Ee
 	where 
 	
 	.
 	There exists a unique solution $F (t,x,v)= f (t,x,v) + \mathfrak{M} F_s (x, v)$ 
 	for \eqref{equation for F mag} with \eqref{field property} and \eqref{diff_F} 
 	where $F_s (x, v)$ solves \eqref{equation for F_infty} and \eqref{diff_F} with 
 	$\iint_{\O \times \R^3} F_s (x, v) \dd x \dd v = 1$ and $\iint_{\O \times \R^3} f_0 (x,v) \dd x \dd v = 0$.
 	In addition, if $\| e^{\theta^\prime (|v|^2+ 2\Phi (x))} f_0\|_{L^\infty_{x,v}}<\infty$ for $0<\theta^\prime< \frac{1}{2b}$ with $b$ defined in \eqref{def:Theta}, and $\| e^{\delta (|v|^2 + 2 \Phi(x))^{1/2} } f_0 \|_{L^1_{x,v}} < \infty$ with $0 < \delta < \infty$, then we have
 \unhide
 \end{theorem}
 
 \hide
 \begin{theorem} \label{theorem}
 	Consider the initial condition $F(t,x, v) |_{t = 0}  = F_0 (x, v)$.
 	There exists a unique solution $F (t,x,v)= f (t,x,v) + \mathfrak{M} F_s (x, v)$ for  \eqref{equation for F mag} and \eqref{diff_F} 
 	where $F_s (x, v)$ is the solution for \eqref{def:f} and \eqref{diff_F} with 
 	$\iint_{\O \times \R^3} F_s (x, v) \dd x \dd v = 1$, and
 	$\iint_{\O \times \R^3} f_0 (x,v) \dd x \dd v = 0$. In addition, if $\| e^{\theta^\prime (|v|^2+ 2\Phi (x))} f_0\|_{L^\infty_{x,v}}<\infty$ for $0<\theta^\prime< \frac{1}{2b}$ with $b$ defined in \eqref{def:Theta}, and $\| e^{\delta (|v|^2 + 2 \Phi(x))^{1/2} } f_0 \|_{L^1_{x,v}} < \infty$ with $0 < \delta < \infty$, then we have
 	\Be \label{theorem_infty_1}
 	\sup_{t\geq0}\| e^{\theta^\prime (|v|^2+ 2\Phi (x))} f (t)\|_{L^\infty_{x,v}} \lesssim \| e^{\theta^\prime (|v|^2+ 2\Phi (x))} f_0\|_{L^\infty_{x,v}}.
 	\Ee
 	Moreover, for all $t \geq 0$ and $0 \leq \theta< \theta^\prime$,
 	\Be \label{theorem_infty}
 	\begin{split}
 		\sup_{x \in \bar{\O}}\int_{\R^3} e^{\theta  (|v|^2+ 2\Phi (x))} |f(t,x,v) |\dd v  
 		\lesssim_{\theta}   e^{-\Lambda t},
 	\end{split} 
 	\Ee
 	where $\Lambda$ is the same as in Theorem \ref{theorem_1}.
 \end{theorem}

\unhide

\hide

We emphasis that there is no simple formula for steady state in the non-trivial field and non-isothermal diffusive boundary case, compared to no field case, pure gravitation case and symmetric domains case where they have the explicit formula in the form of Maxwellian.
Indeed, it is not a global Maxwellian or a local Maxwellian because it would not satisfy the equation or the boundary condition.
In contrast to , we do not need the symmetric assumption of the domains or the constant field.
We also remark that the priori $L^{\infty}$ estimate is the first step to understand the existence of the steady state.

\unhide


\hide
We first consider a free transport equation in a periodic domain $\O \subset \R^3$, 
with an initial condition 
$F(t, x, v) |_{t = 0}  = F_0 (x, v)$. 
$F(t, x, v)$ solves \eqref{equation for F mag}
and \eqref{diff_F} with

We denote the phase boundary 
$\gamma := \{ (x, v) \in \p \O \times \R^3\}$ in the phase space 
$\O \times \R^3$, and decompose it into the outgoing boundary $\gamma_{+}$, and the incoming boundary $\gamma_{-}$:
\Be \label{gamma_pm}
\begin{split}
& \gamma_{+} := \{ (x, v) \in \partial \Omega \times \R^3: n(x) \cdot v > 0 \},
\\& \gamma_{-} := \{ (x, v) \in \partial \Omega \times \R^3: n(x) \cdot v < 0 \}.
\end{split}
\Ee
Moreover, we use the symbol $\mu$ to denote the global Maxwellian at unit temperature:
\be \label{isothermal maxwellian}
\mu (x, v) := \frac{1}{2 \pi}e^{-  \frac{|v|^2}{2}},
\ee
for $(x, v) \in \p\O \times \R^3$. Following \eqref{mu x normal} with $\Theta (x) \equiv 1$, we have 
\be
\int_{n(x) \cdot v > 0} \mu (x, v) \{n(x) \cdot v\} \dd v = 1.
\ee

After achieving the steady state $F (x, v)$ solving 

with $\Phi (x)$ defined in \eqref{field property}, and \textit{diffusive reflection} boundary condition \eqref{diff_F}.
We are mainly interested in the asymptotic behavior of the exponential moments of the fluctuation.

\unhide

\medskip 

\noindent \textbf{Difficulties and Ideas. }The main/basic idea of to construct stationary solutions and prove its stability is our novel $L^1-L^\infty$ bootstrap, which allow us to transfer a velocity mixing of the diffusive reflection boundary condition to a spatial mixing through the characteristics of \textit{the stochastic cycles}:

\begin{definition}[Stochastic Cycles] \label{def_cycles}
	Define the backward exit time $t_{\mathbf{b}}$ and the forward exit time $\tf$, 
	\Be \label{def_tb}
	\begin{split} 
		& t_{\mathbf{b}}(x, v) := \sup \{s \geq 0: X(t - \tau; t, x, v) \in \Omega, \  \forall\tau \in [0, s) \}, \
		x_{\mathbf{b}}(x, v) := X(t - t_{\mathbf{b}}(x, v); t, x, v), 
		\\& t_{\mathbf{f}}(x, v) := \sup \{s \geq 0: X(t + \tau; t, x, v) \in \Omega, \  \forall\tau \in [0, s) \}, \
		x_{\mathbf{f}}(x, v) := X(t + t_{\mathbf{f}}(x, v); t, x, v).
	\end{split} 
	\Ee
	We define the stochastic cycles: 
	\be \notag
	t^1 (t, x, v) = t - t_{\mathbf{b}}(x, v),
	\ x^1 (x, v) = x_{\mathbf{b}}(x, v) = X(t^1, t, x, v), \ \vb(x, v) =  V(t^1, t, x, v),
	\ee
	\Be \label{def:t_k}
	\begin{split}
		& t^k (t, x, v, v^1,..., v^{k-1}) = t^{k-1}
		- t_{\mathbf{b}}(x^{k-1}, v^{k-1}), \ 
		\tb^k = \tf^{k+1} = t^k - t^{k+1},
		\\&  x^k (t, x, v, v^1,..., v^{k-1}) = X(t^{k}; t^{k-1}, x^{k-1}, v^{k-1}), \ 
		\vb^k= V(t^{k+1}; t^k, x^k, v^k),
	\end{split}
	\Ee
	where we define $v^j \in \mathcal{V}_j := \{v^j \in \R^3: n(x^j) \cdot v^j > 0 \}$ with the measure 
	$\dd \sigma_j = \dd \sigma_j (x^j)$ on $\mathcal{V}_j$ which is given by
	\Be \label{def:sigma measure}
	\dd \sigma_j := \mu_{\Theta} (x^{j+1}, \vb^{j}) \{ n(x^j) \cdot v^j \} \dd v^j.
	\Ee
	Here, $n(x)$ is the outward normal at $x \in \p\O$.
\end{definition}

We encounter a difficulty in the presence of magnetic field as the change of variable $v \mapsto (\tb, S_{\xb})$ becomes degenerate: 
\Be 
\dd v = 5 B^2_3 (2 - 2 \cos (B_3 \tbi))^{-1} \dd \tb \dd S_{\xb}.
\Ee
Note that the denominator $2 - 2 \cos (B_3 \tbi)$ can be very small when $B_3 \tbi$ approaches $2k \pi$ with $k \in \Z_{+}$.
Thus, we split velocity $v$ into the following two parts:
\Be 
\begin{split} 
	& \mathcal{V}^{G_{\mathbf{\e}}} (x) := \{v \in  \mathcal{V}: \frac{\e}{B_3} + \frac{2 k \pi}{B_3}
	\leq \tb (x, v) \leq 
	\frac{2 \pi - \e}{B_3} + \frac{2 k \pi}{B_3} \ \text{with} \ k \in \Z_{+} \},  
	\\& \mathcal{V}^{B_{\mathbf{\e}}} (x) :=  \{v \in  \mathcal{V}: - \frac{\e}{B_3} + \frac{2 k \pi}{B_3} \leq \tb (x, v) \leq 
	\frac{\e}{B_3} + \frac{2 k \pi}{B_3} \ \text{with} \ k \in \Z_{+} \}.
\end{split} 
\Ee
Although the Jacobian is large when $v \in \mathcal{V}^{B_{\mathbf{\e}}}$, we can control the integration of such part small inside the stochastic cycles (see Lemma \ref{lem: Be mag}). 
This is because when $v \in \mathcal{V}^{B_{\mathbf{\e}}}$, the size of range of $\tb (x, v)$ in each period $[\frac{2 k \pi}{B_3}, \frac{2 (k+1) \pi}{B_3}]$ with $k \in \Z_{+}$ is $\frac{2 \e}{B_3}$. 
On the other hand, we deduce the upped bound when $v \in \mathcal{V}^{G_{\mathbf{\e}}}$ since the denominator in Jacobian is away from zero (see Lemma \ref{lem:bound1 mag} and Proposition \ref{prop:j linfty bound}). We crucially use this $L^1-L^\infty$ bootstrap to construct stationary solutions based on a pointwise estimate boundary outgoing flux for $x \in \p\O$,
\be
\mathfrak{J}(x) = \int_{n(x)\cdot v >0} F_s(x, v) \{ n(x) \cdot v \} \dd v.
\ee

\hide

We start to consider the dynamical problem. 
Suppose $F (t, x, v)$ solves \eqref{equation for F mag}, \eqref{diff_F} with \eqref{field property} and the initial condition 
$F(t,x, v) |_{t = 0}  = F_0 (x, v)$.
Since the choice of $\mu_{\Theta} (x, v)$ formally guarantees the conservation of mass, this lead us to set the total mass of the initial datum for some $\mathfrak{M} \geq0$ and $F_s (x, v)$ is the steady state (see Remark \ref{positivity on F_infty}), so that
\Be\label{def:f}
f(t, x, v)=
F(t, x, v) - \mathfrak{M} F_s (x, v),
\Ee 
which solves
\begin{align} 
\partial_t f + v \cdot \nabla_{x} f + (v \times B - \nabla \phi - g \mathbf{e}_3)\cdot \nabla_{v} f = 0,& \ \  \text{for} \  \    (t, x, v) \in \R_{+} \times \Omega \times \R^3,  \label{eqtn_f} 
\\
 f (t,x,v) |_{t = 0}  = F_0 (x, v) -   F_s (x, v) = f_0 (x,v),&  \ \  \text{for} \   \   (x, v) \in \Omega \times \R^3,
\label{init_f}
 \\
 f (t, x, v)  = \mu_{\Theta} (x, v) \int_{n(x^1) \cdot v^1>0} f(t, x, v^1) \{ n(x^1) \cdot v^1 \} \dd v^1 ,& \ \  \text{for} \  \   (t,x, v) \in \R_{+} \times \gamma_-. \label{diff_f} 
\end{align}

We stress that the control of moments is a key step toward nonlinear problems such as the Vlasov-Poisson systems. In this paper, we contribute toward establishing a decay of \textit{exponential moments} of the fluctuation in $L^\infty_x$ with the exponential decaying rate $e^{-\Lambda t}$ when the domain is in 3D. 

\subsection{With magnet field}

Here we consider another \textit{free transport} equation in a periodic domain $\O \subset \R^3$, 
with an initial condition 
$F(t, x, v) |_{t = 0}  = F_0 (x, v)$. 
$F(t, x, v)$ solves \eqref{equation for F mag}
and \eqref{diff_F} with

where $g$ is the gravitational constant.

Similar as in the gravitational case, we start with the priori $L^{\infty}$ estimate on the steady state. 
Next, we prove the existence of the steady state $F_{\infty} (x, v)$ which solves
\Be \label{def:f}
v \cdot \nabla_{x} F_{\infty} + (v \times B - \nabla \Phi ) \cdot \nabla_{v} F_{\infty} = 0, \ \  \text{for} \  \    (t, x, v) \in \R_{+} \times \Omega \times \R^3,  
\Ee
with $\Phi (x)$ defined in \eqref{field property mag}, and \textit{diffusive reflection} boundary condition \eqref{diff_F}.

Then it is natural for us to study the asymptotic behaviour of the \textit{fluctuation}
\Be \label{def:f}
f(t, x, v)=
F(t, x, v) - F_{\infty} (x, v), \ \text{with} \ \iint_{\O \times \R^3} f_0 (x,v) \dd x \dd v = 0,
\Ee 
which solves
\begin{align} 
	\partial_t f + v \cdot \nabla_{x} f + (v \times B - \nabla \Phi ) \cdot \nabla_{v} f = 0, \ \  \text{for} \  \    (t, x, v) \in \R_{+} \times \Omega \times \R^3, \label{eqtn_f mag} 
\end{align}  
and satisfies the initial condition \eqref{init_f} and the diffuse boundary condition \eqref{diff_f}.

Finally, we establish the exponential decay of \textit{exponential moments} of the fluctuation in $L^\infty_x$. 


\bigskip

\textbf{Notations.} We shall clarify some notations: 
$A \lesssim B$ if $A\leq C B$ for a constant $C>0$ which is independent on $A,B$;
$A \gtrsim B$ if $A \geq C B$ for a constant $C>0$ which is independent on $A,B$;  
$\| \cdot \|_{L^1_{x,v}}$ or $\| \cdot \|_1$ for the norm of $L^1(\O \times \R^3)$; $\| \cdot \|_{L^\infty_{x,v}}$ or $\| \cdot \|_\infty$  for the norm of $L^\infty(\bar{\O} \times \R^3)$; 
$|g|_{L^1_{\gamma_\pm}}$ denotes $\int_{\gamma_\pm} |g(x,v) | |n(x) \cdot v| \dd S_x \dd v$;  
$| a |$ denotes $\sqrt{a^2_1 + a^2_2 + a^2_3}$ when $a = (a_1, a_2, a_3) \in \R^3$;
$\langle \cdot \rangle$ denotes $1 + |\cdot|$.

\bigskip



%
%
%




As the first part of $L^1-L^{\infty}$ framework, 
we first assume the steady state $F (x, v)$ is $L^1$ bounded, and define the 

The primary tool to obtain $L^\infty$-Estimates on $\J (x)$ is applying the stochastic cycles $k$ times with $t_* = 0$ (see \eqref{expand_G11}-\eqref{expand_G2} for details).
We first pick sufficiently large time $t$ to guarantee the smallness of cases when $t^{1} < 0$ and $t^{2} < 0 \leq t^{1}$.

Next, for $i = 2, \cdots , k-1$ and $t^{i+1} < 0 \leq t^{i }$, we establish the following control:
\Be \notag
  \Big|\int_{\prod_{j=0}^{i} \mathcal{V}_j}   
  \mathbf{1}_{t^{i+1} < 0 \leq t^{i}} F(X(0; t^i, x^i, v^i), V(0; t^i, x^i, v^i)) 
  \dd \tilde{\Sigma}_{i} \Big|
  \lesssim 2^{i} \times \| F(x, v) \|_{L^1_{x,v}}.
\Ee
where $\dd \tilde{\Sigma}_{i} := \frac{ \dd \sigma_{i}}{\mu_{\Theta} (x^{i+1}, v^{i})} \dd \sigma_{i-1} \cdots \dd \sigma_1 \dd \sigma_0$ with $\dd \sigma_j = \mu_{\Theta} (x^{j+1}, \vb^{j}) \{ n(x^j) \cdot v^j \} \dd v^j$ in \eqref{def:sigma measure}.
Here, we will choose and fix large enough iteration umber $i$.

For $t^{k} \geq 0$, we derive a sharper bound from Stirling's formula, such that for $k \gtrsim t$ 
\Be \notag
\sup_{(x,v) \in \bar{\O} \times \R^3}  \Big(\int_{\prod_{j=0}^{k -1} \mathcal{V}_j}   
    \mathbf{1}_{t^{k} (t,x,v,v^1,\cdots, v^{k-1}) \geq 0 } \ \dd \sigma_0 \cdots \dd \sigma_{k-1}\Big) \lesssim e^{-t}.
\Ee  
Collecting all parts together, we obtain the $L^{\infty}$ bound on $\J (x)$ in Proposition \ref{prop:j linfty bound}.


For the magnet field case, we also assume the $L^1$ boundedness on the steady state $F (x, v)$. 
Then we apply the stochastic cycles on the boundary outgoing flux, $\mathfrak{J}(x) = \int_{n(x)\cdot v >0} F(x, v) \{ n(x) \cdot v \} \dd v$ for $x \in \p\O$.


Next we show the existence of the steady state for \eqref{equation for F mag} and \eqref{diff_F}.
To construct the steady state with non-zero mass, we write $F (x, v)$ as
\be
F (x, v) = \mu (x, v) + f (x, v), \ \text{for} \ (x, v) \in \p\O \times \R^3,
\ee
where $\iint_{\O \times \R^3} f (x, v) \dd x \dd v = 0$.

Now it suffices for us to show the existence of the fluctuation $f (x, v)$.
Then we use a penalization sequence to keep the mass zero condition under the following iterative scheme:
\Be 
\begin{split}
& \epsilon f^j + v \cdot \nabla_{x} f^j + (v \times B - \nabla \Phi ) \cdot \nabla_{v} f^j = 0, 
\\& f^j (x, v)  
= (1 - \frac{1}{j}) \mu_{\Theta} (x, v) \int_{n(x) \cdot v^1>0} f^j (x, v^1) \{ n(x^1) \cdot v^1 \} \dd v^1 + \mu_{\Theta} (x, v) - \mu (x, v),
\end{split}
\Ee  
with initial setting $f^0_{j} = 0$ and $B, \Phi (x)$ are defined in \eqref{field property}.

We first consider $j \rightarrow \infty$ for $\{ f^j \}$ with fixed $\epsilon > 0$.
Since the penalization $\e$ keeps the $L^1$ bound, this allows us to use $L^1-L^{\infty}$ framework and  obtain the existence of the limit $\fe$ in Lemma \ref{lem: existence of f^e}. 
Next, we obtain the uniform $L^1$ bound on $\{\fe\}$. Applying $L^1-L^{\infty}$ framework, we derive the uniform $L^{\infty}$ bound on $\{\fe\}$. Thus we achieve a weak $*$ limit such that $\fe \stackrel{\ast}{\rightharpoonup} f \in L^{\infty}$.
Consequently, we show the existence and uniqueness of $f (x, v)$, and prove Theorem \ref{existence on F}. 

For the magnet field case, we also start with the double iteration in both $j$ and $l$ with $B$, $\Phi (x)$ defined in \eqref{field property mag}. 
Note that
$\nabla_{v} \cdot (v \times B - \nabla \Phi (x)) = 0$, thus magnet field doesn't have any influence when we take the integration over $\O \times \R^3$.
The rest of steps follows from the gravitational case.



\unhide

The proof of dynamical stability is based on 1) the Doeblin-type pointwise lower bound with a unreachable defect (Proposition \ref{prop:Doeblin 1})
\Be 
f(NT_0,x,v)\geq  
\mathfrak{m}(x,v) \Big\{
\iint_{\O \times \R^3}f((N-1)T_0,x,v) \dd v \dd x 
-  \iint_{\O \times \R^3} \mathbf{1}_{t_\mathbf{f}(x,v)\geq \frac{T_0}{4}} f((N-1)T_0,x,v)  \dd v \dd x 
\Big\};
\Ee
2) weight energy estimate to control the unreachable defect (Theorem \ref{theorem_1}). In contrast to no-gravity case of \cite{JinKim1}, we are able to employ an exponential weight $\varphi (\tau )= e^{\delta \tau  }$
 \Be  
 \vertiii{f} :=   \|f  \|_{L^1_{x,v}}
 +   \Big(1+ \frac{1}{\delta T_0}\Big) \frac{ 4 \mathfrak{m}_{T_0} }{ \varphi (\frac{T_0}{4})}
 \|  \varphi (\tf) f\|_{L^1_{x,v}}  ;
 \Ee
 3) prove a decay of an exponential moments in $L^\infty$ using $L^1-L^\infty$ bootstrap. 
 
Regarding the first component, we carefully split the characteristics in the presence of magnetic field as the crucial mixing by diffuse reflection boundary operator is degenerate. About the second ingredient, a key observation is that gravity control the vertical trajectory so that $\tb(x,v) , \tf(x,v) \leq |v_3|$, which allows $\varphi$ of an exponential growth can make 
 \be 
 \int_{n(x) \cdot v<0} \varphi(\tf) \mu_{\Theta} (x, v) |n(x) \cdot v| \dd v<\infty.
 \ee 
 We prove an exponential decay in $L^1$ using an energy estimate with an exponential weight.

 \hide

h, and we obtain a weighted energy estimate with exponential $\varphi$ under dominant gravity.
Furthermore, since $\tf (x, v) \leq \tf (\xb, \vb) \lesssim |n(\xb) \cdot \vb|$ and $|v|^2 + 2 \Phi(x) = |\vb|^2$ from Lemma \ref{conservative field}, we derive 
$| \varphi (\tf) |
\lesssim \big| \varphi \big( (|v|^2 + 2 \Phi(x))^{1/2} \big) \big|$.
Therefore, we prove the $L^1$ exponential decay in Theorem \ref{theorem_1}.

As the second part of $L^1$-$L^\infty$ framework, we continue to prove an $L^1$-decay of the fluctuation $f$ as $t\rightarrow \infty$.
We prove a key lower bound with a unreachable defect in :

Next we control the unreachable defect using the weighted $L^1$-estimates. 
From the equality of $\tf$ under Vlasov operator, we find that a weight $\varphi(\tf)$ can provide an effective dissipation, for $\varphi^\prime \geq 0$,
\be \notag
[\p_t + v\cdot \nabla_x - \nabla \Phi \cdot \nabla_{v}] \big(\varphi(\tf) |f| \big) 
= -\varphi^\prime(\tf)| f|.
\ee
Employing a function $\varphi$ with $\varphi^\prime \rightarrow  \infty$ as $\tau \rightarrow \infty$, we establish the following energies:

with $\| \mathfrak{m} \|_{L^1_{x,v}} := \mathfrak{m}_{T_0}$ (see \eqref{est:m}) and $\varphi$ in \eqref{varphis}.

In order to control these energies, we require the estimate on weight integration $\int_{\gamma_+} \varphi(\tf) |f|$. 
Apply the diffusive boundary condition, it suffices to bound

Under a dominant gravitational field, the forward and backward exit time for molecules on the boundary
mostly depends on the velocity on $x_3$-axis.
That said, for $(x, v) \in \gamma_- = \{ (x, v) \in \partial \Omega \times \R^3: n(x) \cdot v < 0 \}$, 
\be \notag
\tf (x, v) \lesssim |n(x) \cdot v| = v_3.
\ee 
For $(x, v) \in \gamma_+ = \{ (x, v) \in \partial \Omega \times \R^3: n(x) \cdot v > 0 \}$, 
\be \notag
\tb (x, v) \lesssim |n(x) \cdot v| = v_3.
\ee 
Here we emphasis that without gravity, $\tf$ can be large when $|v| \ll 1$.

From $\mu_{\Theta} (x, v) = \frac{1}{2 \pi \Theta^2(x)}e^{-  \frac{|v|^2}{2 \Theta(x)}}$ and $|n(x) \cdot v| \leq |v|$, we obtain 
\be \notag
\mu_{\Theta} (x, v) |n(x) \cdot v| \lesssim e^{-  \frac{|v_3|^2}{4 \Theta(x)}}.
\ee


For the magnet field case, we still use the stochastic cycles to show the $L^1$-decay on the fluctuation $f$.
Considering the stochastic cycle $(x, v) \in \gamma_{+}$, $(x^1, \vb) \in \gamma_{-}$, we have
\Be \notag
x^1 + \int^{t}_{t - \tb}  V(s; t, x, v) \dd s = x.
\Ee
From the characteristics trajectories under magnet field, we can compute
\be \notag
\frac{\dd}{\dd s} V_1(s; t, x, v) = B_3 V_2 (s; t, x, v), \ \ \frac{\dd}{\dd s} V_2(s; t, x, v) = - B_3 V_1 (s; t, x, v).
\ee
Then we deduce that
\Be \notag
|v| = |\vb|, \ |v_1|^2 + |v_2|^2
= \frac{B^2_3 |x - x_1 |^2 }{2 - 2 \cos (B_3 \tb)}.
\Ee
Thus we can bound $|\vb|$ by the distance between $x$ and $x_1$ if $B_3 \tb$ is away from $2k \pi$ with $k \in \Z_{+}$.
Since $f (t, x, v) = f (\tb, \xb, \vb) =
\mu_{\Theta} (\xb, \vb) \J (\xb)$, this enables us to find a lower bound on $f (t, x, v)$ when $t > \tb$.

Picking a time gap $T_0$, we compare $f(NT_0)$ with $f((N-1)T_0)$ by the following stochastic cycle:
\be 
f(NT_0, x,v) 
 \geq \mathbf{1}_{ \tb(x,v) \leq \frac{T_0}{4}} \mu_{\Theta} (x^{1}, \vb) \int_{\mathcal{V}_1} 
\int_{\mathcal{V}_2}  
\int_{\mathcal{V}_3} 
\mathbf{1}_{t^3 \geq (N-1)T_0}
 f(t^3, x^3, v^3)  \{n(x^3) \cdot v^3\} \dd v^3 \dd \sigma_2 \dd \sigma_1.
\ee 
 We restrict $v^1 \in \mathcal{V}_1$ and $v^2 \in \mathcal{V}_2$ by controlling the distance between the pair $(x^1, x^2)$ and $(x^2, x^3)$ and the duration on $\tb^1 + \tb^2$.
Moreover, we set $\tb^2$ as 
$\frac{\pi}{2B_3} \leq \tb^2 \leq \frac{3\pi}{2B_3}$ to obtain the pointwise lower bound on $|v^2|$.
Applying the change of variable $(\tb^1, \tb^2) \mapsto (\tb^1, \tb^1 + \tb^2)$, we derive the lower bound with a unreachable defect in Proposition \ref{prop:Doeblin mag}. 
At last, we use the weight function $\varphi$ to obtain the uniform energy estimate, and prove $L^1$ exponential decay in Theorem \ref{theorem_1}.\unhide



Finally, a major difficulty arises at the last ingredient when we control the $L^\infty$-norm of the fluctuation $f (t, x, v)$ via the stochastic cycles (see \eqref{expand_h1}-\eqref{expand_h3} for details). It turns out crucial to control  so-called \textit{$i$-residual measure} uniformly-in-$i$:
\be 
\int_{\prod_{j=0}^{i-1} \mathcal{V}_j}   
\mathbf{1}_{0 < t_{i } \leq t} \
\dd \sigma_{i-1} \cdots \dd \sigma_1 \dd \sigma_0.
\ee
We note that $\dd \sigma_j = \mu_{\Theta} (x^{j+1}, \vb^{j}) \{ n(x^j) \cdot v^j \} \dd v^j$ is not a probability measure in general. Therefore, we could experience an exponential growth of $i$-residual measure where $i$ is proportional to the time $t$. We overcome this difficulty by constructing a continuous outgoing boundary flux $\J (x) $ and bound the $i$-residual measure using an upper/lower bound of $\J (x) $ (see the strategy described at the beginning of Section \ref{sec: continuity}). The downward gravity is crucial to prove $\J(x)$ continuous. The continuity of $\J (x)$ closely relates to continuity/regularity of solutions to the transport equation (\cite{Kim, ChenKim, EGKM2, KL2}).

\hide

However, it is hard to directly control the following terms in stochastic cycle representation:
\Be \label{intro: expand_h2}
\mu_{\Theta} (x^1, \vb) 
\sum\limits^{k-1}_{i=1} 
\int_{\prod_{j=1}^{i} 
\mathcal{V}_j}   
\dd \sigma_{1} \cdots \dd \sigma_{i-1}
     \Big\{   \mathbf{1}_{t^{i+1} < 0 \leq t^{i }} f (0, X(0; t^i, x^i, v^i), V(0; t^i, x^i, v^i))  \Big\}
\{ n(x^i) \cdot v^i \} \dd v^i
\Ee
where
$\dd \sigma_j = \mu_{\Theta} (x^{j+1}, \vb^{j}) \{ n(x^j) \cdot v^j \} \dd v^j$ in \eqref{def:sigma measure}.

Note that under isothermal boundary condition, 
$\int_{\mathcal{V}_i} \dd \sigma_i
= \int_{\mathcal{V}_i} \mu_{\Theta} (x^{j}, v^{j}) \{ n(x^j) \cdot v^j \} \dd v^j 
= 1$, which is a probability measure.
For non-isothermal boundary condition, although we get $|\vb^{j}| = |v^j|$ from the the characteristics trajectories, the problem comes from that the wall temperature $\Theta(x)$ is varying with the boundary, i.e.  $\mu_{\Theta} (x^{j+1}, \vb^{j}) \neq \mu_{\Theta} (x^{j}, \vb^{j})$ when $x^{j+1} \neq x^j$.
This leads to $\dd \sigma_j$ not being a probability measure on $\mathcal{V}_j$.
From $a \leq \Theta(x) \leq b$, we at most deduce that, for $i =1, \cdots, k-1$, 
\Be \label{intro: sigma measure}
\int_{\mathcal{V}_i} \dd \sigma_i
= \int_{\mathcal{V}_i} \mu_{\Theta} (x^{j+1}, v^{j}) \{ n(x^j) \cdot v^j \} \dd v^j \leq \int_{\mathcal{V}_i} \frac{1}{2 \pi a^2}e^{-  \frac{|v|^2}{2 b}} \{ n(x^j) \cdot v^j \} \dd v^j \leq \frac{b^2}{a^2}.
\ee
Different from $L^\infty$-Estimates on steady state $F (x, v)$, we cannot choose and fix $t$ or $k$.
Otherwise, inputting \eqref{intro: sigma measure} in 
\eqref{intro: expand_h2} will give an exponential growth instead of bounded estimate when $t \rightarrow \infty$.

Therefore, we turn to \unhide

\hide
The rest of steps are similar as in the first case.
We can obtain a positive infimum of boundary outgoing flux $\J (x)$ for a non-negative steady state with unit mass from the continuity.
Then we consider the residual measure as the part of the stochastic cycle on $\J (x)$, and establish the uniform estimate.



Lastly, we bootstrap the $L^1$-decay 
to bound the exponential moments.
The most crucial tool is the stochastic cycle representation (see \eqref{expand_k1}-\eqref{expand_k4} for details) on weighted function $\varrho (t) w^\prime (x, v) f(t, x, v)$ where 
$\varrho (t)$ is only time dependent
and $w^\prime (x, v) := e^{\theta^\prime (|v|^2+ 2\Phi (x))}$ with $0< \theta^\prime < \frac{1}{2b}$ and $b$ defined in \eqref{def:Theta}.
Note that the expansion contains some different terms from stochastic cycles since $\varrho^\prime \neq 0$.

First, we set $\varrho (t) = t + 1$ with $\varrho^\prime = 1$. Then we use estimates on residual measures and Lemma \ref{lem:bound1} to bound
\Be \notag
\int_{\prod_{j=0}^{i} \mathcal{V}_j}      
      \mathbf{1}_{t^{i+1} < 0 \leq t^{i }}
 \int^{t^{i}}_{0} \varrho^\prime(s) f(s, X(s; t^i, x^i, v^i), V(s; t^i, x^i, v^i)) \dd s 
 \dd \Sigma_{i}
\lesssim  \int^t_0 \| \varrho^\prime (s) f(s) \|_{L^1_{x,v}} \dd s.
\Ee
On the other hand, we apply the sharper bound from Stirling's formula to control
\Be \notag
 \sup_{(x,v) \in \bar{\O} \times \R^3}  \Big(\int_{\prod_{j=1}^{k -1} \mathcal{V}_j}   
    \mathbf{1}_{t_{k }(t,x,v,v^1,\cdots, v^{k-1}) \geq 0 }
\dd \sigma_1 \cdots \dd \sigma_{k-1}\Big)
\lesssim e^{-t}.
\Ee
Collecting all parts in stochastic cycles, we deduce
\be 
\sup_{t\geq0}\| e^{\theta^\prime (|v|^2+ 2\Phi (x))} f (t)\|_{L^\infty_{x,v}} \lesssim \| e^{\theta^\prime (|v|^2+ 2\Phi (x))} f_0\|_{L^\infty_{x,v}}.
\ee

Second, we introduce $w (x, v) := e^{\theta (|v|^2+ 2\Phi (x))}$ with $0< \theta < \theta^\prime$.
Applying the stochastic cycles twice, 
we establish the time integration terms of $\int_{\R^3} w(x, v) |f(t,x,v)| \dd v$  (see \eqref{v<t/2 term in wf}-\eqref{f_exp2} for details). And it suffices to prove the decay of 
\be
\int_{\mathcal{V}_2} 
\mathbf{1}_{t^2 \geq t/2} f(t^2,x^2,v^2) \{ n(x^{2}) \cdot v^{2} \} \dd v^{2}.
\ee

Now we pick $\varrho (t) = e^{\Lambda t}$ 
where $\Lambda$ defined in \eqref{def:M}, define a new weighted function $g (t, x, v) := \varrho (t) w (x, v) f (t, x, v)$, and note that
\Be 
\frac{1}{\varrho (t^2)} \int_{\mathcal{V}_2}  \frac{|n(x^2) \cdot v^2|}{w (x^2, v^2)} g (t^2, x^2, v^2) \dd v^2 = \int_{\mathcal{V}_2} f(t^2,x^2,v^2) \{ n(x^2) \cdot v^2 \} \dd v^2.
\Ee 
Next we apply the stochastic cycles on $g (t^2, x^2, v^2)$.
From $\| f(t) \|_{L^1_{x,v}} 
\lesssim  e^{- \Lambda t}$ in Theorem \ref{theorem_1}, we obtain
\be \notag
\int^t_0 \| \varrho^\prime (s) f(s) \|_{L^1_{x,v}} \dd s \lesssim \int^t_0 \varrho^\prime (s) e^{- \Lambda s} \dd s \lesssim t.
\ee
Collecting all estimates on stochastic cycles, we derive 
\Be 
\frac{1}{\varrho (t^2)} \int_{\mathcal{V}_2} \mathbf{1}_{t^2 \geq t/2} \frac{|n(x^2) \cdot v^2|}{w (x^2, v^2)} g (t^2, x^2, v^2) \dd v^2 
 \lesssim \frac{\langle t\rangle^2}{\varrho (t)},
\Ee
and prove Theorem \ref{theorem}.


For the magnet field case, the crucial tool to show the $L^\infty$-estimate of Moments is still the stochastic cycle on a weighted function $\varrho (t) w^\prime (x, v) f(t, x, v) $. 

Since the time dependent weight $\varrho (t)$ is involved, we have the following term in stochastic cycles:
\Be \notag
\int_{\prod_{j=0}^{i} \mathcal{V}_j} \mathbf{1}_{t^{i+1} < 0 \leq t^{i }}
 \int^{t^{i}}_{0} \varrho^\prime(s) f(s, X(s; t^i, x^i, v^i), V(s; t^i, x^i, v^i)) \dd s 
 \dd \Sigma_{i}.
\Ee
Similar as in $L^\infty$-estimate, we need to overcome the Jacobian from the change of variable $v \mapsto (\tb, S_{\xb})$.
Thus, we split $v$ into $v \in \mathcal{V}^{B_{\mathbf{\e}}}$ and $v \in \mathcal{V}^{G_{\mathbf{\e}}}$.
Using estimates on residual measures, we derive its upper bound (see Lemma \ref{lem:bound1_2 mag}). 

The rest of steps are similar as in the first case.
We first choose $\varrho (t) = t^6 + 1$ and apply the stochastic cycles to obtain the $L^\infty$-estimate on $ w^\prime (x, v) f (t,x,v)$. 
Then we set $\varrho(t) = e^{\Lambda t}$, and consider the time integration terms of 
$\int_{\R^3} w(x, v) |f(t,x,v)| \dd v$.
After applying the stochastic cycles on $\varrho (t) w (x, v) f (t, x, v)$, we conclude Theorem \ref{theorem} of the case \eqref{field property mag}. \unhide

%
\hide
\subsection{plan of the paper}

In the rest of the paper, we collect some basic preliminaries in Section \ref{sec: background}; 
secondly we study priori $L^\infty$-Estimates on $\J (x)$, and prove Proposition \ref{prop:j linfty bound} and \ref{prop:j linfty bound} in Section \ref{sec: Linf estimate};
next we show the existence of the steady state, and prove Proposition \ref{existence on f}, Theorem \ref{existence on F} and \ref{existence on F mag} in Section \ref{sec: steady state};
after that we work on the weighted $L^1$-estimates,
and prove Theorem \ref{theorem_1} and \ref{theorem_1} in Section \ref{sec: L1 estimate};
then we show the continuity on $\J (x)$, uniform $L^1$ bound of residual measure, and prove Proposition \ref{prop: est:measure} in Section \ref{sec: continuity};
finally we prove key results Theorem \ref{theorem} on the $L^\infty$-estimate of moments in Section \ref{sec: exponential moments}. 



\section{Main result}
\label{sec: result}

Our main results are as follows.

An important consequence of the priori $L^{\infty}$ estimates in Proposition \ref{prop:j linfty bound} and \ref{prop:j linfty bound} are the Theorems below which specializes the existence of a steady state solution.

\hide
\begin{theorem} \label{existence on F mag}
Given $\mathfrak{m} > 0$, there exists a unique solution $F (x, v)$ to the steady problem 
\eqref{def:f} and \eqref{diff_F}
such that
\Be \label{initial mass on F mag}
\iint_{\O \times \R^3} F (x, v) \dd x \dd v = \mathfrak{m}.
\Ee
\end{theorem}
\unhide

We have the following dynamical stability results:

The $L^1$-decay result in Theorem \ref{theorem_1} and Theorem \ref{theorem_1} also enable us to establish an exponential decay on exponential moments of the fluctuation in $L^{\infty}_x$.
\unhide

\subsubsection*{Acknowledgements}
This project is supported in part by National Science Foundation under Grant No. 1900923 and 2047681 (NSF-Career), and the Wisconsin Alumni Research Foundation. CK is supported partly by the Brain Pool program (NRF-2021H1D3A2A01039047) of the Ministry of Science and ICT in Korea. CK thanks Professor Seung Yeal Ha for the kind hospitality during his stay at the Seoul National University. In particular, CK thanks Professor Andrej Zlato\u{s} for his insightful comments about \cite{JinKim1} in CK's talk on the occasion of the Virtual Analysis and PDE Seminar (VAPS) hosted by Professor Hung Tran in May 2021. 

\tableofcontents


\section{Background}
\subsection{Characteristics and basic properties} 

For the zero-magnetic field case of  \eqref{field property}
, the characteristics of \eqref{equation for F mag} are determined by the Hamilton ODEs
\Be \label{characteristics}
\begin{cases}
	\frac{\dd}{\dd s} X(s; t, x, v) = V(s; t, x, v),
	\\
	\frac{\dd}{\dd s} V(s; t, x, v) = - \nabla \Phi (X(s; t, x, v)),
\end{cases}  
\Ee
for $- \infty < s, t < \infty$ and $\Phi (x)$ defined in \eqref{field property} with $(X(t; t, x, v), V(t; t, x, v)) = (x, v)$.

For the nonzero-magnetic field case of \eqref{field property mag}, 
the characteristics of \eqref{equation for F mag} solve
\Be \label{characteristics mag}
\begin{cases}
	\frac{\dd}{\dd s} X(s; t, x, v) = V(s; t, x, v),
	\\
	\frac{\dd}{\dd s} V(s; t, x, v) = V(s; t, x, v) \times B - \nabla \Phi (X(s; t, x, v)),    
\end{cases}  
\Ee
where $B, \Phi (x)$ take a form of \eqref{field property mag} with $(X(t; t, x, v), V(t; t, x, v)) = (x, v)$.

We list some properties that both zero-magnetic \eqref{characteristics} and nonzero-magnetic field cases \eqref{characteristics mag}.

\begin{lemma}[\cite{CKL,CKL2}] \label{lem:COV}
For any $g$ and $(X,V)$ in \eqref{characteristics} (resp. \eqref{characteristics mag}), we have   
	\begin{align}
		\int_{\gamma_{+}} \int_0^{t_{-}} g(t,X(t,t+s,x,v),V(t,t+s,x,v)) |n(x) \cdot v|\dd s \dd v \dd S_x 
		= \iint_{\O \times \R^3} g(t,y,v) \dd y \dd v, 
		\label{COV} \\
		\int_{\gamma_{-}} \int_0^{t_{+}} g(t,X(t,t-s,x,v),V(t,t-s,x,v)) |n(x) \cdot v|\dd s \dd v \dd S_x 
		= \iint_{\O \times \R^3} g(t,y,v) \dd y \dd v, 
		\label{COV+} \\
		\int_{\gamma_{\pm}}
		g(t,x_{\mp}(x, v),v_{\mp}(x, v))
		|n(x) \cdot v| \dd v \dd S_x
		= \int_{\gamma_{\mp}}
		g(t,y,v)
		|n(y) \cdot v|\dd v \dd S_y. \label{COV_bdry}
	\end{align}
	Here, for the sake of simplicity, we have abused the notations temporarily: $t_-=\tb, x_-= \xb$ and $t_+= \tf, x_+= \xf$.  
\end{lemma} 

\begin{proof}
	The proof is similar to the proof for Lemma 3 in \cite{CKL}. For the completeness of the paper, we only sketch the proof. 

First we consider zero-magnetic case. To derive \eqref{COV} and \eqref{COV+}, we consider for fixed t, the map
\be \notag
(x, v) \in \O \times \R^3 
\mapsto 
(t - \tb (x, v), \xb(x, v), \vb(x, v)) \in \R \times \gamma_{-}.
\ee
We can check that this map is one-to-one, and we compute the determinate of the matrix from the change of variables. Applying the following the equality
\be \notag
\begin{split}
	& \nabla_{x, v} \xb (x, v) = \nabla_{x, v} \tb \vb (x, v) + \nabla_{x, v} X (t - \tb; t, x, v),
	\\& \nabla_{x, v} \vb (x, v) = - \nabla_{x, v} \tb \nabla \Phi (\xb) + \nabla_{x, v} V (t - \tb; t, x, v),
\end{split}
\ee
and the matrix transformation, the Liouville theorem that $\det [ \frac{\p (X, V)}{\p (x, v)} ] = 1$, we derive 
\Be \notag
\big| \det \Big[ 
\frac{\p (t - \tb, \xb, \vb)}{\p (x, v)} 
\Big] \big| 
= \frac{1}{|\vb \cdot n (\xb)|}.
\Ee
Then we show \eqref{COV+}. Similarly, we can derive \eqref{COV} if we consider the map 
\be \notag
(x, v) \in \O \times \R^3 
\mapsto 
(t + \tf (x, v), \xf(x, v), \vf(x, v)) \in \R \times \gamma_{+}.
\ee

Next, for \eqref{COV_bdry}, we consider the map
\be \notag
(t, x, v) \in \R \times \gamma_{+} 
\mapsto 
(t - \tb (x, v), \xb(x, v), \vb(x, v)) \in \R \times \gamma_{-}.
\ee

Again we can get the map is one-to-one, and  compute the determinate of the matrix from the change of variables. \hide For any $(x, v) \in \gamma_{+}$ and $t - \tb \leq s \leq t$,
\be \notag
\begin{split}
	& \big[\p_t + v \cdot \nabla_{x} - \nabla \Phi \cdot \nabla_{v} \big] X (s; t, x, v) = 0,
	\\& \big[\p_t + v \cdot \nabla_{x} - \nabla \Phi \cdot \nabla_{v} \big] V (s; t, x, v) = 0.
\end{split}
\ee \unhide
	From the matrix transformation and the fact $\tb$ is independent of $t$, we derive 
	\be \notag
	\big| \det \Big[ 
	\frac{\p (\xb, \vb)}{\p (x, v)} 
	\Big] \big|
	= \big| \det \Big[ 
	\frac{\p (t - \tb, \xb, \vb)}{\p (t, x, v)} 
	\Big] \big| 
	= \frac{|v \cdot n (x)|}{|\vb \cdot n (\xb)|}.
	\ee
	which implies \eqref{COV_bdry}. 
	
	The proof for the magnetic case can be done in a parallel way of zero-magnetic proof as the Liouville theorem shows $\det [ \frac{\p (X, V)}{\p (x, v)} ] = 1$ from $\nabla_{v} (V \times B - \nabla \Phi (X)) = 0$. 
\end{proof}

\hide
\begin{lemma}
	Consider $(X,V)$ in \eqref{characteristics mag}, for any 
	$g$, 
	\begin{align}
		\int_{\gamma_{+}} \int_0^{t_{-}} g(t,X(t,t+s,x,v),V(t,t+s,x,v)) |n(x) \cdot v|\dd s \dd v \dd S_x 
		= \iint_{\O \times \R^3} g(t,y,v) \dd y \dd v,
		\label{COV} \\
		\int_{\gamma_{-}} \int_0^{t_{+}} g(t,X(t,t-s,x,v),V(t,t-s,x,v)) |n(x) \cdot v|\dd s \dd v \dd S_x 
		= \iint_{\O \times \R^3} g(t,y,v) \dd y \dd v,
		\label{COV+} \\
		\int_{\gamma_{\pm}}
		g(t,x_{\mp}(x, v),v_{\mp}(x, v))
		|n(x) \cdot v| \dd v \dd S_x
		= \int_{\gamma_{\mp}}
		g(t,y,v)
		|n(y) \cdot v|\dd v \dd S_y.\label{COV_bdry}
	\end{align}
	Here, for the sake of simplicity, we have abused the notations temporarily: $t_-=\tb, x_-= \xb$ and $t_+= \tf, x_+= \xf$.  
\end{lemma} \unhide

Following Lemma \ref{sto_cycle_1}, we derive the stochastic cycles 

\begin{lemma}[\cite{CKL, CKL2}] \label{sto_cycle_1}
	Suppose $F$ solve \eqref{equation for F_infty} and \eqref{diff_F} (resp. \eqref{equation for F mag} and \eqref{diff_F}) with $0 \leq t_* \leq t$, then for $k \geq 1$,
	\begin{align}
		F (x, v) = 
		& \mathbf{1}_{t^1 < t_*}
		F (X(t_*; t, x, v), V(t_*; t, x, v)) \label{expand_F1} 
		\\&  + \mu_{\Theta} (x^1, \vb) \int_{\prod_{j=1}^{i} \mathcal{V}_j}   
		\sum\limits^{k-1}_{i=1} 
		\Big\{   \mathbf{1}_{t^{i+1} < t_* \leq t^{i }} F (X(t_*; t^i, x^i, v^i), V(t_*; t^i, x^i, v^i))  \Big\}
		\dd  \Sigma_{i}
		\label{expand_F2}
		\\& + \mu_{\Theta} (x^1, \vb) \int_{\prod_{j=1}^{k } \mathcal{V}_j}   
		\mathbf{1}_{t^{k } \geq t_* }
		F (x^{k }, v^{k })
		\dd  \Sigma_{k}
		, \label{expand_F3}
	\end{align} 
	where 
	$\dd  {\Sigma}_{i} := \frac{ \dd \sigma_{i}}{ \mu_{\Theta} (x^{i+1}, \vb^{i})} \dd \sigma_{i-1} \cdots  \dd \sigma_1$, with $\dd \sigma_j = \mu_{\Theta} (x^{j+1}, \vb^{j}) \{ n(x^j) \cdot v^j \} \dd v^j$ in \eqref{def:sigma measure}, and $\vb^{j} = \vb (x^j, v^j)$ defined in \eqref{def:t_k}.
	Here, $(X,V)$ in \eqref{characteristics} (resp. \eqref{characteristics mag}).
\end{lemma}

\begin{proof}
	We will prove this lemma by induction for all $n \in \Z_{+}$. 
	When $n = 1$, then it has two cases $t^{1} < t_* \leq t$ and $t_* \leq t^{1} \leq t$. For $t^{1} < t_* \leq t$, we have
	\be \notag
	F (x, v) 
	= F (X(t_*; t, x, v), V(t_*; t, x, v)).
	\ee
	For $t_* \leq t^{1} \leq t$, we get
	\be \notag
	F (x, v) 
	= F (x^1, \vb)
	= \mu_{\Theta} (x^1, \vb) \int_{\mathcal{V}_1}   
	F (x^{1}, v^{1}) \dd  \Sigma_{1}.
	\ee
	Combining the above two cases, we derive that
	\be \notag
	F (x, v) 
	= \mathbf{1}_{t^1 < t_*}
	F (X(t_*; t, x, v), V(t_*; t, x, v)) 
	+ \mu_{\Theta} (x^1, \vb) \int_{\mathcal{V}_1} 
	\mathbf{1}_{t^{1} \geq t_* }  
	F (x^{1}, v^{1}) \dd  \Sigma_{1},
	\ee
	so the equality \eqref{expand_F1}-\eqref{expand_F3} is true for $n = 1$.
	
	Let $k > 1$ and the equality \eqref{expand_F1}-\eqref{expand_F3} holds true for $n = k$. Then we obtain
	\be \label{n=k sto_cycle on F}
	\begin{split}
		F (x, v) = 
		& \mathbf{1}_{t^1 < t_*}
		F (X(t_*; t, x, v), V(t_*; t, x, v)) 
		\\&  + \mu_{\Theta} (x^1, \vb) \int_{\prod_{j=1}^{i} \mathcal{V}_j}   
		\sum\limits^{k-1}_{i=1} 
		\Big\{   \mathbf{1}_{t^{i+1} < t_* \leq t^{i }} F (X(t_*; t^i, x^i, v^i), V(t_*; t^i, x^i, v^i))  \Big\}
		\dd  \Sigma_{i}
		\\& + \mu_{\Theta} (x^1, \vb) \int_{\prod_{j=1}^{k } \mathcal{V}_j}   
		\mathbf{1}_{t^{k } \geq t_* }
		F (x^{k }, v^{k })
		\dd  \Sigma_{k}.
	\end{split}
	\ee 
	Now considering $t^{k+1}$, we have two cases $t^{k+1} < t_* \leq t^k$ and $t_* \leq t^{k+1} \leq t^k$. 
	For $t^{k+1} < t_* \leq t^k$, we have
	\be \notag
	F (x^k, v^k) 
	= F (X(t_*; t^k, x^k, v^k), V(t_*; t^k, x^k, v^k)).
	\ee
	For $t_* \leq t^{k+1}$, we get
	\be \notag
	F (x^k, v^k) 
	= F (x^{k+1}, \vb^k)
	= \mu_{\Theta} (x^{k+1}, \vb^k) \int_{\mathcal{V}_{k+1}}   
	F (x^{k+1}, v^{k+1}) \{ n(x^{k+1}) \cdot v^{k+1} \} \dd v^{k+1}.
	\ee
	Replacing $F (x^k, v^k)$ in \eqref{n=k sto_cycle on F} with the above, we derive that
	\be \notag
	\begin{split}
		F (x, v) =
		& \mathbf{1}_{t^1 < t_*}
		F (X(t_*; t, x, v), V(t_*; t, x, v)) 
		\\&  + \mu_{\Theta} (x^1, \vb) \int_{\prod_{j=1}^{i} \mathcal{V}_j}   
		\sum\limits^{k-1}_{i=1} 
		\Big\{   \mathbf{1}_{t^{i+1} < t_* \leq t^{i }} F (X(t_*; t^i, x^i, v^i), V(t_*; t^i, x^i, v^i))  \Big\}
		\dd  \Sigma_{i}
		\\& + \mu_{\Theta} (x^1, \vb) \int_{\prod_{j=1}^{k} \mathcal{V}_j}   
		\mathbf{1}_{t^{k+1} < t_* \leq t^{k}} F (X(t_*; t^k, x^k, v^k), V(t_*; t^k, x^k, v^k)) \dd  \Sigma_{k}
		\\& + \mu_{\Theta} (x^1, \vb) \int_{\prod_{j=1}^{k+1} \mathcal{V}_j}   
		\mathbf{1}_{t^{k+1} \geq t_* }
		F (x^{k+1}, v^{k+1})
		\dd  \Sigma_{k+1}.
	\end{split}
	\ee 
	Thus, the equality \eqref{expand_F1}-\eqref{expand_F3} holds for $n = k + 1$.
	By induction, we prove the Lemma.
\end{proof}

  \hide
 
 \begin{lemma} \label{sto_cycle_1 mag}
 	Suppose $F (x, v)$ solves \eqref{equation for F mag} and \eqref{diff_F} with $0 \leq t_* \leq t$, then for $k \geq 1$,
 	\begin{align}
 		F (x, v) 
 		& 	= \mathbf{1}_{t^1 < t_*}
 		F (X(t_*; t, x, v), V(t_*; t, x, v)) \label{expand_F1} 
 		\\&  + \mu_{\Theta} (x^1, \vb) \int_{\prod_{j=1}^{i} \mathcal{V}_j}   
 		\sum\limits^{k-1}_{i=1} 
 		\Big\{   \mathbf{1}_{t^{i+1} < t_* \leq t^{i }} F (X(t_*; t^i, x^i, v^i), V(t_*; t^i, x^i, v^i))  \Big\}
 		\dd  \Sigma_{i}
 		\label{expand_F2}
 		\\& + \mu_{\Theta} (x^1, \vb) \int_{\prod_{j=1}^{k } \mathcal{V}_j}   
 		\mathbf{1}_{t^{k } \geq t_* }
 		F (x^{k }, v^{k })
 		\dd  \Sigma_{k},
 		\label{expand_F3}
 	\end{align} 
 	where 
 	$\dd  {\Sigma}_{i} := \frac{ \dd \sigma_{i}}{ \mu_{\Theta} (x^{i+1}, \vb^{i})} \dd \sigma_{i-1} \cdots  \dd \sigma_1$, with $\dd \sigma_j = \mu_{\Theta} (x^{j+1}, \vb^{j}) \{ n(x^j) \cdot v^j \} \dd v^j$ in \eqref{def:sigma measure}, and 
 	$\vb^{j} = \vb (x^j, v^j)$ defined in \eqref{def:t_k}.
 	Here, $(X,V)$ in \eqref{characteristics mag}.
 \end{lemma}

\unhide
 
 \begin{remark} \label{probability}
 	From the Lemma \ref{conservative field} (resp. Lemma \ref{conservative field mag})
 	and $\mu_{\Theta} (x, v)= \mu_{\Theta} (x, |v|)$, we obtain 
 	$\mu_{\Theta} (x^{j+1}, \vb^{j}) = \mu_{\Theta} (x^{j+1}, v^{j})$ with $\vb^{j} = \vb (x^j, v^j)$. 
 	Therefore, in the rest of paper, we write $\dd \sigma_j$ as
 	\Be \notag
 	\dd \sigma_j = \mu_{\Theta} (x^{j+1}, v^{j}) \{ n(x^j) \cdot v^j \} \dd v^j,
 	\Ee
 	where $v^j \in  \mathcal{V}_j  := \{v^j \in \R^3: n(x^j) \cdot v^j > 0 \}$ and $(X,V)$ in \eqref{characteristics} (resp. \eqref{characteristics mag}).
 \end{remark}
 
 \begin{remark} \label{sigma measure estiamte}
 	Though $F$ follows the characteristic trajectory, $F$ is time independent. Moreover, from Remark \ref{probability} and $0 < a \leq \Theta(x) \leq b$, we derive for all $i =1, \cdots, k$ and $(X,V)$ in \eqref{characteristics} (resp. \eqref{characteristics mag}), 
 	\Be \notag
 	\int_{\mathcal{V}_i} \dd \sigma_i
 	= \int_{\mathcal{V}_i} \mu_{\Theta} (x^{j+1}, v^{j}) \{ n(x^j) \cdot v^j \} \dd v^j \leq \int_{\mathcal{V}_i} \frac{1}{2 \pi a^2}e^{-  \frac{|v|^2}{2 b}} \{ n(x^j) \cdot v^j \} \dd v^j \leq \frac{b^2}{a^2}.
 	\ee
 	W.l.o.g. we assume $b^2 = 2 a^2$ (i.e. $\int_{\mathcal{V}_i} \dd \sigma_i \leq 2$) in the rest of paper.
 \end{remark}

\subsection{Gravitational case}


\begin{lemma} \label{conservative field}
Consider $(X,V)$ in \eqref{characteristics}, for $v \in  \mathcal{V}  := \{v \in \R^3: n(x) \cdot v > 0 \}$, $x \in \p\O$ and $\vb = \vb (x, v)$, we have
\Be 
|\vb| = |v|.
\Ee
\end{lemma}

\begin{proof}
Following $(X,V)$ in \eqref{characteristics} with $v \in  \mathcal{V}$, we compute the following derivative:
\Be \notag
\begin{split} 
& \ \ \ \ \frac{\dd}{\dd s}
\Big( 
\frac{|V(s;t,x,v)|^2}{2} + \Phi(X(s;t,x,v)) 
\Big)
\\& = V(s;t,x,v) \cdot \frac{\dd V}{\dd s} + \nabla \Phi \cdot \frac{\dd X}{\dd s} 
\\& = - V(s;t,x,v) \cdot \nabla \Phi (X(s; t, x, v)) + \nabla \Phi \cdot V(s; t, x, v) = 0.
\end{split}
\Ee
Recall that
$(X(t; t, x, v), V(t; t, x, v)) = (x, v)$, $(X(t - t_\mathbf{b}; t, x, v), V(t - t_\mathbf{b}; t, x, v)) = (x_\mathbf{b}, \vb)$. By taking $s = t - \tb$ and $s = t$, we obtain
\Be \notag
|v|^2/2 + \Phi(x) = |\vb|^2/2 + \Phi(x_\mathbf{b}).
\Ee  
Since $\Phi (x)|_{x_3 = 0} \equiv 0$ and $x, \xb \in \p\O$, we have
\be \notag
\Phi(x) = \Phi(x_\mathbf{b}) = 0,
\ee
which implies $|\vb| = |v|$.
\end{proof}

\begin{lemma}
Consider $(X,V)$ in \eqref{characteristics}, for $v \in \mathcal{V}  := \{v \in \R^3: n(x) \cdot v > 0 \}$ and $x \in \p\O$,
there exists a unique time $t_m \in [t-\tb, t]$ such that 
\Be \notag
V_3 (t_m,t, x, v) = 0, \ X_3 (t_m,t, x, v) = \max\limits_{t-\tb \leq s \leq t} X_3 (s,t, x, v).
\Ee
Moreover, we have
\Be \notag
\frac{\tb}{3} \leq t-t_m, \ t_m - (t-\tb) \leq \frac{2 \tb}{3}.
\Ee
\end{lemma}

\begin{proof}
Recall $\|\nabla_{x} \phi\|_{\infty} \leq \varrho_3$, we have 
\Be \label{g3 estimate}
g - \varrho_3 \leq
\big(
\nabla_{x} \Phi
\big)_3
\leq g + \varrho_3.
\Ee
From \eqref{characteristics}, \eqref{g3 estimate} and $0 \leq \varrho_3 \ll 1$, we get
\Be \label{dV3 estimate}
- g - \varrho_3 \leq
\frac{\dd}{\dd s} V_3 (s; t, x, v) 
= 
- \big(
\nabla_X \Phi (X(s; t, x, v))
\big)_3 
\leq - g + \varrho_3 < 0.
\Ee
We define
\Be \label{def: tm}
t_m(t, x, v)
= \{t-\tb \leq s \leq t: X_3 (t_m,t, x, v) = \max\limits_{t-\tb \leq s \leq t} X_3 (s,t, x, v), V_3 (t_m,t, x, v) = 0 \}.
\Ee
The negativity of $\frac{\dd}{\dd s} V_3 (s; t, x, v)$ implies that $t_m$ is unique for any $(t, x, v)$. Moreover, for $s \in [t-\tb, t_m]$, $V_3 \geq 0$ implies $X_3$ is increasing and for $s \in [t_m, t]$, $V_3 \leq 0$ implies $X_3$ is decreasing.

From \eqref{dV3 estimate} and the uniqueness of $t_m$, we deduce 
\Be \notag
\begin{split}
\frac{1}{2} (g - \varrho_3) \big( t_m - (t-\tb) \big)^2 \leq 
& X_3 (t_m,t, x, v) \leq \frac{1}{2} (g + \varrho_3) \big( t_m - (t-\tb) \big)^2,
\\ \frac{1}{2} (g - \varrho_3)( t - t_m )^2 \leq 
& X_3 (t_m,t, x, v) \leq \frac{1}{2} (g + \varrho_3) ( t - t_m )^2.
\end{split}
\Ee
Thus we obtain that
\Be \label{tm first estimate}
\sqrt{\frac{2 X_3 (t_m,t, x, v)}{g + \varrho_3}} \leq t_m - (t-\tb), \ t - t_m \leq \sqrt{\frac{2 X_3 (t_m,t, x, v)}{g - \varrho_3}}.
\Ee
From \eqref{tm first estimate}, we have
\Be \label{tm rate estimate}
\sqrt{\frac{g - \varrho_3}{g + \varrho_3}} \leq \frac{t_m - (t-\tb)}{t - t_m} \leq \sqrt{\frac{g + \varrho_3}{g - \varrho_3}}.
\Ee
Since $\big(t_m - (t-\tb)\big) + \big(t - t_m\big) = \tb$ and $\varrho_3 \ll 1$, we have
\Be \label{tm estimate}
\begin{split}
& \frac{\tb}{3} \leq \Big(1 + \sqrt{\frac{g + \varrho_3}{g - \varrho_3}} \Big)^{-1} \tb \leq t-t_m \leq \Big(1 + \sqrt{\frac{g - \varrho_3}{g + \varrho_3}} \Big)^{-1} \tb \leq \frac{2 \tb}{3}, 
\\& \frac{\tb}{3} \leq \Big(1 + \sqrt{\frac{g + \varrho_3}{g - \varrho_3}} \Big)^{-1} \tb \leq t_m - (t-\tb) \leq \Big(1 + \sqrt{\frac{g - \varrho_3}{g + \varrho_3}} \Big)^{-1} \tb \leq \frac{2 \tb}{3}.
\end{split}
\Ee
\end{proof}

\begin{remark}
From \eqref{g3 estimate} and \eqref{tm estimate}, we derive that 
\Be \notag
\frac{1}{6} g \tb \leq \frac{\tb}{3} (g - \varrho_3) \leq v_3, \ v_{\mathbf{b}, 3} \leq \frac{2 \tb}{3} (g + \varrho_3) \leq \frac{4}{3} g \tb.
\Ee
For the sake of simplicity, we write:
suppose $(X,V)$ in \eqref{characteristics}, for all $(x,v) \in \gamma_+$,
\Be \label{v3t estimate}
\begin{split}
& \tb (x, v) \lesssim v_3 \lesssim \tb (x, v), \ \ \tb (x, v) \lesssim v_{\mathbf{b}, 3} (x, v) \lesssim \tb (x, v).
\end{split}
\Ee
\end{remark}

Now we consider the change of variables 
$v \in \{v \in \R^3
: n(x) \cdot v  >  0\}
\mapsto
(\xb(x,v), \tb(x,v)) \in \partial \Omega \times \R_{+}.$
Since the domain is periodic, this is a local bijective mapping. For fixed $x, \tb$ and $\xb$, we introduce the set of velocities $\{v^{m, n}\}$ with $m, n \in \mathbb{Z}$ such that
\Be \label{def:vmn}
v^{m, n}
\in \{v^{m, n} \in \R^3
: n(x) \cdot v^{m, n} > 0\}
 \mapsto
(\xb + (m, n, 0), \tb)
= (\xb, \tb)
\in \partial \Omega \times \R_{+}.
\ee

\begin{proposition} \label{prop:mapV}
Consider $(X,V)$ in \eqref{characteristics}, 
\begin{itemize}
\item For fixed $x \in \p\O$, and $m, n \in \mathbb{Z}$, we introduce the following map:
\Be \label{mapV}
v \in \{v \in \R^3
: n(x) \cdot v  >  0\}
\mapsto (\xb, \tb):=
(\xb(x,v)  + (m, n, 0), \tb(x,v)) \in \partial \Omega \times \R_{+}.
\Ee 
Then the map \eqref{mapV} is locally bijective and has the change of variable formula as
\Be \label{jacob:mapV}
    \dd v \lesssim |\tb|^{-3}|n (\xb ) \cdot \vb | \dd \tb \dd S_{\xb }.
\Ee
Therefore, for any non-negative function $\phi(v)$, there exists a constant $C$ such that
\be \notag
\int_{v \in v^{m, n}} \phi (v) \dd v \leq C \iint_{\p\O \times \R_{+}} \phi (\xb, \tb) |\tb|^{-3}|n (\xb ) \cdot \vb | \dd \tb \dd S_{\xb }.
\ee
\item Similarly we have a locally bijective map: 
\Be \notag
v \in \{v \in \R^3
: n(x) \cdot v<0\}
\mapsto (\xf , \tf ):=
(\xf(x,v) + (m, n, 0), \tf(x,v)) \in \partial \Omega \times \R_{+},
\Ee 
with 
\Be \label{mapV_f}
\dd v \lesssim |\tf|^{-3} |n (\xf) \cdot \vf | \dd \tf \dd S_{\xf}.
\Ee 
\end{itemize}
\end{proposition} 

\begin{proof}
We just need to show \eqref{jacob:mapV}, since \eqref{mapV_f} can be deduced after changing the backward variables into forward variables. For the sake of simplicity, we have abused the notations temporarily: 
\be \notag
\begin{split}
& \xb := x^1 = (x^1_1, x^1_2, x^1_3)= (x^1_\parallel, x^1_3), \ \
v = (v_1, v_2, v_3) = (v_\parallel, v_3), 
\\& \vb = (v_{\mathbf{b}, 1}, v_{\mathbf{b}, 2}, v_{\mathbf{b}, 3}), \ \
t_{\mathbf{b}} = \tb(x,v).
\end{split}
\ee
Recall that $\Phi (x) = g x_3 + \phi(x)$, we get $
\nabla \nabla \Phi= \nabla \nabla (g x_3 + \phi(x))=\nabla  ( (0, 0, g) + \nabla \phi(x))= \nabla \nabla \phi (x).$

Now we compute the determinant of the Jacobian matrix. 
Fixing $x, t$ and following the characteristics trajectory, we deduce
\Be \label{first cov first part}
 x^1 + \int^{t}_{t - \tb}  V(s;t,x, v) \dd s + (m, n, 0) = x,
\Ee
\Be \label{first cov second part}
\vb +  \int^{t}_{t - \tb}  - \nabla \Phi ( X(s;t,x, v)) \dd s = v.
\Ee
Inputting \eqref{first cov second part} into \eqref{first cov first part}, 
\Be \notag
 x^1 + \int^{t}_{t - \tb} (v - \int^{t}_{s} - \nabla \Phi(X(\alpha, t, x, v)) \dd  \alpha) \dd s + (m, n, 0) = x.
\Ee
Taking $\partial_{v}$ on $\nabla \Phi$, we have
\Be \notag
\nabla_{v} 
\big[ \nabla\Phi(X(\alpha, t, x, v) 
\big]
= \nabla_{v} X(\alpha, t, x, v) \cdot \nabla \nabla\Phi(X) =  \nabla_{v} X(\alpha, t, x, v) \cdot \nabla \nabla \phi.
\Ee
Then, we get
\Be \notag
\begin{split}
    & \frac{\partial \tb}{\partial v} = \frac{1}{v_{\mathbf{b}, 3}} \Big(
     -(0, 0, 1) \tb - \int^{
    t - \tb}_{t}
    \int^{s}_{t} - \nabla_{v} X \cdot \nabla \nabla \phi_3 (X(\alpha, t, x, v)) \dd \alpha \dd s
    \Big),
    \\& \frac{\partial x^1_\pll}{\partial v} = -
    \begin{pmatrix}
    1 & 0 & 0\\
    0 & 1 & 0
    \end{pmatrix} \tb - \int^{
   t - \tb
    }_{t}
    \int^{s}_{t} - \nabla_{v} X \cdot \nabla \nabla \phi_{\parallel} (X(
    \alpha, t, x, v)) \dd \alpha \dd s - \frac{\partial \tb}{\partial v} v_{\parallel},
\end{split}
\Ee
where $\phi = (\phi_1, \phi_2, \phi_3) = (\phi_\parallel, \phi_3)$.
Therefore, we get
\Be \notag
    \det 
    \Big(
    \frac{\partial x^1_{\pll}}{\partial v}, \frac{\partial \tb}{\partial v}
    \Big)
    = \frac{(\tb)^3}{v_{\mathbf{b}, 3}} \times \det 
    \Big[
    -{Id}_{3 \times 3} - \frac{1}{\tb} \int^{t - \tb}_{t}
    \int^{s}_{t} - \nabla_{v} X \cdot \nabla \nabla \phi \dd \alpha \dd s
    \Big].
\Ee
We set 
$d_0 = 
	\Big(
	\det 
    \big[
    -{Id}_{3 \times 3} - \frac{1}{\tb} \int^{t - \tb}_{t}
    \int^{s}_{t} - \nabla_{v} X \cdot \nabla \nabla \phi \dd \alpha \dd s
    \big]	
	 \Big)^{-1},$
and recall that 
\Be \notag
\begin{split}
 \sup_{x_3 \geq 0} e^{\varrho_1 x_3} 
  \|\nabla \nabla_{x} \phi\|_{\infty} \leq \varrho_2 \ll 1 \ \text{with} \ \varrho_1 > 1,
  \ \ \ \text{and} \  \   \|\nabla_{x} \phi\|_{\infty} \leq \varrho_3 \ll 1.
\end{split}
\Ee
From \eqref{characteristics}, we have 
\Be \notag
    \frac{\dd}{\dd s}|\nabla_v X(s;t, x, v)| \lesssim |\nabla_v V(s;t, x, v)|,
\Ee
\Be \notag
\begin{split} 
    \frac{\dd}{\dd s}|\nabla_v V(s;t, x, v)| 
    & \lesssim |\nabla \nabla_X \phi (X(s;t, x, v))||\nabla_v X(s;t, x, v)|
      \lesssim \varrho_2 e^{- \varrho_1 X_3 (s;t, x, v)} |\nabla_v X(s;t, x, v)|.
\end{split}
\Ee
Since $\nabla_v X(t;t, x, v) = \nabla_v x = 0$, $\nabla_v V(t;t, x, v) = \nabla_v v = Id_{3 \times 3}$ and $t - \tb \leq s \leq t$, we get
\Be \label{Xv first}
\begin{split} 
    |\nabla_v X(s;t, x, v)| 
    & \lesssim \int^{t}_{s} |\nabla_v V(s_1;t, x, v)| \dd s_1
    \\& \lesssim |t-s| + \int^{t}_{s} \int^{t}_{s_1}
    \varrho_2 e^{- \varrho_1 X_3 (s_2;t, x, v)} |\nabla_v X(s_2;t, x, v)| \dd s_2 \dd s_1.
\end{split}
\Ee
Using the Fubini's theorem, we derive
\Be \notag
\begin{split}
\eqref{Xv first} 
& \lesssim |t-s| + \int^{t}_{s} \int^{s_2}_{s}
    \varrho_2 e^{- \varrho_1 X_3 (s_2;t, x, v)} |\nabla_v X(s_2;t, x, v)| \dd s_1 \dd s_2
\\& \lesssim |t-s| + \int^{t}_{s} \varrho_2 
(s_2-s) e^{- \varrho_1 X_3 (s_2;t, x, v)}
|\nabla_v X(s_2;t, x, v)| \dd s_2.
\end{split}
\Ee
Using the Gronwall’s inequality, for $t-\tb \leq s \leq t$,
\Be \label{first estimate on X_v}
\begin{split} 
    |\nabla_v X(s;t, x, v)| 
    & \lesssim |t-s| \exp 
    \big(
    \int^{t}_{s} \varrho_2 
(s_2-s) e^{- \varrho_1 X_3 (s_2;t, x, v)} \dd s_2 
	\big)
	\\& \leq |t-s| \exp 
    \Big(
    \int^{t}_{t-\tb} \varrho_2 
	\big(
	s_2- (t-\tb)
	\big)
 	e^{- \varrho_1 X_3 (s_2;t, x, v)} \dd s_2 
	\Big).
\end{split}
\Ee

Now we consider $\int^{t}_{t-\tb} 
\big( s_2-(t - \tb) \big) e^{- \varrho_1 X_3 } \dd s_2$. 
From \eqref{g3 estimate}, $t_m$ in \eqref{def: tm} and \eqref{tm estimate}, in the case $s \in [t_m, t]$ where $X_3 (s,t, x, v)$ is decreasing, 
we obtain 
\Be \notag
t-t_m \gtrsim \frac{\tb}{2}, \
v_3 \geq (g - \varrho_3) (t-t_m) \gtrsim g (t-t_m) \gtrsim 5 \tb. 
\Ee
Thus we have
\Be \label{x3 estimate}
X_3 (s,t,x,v) \gtrsim \frac{1}{2} v_3 (t-s) \gtrsim \frac{5}{2} \tb (t-s).
\Ee
Using \eqref{x3 estimate}, we derive 
\begin{align}
 \int^{t}_{t_m} 
(s_2 - (t - \tb) ) e^{- \varrho_1 X_3} \dd s_2   \leq \int^{t}_{t - \tb /2} 
\big(
s_2 - (t - \tb)
\big) e^{- \varrho_1 \frac{5}{2} \tb (t-s_2)} \dd s_2. \label{second estimate on X_v}
\end{align}
Setting $a = t - s_2$, we get
\Be \notag
\begin{split}
\eqref{second estimate on X_v} 
& = \int^{\tb /2}_{0} 
(\tb - a) e^{- \varrho_1 \frac{5}{2} \tb a} \dd a
\\& = - \frac{2}{5 \varrho_1} e^{- \varrho_1 \frac{5}{2} \tb a} \Bigg|^{\tb /2}_{0} -
\bigg(
\frac{a}{- \varrho_1 \frac{5}{2} \tb}
e^{- \varrho_1 \frac{5}{2} \tb a} - \frac{e^{- \varrho_1 \frac{5}{2} \tb a}}{(- \varrho_1 \frac{5}{2} \tb)^2}
\bigg) \Bigg|^{\tb /2}_{0}
\lesssim \frac{2}{5 \varrho_1}.
\end{split}
\Ee

In the case $s \in [t- \tb, t_m]$ where $X_3 (s,t, x, v)$ is increasing. 
From \eqref{tm estimate}, 
$$t_m - (t- \tb) \gtrsim \frac{\tb}{2}, \
v_{\mathbf{b}, 3} \geq (g - \varrho_3) \big( t_m - (t- \tb) \big) \gtrsim g (t-t_m) \gtrsim 5 \tb. 
$$
From \eqref{g3 estimate}, for $x \in \p\O$,  $s \in [t-\tb, t_m]$, we have
\Be \label{x3 estimate second case}
X_3 (s,t,x,v) \gtrsim \frac{1}{2} v_{\mathbf{b}, 3} \big( s - (t- \tb) \big) \gtrsim \frac{5}{2} \tb \big( s - (t- \tb) \big).
\Ee
Using \eqref{x3 estimate second case}, we derive 
\begin{align}
 \int^{t_m}_{t- \tb} 
(s_2- (t - \tb) ) e^{- \varrho_1 X_3} \dd s_2   \leq \int^{t - \tb /2}_{t - \tb} 
\big(
s_2 - (t - \tb)
\big) e^{- \varrho_1 \frac{5}{2} \tb \big( s_2 - (t- \tb) \big)} \dd s_2. \label{second estimate on X_v second case}
\end{align}
Setting $a = s_2 - (t- \tb)$, we get
\Be \notag
\begin{split}
\eqref{second estimate on X_v second case} 
  = \int^{\tb /2}_{0} 
a e^{- \varrho_1 \frac{5}{2} \tb a} \dd a
  = 
\bigg(
\frac{a}{- \varrho_1 \frac{5}{2} \tb}
e^{- \varrho_1 \frac{5}{2} \tb a} - \frac{e^{- \varrho_1 \frac{5}{2} \tb a}}{(- \varrho_1 \frac{5}{2} \tb)^2}
\bigg) \Bigg|^{\tb /2}_{0}
\lesssim \frac{2}{5 \varrho_1}.
\end{split}
\Ee
Combining \eqref{second estimate on X_v} and \eqref{second estimate on X_v second case}, we conclude
\Be \label{estimate on X_v}
\eqref{first estimate on X_v} \leq |t-s| \exp   
(\varrho_2 / \varrho_1).
\Ee
From \eqref{estimate on X_v} and Fubini's theorem, we have 
\Be \notag
\begin{split}
    & \ \ \ \ 
    \Big|
    \int^{t - \tb}_{t}
\int^{s}_{t} - \nabla_{v} X \cdot \nabla^2 \phi (X(\alpha; t, x, v)) \dd \alpha \dd s
	\Big|
    \\& \leq \int^{t}_{t - \tb}
\int^{t}_{s} |\nabla^2 \phi(X)| |\nabla_{v} X(\alpha; t, x, v)| \dd \alpha \dd s
     \leq \int^{t}_{t - \tb}
\int^{t}_{s} \varrho_2 e^{- \varrho X_3} |t - \alpha| e^{\varrho_2/\varrho_1} \dd \alpha \dd s
    \\& \leq \int^{t}_{t - \tb} \varrho_2 
\big(
\alpha - (t - \tb)
\big)
e^{- \varrho_1 X_3} |t - \alpha| e^{\varrho_2/\varrho_1} \dd \alpha 
    \\& \leq \varrho_2 e^{\varrho_2/\varrho_1} \tb \int^{t}_{t - \tb}  
\big(
\alpha - (t - \tb)
\big)
  e^{- \varrho_1 X_3 (\alpha; t, x, v)} \dd \alpha
    \leq e^{\varrho_2/\varrho_1} \tb \frac{\varrho_2}{\varrho_1},
\end{split}
\Ee
where the last inequality holds from \eqref{estimate on X_v}. 
Since $\varrho_2 \ll 1$, and
\Be \notag
|\frac{1}{\tb} \int^{t - \tb}_{t}
    \int^{s}_{t} - \nabla_{v} \phi d \alpha ds| \leq e^{\varrho_2/\varrho_1} \frac{\varrho_2}{\varrho_1},
\Ee
thus we derive the uniform-in-time bound on $d_0$.
From
$ 
\det (
    \frac{\partial x^1_{\pll}}{\partial v}, \frac{\partial \tb}{\partial v} )
    = \frac{(\tb)^3}{v_{\mathbf{b}, 3}} \times d_0
$, we prove \eqref{jacob:mapV}.
\end{proof}

\begin{remark}
We can apply Proposition \ref{prop:mapV} on 
$v^{j} \in \mathcal{V}_{j} 
\mapsto
(x^{j+1}, \tb^{j})
:=
(x_{\mathbf{b}} (x^{j}, v^{j}), \tb (x^{j}, v^{j}))$, and this is also a local bijective mapping. For fixed $\tb^{j}$ and $x^{j+1}$, we introduce the set of velocities $\{v^{m, n}_{j}\}$ with $m, n \in \mathbb{Z}$ such that
\Be \label{def:vjmn}
v^{m, n}_{j} 
\in \mathcal{V}_{j} 
 \mapsto
(x^{j+1} + (m, n, 0), \tb^{j})
= (x^{j+1}, \tb^{j})
\in \partial \Omega \times 
[0, t_{j}],
\ee
with the change of variable formula as
\be \label{vjmn}
\dd v^{m, n}_{j} \lesssim |\tb^{j}| ^{-3} {|n (x^{j+1}) \cdot v^{m, n}_{j, \mathbf{b}}|}\dd \tb^{j} \dd S_{x^{j+1}}.
\Ee 
\end{remark}


\subsection{With Magnet field}

For the technique reason, w.l.o.g. we set 
$g = 10$ for the magnet field case in the rest of the paper.

\begin{lemma} \label{conservative field mag}
Consider $(X,V)$ in \eqref{characteristics mag}, for $v \in  \mathcal{V}  := \{v \in \R^3: n(x) \cdot v > 0 \}$, $x \in \p\O$ and $\vb = \vb (x, v)$, we have
\be \label{tb expression mag}
\tb (x, v) = \frac{|v_3|}{5}, \ \ |\vb| = |v|.
\ee
\end{lemma}

\begin{proof}
From $(X,V)$ in \eqref{characteristics mag}, we have
\be \label{V characteristics mag}
\frac{\dd}{\dd s} V(s; t, x, v) = (B_3 V_2 (s; t, x, v), - B_3 V_1 (s; t, x, v), -g).
\ee
Then we obtain the following derivative, 
\Be \label{V3 characteristics mag}
\frac{\dd}{\dd s}
\Big( 
\frac{|V_3 (s;t,x,v)|^2}{2} + \Phi(X(s;t,x,v)) 
\Big)
= V_3 \cdot \frac{\dd V_3}{\dd s} + \nabla \Phi (X) \cdot \frac{\dd X}{\dd s} 
= - g V_3 + g V_3 = 0.
\Ee
Since
$(X(t; t, x, v), V(t; t, x, v)) = (x, v)$, $(X(t - t_\mathbf{b}; t, x, v), V(t - t_\mathbf{b}; t, x, v)) = (x_\mathbf{b}, v_\mathbf{b})$, we have 
\Be \notag
|v_3|^2/2 + \Phi(x) = |\vbn|^2/2 + \Phi(\xb).
\Ee  
Since $\Phi (x)|_{x_3 = 0} = 0$, we have $\Phi(x) = \Phi(x_\mathbf{b})$ which implies $|\vbn| = |v_3|$.
Moreover, using $\frac{\dd}{\dd s} V_3 (s; t, x, v) = -g$ in \eqref{V characteristics mag}, we derive that
\be \notag 
\tb (x, v) = 2 v_3 / g = v_3 / 5.
\ee

On the other side, from \eqref{V characteristics mag}, we also have
\be \label{V12 characteristics}
\frac{\dd}{\dd s} V_1(s; t, x, v) = B_3 V_2 (s; t, x, v), \ \ \frac{\dd}{\dd s} V_2(s; t, x, v) = - B_3 V_1 (s; t, x, v).
\ee
Thus we obtain that,
\be \label{V12 conservative mag}
\frac{\dd}{\dd s}
\Big( 
|V_1 (s;t,x,v)|^2 + |V_2 (s;t,x,v)|^2) 
\Big)
= 2 V_1 B_3 V_2 - 2 V_2 B_3 V_1 = 0.
\ee
Therefore we prove $|v_\mathbf{b}| = |v|$.
\end{proof}

Similarly, we consider the change of variables 
$v \in \{v \in \R^3
: n(x) \cdot v  >  0\}
\mapsto
(\xb(x,v), \tb(x,v)) \in \partial \Omega \times \R_{+}.$
Since the domain is periodic, this is a local bijective mapping. For fixed $x, \tb$ and $\xb$, we introduce the set of velocities $\{v^{m, n}\}$ with $m, n \in \mathbb{Z}$ such that
\Be \label{def:vmn mag}
v^{m, n}
\in \{v^{m, n} \in \R^3
: n(x) \cdot v^{m, n} > 0\}
 \mapsto
(\xb + (m, n, 0), \tb)
= (\xb, \tb)
\in \partial \Omega \times \R_{+}.
\ee

\begin{proposition} \label{prop:mapV mag}
Consider $(X,V)$ in \eqref{characteristics mag},
\begin{itemize}
\item For fixed $x \in \p\O$, and $m, n \in \mathbb{Z}$, we introduce the following map:
\Be \label{mapV mag}
v \in \{v \in \R^3
: n(x) \cdot v  >  0\}
\mapsto (\xb, \tb):=
(\xb(x,v) + (m, n, 0), \tb(x,v)) \in \partial \Omega \times \R_{+}.
\Ee Then the map \eqref{mapV mag} is locally bijective and has the change of variable formula as   
\Be \label{jacob:mapV mag}
    \dd v = 5 B^2_3 (2 - 2 \cos (B_3 \tb))^{-1} \dd \tb \dd S_{\xb}.
\Ee
Therefore, for any non-negative function $\phi(v)$, there exists a constant $C$ such that
\be \notag
\int_{v \in v^{m, n}} \phi (v) \dd v 
= \iint_{\p\O \times \R_{+}} \phi (\xb, \tb) 5 B^2_3 (2 - 2 \cos (B_3 \tb))^{-1} \dd \tb \dd S_{\xb}.
\ee
\item Similarly we have a locally bijective map 
\Be \notag
v \in \{v \in \R^3
: n(x) \cdot v<0\}
\mapsto (\xf , \tf ):=
(\xf(x,v), \tf(x,v)) \in \partial \Omega \times \R_{+},
\Ee with 
\Be \label{mapV_f mag}
\dd v  = 5 B^2_3 (2 - 2 \cos (B_3 \tf))^{-1} \dd \tf \dd S_{\xf}.
\Ee 
\end{itemize}
\end{proposition} 

\begin{proof}
We just need to show \eqref{jacob:mapV mag} since \eqref{mapV_f mag} can be deduced after changing the backward variables into forward variables.
Here, we write
$x = (x_1, x_2, x_3), v = (v_1, v_2, v_3)$.

From \eqref{V12 characteristics}, for $t - \tb \leq s \leq t$, we assume for $0 \leq a$ and $0 \leq \theta_0 < 2 \pi$,
\be \label{V12 expression mag}
V_1(s; t, x, v) = a \sin (B_3 s + \theta_0), 
\ \ V_2(s; t, x, v) = a \cos (B_3 s + \theta_0),
\ee
where we set
\be \label{v12 expression mag}
v_1 = V_1 (t; t, x, v) = a \sin (B_3 t + \theta_0), 
\ \ v_2 = V_2 (t; t, x, v) = a \cos (B_3 t + \theta_0).
\ee
Given fixed $t$ and $x$, we can compute $\xb$, 
\Be \notag
\xb + \int^{t}_{t - \tb}  V(s; t, x, v) \dd s = x.
\Ee
From \eqref{V12 expression mag} and \eqref{v12 expression mag}, we get
\be \label{xb1 expression mag}
\begin{split}
x_{\mathbf{b}, 1} & = x_{1} - \int^{t}_{t - \tb}  a \sin (B_3 s + \theta_0) \dd s 
= x_{1} + \frac{1}{B_3}  a \cos (B_3 s + \theta_0) \Big|^{t}_{t - \tb} 
\\& = x_{1} + \big( \frac{1}{B_3} a \cos (B_3 t + \theta_0) - \frac{1}{B_3} a \cos (B_3 (t - \tb) + \theta_0) \big)
\\& = x_{1} + \frac{1}{B_3} v_2 - \frac{1}{B_3} v_2 \cos (B_3 \tb) - \frac{1}{B_3} v_1 \sin (B_3 \tb),
\end{split}
\ee
\be \label{xb2 expression mag}
\begin{split}
x_{\mathbf{b}, 2} & = x_{2} - \int^{t}_{t - \tb}  a \cos (B_3 s + \theta_0) \dd s 
= x_{2} - \frac{1}{B_3}  a \sin (B_3 s + \theta_0) \Big|^{t}_{t - \tb} 
\\& = x_{2} - \big( \frac{1}{B_3} a \sin (B_3 t + \theta_0) - \frac{1}{B_3} a \sin (B_3 (t - \tb) + \theta_0) \big)
\\& = x_{2} - \frac{1}{B_3} v_1 + \frac{1}{B_3} v_1 \cos (B_3 \tb) - \frac{1}{B_3} v_2 \sin (B_3 \tb).
\end{split}
\ee
Thus we can compute out the matrix,
\be \notag
\Big| \det \big[ 
\frac{\p (\xb, \tb)}{\p v} 
\big] \Big| 
= \frac{2 - 2 \cos (B_3 \tb)}{5 B^2_3}.
\ee
Therefore, we prove \eqref{jacob:mapV mag}.
\end{proof}

\begin{remark}
We can apply Proposition \ref{prop:mapV mag} on 
$v^{j} \in \mathcal{V}_{j} 
\mapsto
(x^{j+1}, \tb^{j})
:=
(x_{\mathbf{b}} (x^{j}, v^{j}), \tb (x^{j}, v^{j}))$, and this is also a local bijective mapping. For fixed $\tb^{j}$ and $x^{j+1}$, we introduce the set of velocities $\{v^{m, n}_{j}\}$ with $m, n \in \mathbb{Z}$ such that
\Be \notag
v^{m, n}_{j} 
\in \mathcal{V}_{j} 
 \mapsto
(x^{j+1} + (m, n, 0), \tb^{j})
= (x^{j+1}, \tb^{j})
\in \partial \Omega \times 
[0, t_{j}],
\ee
with the change of variable formula as   
\be \label{vjmn mag}
\dd v^{m, n}_{j} = 5 B^2_3 (2 - 2 \cos (B_3 \tbj))^{-1} \dd \tbj \dd S_{x^{j+1}}.
\Ee 
\end{remark}


\section{Pointwise control of the Outgoing Flux \texorpdfstring{$\mathfrak{J}(x)$}{J(x)}}
\label{sec: Linf estimate}

The main purpose of this section is to prove the following pointwise bound of the stationary solution.

\begin{proposition} \label{prop:j linfty bound}
	Suppose $F (x, v)$ solves \eqref{equation for F_infty} with \eqref{field property} and \eqref{diff_F} (resp. \eqref{equation for F_infty} with \eqref{field property mag} and \eqref{diff_F}). Recall $b>0$ defined in \eqref{def:Theta}. Then there exists constant $k > 0$ (see \eqref{choice:k} and \eqref{choice:t} for the precise choice) such that
	\Be  \label{est: F mag}
	\| e^{ \frac{1}{2 b} (|v|^2/2 + \Phi(x))} F (x, v)\|_{L^\infty_{x,v}} 
	\lesssim 2^{k+1} \times \| F(x, v) \|_{L^1_{x,v}}.
	\Ee
\end{proposition}
\hide
\begin{proposition} \label{prop:j linfty bound}
	Suppose $F (x, v)$ solves \eqref{def:f} and \eqref{diff_F}  with $b$ defined in \eqref{def:Theta}.
	There exists constant $k > 0$ (see \eqref{choice:k} and  for the precise choice) such that
	\be \label{est: F mag}
	\| e^{ \frac{1}{2 b} (|v|^2/2 + \Phi(x))} F (x, v)\|_{L^\infty_{x,v}} 
	\lesssim 2^{k+1} \times \| F(x, v) \|_{L^1_{x,v}}. 
	\ee
\end{proposition}
\unhide

\subsection{Gravitational case}

\begin{definition}
Suppose $F (x, v)$ solve \eqref{equation for F_infty} and \eqref{diff_F}, define the boundary outgoing flux on $F$, denoted by $\mathfrak{J}(x)$:
\Be \label{diff_J}
\J (x) = \int_{n(x)\cdot v >0} F(x, v) \{ n(x) \cdot v \} \dd v, \ \text{for} \ x \in \p\O.
\Ee
\end{definition}
From \eqref{diff_F}, we have for $(x, v) \in \gamma_-$,
\Be \label{FJ relation}
F(x, v) = \mu_{\Theta} (x, v) \J (x).
\Ee
An immediate consequence of Lemma \ref{sto_cycle_1} and Remark \ref{probability} follows:
\begin{lemma} \label{sto_cycle J}
Suppose $F$ solves \eqref{equation for F_infty} and \eqref{diff_F}, and $\J (x)$ defined in \eqref{diff_J}, then for $t \geq 0$ and $k \geq 1$,
\begin{align}
	\J (x)
    & = \int_{\mathcal{V}_0}   
    \mathbf{1}_{t^{1} < 0 }  F (X(0; t, x, v^0), V(0; t, x, v^0)) \{ n(x) \cdot v^0 \} \dd v^0
\label{expand_G11}
	\\& + \int_{\prod_{j=0}^{1} \mathcal{V}_j}   
    \mathbf{1}_{t^{2} < 0 \leq t^{1}} F (X(0; t^1, x^1, v^1), V(0; t^1, x^1, v^1)) \dd \tilde{\Sigma}_{1}
\label{expand_G12}
    \\& 	+ \int_{\prod_{j=0}^{i} \mathcal{V}_j}   
     \sum\limits^{k-1}_{i=2} 
     \Big\{   \mathbf{1}_{t^{i+1} < 0 \leq t^{i }} 		F (X(0; t^i, x^i, v^i), V(0; t^i, x^i, v^i))  \Big\} \dd \tilde{\Sigma}_{i}
\label{expand_G1}
    \\& + \int_{\prod_{j=0}^{k } \mathcal{V}_j}   
    \mathbf{1}_{t^{k } \geq 0 } F (x^{k }, v^{k }) \dd \tilde{\Sigma}_{k},
\label{expand_G2}
\end{align} 
where 
$\dd \tilde{\Sigma}_{i} := \frac{ \dd \sigma_{i}}{\mu_{\Theta} (x^{i+1}, v^{i})} \dd \sigma_{i-1} \cdots \dd \sigma_1 \dd \sigma_0$, with $\dd \sigma_j = \mu_{\Theta} (x^{j+1}, v^{j}) \{ n(x^j) \cdot v^j \} \dd v^j$ in \eqref{def:sigma measure}.
Here, $(X,V)$ in \eqref{characteristics}.
\end{lemma}

For the proof of Proposition \ref{prop:j linfty bound}, we start with first two terms in $\J (x)$ stochastic cycles.

\subsubsection{Estimate on first two terms in $\J (x)$}

We first do estimate on \eqref{expand_G11}. From Lemma \ref{conservative field}, Remark \ref{probability} and \eqref{FJ relation}, we have
\Be \label{est1: G11}
\begin{split}
\eqref{expand_G11}
& = \int_{n(x) \cdot v^0 > 0} \mathbf{1}_{t^{1} < 0 } F (x^1, \vb^{0}) \{ n(x) \cdot v^0 \} \dd v^0
\\& = \int_{n(x) \cdot v^0 > 0} \mathbf{1}_{t^{1} < 0 } \ \mu_{\Theta} (x^1, v^{0}) \J (x^1) \{ n(x) \cdot v^0 \} \dd v^0
\\& \leq \| \J (x)\|_{L^{\infty}_x} \int_{n(x) \cdot v^0 > 0} \mathbf{1}_{t^{1} < 0 } \ \mu_{\Theta} (x^1, v^{0})  \{ n(x) \cdot v^0 \} \dd v^0.
\end{split}
\Ee
Since $t^1 < 0$, $|n(x) \cdot v^0| \gtrsim \tb (x, v^0) = t - t^1 > t \gg 1$
and using the notation $v^0 = (v^0_1, v^0_2, v^0_3)$, we deduce
\Be \label{est2: G11}
\begin{split}
  \int_{n(x) \cdot v^0 > 0} \mathbf{1}_{t^{1} < 0} \ \mu_{\Theta} (x^1, v^{0})  \{ n(x) \cdot v^0 \} \dd v^0
  \lesssim \int_{v^0_{3} \gtrsim t} \ e^{- \frac{|v^0_{3}|^2}{2b}} \{ v^0_{3} \} \dd v^0_{3} \lesssim e^{- \frac{t^2}{2b}}.
\end{split}
\Ee
From \eqref{est1: G11} and \eqref{est2: G11}, we derive
\Be \label{est: G11}
\int_{\mathcal{V}_0} \mathbf{1}_{t^{1} < 0 }  F (X(0; t, x, v^0), V(0; t, x, v^0)) \{ n(x) \cdot v^0 \} \dd v^0 
\lesssim \| \J (x)\|_{L^{\infty}_x} \times e^{- \frac{t^2}{2b}}.
\Ee

Next we do estimate on \eqref{expand_G12}. Again from Lemma \ref{conservative field} and \eqref{FJ relation}, we have 	
\Be \label{est1: G12}
\begin{split}
\eqref{expand_G12} 
& = \int_{\mathcal{V}_0} \int_{\mathcal{V}_1} \mathbf{1}_{t^{2} < 0 \leq t^{1}} F (x^2, \vb^{1}) \{ n(x^1) \cdot v^1 \} \dd v^1 \dd \sigma_0
\\& = \int_{\mathcal{V}_0} \int_{\mathcal{V}_1} \mathbf{1}_{t^{2} < 0 \leq t^{1}} \ \mu_{\Theta} (x^2, v^{1}) \J (x^2) \{ n(x^1) \cdot v^1 \} \dd v^1 \dd \sigma_0
\\& \lesssim \| \J (x)\|_{L^{\infty}_x} \int_{\mathcal{V}_0} \int_{\mathcal{V}_1} \mathbf{1}_{t^{2} < 0 \leq t^{1}} \ \mu_{\Theta} (x^2, v^{1}) \{ n(x^1) \cdot v^1 \} \dd v^1 \dd \sigma_0.
\end{split}
\Ee
Since $t^2 < 0$, $t - t^2 = \tb (x, v^0) + \tb (x^1, v^1)$ and $|n(x) \cdot v^0| \gtrsim \tb (x, v^0)$, $|n(x^1) \cdot v^1| \gtrsim \tb (x^1, v^1)$, we get
\Be \notag
|n(x) \cdot v^0| + |n(x^1) \cdot v^1| \gtrsim t - t^2 > t.
\Ee
After setting $v^0_{3} = n(x) \cdot v^0$ and $v^1_{3} = n(x^1) \cdot v^1$, we have 
\Be \label{est2: G12}
\begin{split}
& \ \ \ \ \int_{\mathcal{V}_0} \int_{\mathcal{V}_1} \mathbf{1}_{t^{2} < 0 \leq t^{1}} \ \mu_{\Theta} (x^2, v^{1}) \{ n(x^1) \cdot v^1 \} \dd v^1 \dd \sigma_0
\\& \lesssim \int_{v^0_{3} \gtrsim t/2} \ e^{- \frac{|v^0_{3}|^2}{2b}} \{ v^0_{3} \} \dd v^0_{3} \int_{\mathcal{V}_1} e^{- \frac{|v^1_{3}|^2}{2b}} \{ v^1_{3} \} \dd v^1_{3} + \int_{\mathcal{V}_0} \ e^{- \frac{|v^0_{3}|^2}{2b}} \{ v^0_{3} \} \dd v^0_{3} \int_{v^1_{3} \gtrsim t/2} e^{- \frac{|v^1_{3}|^2}{2b}} \{ v^1_{3} \} \dd v^1_{3}, 
\\& \lesssim \int_{v^0_{3} \gtrsim t/2} \ e^{- \frac{|v^0_{3}|^2}{2b}} \{ v^0_{3} \} \dd v^0_{3} + \int_{v^1_{3} \gtrsim t/2} e^{- \frac{|v^1_{3}|^2}{2b}} \{ v^1_{3} \} \dd v^1_{3} \lesssim e^{- \frac{t^2}{8b}}.
\end{split}
\Ee
From \eqref{est1: G12} and \eqref{est2: G12}, we derive
\Be \label{est: G12}
\int_{\prod_{j=0}^{1} \mathcal{V}_j} \mathbf{1}_{t^{2} < 0 \leq t^{1}} F (X(0; t^1, x^1, v^1), V(0; t^1, x^1, v^1)) \dd \tilde{\Sigma}_{1}
\lesssim \| \J (x)\|_{L^{\infty}_x} \times e^{- \frac{t^2}{8b}}.
\Ee


\subsubsection{Estimate on last two terms in $\J (x)$}

To estimate \eqref{expand_G1}, we first apply Proposition \ref{prop:mapV} on $\mathcal{V}_{j}$ for $j = i-2, i-1$ and using $|n(x^{j}) \cdot v^{j}| \lesssim \tb^{j}$, we derive that
\begin{align}
& \ \ \ \ \int_{\mathcal{V}_{i-2}} \int_{\mathcal{V}_{i-1}} 
\int_{\mathcal{V}_i} 
\mathbf{1}_{t^{i+1} < 0 \leq t^i}
|F (X(0; t^i, x^i, v^i), V(0; t^i, x^i, v^i)) |  \{n(x^i) \cdot v^i\} \dd v^i \dd \sigma_{i-1} \dd \sigma_{i-2} \notag
\\& \lesssim \int_0^{t^{i-2}} \dd \tb^{i-1}  \int_{\p\O}  
\frac{\dd S_{x^{i-1}}}{\tb^{i-2}} 
\sum\limits_{m, n \in \mathbb{Z}} \mu_{\Theta} (x^{i-1}, v^{m, n}_{i-2, \mathbf{b}})
\int_0^{t^{i-2}- \tb^{i-1}} \dd \tb^{i-2}  \int_{\p\O}   
\frac{\dd S_{x^i}}{\tb^{i-1}} \sum\limits_{m, n \in \mathbb{Z}} \mu_{\Theta} (x^{i}, v^{m, n}_{i-1, \mathbf{b}}) \label{integrand}
\\& \ \ \ \ \times \bigg( 
\int_{\mathcal{V}_i} 
\mathbf{1}_{t^{i} - \tb(x^i,v^i) < 0 } \
|F (X(0; t^i, x^i, v^i), V(0; t^i, x^i, v^i))|  \{n(x^i) \cdot v^i\} \dd v^i\bigg), \label{est1:forcing}
\end{align}
with $t^{i-1} = t^{i-2} - \tb^{i-2}$, $t^{i}= t^{i-1} - \tb^{i-1}$ and $v^{m, n}_{i-1, \mathbf{b}} = \vb (x^{i}, v^{m, n}_{i-1})$, $v^{m, n}_{i-2, \mathbf{b}} = \vb (x^{i-1}, v^{m, n}_{i-2})$.

\begin{lemma} \label{lem: sum of mu}
Consider $(X,V)$ in \eqref{characteristics} with $x^{i-1} \in \p\O$, for $\tb^{i-2} \geq 1$,
\Be \label{t>1,i-2}
\sum\limits_{m, n \in \mathbb{Z}} \mu_{\Theta} (x^{i-1}, v^{m, n}_{i-2, \mathbf{b}}) \lesssim (\tb^{i-2})^4 e^{- (\tb^{i-2})^2 / 2b}.
\Ee
For $0 \leq \tb^{i-2} < 1$,
\Be \label{t<1,i-2}
\sum\limits_{m, n \in \mathbb{Z}} \mu_{\Theta} (x^{i-1}, v^{m, n}_{i-2, \mathbf{b}})
\lesssim 
\sum\limits_{|m| < 2, |n| < 2} \mu_{\Theta} (x^{i-1}, v^{m, n}_{i-2, \mathbf{b}}) +
e^{-\frac{1}{2b (\tb^{i-2})^2}},
\Ee
where $v^{m, n}_{i-2, \mathbf{b}} = \vb (x^{i-1}, v^{m, n}_{i-2})$, which was defined in \eqref{def:t_k} and \eqref{def:vmn}.
\end{lemma} 

\begin{proof}

Here, for the sake of simplicity, we have abused the notations temporarily: 
$$
x^i = (x^i_1, x^i_2, x^i_3)= (x^i_\parallel, x^i_3), \ 
v^{m, n}_{i, \mathbf{b}} = (v^{m, n}_{i, \mathbf{b}_1}, v^{m, n}_{i, \mathbf{b}_2}, v^{m, n}_{i, \mathbf{b}_3}) = (v^{m, n}_{i, \mathbf{b}_\parallel}, v^{m, n}_{i, \mathbf{b}_3}).
$$
To estimate $v^{m, n}_{i-2, \mathbf{b}_\parallel}$, we recall that
$\|\nabla_{x} \phi\|_{\infty} \leq \varrho_3 \ll 1.$ Thus, we have
\Be \label{vmnb12 first estimate}
\begin{split}
& |v^{m, n}_{i-2, \mathbf{b}_1}| \geq \min \big\{ \frac{|x^{i-1}_1 + m - x^{i-2}_1|}{\tb^{i-2}} - \varrho_3 \tb^{i-2}, 0 \big\},
\\& |v^{m, n}_{i-2, \mathbf{b}_2}| \geq \min \big\{ \frac{|x^{i-1}_2 + n - x^{i-2}_2|}{\tb^{i-2}} - \varrho_3 \tb^{i-2}, 0 \big\}.
\end{split}
\Ee
Now we split the length of $\tb^{i-2}$ into two cases: 

\textbf{Case 1:} $\tb^{i-2} \geq 1$. From \eqref{vmnb12 first estimate} and $\varrho_3 \ll 1$, for $|m| \geq (\tb^{i-2})^2$, we bound 
\Be \notag
\frac{|x^{i-1}_1 + m - x^{i-2}_1|}{\tb^{i-2}} - \varrho_3 \tb^{i-2} \gtrsim \frac{|x^{i-1}_1 + m - x^{i-2}_1|}{2 \tb^{i-2}} \gtrsim \frac{|m|}{2 \tb^{i-2}}.
\Ee
Similarly, for $|n| \geq (\tb^{i-2})^2$, we bound
\Be \notag
\frac{|x^{i-1}_2 + n - x^{i-2}_2|}{\tb^{i-2}} - \varrho_3 \tb^{i-2} \gtrsim \frac{|x^{i-1}_2 + n - x^{i-2}_2|}{2 \tb^{i-2}} \gtrsim \frac{|n|}{2 \tb^{i-2}}.
\Ee
For $|m| < (\tb^{i-2})^2$, we bound $|v^{m, n}_{i-2, \mathbf{b}_1}| \geq 0$, and for $|n| < (\tb^{i-2})^2$, we bound $|v^{m, n}_{i-2, \mathbf{b}_2}| \geq 0$. Note that from \eqref{v3t estimate}, we know
$\tb^{i-2} \lesssim v^{m, n}_{i-2, \mathbf{b}_3} \lesssim \tb^{i-2}$. 
To derive \eqref{t>1,i-2}, we divide $\{v^{m, n}_{i-2, \mathbf{b}}\}_{m, n \in \mathbb{Z}}$ into four parts.

For $|m| < (\tb^{i-2})^2$ and $|n| < (\tb^{i-2})^2$, we bound
$|v^{m, n}_{i-2, \mathbf{b}}| \geq |v^{m, n}_{i-2, \mathbf{b}_3}| \gtrsim \tb^{i-2}$. Therefore, we have
\Be \label{t>1m<n<}
\sum\limits_{|m| < (\tb^{i-2})^2, |n| < (\tb^{i-2})^2} \mu_{\Theta} (x^{i-1}, v^{m, n}_{i-2, \mathbf{b}}) \lesssim (\tb^{i-2})^4 e^{- (\tb^{i-2})^2 / 2b}.
\Ee

For $|m| < (\tb^{i-2})^2$ and $|n| \geq (\tb^{i-2})^2$, we bound
$|v^{m, n}_{i-2, \mathbf{b}_2}| \gtrsim \frac{|n|}{2 \tb^{i-2}}$. Thus, we have 
\Be \label{t>1m<n>}
\begin{split}
 \sum\limits_{|m| < (\tb^{i-2})^2, |n| \geq (\tb^{i-2})^2} \mu_{\Theta} (x^{i-1}, v^{m, n}_{i-2, \mathbf{b}})
& \lesssim \sum\limits_{|m| < (\tb^{i-2})^2} \mu_{\Theta} (x^{i-1}, \frac{|n|}{2 \tb^{i-2}} + \tb^{i-2})
\\& \lesssim (\tb^{i-2})^2 e^{- (\tb^{i-2})^2 / 2b} \sum\limits^{\infty}_{n=0} \mu_{\Theta} (x^{i-1}, \frac{|n|}{2 \tb^{i-2}})
\\& \leq (\tb^{i-2})^2 e^{- (\tb^{i-2})^2 / 2b} 
(1 - e^{-\frac{1}{8b (\tb^{i-2})^2}})^{-1} 
\lesssim (\tb^{i-2})^4 e^{- (\tb^{i-2})^2 / 2b},
\end{split}
\Ee
where the last inequality holds from the Taylor expansion. 

For $|m| \geq (\tb^{i-2})^2$ and $|n| < (\tb^{i-2})^2$ case, we bound
$|v^{m, n}_{i-2, \mathbf{b}_1}| \gtrsim \frac{|m|}{2 \tb^{i-2}}$. Similar as in \eqref{t>1m<n>}, we get
\Be \label{t>1m>n<}
\sum\limits_{|m| \geq (\tb^{i-2})^2, |n| < (\tb^{i-2})^2} \mu_{\Theta} (x^{i-1}, v^{m, n}_{i-2, \mathbf{b}}) \lesssim (\tb^{i-2})^4 e^{- (\tb^{i-2})^2 / 2b}.
\Ee

For $|m| \geq (\tb^{i-2})^2$ and $|n| \geq (\tb^{i-2})^2$, we need to use two lower bounds
$|v^{m, n}_{i-2, \mathbf{b}_2}| \gtrsim \frac{|n|}{2 \tb^{i-2}}$, $|v^{m, n}_{i-2, \mathbf{b}_1}| \gtrsim \frac{|m|}{2 \tb^{i-2}}$. Then, we derive
\Be \label{t>1m>n>}
\begin{split}
\sum\limits_{|m| \geq (\tb^{i-2})^2, |n| \geq (\tb^{i-2})^2} \mu_{\Theta} (x^{i-1}, v^{m, n}_{i-2, \mathbf{b}}) & \lesssim \sum\limits^{\infty}_{m, n = 0} \mu_{\Theta} (x^{i-1}, \frac{|m|}{2 \tb^{i-2}} + \frac{|n|}{2 \tb^{i-2}} + \tb^{i-2})
\\& \lesssim e^{- (\tb^{i-2})^2 / 2b} \sum\limits^{\infty}_{n=0} \mu_{\Theta} (x^{i-1}, \frac{|n|}{2 \tb^{i-2}})
(1 - e^{-\frac{1}{8 (\tb^{i-2})^2}})^{-1} 
\\& \lesssim e^{- (\tb^{i-2})^2 / 2b} 
(1 - e^{-\frac{1}{8 (\tb^{i-2})^2}})^{-2} 
\lesssim (\tb^{i-2})^4 e^{- (\tb^{i-2})^2 / 2b}.
\end{split}
\Ee
From \eqref{t>1m<n<}, \eqref{t>1m<n>}, \eqref{t>1m>n<} and \eqref{t>1m>n>}, we conclude that, for $\tb^{i-2} \geq 1$,
\Be \notag
\sum\limits_{m, n \in \mathbb{Z}} \mu_{\Theta} (x^{i-1}, v^{m, n}_{i-2, \mathbf{b}}) \lesssim (\tb^{i-2})^4 e^{- (\tb^{i-2})^2 / 2b}.
\Ee

\textbf{Case 2:} $0 \leq \tb^{i-2} < 1$.
In this case $\tb^{i-2}$ is small, thus from \eqref{vmnb12 first estimate} and $\varrho_3 \ll 1$, for $|m| \geq 2$ and $|n| \geq 2$, we bound 
\Be \notag
\begin{split}
& \frac{|x^{i-1}_1 + m - x^{i-2}_1|}{\tb^{i-2}} - \varrho_3 \tb^{i-2} \gtrsim \frac{|x^{i-1}_1 + m - x^{i-2}_1|}{2 \tb^{i-2}} \gtrsim \frac{|m|}{2 \tb^{i-2}},
\\& \frac{|x^{i-1}_2 + n - x^{i-2}_2|}{\tb^{i-2}} - \varrho_3 \tb^{i-2} \gtrsim \frac{|x^{i-1}_2 + n - x^{i-2}_2|}{2 \tb^{i-2}} \gtrsim \frac{|n|}{2 \tb^{i-2}}.
\end{split}
\Ee
For $|m| < 2$, we bound $|v^{m, n}_{i-2, \mathbf{b}_1}| \geq 0$, and for $|n| < 2$, we bound $|v^{m, n}_{i-2, \mathbf{b}_2}| \geq 0$. Note that from \eqref{v3t estimate}, we know
$\tb^{i-2} \lesssim v^{m, n}_{i-2, \mathbf{b}_3} \lesssim \tb^{i-2}$. 
To obtain \eqref{t<1,i-2}, we again divide $\{v^{m, n}_{i-2, \mathbf{b}}\}_{m, n \in \mathbb{Z}}$ into four parts.

For $|m| < 2$ and $|n| < 2$, we keep the following five terms summation: 
\Be \label{t<1m<n<}
\sum\limits_{|m| < 2, |n| < 2} \mu_{\Theta} (x^{i-1}, v^{m, n}_{i-2, \mathbf{b}}).
\Ee

For $|m| < 2$ and $|n| \geq 2$, we bound
$|v^{m, n}_{i-2, \mathbf{b}_2}| \gtrsim \frac{|n|}{2 \tb^{i-2}}$. Thus, we have 
\Be \label{t<1m<n>}
\begin{split}
\sum\limits_{|m| < 2, |n| \geq 2} \mu_{\Theta} (x^{i-1}, v^{m, n}_{i-2, \mathbf{b}})
& \lesssim \sum\limits_{|m| < 2, |n| \geq 2} \mu_{\Theta} (x^{i-1}, \frac{|n|}{2 \tb^{i-2}} + \tb^{i-2})
\\& \lesssim 3 e^{- (\tb^{i-2})^2 / 2b} \sum\limits^{\infty}_{n=2} \mu_{\Theta} (x^{i-1}, \frac{|n|}{2 \tb^{i-2}})
\lesssim \sum\limits^{\infty}_{n = 2} e^{-\frac{n^2}{8b (\tb^{i-2})^2}}
\\& \leq e^{-\frac{1}{2b (\tb^{i-2})^2}}
(1 - e^{-\frac{1}{8b (\tb^{i-2})^2}})^{-1} 
\lesssim e^{-\frac{1}{2b (\tb^{i-2})^2}},
\end{split}
\Ee
where the last inequality holds from $0 \leq \tb^{i-2} < 1$. 

For $|m| \geq 2$ and $|n| < 2$ case, we bound
$|v^{m, n}_{i-2, \mathbf{b}_1}| \gtrsim \frac{|m|}{2 \tb^{i-2}}$. Similar as in \eqref{t<1m<n>}, we get
\Be \label{t<1m>n<}
\sum\limits_{|m| \geq 2, |n| < 2} \mu_{\Theta} (x^{i-1}, v^{m, n}_{i-2, \mathbf{b}}) \lesssim e^{-\frac{1}{2b (\tb^{i-2})^2}}.
\Ee

For $|m| \geq 2$ and $|n| \geq 2$, we bound
$|v^{m, n}_{i-2, \mathbf{b}_2}| \gtrsim \frac{|n|}{2 \tb^{i-2}}$, $|v^{m, n}_{i-2, \mathbf{b}_1}| \gtrsim \frac{|m|}{2 \tb^{i-2}}$ and we derive
\Be \label{t<1m>n>}
\begin{split}
\sum\limits_{|m| \geq 2, |n| \geq 2} \mu_{\Theta} (x^{i-1}, v^{m, n}_{i-2, \mathbf{b}}) & \lesssim \sum\limits^{\infty}_{m, n = 2} \mu_{\Theta} (x^{i-1}, \frac{|m|}{2 \tb^{i-2}} + \frac{|n|}{2 \tb^{i-2}} + \tb^{i-2})
\\& \lesssim e^{- (\tb^{i-2})^2 / 2b} \sum\limits^{\infty}_{n=2} \mu_{\Theta} (x^{i-1}, \frac{|n|}{2 \tb^{i-2}})
e^{-\frac{1}{2b (\tb^{i-2})^2}}
\lesssim e^{-\frac{1}{2b (\tb^{i-2})^2}}.
\end{split}
\Ee
From \eqref{t<1m<n<}, \eqref{t<1m<n>}, \eqref{t<1m>n<} and \eqref{t<1m>n>}, we conclude that, for $0 \leq \tb^{i-2} < 1$,
\Be \notag
\sum\limits_{m, n \in \mathbb{Z}} \mu_{\Theta} (x^{i-1}, v^{m, n}_{i-2, \mathbf{b}})
\lesssim 
\sum\limits_{|m| < 2, |n| < 2} \mu_{\Theta} (x^{i-1}, v^{m, n}_{i-2, \mathbf{b}}) +
e^{-\frac{1}{2b (\tb^{i-2})^2}},
\Ee
so we prove \eqref{t>1,i-2} and \eqref{t<1,i-2}.
\end{proof}

Now we can control \eqref{expand_G1} via the following lemma:

\begin{lemma}  \label{lem:bound1}
For $i = 2, \cdots , k-1$ and  $(X,V)$ in \eqref{characteristics}, we have
\Be \label{bound1:expand_h}
  \Big|\int_{\prod_{j=0}^{i} \mathcal{V}_j}   
  \mathbf{1}_{t^{i+1} < 0 \leq t^{i}} F(X(0; t^i, x^i, v^i), V(0; t^i, x^i, v^i)) 
  \dd \tilde{\Sigma}_{i} \Big|
  \lesssim 2^{i} \times \| F(x, v) \|_{L^1_{x,v}}.
\Ee
where 
$\dd \tilde{\Sigma}_{i} := \frac{ \dd \sigma_{i}}{\mu_{\Theta} (x^{i+1}, v^{i})} \dd \sigma_{i-1} \cdots \dd \sigma_1 \dd \sigma_0$ with $\dd \sigma_j = \mu_{\Theta} (x^{j+1}, v^{j}) \{ n(x^j) \cdot v^j \} \dd v^j$ in \eqref{def:sigma measure}.
\end{lemma}

\begin{proof}
For \eqref{bound1:expand_h}, it suffices to prove this upper bound for $i = 2,...,k-1$, 
\Be \label{forcing}
\int_{\mathcal{V}_0} \cdots \int_{\mathcal{V}_{i-1}} 
\int_{\mathcal{V}_i} 
\mathbf{1}_{t^{i+1} < 0 \leq t^i}
|F (X(0; t^i, x^i, v^i), V(0; t^i, x^i, v^i))|  \{n(x^i) \cdot v^i\} \dd v^i \dd \sigma_{i-1}  \cdots \dd \sigma_0.
\Ee  

\textbf{Step 1.} 
Applying Proposition \ref{prop:mapV} , \eqref{integrand} and \eqref{est1:forcing}, 
we bound the above integration as  
\Be \label{est2:forcing}
\begin{split}
\eqref{forcing} \lesssim 
&\int_{\mathcal{V}_0} \dd \sigma_0 \cdots \int_{\mathcal{V}_{i-3}} \dd \sigma_{i-3}
\int^{t^{i-2}}_0 \dd \tb^{i-1}
 \int_0^{t^{i-2}- \tb^{i-1}} \dd \tb^{i-2} \int_{\p\O} \dd S_{x^i}
\\& \times
\underbrace{\bigg(\int_{\p\O} 
\frac{\sum\limits_{m, n \in \mathbb{Z}} \mu_{\Theta} (x^{i-1}, v^{m, n}_{i-2, \mathbf{b}})}{|\tb^{i-2}|} 
\times
\frac{\sum\limits_{m, n \in \mathbb{Z}} \mu_{\Theta} (x^{i}, v^{m, n}_{i-1, \mathbf{b}})}{|\tb^{i-1}|}   \dd S_{x^{i-1}} \bigg)}_{\eqref{est2:forcing}_*}
\eqref{est1:forcing}.
\end{split}\Ee

\textbf{Step 2.} We claim that 
\Be \label{est3:forcing}
\eqref{est2:forcing}_* \lesssim \mathbf{1}_{\tb^{i-1} \leq \tb^{i-2}}
\langle \tb^{i-2}\rangle^{-5}
+  \mathbf{1}_{\tb^{i-1} \geq \tb^{i-2}}
\langle \tb^{i-1}\rangle^{-5}.
\Ee

We split the cases: 
\textit{Case 1:} $\tb^{i-1} \leq \tb^{i-2}$.
Using \eqref{t>1,i-2} and \eqref{t<1,i-2} in Lemma \ref{lem: sum of mu}, we bound 
\Be \label{bound1_1and5}
\begin{split}
\frac{\sum\limits_{m, n \in \mathbb{Z}} \mu_{\Theta} (x^{i-1}, v^{m, n}_{i-2, \mathbf{b}})}{|\tb^{i-2}|}  
\lesssim  
\mathbf{1}_{\tb^{i-2} \leq 1} \frac{1}{
|\tb^{i-2}| 
}+ \mathbf{1}_{\tb^{i-2} \geq 1}
\frac{1}{|\tb^{i-2}|^5}, 
\end{split}
\Ee
Replacing $i$ with $i+1$ in \eqref{t>1,i-2} and \eqref{t<1,i-2}, we bound
\Be \label{bound2_1and5}
\begin{split}
\frac{\sum\limits_{m, n \in \mathbb{Z}} \mu_{\Theta} (x^{i}, v^{m, n}_{i-1, \mathbf{b}})}{|\tb^{i-1}|}  
& \lesssim 
\underbrace{\sum\limits_{|a| < 2, |b| < 2} \mathbf{1}_{\tb^{i-1} \leq 1} \mathbf{1}_{|x^i + (a, b) - x^{i-1}| \leq 2 \varrho_3 (\tb^{i-1})^2} \frac{1}{|\tb^{i-1}|}}_{\eqref{bound2_1and5}_1}
\\& + \underbrace{\sum\limits_{|a| < 2, |b| < 2} \mathbf{1}_{\tb^{i-1} \leq 1} \mathbf{1}_{|x^i + (a, b) - x^{i-1}| \geq 2 \varrho_3 (\tb^{i-1})^2} \frac{1}{|\tb^{i-1}|} \mu \big(x^{i}, \frac{|x^i + (a, b) - x^{i-1}|}{2|\tb^{i-1}|} \big)}_{\eqref{bound2_1and5}_2} 
\\& + \underbrace{\mathbf{1}_{\tb^{i-1} \leq 1} \frac{1}{|\tb^{i-1}|} e^{-\frac{1}{2b (\tb^{i-1})^2}}}_{\eqref{bound2_1and5}_3} 
+ \underbrace{
\mathbf{1}_{\tb^{i-1} \geq 1} 
\mu_{\Theta} (x^{i}, |\tb^{i-1}|)}_{\eqref{bound2_1and5}_4}.
\end{split}
\Ee

For $\eqref{bound2_1and5}_1$, we have 
\Be \label{bound3_1and5_1}
\begin{split}
& \ \ \ \ \mathbf{1}_{\tb^{i-1} \leq \tb^{i-2}} \int_{\p\O} 
\frac{\sum\limits_{m, n \in \mathbb{Z}} \mu_{\Theta} (x^{i-1}, v^{m, n}_{i-2, \mathbf{b}})}{|\tb^{i-2}|} 
\times \eqref{bound2_1and5}_1 \ \dd S_{x^{i-1}} 
\\& \lesssim \sum\limits_{|a| < 2, |b| < 2} \mathbf{1}_{\tb^{i-1} \leq \tb^{i-2}} \eqref{bound1_1and5} \bigg(\int_{\tb^{i-1} \leq 1, |x^i + (a, b) - x^{i-1}| \leq 2 \varrho_3 (\tb^{i-1})^2}  
\frac{1}{|\tb^{i-1}|}  \dd S_{x^{i-1}} \bigg)
\\& \lesssim \mathbf{1}_{\tb^{i-1} \leq \tb^{i-2}}
 \Big\{
 \mathbf{1}_{\tb^{i-2} \leq 1} \frac{1}{
|\tb^{i-2}| 
}+ \mathbf{1}_{\tb^{i-2} \geq 1}
\frac{1}{|\tb^{i-2}|^5} \Big\}
\mathbf{1}_{\tb^{i-1} \leq 1} \frac{(\varrho_3 (\tb^{i-1})^2)^2}{|\tb^{i-1}|}
\\& \leq \mathbf{1}_{\tb^{i-1} \leq \tb^{i-2}} \frac{(\tb^{i-1})^3}{|\tb^{i-2}|}
\mathbf{1}_{\tb^{i-2} \leq 1} 
+ \mathbf{1}_{\tb^{i-2} \geq 1}
\frac{1}{|\tb^{i-2}|^5}
\\& \leq \mathbf{1}_{\tb^{i-2} \leq 1} + \mathbf{1}_{\tb^{i-2} \geq 1}
\frac{1}{|\tb^{i-2}|^5}. 
\end{split}
\Ee

For $\eqref{bound2_1and5}_2$, we employ a change of variables, for $x^{i} \in \p\O$, $|a| < 2$, $|b| < 2$ and $\tb^{i-1}\geq0$, $x^{i-1} \in  \p\O
\mapsto z := \frac{1}{\tb^{i-1}} (x^{i-1} + (a, b) - x^{i}) \in \mathfrak{S}^{a, b}_{x^{i}, \tb^{i-1}},$ where the image $\mathfrak{S}^{a, b}_{x^{i}, \tb^{i-1}}$ of the map is a two dimensional smooth plane. Using the local chart of $\p\O$ we have 
$\dd S_{x^{i-1}} \lesssim |\tb^{i-1}|^2 \dd S_z.$
From this change of variables and \eqref{bound1_1and5}, we conclude that
\Be \label{bound3_1and5_2}
\begin{split}
& \ \ \ \ \mathbf{1}_{\tb^{i-1} \leq \tb^{i-2}} \mathbf{1}_{|x^i + (a, b) - x^{i-1}| \geq 2 \epsilon (\tb^{i-1})^2} \int_{\p\O} 
\frac{\sum\limits_{m, n \in \mathbb{Z}} \mu_{\Theta} (x^{i-1}, v^{m, n}_{i-2, \mathbf{b}})}{|\tb^{i-2}|} 
\times \eqref{bound2_1and5}_2 \ \dd S_{x^{i-1}}
\\& \lesssim \mathbf{1}_{\tb^{i-1} \leq \tb^{i-2}}
\eqref{bound1_1and5}
\sum\limits_{|a| < 2, |b| < 2}
\int_{ \mathfrak{S}^{a, b}_{x^{i}, \tb^{i-1}} } \mathbf{1}_{\tb^{i-1} \leq 1} e^{- z^2 / 8b} |\tb^{i-1}| \dd S_{z} 
\\& \lesssim  \mathbf{1}_{\tb^{i-1} \leq \tb^{i-2}}
 \Big\{
 \mathbf{1}_{\tb^{i-2} \leq 1} \frac{1}{
|\tb^{i-2}|} + \mathbf{1}_{\tb^{i-2} \geq 1}
\frac{1}{|\tb^{i-2}|^5} \Big\}
\mathbf{1}_{\tb^{i-1} \leq 1}|\tb^{i-1}|
\\& \lesssim \mathbf{1}_{\tb^{i-1} \leq \tb^{i-2}}\Big\{
 \mathbf{1}_{\tb^{i-2} \leq 1} \frac{|\tb^{i-1}|}{|\tb^{i-2}|} + \mathbf{1}_{\tb^{i-2} \geq 1}
\frac{1}{|\tb^{i-2}|^5}
 \Big\} 
\lesssim \mathbf{1}_{\tb^{i-2} \leq 1} + \mathbf{1}_{\tb^{i-2} \geq 1}
\frac{1}{|\tb^{i-2}|^5}.
\end{split}
\Ee
 
For $\eqref{bound2_1and5}_3$, since $e^{-\frac{1}{2b t^2}} \lesssim t^2$ for $0 < t \leq 1$, and we have that 
\Be \label{bound3_1and5_3}
\begin{split}
& \ \ \ \ \mathbf{1}_{\tb^{i-1} \leq \tb^{i-2}} \int_{\p\O} 
\frac{\sum\limits_{m, n \in \mathbb{Z}} \mu_{\Theta} (x^{i-1}, v^{m, n}_{i-2, \mathbf{b}})}{|\tb^{i-2}|} 
\times \eqref{bound2_1and5}_3 \ \dd S_{x^{i-1}}
\\& \lesssim \mathbf{1}_{\tb^{i-1} \leq \tb^{i-2}} \eqref{bound1_1and5}
\mathbf{1}_{\tb^{i-1} \leq 1} \frac{1}{|\tb^{i-1}|} e^{-\frac{1}{2b (\tb^{i-1})^2}}
\int_{\p\O}  
 \dd S_{x^{i-1}} 
\\& \lesssim \mathbf{1}_{\tb^{i-1} \leq \tb^{i-2}}
 \Big\{
 \mathbf{1}_{\tb^{i-2} \leq 1} \frac{1}{
|\tb^{i-2}| 
} + \mathbf{1}_{\tb^{i-2} \geq 1}
\frac{1}{|\tb^{i-2}|^5} \Big\}
\mathbf{1}_{\tb^{i-1} \leq 1} |\tb^{i-1}|
\\& \leq \mathbf{1}_{\tb^{i-1} \leq \tb^{i-2}} \frac{|\tb^{i-1}|}{|\tb^{i-2}|}
\mathbf{1}_{\tb^{i-2} \leq 1} 
+ \mathbf{1}_{\tb^{i-2} \geq 1}
\frac{1}{|\tb^{i-2}|^5}
\\& \leq \mathbf{1}_{\tb^{i-2} \leq 1} + \mathbf{1}_{\tb^{i-2} \geq 1}
\frac{1}{|\tb^{i-2}|^5}. 
\end{split}
\Ee

For $\eqref{bound2_1and5}_4$, from
$\mathbf{1}_{\tb^{i-1} \geq 1} 
\mu_{\Theta} (x^{i}, |\tb^{i-1}|) \lesssim \mathbf{1}_{\tb^{i-1} \geq 1}$, we can derive that 
\Be \label{bound3_1and5_4}
\begin{split}
& \ \ \ \ \mathbf{1}_{\tb^{i-1} \leq \tb^{i-2}} \int_{\p\O} 
\frac{\sum\limits_{m, n \in \mathbb{Z}} \mu_{\Theta} (x^{i-1}, v^{m, n}_{i-2, \mathbf{b}})}{|\tb^{i-2}|} 
\times \eqref{bound2_1and5}_4 \ \dd S_{x^{i-1}}
\\& \lesssim \mathbf{1}_{\tb^{i-1} \leq \tb^{i-2}} \eqref{bound1_1and5}
\mathbf{1}_{\tb^{i-1} > 1} \mu_{\Theta} (x^{i}, |\tb^{i-1}|) \int_{\p\O} \dd S_{x^{i-1}} 
\\& \lesssim  \mathbf{1}_{\tb^{i-1} \leq \tb^{i-2}}
 \Big\{
 \mathbf{1}_{\tb^{i-2} \leq 1} \frac{1}{
|\tb^{i-2}| 
}+ \mathbf{1}_{\tb^{i-2} \geq 1}
\frac{1}{|\tb^{i-2}|^5} \Big\} \mathbf{1}_{\tb^{i-1} \geq 1}
\\& \lesssim \mathbf{1}_{\tb^{i-2} \geq 1} \frac{1}{|\tb^{i-2}|^5}.
\end{split}
\Ee

Collecting estimate from \eqref{bound3_1and5_1}-\eqref{bound3_1and5_4}, we deduce that
\Be \label{bound3_1and5}
\mathbf{1}_{\tb^{i-1} \leq \tb^{i-2}} \eqref{est2:forcing}_*  \lesssim 
 \mathbf{1}_{\tb^{i-2} \leq 1} 
+ \mathbf{1}_{\tb^{i-2} \geq 1}
 {|\tb^{i-2}|^{-5}}.
\Ee

\textit{Case 2:} $\tb^{i-1} \geq \tb^{i-2}$.
We change the role of $i-1$ and $i-2$ and follow the argument of the previous case.   We employ a change of variables, for $x^{i-1} \in \p\O$ $|a| < 2$, $|b| < 2$  and $\tb^{i-2}\geq0$, $
x^{i-2} \in  \p\O
\mapsto z := \frac{1}{\tb^{i-2}} (x^{i-2}-x^{i-1}) \in \mathfrak{S}_{x^{i-1}, \tb^{i-2}} ,
$ with $\dd S^{a, b}_{x^{i-2}} \lesssim |\tb^{i-2}|^2 \dd S_z.$ 
Then we can conclude that 
\Be \label{bound4_1and5}
\begin{split}
\mathbf{1}_{\tb^{i-1} \geq \tb^{i-2}} \eqref{est2:forcing}_*  \lesssim 
 \mathbf{1}_{\tb^{i-1} \leq 1} 
+ \mathbf{1}_{\tb^{i-1} \geq 1}
 {|\tb^{i-1}|^{-5}}.
\end{split}
\Ee
Clearly the above bound and \eqref{bound3_1and5} imply \eqref{est3:forcing}. 
 
\textbf{Step 3.} Now we apply \eqref{est3:forcing} on \eqref{est2:forcing}. Then we have 
\begin{align}
\eqref{forcing}
& \lesssim \int_{\mathcal{V}_0} \dd \sigma_0 \cdots \int_{\mathcal{V}_{i-3}} \dd \sigma_{i-3}  
\int^{t^{i-2}}_0 
\frac{\dd \tb^{i-1}}{\langle \tb^{i-1} \rangle^{5}}
\int_0^{\min\{t^{i-2}- \tb^{i-1},  \tb^{i-1}  \}} \dd \tb^{i-2} \int_{\p\O} \dd S_{x^i} \eqref{est1:forcing} \label{est_a:forcing}
\\& + \int_{\mathcal{V}_0} \dd \sigma_0 \cdots \int_{\mathcal{V}_{i-3}} \dd \sigma_{i-3} 
\int^{t^{i-2}}_0 \frac{\dd \tb^{i-2}}{\langle  \tb^{i-2}\rangle^{5}}
\int_0^{ \min\{t^{i-2}- \tb^{i-2},\tb^{i-2}  \}} \dd \tb^{i-1} \int_{\p\O} \dd S_{x^i} \eqref{est1:forcing}.\label{est_b:forcing}
\end{align} 

For \eqref{est_a:forcing}, we employ the change of variables 
\be \notag
(x^{i}, \tb^{i-2}, v^{i}) 
\mapsto (y, w) = (X(0; t^{i-2} -\tb^{i-2} - \tb^{i-1}, x^{i}, v^{i}), V(0; t^{i-2} -\tb^{i-2} - \tb^{i-1}, x^{i}, v^{i})) \in \O \times\R^3. 
\ee
From \eqref{COV}, we derive
$|n(x^i) \cdot v^i| \dd S_{x^i} \dd \tb^{i-2} \dd v^i \lesssim \dd y \dd w$. 
Thus, we bound \eqref{est_a:forcing} as
\Be \notag
\begin{split}
\eqref{est_a:forcing}
& \leq \int_{\mathcal{V}_0} \dd \sigma_0 \cdots \int_{\mathcal{V}_{i-3}} \dd \sigma_{i-3}
\int^{t^{i-2}}_0 \dd \tb^{i-1}\langle \tb^{i-1} \rangle^{-5} 
 \iint_{\O \times\R^3 } 
|F (y, w)|  \dd y \dd w
\\& \lesssim 2^{i} \times \| F(x, v) \|_{L^1_{x,v}}. 
\end{split}
\Ee

A bound of \eqref{est_b:forcing} can be derived similarly, by using the change of variables 
\be \notag
(x^{i}, \tb^{i-1}, v^{i}) 
\mapsto (y, w) = (X(0; t^{i-2} -\tb^{i-2} - \tb^{i-1}, x^{i}, v^{i}), V(0; t_{i-2} -\tb^{i-2} - \tb^{i-1}, x^{i}, v^{i})) \in \O \times\R^3, 
\ee
with $
|n(x^i) \cdot v^i| \dd S_{x^i} \dd \tb^{i-1} \dd v^i \lesssim \dd y \dd w$.
\end{proof}

Finally, we control \eqref{expand_G2} by establishing the following estimate:

\begin{lemma} \label{lem:small_largek}
Consider $(X,V)$ in \eqref{characteristics}, there exists $\mathfrak{C}= \mathfrak{C}(\O)>0$ (see \eqref{choice:k} for the precise choice),
such that  
\Be \label{small_largek}
\text{if }  \  k \geq \mathfrak{C}t,  \text{ then } 
\sup_{(x,v) \in \bar{\O} \times \R^3}  \Big(\int_{\prod_{j=0}^{k -1} \mathcal{V}_j}   
    \mathbf{1}_{t^{k} (t,x,v,v^1,\cdots, v^{k-1}) \geq 0 } \ \dd \sigma_0 \cdots \dd \sigma_{k-1}\Big) \lesssim e^{-t},
\Ee  
where $\dd \sigma_j = \mu_{\Theta} (x^{j+1}, v^{j}) \{ n(x^j) \cdot v^j \} \dd v^j$ in \eqref{def:sigma measure}.
\end{lemma}

\begin{proof}

\textbf{Step 1.} 
Since $|n(x) \cdot v| \lesssim \tb (x, v)$ for $(x,v) \in \gamma_{+}$, we get 
\be \notag
\mathbf{1}_{ |n(x) \cdot v| \lesssim \delta} \geq \mathbf{1} _{\delta> \tb (x,v)}.
\ee
Setting $\vartheta$ as the angle between $v $ and $n(x)$ and $r = |v|$, we have 
\Be \label{estimate on delta}
\begin{split} 
& \ \ \ \ \int_{n(x) \cdot v > 0}
\mathbf{1} _{\delta> \tb (x, v)} \mu_{\Theta} (\xb, v) |n(x) \cdot v|\dd v
\\& \lesssim \int_{ |n(x) \cdot v| \lesssim \delta} \mu_{\Theta} (\xb, v) | n(x) \cdot v | \dd v  \\& \leq \int_{ |n(x) \cdot v| \lesssim \delta} \mu_{\Theta} (\xb, v) \delta \dd v
    \leq C \int^{\infty}_{0} \delta e^{- \frac{r^2}{2b}} r^2 \dd r \int_{\cos \vartheta < \delta / r} \sin \vartheta \dd \vartheta
\\& \leq C \int^{\infty}_{0} \delta e^{- \frac{r^2}{2b}} \delta r \dd r \lesssim C \delta^2.
\end{split}
\Ee
Now we define $\mathcal{V}^{\delta}_i := \{ v^{i} \in \mathcal{V}_i:  {| n(x^{i}) \cdot v^{i} |} < \delta \}$. Recall \eqref{estimate on delta}, we have
\Be \notag
\int_{\mathcal{V}^{\delta}_j} \dd \sigma_j \leq C \delta^2.
\Ee
On the other hand, we have $t_{\mathbf{b}} (x^{i}, v^{i}) \gtrsim {| n(x^{i}) \cdot v^{i} |}$. Therefore, if $v^{i} \in \mathcal{V}_i \backslash \mathcal{V}^{\delta}_i$, we derive that 
\Be \notag
t_{\mathbf{b}} (x^{i}, v^{i}) \geq C_{\Omega} \delta.
\Ee

 If $t_{k}(t,x,v^0,v^1,\cdots, v^{k-1}) \geq 0$, we conclude such $v^{i} \in \mathcal{V}_i \backslash \mathcal{V}^{\delta}_i$ can exist at most $[\frac{t  }{C_{\Omega} \delta}] + 1$ times. Denote the combination 
$\begin{pmatrix}
M \\ N
\end{pmatrix}= \frac{M(M-1) \cdots (M-N+1)}{N(N-1) \cdots 1}= \frac{M!}{N! (M-N)!}$ for $M,N \in \mathbb{N}$ and $M\geq N$. From Remark \ref{sigma measure estiamte} and $0<\delta \ll 1$, we have
\Be \label{sum_bound}
\begin{split} 
& \ \ \ \ \int_{\prod_{j=0}^{k-1} \mathcal{V}_j}   \mathbf{1}_{t_{k} (t,x,v^0,v^1,\cdots, v^{k-1}) \geq 0} \ \dd \sigma_{k-1} \cdots \dd \sigma_0
\\& \leq 2^k \sum\limits^{[\frac{t}{C_{\Omega} \delta}] + 1}_{m=0} 
\begin{pmatrix}
k
 \\
m
\end{pmatrix} 
\big( \int_{\mathcal{V}_i^\delta} \dd \sigma_i \big)^{k-m}
 \leq 2^k (C\delta^2)^{k - [\frac{t  }{C_{\Omega} \delta}]}
  \underbrace{ \sum\limits^{[\frac{t  }{C_{\Omega} \delta}] + 1}_{m=0} 
\begin{pmatrix}
k
 \\
m
\end{pmatrix}}_{\eqref{sum_bound}_*}.
\end{split}
\Ee

\textbf{Step 2.} 
Recall the Stirling's formula,
\Be \label{Stirling}
\sqrt{2 \pi} k^{k+\frac{1}{2}} e^{-k} \leq k ! \leq k^{k+\frac{1}{2}} e^{-k+1}.
\Ee
Using $(1 + \frac{1}{\mathfrak{a}-1})^{\mathfrak{a}-1} \leq e$ and \eqref{Stirling}, we have for $\mathfrak{a} \geq 2$,
\Be \notag
\begin{split}
\begin{pmatrix}
k \\
\frac{k}{\mathfrak{a}}
\end{pmatrix} 
  = \frac{k !}{ (k - \frac{k}{\mathfrak{a}}) ! \frac{k}{ \mathfrak{a}} !} 
& \leq 
 (\frac{\mathfrak{a}}{\mathfrak{a}-1})^{\frac{\mathfrak{a}}{\mathfrak{a}-1} k} \mathfrak{a}^{\frac{k}{\mathfrak{a}}} \sqrt{\frac{\mathfrak{a}^2}{k (\mathfrak{a}-1)}} 
\\&  = 
 \frac{1}{\sqrt{k}} \bigg( \mathfrak{a}^{\frac{1}{\mathfrak{a}}} \big( \frac{\mathfrak{a}}{\mathfrak{a}-1} \big)^{\frac{\mathfrak{a}}{\mathfrak{a}-1}} \bigg)^k \sqrt{\frac{ \mathfrak{a}^2}{\mathfrak{a}-1}}
  \leq
  \frac{1}{\sqrt{k}} (e \mathfrak{a})^{\frac{k}{\mathfrak{a}}} \sqrt{\frac{\mathfrak{a}^2}{ \mathfrak{a}-1}}.
\end{split}
\Ee
Hence, we derive that
\Be \label{est:sum_com}
\begin{split}
\sum\limits^{[\frac{k}{\mathfrak{a}}]}_{i=1} 
\begin{pmatrix}
k \\
i
\end{pmatrix} \leq 
\frac{k}{\mathfrak{a}} 
\begin{pmatrix}
k \\
\frac{k}{\mathfrak{a}}
\end{pmatrix} \leq \frac{e}{2 \pi} \sqrt{\frac{k}{\mathfrak{a}}} (e \mathfrak{a})^{\frac{k}{\mathfrak{a}}}.
\end{split}
\Ee

\textbf{Step 3.} 
Now we estimate $\eqref{sum_bound}_*$. 
For fixed $0< \delta \ll 1/2$ which is independent of $t$, we choose  
\Be \label{choice:k}
\mathfrak{a} \in \mathbb{N} \ \text{ such that }
(4^{\mathfrak{a}} \delta^{2 \mathfrak{a}} e \mathfrak{a})^{\frac{1}{C_\O \delta}} \leq e^{-2}, \ \text{ and set } 
k :=  \mathfrak{a} \Big(\big[\frac{t  }{C_{\Omega} \delta}\big] + 1\Big).
\Ee 
Using \eqref{est:sum_com}, we have
\Be \notag
\eqref{sum_bound}_* \lesssim   \sqrt{\big[\frac{t  }{C_{\Omega} \delta}\big] + 1} \Big( e \frac{k}{[\frac{t  }{C_{\Omega} \delta}] + 1} \Big)^{[\frac{t  }{C_{\Omega} \delta}] + 1} 
\lesssim \sqrt{\big[\frac{t  }{C_{\Omega} \delta}\big] + 1} (e \mathfrak{a})^{[\frac{t  }{C_{\Omega} \delta}] +1 }.
\Ee
Hence, we bound \eqref{sum_bound} by
\Be \notag
2^k (\delta^{2\mathfrak{a}} e \mathfrak{a} )^{[\frac{t  }{C_{\Omega} \delta}] + 1} \sqrt{\big[\frac{t  }{C_{\Omega} \delta}\big] + 1} 
 \lesssim e^{-t}.
\Ee
This completes the proof and we can conclude
\Be \label{est: small_largek}
\begin{split}
\eqref{expand_G2} 
& = \int_{\prod_{j=0}^{k } \mathcal{V}_j} \mathbf{1}_{t^{k } \geq 0 } F (x^{k }, v^{k }) \dd \tilde{\Sigma}_{k} 
\\& = \int_{\prod_{j=0}^{k -1} \mathcal{V}_j}  \mathbf{1}_{t_{k }(t,x,v,v^1,\cdots, v^{k-1}) \geq 0 } \ \dd \sigma_0 \cdots \dd \sigma_{k-1} \int_{\mathcal{V}_k} \mu_{\Theta} (x_{k+1}, v_{k}) \J (x_{k+1}) \{ n(x_k) \cdot v_k \} \dd v_k
\\& \lesssim \| \J (x)\|_{L^{\infty}_x} \times e^{-t}.
\end{split}
\Ee
\end{proof}

\begin{proof}[\textbf{Proof of Proposition \ref{prop:j linfty bound}}]
Using \eqref{est: G11}, \eqref{est: G12}, \eqref{est: small_largek} and Lemma \ref{lem:bound1}, we have
\Be \notag
\J (x) \leq \| \J (x)\|_{L^{\infty}_x} \times e^{- \frac{t^2}{2b}} + \| \J (x)\|_{L^{\infty}_x} \times e^{- \frac{t^2}{8b}} + \sum^{k}_{i = 2} 2^{i} \times \| F(x, v) \|_{L^1_{x,v}} + \| \J (x)\|_{L^{\infty}_x} \times e^{-t}.
\Ee
We can pick sufficiently large but fixed $t$ such that 
\be \label{choice:t}
e^{- \frac{t^2}{2b}} + e^{- \frac{t^2}{8b}} + e^{-t} \leq 1/2,
\ee 
then we derive
\Be \notag
(1 - e^{- \frac{t^2}{2b}} - e^{- \frac{t^2}{8b}} - e^{-t}) \| \J (x)\|_{L^{\infty}_x} \leq \sum^{k}_{i = 2} 2^{i} \times \| F(x, v) \|_{L^1_{x,v}}.
\Ee
Since $t$ and $k = \mathfrak{a} \Big(\big[\frac{t  }{C_{\Omega} \delta}\big] + 1\Big)$ are fixed, we conclude that
\Be \label{est: j}
\| \J (x)\|_{L^{\infty}_x} 
\leq (1 - e^{- \frac{t^2}{2b}} - e^{- \frac{t^2}{8b}} - e^{-t})^{-1} \sum^{k}_{i = 2} 2^{i} \times \| F(x, v) \|_{L^1_{x,v}}
\lesssim 2^{k+1} \times \| F(x, v) \|_{L^1_{x,v}}.
\Ee

Moreover, $F (x ,v) = F (x_\mathbf{b}, \vb) = \mu_{\Theta} (x_\mathbf{b}, \vb) \J (x_\mathbf{b})$, the $L^{\infty}_x$-estimate in \eqref{est: j} implies
\Be \notag
\|\mu^{-1}_{\Theta} (x_\mathbf{b}, \vb) F (x, v)\|_{L^{\infty}_{x,v}} 
= \| \J (x)\|_{L^{\infty}_x}
\lesssim 2^{k+1} \times \| F(x, v) \|_{L^1_{x,v}}.
\Ee

Using $|v|^2/2 + \Phi(x) = |\vb|^2/2$ in Lemma \ref{conservative field} and \eqref{def:Theta}, we get
\be
\mu^{-1}_{\Theta} (\xb, \vb) = 2 \pi \Theta^2(\xb) e^{ \frac{|\vb|^2}{2 \Theta(\xb)}}
\gtrsim e^{ \frac{|v|^2/2 + \Phi(x)}{2 b}}.
\ee
Therefore, we prove the proposition.
\end{proof}

\subsection{With Magnet field} 

Similar as Lemma \ref{sto_cycle J}, the stochastic cycles for the boundary outgoing flux $\J (x)$ with magnet field follows:
\begin{lemma}
Suppose $F (x, v)$ solves \eqref{def:f} and \eqref{diff_F}, with $\J (x)$ defined in \eqref{diff_J}, then for $t \geq 0$ and $k \geq 1$,
\begin{align}
	\J (x)
    & = \int_{\mathcal{V}_0}   
    \mathbf{1}_{t^{1} < 0 }  F (X(0; t, x, v^0), V(0; t, x, v^0)) \{ n(x) \cdot v^0 \} \dd v^0
\label{expand_G11 mag}
	\\& + \int_{\prod_{j=0}^{1} \mathcal{V}_j}   
    \mathbf{1}_{t^{2} < 0 \leq t^{1}} F (X(0; t^1, x^1, v^1), V(0; t^1, x^1, v^1)) \dd \tilde{\Sigma}_{1}
\label{expand_G12 mag}
    \\& 	+ \int_{\prod_{j=0}^{i} \mathcal{V}_j}   
     \sum\limits^{k-1}_{i=2} 
     \Big\{   \mathbf{1}_{t^{i+1} < 0 \leq t^{i }} 		F (X(0; t^i, x^i, v^i), V(0; t^i, x^i, v^i))  \Big\} \dd \tilde{\Sigma}_{i}
\label{expand_G1 mag}
    \\& + \int_{\prod_{j=0}^{k } \mathcal{V}_j}   
    \mathbf{1}_{t^{k } \geq 0 } F (x^{k }, v^{k }) \dd \tilde{\Sigma}_{k},
\label{expand_G2 mag}
\end{align} 
where 
$\dd \tilde{\Sigma}_{i} := \frac{ \dd \sigma_{i}}{\mu_{\Theta} (x^{i+1}, v^{i})} \dd \sigma_{i-1} \cdots \dd \sigma_1 \dd \sigma_0$ with $\dd \sigma_j = \mu_{\Theta} (x^{j+1}, v^{j}) \{ n(x^j) \cdot v^j \} \dd v^j$ in \eqref{def:sigma measure}.
Here, $(X,V)$ in \eqref{characteristics mag}.
\end{lemma} 

To prove Proposition \ref{prop:j linfty bound} for the magnetic case, we start with the estimate on the first two terms in stochastic cycles on $\J (x)$.
Then we use lemma \ref{lem:bound1 mag} and \ref{lem:small_largek mag} to control the rest of terms.


For the first two terms \eqref{expand_G11 mag} and \eqref{expand_G12 mag}. 
Following \eqref{est1: G11}-\eqref{est: G11} and $\tb (x, v) = v_3 / 5$, we derive
\Be \label{est: G11 mag}
\eqref{expand_G11 mag} = \int_{\mathcal{V}_0} \mathbf{1}_{t^{1} < 0 }  F (X(0; t, x, v^0), V(0; t, x, v^0)) \{ n(x) \cdot v^0 \} \dd v^0 
\lesssim \| \J (x)\|_{L^{\infty}_x} \times e^{- \frac{t^2}{2b}}.
\Ee
Similarly, from \eqref{est1: G12}-\eqref{est: G12} and using $\tb (x, v) = v_3 / 5$, we get
\Be \label{est: G12 mag}
\eqref{expand_G12 mag} = \int_{\prod_{j=0}^{1} \mathcal{V}_j} \mathbf{1}_{t^{2} < 0 \leq t^{1}} F (X(0; t^1, x^1, v^1), V(0; t^1, x^1, v^1)) \dd \tilde{\Sigma}_{1}
\lesssim \| \J (x)\|_{L^{\infty}_x} \times e^{- \frac{t^2}{8b}}.
\Ee


To estimate \eqref{expand_G1 mag}, we need to apply Proposition \ref{prop:mapV mag} on $\mathcal{V}_{i-1}$ and derive
\begin{align}
& \ \ \ \ \int_{\mathcal{V}_{i-1}} 
\int_{\mathcal{V}_i} 
\mathbf{1}_{t^{i+1} < 0 \leq t^i}
|F (X(0; t^i, x^i, v^i), V(0; t^i, x^i, v^i)) |  \{n(x^i) \cdot v^i\} \dd v^i \dd \sigma_{i-1} \notag
\\& \lesssim \int^{t^{i-1}}_0 \int_{\p\O}  
\frac{5 B^2_3 \sum\limits_{m, n \in \mathbb{Z}} \mu_{\Theta} (x^{i}, v^{m, n}_{i-1, \mathbf{b}})}{2 - 2 \cos (B_3 \tb^{i-1})} \{ n(x^{i-1}) \cdot v^{i-1} \} \dd \tb^{i-1} \dd S_{x^{i}} \notag
\\& \ \ \ \ \times
\int_{\mathcal{V}_i} 
\mathbf{1}_{t^{i+1} < 0 \leq t^i} \
|F (X(0; t^i, x^i, v^i), V(0; t^i, x^i, v^i))|  \{n(x^i) \cdot v^i\} \dd v^i. \label{est0:forcing mag}
\end{align}

To control the integrand term $\frac{\sum\limits_{m, n \in \mathbb{Z}} \mu_{\Theta} (x^{i}, v^{m, n}_{i-1, \mathbf{b}})}{2 - 2 \cos (B_3 \tbi)}$ in \eqref{est0:forcing mag}, we introduce the following:
\begin{definition} \label{def:BG mag}
For fixed $x \in \O$, we categorize $v \in  \mathcal{V} := \{v \in \R^3: n(x) \cdot v > 0 \}$ into two sets:
\Be \notag
\begin{split} 
& \mathcal{V}^{G_{\mathbf{\e}}} (x) := \{v \in  \mathcal{V}: \frac{\e}{B_3} + \frac{2 k \pi}{B_3}
  \leq \tb (x, v) \leq 
 \frac{2 \pi - \e}{B_3} + \frac{2 k \pi}{B_3} \ \text{with} \ k \in \Z_{+} \},  
\\& \mathcal{V}^{B_{\mathbf{\e}}} (x) :=  \{v \in  \mathcal{V}: - \frac{\e}{B_3} + \frac{2 k \pi}{B_3} \leq \tb (x, v) \leq 
 \frac{\e}{B_3} + \frac{2 k \pi}{B_3} \ \text{with} \ k \in \Z_{+} \}.
\end{split} 
\Ee
In the rest of the paper, we denote
$\mathcal{V}^{B_{\e}}_i (x^i) := \{ v^{i} \in \mathcal{V}^{B_{\mathbf{\e}}} (x^i) \subset \mathcal{V}_{i} \}$ and $\mathcal{V}^{G_{\e}}_i (x^i) := \{ v^{i} \in \mathcal{V}^{G_{\mathbf{\e}}} (x^i) \subset \mathcal{V}_{i} \}$.
\end{definition}

\begin{lemma} \label{lem: Be mag}
For $0 \leq \e \ll 1$ and $(X,V)$ in \eqref{characteristics mag}, we have
\be \label{estimate on Be mag}
\int_{\mathcal{V}^{B_{\e}}} \dd \sigma \lesssim \e.
\ee
\end{lemma}

\begin{proof}
Since $|n(x) \cdot v| = | v_3 |  = 5 \tb (x, v)$ for $(x,v) \in \gamma_{+}$ and $0 \leq \e \ll 1$, we get  
\Be \notag 
\begin{split} 
& \ \ \int_{v \in \mathcal{V}^{B_{\mathbf{\e}}} (x)}
\mu_{\Theta} (\xb, v) |n(x) \cdot v|\dd v 
\\& \lesssim \int_{v \in \mathcal{V}^{B_{\mathbf{\e}}} (x)} e^{- \frac{| v_3 |^2}{2b}} | v_3 | \dd v_3 
= \int_{v \in \mathcal{V}^{B_{\mathbf{\e}}} (x)} e^{- \frac{25 \tb^2}{2b}} \times 25 \tb \dd \tb 
\\& \lesssim - e^{- \tb^2 / b} \Bigg|^{\frac{\e}{B_3}}_{0} + \sum^{\infty}_{k = 1} 
\int_{- \frac{\e}{B_3} \leq \tb - \frac{2 k \pi}{B_3} \leq 
 \frac{\e}{B_3} } e^{- \tb^2 / b} \dd \tb
\\& \lesssim 
\e^2 + \e \times \sum^{\infty}_{k = 1} e^{- (\frac{2 k \pi}{B_3})^2 / b}
\lesssim \e,
\end{split}
\Ee
where the last line follows from the Taylor expansion.
Therefore, we conclude \eqref{estimate on Be mag}.
\end{proof}

\begin{lemma} \label{lem: sum of mu mag}
Consider $(X,V)$ in \eqref{characteristics mag} with $0 \leq \e \ll 1$ and $x^{i+1} \in \p\O$, for $v^{m, n}_{i} \in \mathcal{V}^{G_{\mathbf{\e}}} (x^{i+1})$,
\Be \label{t> i mag}
\sum\limits_{m, n \in \mathbb{Z}} \mu_{\Theta} (x^{i+1}, v^{m, n}_{i, \mathbf{b}}) 
\lesssim e^{- 25 (\tbi)^2 / 2b}.
\Ee
For $v^{m, n}_{i} \in \mathcal{V}^{B_{\mathbf{\e}}} (x^{i+1})$,
\Be \label{t< i mag}
\sum\limits_{m, n \in \mathbb{Z}} \mu_{\Theta} (x^{i+1}, v^{m, n}_{i, \mathbf{b}})
\lesssim 
\sum\limits_{|m| < 2, |n| < 2} \mu_{\Theta} (x^{i+1}, v^{m, n}_{i, \mathbf{b}}) +
e^{-\frac{B^2_3 }{b (1 - \cos (B_3 \tbi))}} e^{- 25 (\tbi)^2 / 2b},
\Ee
where $v^{m, n}_{i, \mathbf{b}} = \vb (x^{i+1}, v^{m, n}_{i})$, which was defined in \eqref{def:t_k} and \eqref{def:vmn mag}.
\end{lemma} 

\begin{proof}
Here, for the sake of simplicity, we have abused the notations temporarily: 
\Be \notag
x^i = (x^i_1, x^i_2, x^i_3)= (x^i_\parallel, x^i_3), \ 
v^{m, n}_{i, \mathbf{b}} = (v^{m, n}_{i, \mathbf{b}_1}, v^{m, n}_{i, \mathbf{b}_2}, v^{m, n}_{i, \mathbf{b}_3}) = (v^{m, n}_{i, \mathbf{b}_\parallel}, v^{m, n}_{i, \mathbf{b}_3}).
\ee
Note that from \eqref{tb expression mag}, we know $|v^{m, n}_{i, \mathbf{b}_3}|^2 = 25 ( \tb^{i})^2$.
From \eqref{xb1 expression mag} and \eqref{xb2 expression mag}, we have
\Be \notag
\begin{split}
& \ \ \ \ B^2_3 |x_{\mathbf{b}, 1} - x_1|^2 + B^2_3 |x_{\mathbf{b}, 2} - x_2|^2
\\& = (v_2 - v_2 \cos (B_3 \tb) - v_1 \sin (B_3 \tb))^2 + (- v_1 + v_1 \cos (B_3 \tb) - v_2 \sin (B_3 \tb))^2
\\& = (v^2_1 + v^2_2) \times (2 - 2 \cos (B_3 \tb)).
\end{split}
\ee
Using \eqref{V12 conservative mag} and the above equality, we obtain
\Be \label{vmnb12 first estimate mag}
\big| v^{m, n}_{i, \mathbf{b}_\parallel} \big|^2 
= |v^{m, n}_{i, \mathbf{b}_1}|^2 + |v^{m, n}_{i, \mathbf{b}_2}|^2
= |v^{m, n}_{i_1}|^2 + |v^{m, n}_{i_2}|^2 = \frac{B^2_3 \big( |x^{i+1}_1 - x^i_1 + m|^2 + |x^{i+1}_2 - x^i_2 + n|^2 \big)}{2 - 2 \cos (B_3 \tbi)}.
\Ee
Now we split the length of $\tbi$ into two cases: 

\textbf{Case 1:} $v^{m, n}_{i} \in \mathcal{V}^{G_{\mathbf{\e}}} (x^{i+1})$.
From \eqref{vmnb12 first estimate mag}, for $|m| \geq 2$ and $|n| \geq 2$, we bound 
\Be \label{m>n> first estimate mag}
\big| v^{m, n}_{i, \mathbf{b}_\parallel} \big|^2 
\gtrsim \frac{B^2_3}{4} ( m^2 + n^2 )
\Ee
For $|m| < 2$, we bound $|v^{m, n}_{i-2, \mathbf{b}_1}| \geq 0$, and for $|n| < 2$, we bound $|v^{m, n}_{i-2, \mathbf{b}_2}| \geq 0$.  
To derive \eqref{t> i mag}, we divide $\{v^{m, n}_{i, \mathbf{b}}\}_{m, n \in \mathbb{Z}}$ into four parts.

For $|m| < 2$ and $|n| < 2$, we collect and bound all terms as
\be \label{t> m<n< mag}
\sum\limits_{|m| < 2, |n| < 2} \mu_{\Theta} (x^{i+1}, v^{m, n}_{i, \mathbf{b}}) 
\lesssim e^{- 25 (\tbi)^2 / 2b}.
\ee

For $|m| < 2$ and $|n| \geq 2$, we bound
$|v^{m, n}_{i-2, \mathbf{b}_2}| \gtrsim \frac{B_3}{2} |n|$. Thus, we have 
\Be \label{t> m<n> mag}
\begin{split}
 \sum\limits_{|m| < 2, |n| \geq 2} \mu_{\Theta} (x^{i+1}, v^{m, n}_{i, \mathbf{b}})
& \lesssim \sum\limits_{|m| < 2} \mu_{\Theta} (x^{i+1}, \frac{B_3}{2} |n| + 5 \tbi)
\\& \lesssim e^{- 25 (\tbi)^2 / 2b} \sum\limits^{\infty}_{n=0} \mu_{\Theta} (x^{i+1}, \frac{B_3}{2} |n|)
\lesssim e^{- 25 (\tbi)^2 / 2b},
\end{split}
\Ee
where the last inequality holds from $\sum\limits^{\infty}_{n=0} \frac{1}{n^2} \leq \infty$. Similarly, for $|m| \geq 2$ and $|n| < 2$ case, we bound
$|v^{m, n}_{i-2, \mathbf{b}_1}| \gtrsim \frac{B_3}{2} |m|$, and get
\Be \label{t> m>n< mag}
\sum\limits_{|m| \geq 2, |n| < 2} \mu_{\Theta} (x^{i-1}, v^{m, n}_{i-2, \mathbf{b}}) \lesssim e^{- 25 (\tbi)^2 / 2b}.
\Ee

For $|m| \geq 2$ and $|n| \geq 2$, from \eqref{m>n> first estimate mag} and $\sum\limits_{m, n} \frac{1}{m^2 + n^2} \leq \infty$, we obtain
\Be \label{t> m>n> mag}
\sum\limits_{|m| \geq 2, |n| \geq 2} \mu_{\Theta} (x^{i-1}, v^{m, n}_{i-2, \mathbf{b}}) \lesssim e^{- 25 (\tbi)^2 / 2b}.
\Ee
Using \eqref{t> m<n< mag}, \eqref{t> m<n> mag}, \eqref{t> m>n< mag} and \eqref{t> m>n> mag}, we conclude \eqref{t> i mag}.

\textbf{Case 2:} $v^{m, n}_{i} \in \mathcal{V}^{B_{\mathbf{\e}}} (x^{i+1})$.
Since in this case $\tbi$ is small, we bound 
$|v^{m, n}_{i-2, \mathbf{b}_3}| \geq 0$.
Again, for $|m| < 2$, we bound $|v^{m, n}_{i-2, \mathbf{b}_1}| \geq 0$, and for $|n| < 2$, we bound $|v^{m, n}_{i, \mathbf{b}_2}| \geq 0$. 
To obtain \eqref{t< i mag}, again we divide $\{v^{m, n}_{i, \mathbf{b}}\}_{m, n \in \mathbb{Z}}$ into four parts.

For $|m| < 2$ and $|n| < 2$, we collect and keep all terms 
\Be \label{t< m<n< mag}
\sum\limits_{|m| < 2, |n| < 2} \mu_{\Theta} (x^{i+1}, v^{m, n}_{i, \mathbf{b}}).
\Ee

For $|m| < 2$ and $|n| \geq 2$, we bound
$|v^{m, n}_{i, \mathbf{b}_2}|^2 \geq \frac{B^2_3 n^2}{2 - 2 \cos (B_3 \tbi)}$, and have 
\Be \label{t< m<n> mag}
\begin{split}
 \sum\limits_{|m| < 2, |n| \geq 2} \mu_{\Theta} (x^{i+1}, v^{m, n}_{i, \mathbf{b}})
  \lesssim e^{- 25 (\tbi)^2 / 2b} \sum\limits^{\infty}_{n = 2} e^{-\frac{B^2_3 n^2}{4b (1 - \cos (B_3 \tbi))}}
  \lesssim e^{-\frac{B^2_3 }{b (1 - \cos (B_3 \tbi))}} e^{- 25 (\tbi)^2 / 2b},
\end{split}
\Ee

For $|m| \geq 2$ and $|n| < 2$, we bound
$|v^{m, n}_{i, \mathbf{b}_1}|^2 \geq \frac{B^2_3 m^2}{2 - 2 \cos (B_3 \tbi)}$, and get
\Be \label{t< m>n< mag}
\begin{split}
\sum\limits_{|m| \geq 2, |n| < 2} \mu_{\Theta} (x^{i+1}, v^{m, n}_{i, \mathbf{b}})
& \lesssim e^{- 25 (\tbi)^2 / 2b} \sum\limits^{\infty}_{m = 2} e^{-\frac{B^2_3 m^2}{4b (1 - \cos (B_3 \tbi))}}
\\& \lesssim e^{-\frac{B^2_3 }{b (1 - \cos (B_3 \tbi))}} e^{- 25 (\tbi)^2 / 2b},
\end{split}
\Ee

For $|m| \geq 2$ and $|n| \geq 2$, we bound
$\big| v^{m, n}_{i, \mathbf{b}_\parallel} \big|^2 
\geq \frac{B^2_3 ( m^2 +  n^2 )}{2 - 2 \cos (B_3 \tbi)}$, and derive
\Be \label{t< m>n> mag}
\begin{split}
\sum\limits_{|m| \geq 2, |n| \geq 2} \mu_{\Theta} (x^{i+1}, v^{m, n}_{i, \mathbf{b}}) 
& \lesssim \sum\limits^{\infty}_{m, n = 2} e^{-\frac{B^2_3 ( m^2 +  n^2 )}{4b (1 - \cos (B_3 \tbi))}}
\\& \lesssim e^{- \frac{2 B^2_3 }{b (1 - \cos (B_3 \tbi))}} e^{- 25 (\tbi)^2 / 2b}.
\end{split}
\Ee
From \eqref{t< m<n< mag}, \eqref{t< m<n> mag}, \eqref{t< m>n< mag} and \eqref{t< m>n> mag}, we conclude \eqref{t< i mag}.
\end{proof} 

Now we are ready to control \eqref{expand_G1 mag} via the following lemma:

\begin{lemma} \label{lem:bound1 mag}
For $i = 2, \cdots, k-1$, $0 \leq \e \ll 1$ and $(X,V)$ in \eqref{characteristics mag}, we have
\Be \label{bound1:expand_h mag}
  \Big|\int_{\prod_{j=0}^{i} \mathcal{V}_j}   
  \mathbf{1}_{t^{i+1} < 0 \leq t^{i}} F(X(0; t^i, x^i, v^i), V(0; t^i, x^i, v^i)) 
  \dd \tilde{\Sigma}_{i} \Big|
  \lesssim 2^{i} \e^{-2} \| F(x, v) \|_{L^1_{x,v}} + \e^{i-1} \| \J (x)\|_{L^{\infty}_x}.
\Ee
where 
$\dd \tilde{\Sigma}_{i} := \frac{ \dd \sigma_{i}}{\mu_{\Theta} (x^{i+1}, v^{i})} \dd \sigma_{i-1} \cdots \dd \sigma_1 \dd \sigma_0$ with $\dd \sigma_j = \mu_{\Theta} (x^{j+1}, v^{j}) \{ n(x^j) \cdot v^j \} \dd v^j$ in \eqref{def:sigma measure}.
\end{lemma}

\begin{proof}
For \eqref{bound1:expand_h mag}, it suffices to prove this upper bound for $i = 2,...,k-1$, 
\Be \label{forcing mag}
\int_{\mathcal{V}_0} \cdots \int_{\mathcal{V}_{i-1}} 
\int_{\mathcal{V}_i} 
\mathbf{1}_{t^{i+1} < 0 \leq t^i}
|F (X(0; t^i, x^i, v^i), V(0; t^i, x^i, v^i))|  \{n(x^i) \cdot v^i\} \dd v^i \dd \sigma_{i-1}  \cdots \dd \sigma_0.
\Ee  

\textbf{Step 1.} Recall from Definition \ref{def:BG mag}, we obtain for $j = 0, \cdots, i-1$,
\be \notag
v^{j} \in \mathcal{V}^{B_{\e}}_j, \ \text{or} \ v^{j} \in \mathcal{V}^{G_{\e}}_j.
\ee
Now we split $\{v^{j}\}^{i-1}_{j=0}$ into two cases: 

\textit{Case 1:} For all $0 \leq j \leq i-2$, $v^{j} \in \mathcal{V}^{B_{\e}}_j$.
From Lemma \ref{lem: Be mag}, we derive that
\Be \notag
\begin{split}
\eqref{forcing mag}
& = \int_{\mathcal{V}^{B_{\e}}_0} \cdots \int_{\mathcal{V}^{B_{\e}}_{i-2}} 
\int_{\mathcal{V}_{i-1}}
\int_{\mathcal{V}_i} 
\mathbf{1}_{t^{i+1} < 0 \leq t^i}
|F (x^{i+1}, \vb^i)|  \{n(x^i) \cdot v^i\} \dd v^i \dd \sigma_{i-1}  \cdots \dd \sigma_0
\\& \lesssim \int_{\mathcal{V}^{B_{\e}}_0} \cdots \int_{\mathcal{V}^{B_{\e}}_{i-2}} \int_{\mathcal{V}_i} \mathbf{1}_{t^{i+1} < 0 \leq t^{i}} \ \mu_{\Theta} (x^{i+1}, v^{i}) \J (x^{i+1}) \{ n(x^i) \cdot v^i \} \dd v^i \dd \sigma_{i-2} \cdots \dd \sigma_0
\\& \lesssim \| \J (x)\|_{L^{\infty}_x} \int_{\mathcal{V}^{B_{\e}}_0} \dd \sigma_0 \cdots \int_{\mathcal{V}^{B_{\e}}_{i-2}} \dd \sigma_{i-2} \lesssim \| \J (x)\|_{L^{\infty}_x} \times \e^{i-1}.
\end{split}
\ee

\textit{Case 2:} There exists $0 \leq j \leq i-2$ such that $v^{j} \in \mathcal{V}^{G_{\e}}_j$.
Using Proposition \ref{prop:mapV mag}, we get 
\begin{align}
\eqref{forcing mag} 
& \lesssim \int_{\mathcal{V}_0} \dd \sigma_0 \cdots \int_{\mathcal{V}_{j-1}} \dd \sigma_{j-1} \times 
\int_{v^j \in \mathcal{V}^{G_{\e}}_j} \frac{5 B^2_3 \sum\limits_{m, n \in \mathbb{Z}} \mu_{\Theta} (x^{j+1}, v^{m, n}_{j, \mathbf{b}})}{2 - 2 \cos (B_3 \tb^{j})} \{ n(x^j) \cdot v^j \} \dd S_{x^{j+1}} \dd \tb^{j}  \notag
\\& \ \ \ \ \times \cdots \times \int_0^{t^{i-2}} \int_{\p\O} \frac{5 B^2_3 \sum\limits_{m, n \in \mathbb{Z}} \mu_{\Theta} (x^{i-1}, v^{m, n}_{i-2, \mathbf{b}})}{2 - 2 \cos (B_3 \tb^{i-2})} \{ n(x^{i-2}) \cdot v^{i-2} \}  \dd \tb^{i-2} \dd S_{x^{i-1}} \notag
\\& \ \ \ \ \times \int^{t^{i-1}}_0 \int_{\p\O}  
\frac{5 B^2_3 \sum\limits_{m, n \in \mathbb{Z}} \mu_{\Theta} (x^{i}, v^{m, n}_{i-1, \mathbf{b}})}{2 - 2 \cos (B_3 \tb^{i-1})} \{ n(x^{i-1}) \cdot v^{i-1} \} \dd \tb^{i-1} \dd S_{x^{i}} \notag
\\& \ \ \ \ \times
\int_{\mathcal{V}_i} 
\mathbf{1}_{t^{i+1} < 0 \leq t^i} \
|F (X(0; t^i, x^i, v^i), V(0; t^i, x^i, v^i))|  \{n(x^i) \cdot v^i\} \dd v^i. \label{est1:forcing mag}
\end{align}
Using the Tonelli's theorem, we deduce that
\Be \label{est2:forcing mag}
\begin{split}
\eqref{forcing mag} 
\lesssim 
& \int_{\mathcal{V}_0} \dd \sigma_0 \cdots \int_{\mathcal{V}_{j-1}} \dd \sigma_{j-1} \times \int_{v^j \in \mathcal{V}^{G_{\e}}_j} \dd \tb^{j} \int_{\p\O} \dd S_{x^{i}}
\\& \times 
\underbrace{ \int_0^{t^{j+1}} \int_{\p\O} \frac{\tb^{j+1} \sum\limits_{m, n \in \mathbb{Z}} \mu_{\Theta} (x^{j+2}, v^{m, n}_{j+1, \mathbf{b}})}{2 - 2 \cos (B_3 \tb^{j+1})} \times \frac{\tb^{j} \sum\limits_{m, n \in \mathbb{Z}} \mu_{\Theta} (x^{j+1}, v^{m, n}_{j, \mathbf{b}})}{2 - 2 \cos (B_3 \tb^{j})} \dd S_{x^{j+1}} \dd \tb^{j+1} }_{\eqref{est2:forcing mag}^*}
\\& \times \int_0^{t^{j+2}} \int_{\p\O} \frac{5 B^2_3 \sum\limits_{m, n \in \mathbb{Z}} \mu_{\Theta} (x^{j+3}, v^{m, n}_{j+2, \mathbf{b}})}{2 - 2 \cos (B_3 \tb^{j+2})} \{ n(x^{j+2}) \cdot v^{j+2} \} \dd S_{x^{j+2}} \dd \tb^{j+2} 
\\& \times \cdots \times \int_0^{t^{i-2}} \int_{\p\O} \frac{5 B^2_3 \sum\limits_{m, n \in \mathbb{Z}} \mu_{\Theta} (x^{i-1}, v^{m, n}_{i-2, \mathbf{b}})}{2 - 2 \cos (B_3 \tb^{i-2})} \{ n(x^{i-2}) \cdot v^{i-2} \} \dd S_{x^{i-2}} \dd \tb^{i-2}
\\& \times \underbrace{ \int^{t^{i-1}}_0 \int_{\p\O}
\frac{5 B^2_3 \sum\limits_{m, n \in \mathbb{Z}} \mu_{\Theta} (x^{i}, v^{m, n}_{i-1, \mathbf{b}})}{2 - 2 \cos (B_3 \tb^{i-1})} \{ n(x^{i-1}) \cdot v^{i-1} \} \dd S_{x^{i-1}}   \dd \tb^{i-1}}_{\eqref{est2:forcing mag}^{**}}
\times \eqref{est1:forcing mag}.
\end{split}
\Ee

\textbf{Step 2.} First we claim that
\be \label{bound3_6 mag}
\begin{split}
\eqref{est2:forcing mag}^{**} 
& = \int^{t^{i-1}}_0 \int_{\p\O}
\frac{5 B^2_3 \sum\limits_{m, n \in \mathbb{Z}} \mu_{\Theta} (x^{i}, v^{m, n}_{i-1, \mathbf{b}})}{2 - 2 \cos (B_3 \tb^{i-1})} \{ n(x^{i-1}) \cdot v^{i-1} \} \dd S_{x^{i-1}}   \dd \tb^{i-1}
\leq 1.
\end{split}
\ee
From Lemma \ref{conservative field mag}, we deduce
\be \label{bound3_6_1 mag}
\begin{split}
\eqref{est2:forcing mag}^{**} 
& = \int^{t^{i-1}}_0 \int_{\p\O}
\frac{5 B^2_3 \sum\limits_{m, n \in \mathbb{Z}} \mu_{\Theta} (x^{i}, v^{m, n}_{i-1, \mathbf{b}})}{2 - 2 \cos (B_3 \tb^{i-1})} \times | n(x^{i}) \cdot v^{m, n}_{i-1, \mathbf{b}} | \dd S_{x^{i-1}}   \dd \tb^{i-1}.
\end{split}
\ee
Using Proposition \ref{prop:mapV mag} and the fact $\tf (x^i, v^{m, n}_{i-1, \mathbf{b}}) = \tb (x^{i-1}, v^{m, n}_{i-1})$, we get
\be \label{bound3_6_2 mag}
\begin{split}
\eqref{bound3_6_1 mag}
& \leq \int^{\infty}_0 \int_{\p\O}
\frac{5 B^2_3 \sum\limits_{m, n \in \mathbb{Z}} \mu_{\Theta} (x^{i}, v^{m, n}_{i-1, \mathbf{b}})}{2 - 2 \cos (B_3 \tf^{i})} | n(x^{i}) \cdot v^{m, n}_{i-1, \mathbf{b}} | \dd S_{x^{i-1}}   \dd \tf^{i}
\\& = \int_{n(x^i) \cdot v < 0}
 \mu_{\Theta} (x^{i}, v) | n(x^{i}) \cdot v | \dd v = 1.
\end{split}
\ee

\textbf{Step 3.} Next we claim 
\Be \label{est3:forcing mag}
\eqref{est2:forcing mag}_* 
= \int_0^{t^{j+1}} \int_{\p\O} \frac{\tb^{j+1} \sum\limits_{m, n \in \mathbb{Z}} \mu_{\Theta} (x^{j+2}, v^{m, n}_{j+1, \mathbf{b}})}{2 - 2 \cos (B_3 \tb^{j+1})} \times \frac{\tb^{j} \sum\limits_{m, n \in \mathbb{Z}} \mu_{\Theta} (x^{j+1}, v^{m, n}_{j, \mathbf{b}})}{2 - 2 \cos (B_3 \tb^{j})} \dd S_{x^{j+1}} \dd \tb^{j+1}
\lesssim \e^{-2}.
\Ee
From the setting $v^{j} \in \mathcal{V}^{G_{\e}}_j$ with $\frac{\e}{B_3} + \frac{2 k \pi}{B_3}
  \leq \tb^j (x, v) \leq 
 \frac{2 \pi - \e}{B_3} + \frac{2 k \pi}{B_3}$, we derive that
\Be \notag
\frac{1}{2 - 2 \cos (B_3 \tb^{j})} \lesssim \e^{-2}.
\ee
Then using Lemma \ref{lem: sum of mu mag} and $t e^{- t^2} \lesssim 1$, we bound
\Be \label{bound1_1and5 mag}
\begin{split}
\frac{\tb^{j} \sum\limits_{m, n \in \mathbb{Z}} \mu_{\Theta} (x^{j+1}, v^{m, n}_{j, \mathbf{b}})}{2 - 2 \cos (B_3 \tb^{j})} 
\lesssim \tb^{j} e^{- (\tb^{j})^2} \e^{-2} 
\lesssim \e^{-2}.
\end{split}
\Ee
Applying \eqref{bound1_1and5 mag} on $\eqref{est2:forcing mag}_*$, we derive that
\Be \label{bound2_1and5 mag}
\eqref{est2:forcing mag}_*
\lesssim \e^{-2} \int_0^{t^{j+1}} \int_{\p\O} \frac{\tb^{j+1} \sum\limits_{m, n \in \mathbb{Z}} \mu_{\Theta} (x^{j+2}, v^{m, n}_{j+1, \mathbf{b}})}{2 - 2 \cos (B_3 \tb^{j+1})} \dd S_{x^{j+1}} \dd \tb^{j+1} \lesssim \e^{-2},
\Ee
where the last inequality holds from 
\eqref{bound3_6 mag}, and we prove \eqref{est3:forcing mag}. 

\textbf{Step 4.}
Now applying \eqref{bound3_6 mag} on $\{ v^k \}^{i-1}_{k = j+2}$ and \eqref{est3:forcing mag}, we derive that
\be \label{est_1:forcing mag}
\begin{split} 
\eqref{est2:forcing mag} 
\lesssim 
& \int_{\mathcal{V}_0} \dd \sigma_0 \cdots \int_{\mathcal{V}_{j-1}} \dd \sigma_{j-1} \times 
\e^{-2} 
 \times \int_{v^j \in \mathcal{V}^{G_{\e}}_j} \dd \tb^{j} \int_{\p\O} \dd S_{x^{i}} \eqref{est1:forcing mag}. 
\end{split}
\ee

For \eqref{est_1:forcing mag}, we employ the change of variables 
\be
(x^{i}, \tb^{j}, v^{i}) 
\mapsto (y, w) = (X(0; t^{j} -\tb^{j} \cdots - \tb^{i-1}, x^{i}, v^{i}), V(0; t^{j} -\tb^{j} \cdots - \tb^{i-1}, x^{i}, v^{i})) \in \O \times\R^3.
\ee
From \eqref{COV}, we derive
$|n(x^i) \cdot v^i| \dd S_{x^i} \dd \tb^{j} \dd v^i \lesssim \dd y \dd w$. 
Thus, we bound \eqref{est_1:forcing mag} as
\Be \notag
\begin{split}
\eqref{est_1:forcing mag}
  \leq \int_{\mathcal{V}_0} \dd \sigma_0 \cdots \int_{\mathcal{V}_{j-1}} \dd \sigma_{j-1}
\times \e^{-2} \times 
 \iint_{\O \times\R^3 } 
|F (y, w)|  \dd w \dd y
  \lesssim  2^{j} \e^{-2} \times \| F(x, v) \|_{L^1_{x,v}}. 
\end{split}
\Ee

Applying the above on each situation when $j = 0, \cdots, i-2$ in second cases, we deduce
\Be \notag
\eqref{forcing mag} 
\lesssim \sum\limits^{i-2}_{j = 0} 2^{j} \e^{-2 } \times \| F(x, v) \|_{L^1_{x,v}}
\leq 2^{i} \e^{-2 } \| F(x, v) \|_{L^1_{x,v}}.
\ee
Collecting the estimates from two cases, we prove \eqref{bound1:expand_h mag}.
\end{proof}

Following Lemma \ref{lem:small_largek} and Remark \ref{sigma measure estiamte}, we control the last term \eqref{expand_G2 mag}.

\begin{lemma} \label{lem:small_largek mag}
Consider $(X,V)$ in \eqref{characteristics mag}, there exists $\mathfrak{C}= \mathfrak{C}(\O)>0$ (see \eqref{choice:k} for the precise choice), such that  
\Be \label{small_largek mag}
\text{if }  \  k \geq \mathfrak{C}t,  \text{ then } 
\sup_{(x,v) \in \bar{\O} \times \R^3}  \Big(\int_{\prod_{j=0}^{k -1} \mathcal{V}_j}   
    \mathbf{1}_{t^{k} (t,x,v,v^1,\cdots, v^{k-1}) \geq 0 } \ \dd \sigma_0 \cdots \dd \sigma_{k-1}\Big) \lesssim e^{-t},
\Ee  
where $\dd \sigma_j = \mu_{\Theta} (x^{j+1}, v^{j}) \{ n(x^j) \cdot v^j \} \dd v^j$ in \eqref{def:sigma measure}.
\end{lemma}

\begin{proof}

For the magnet field case, from \eqref{estimate on delta} and $|n(x) \cdot v| = 5 \tb (x, v)$ for $(x,v) \in \gamma_{+}$, we have 
\Be \notag
\begin{split} 
& \ \ \ \ \int_{n(x) \cdot v > 0}
\mathbf{1} _{\delta> \tb (x, v)} \mu_{\Theta} (\xb, v) |n(x) \cdot v|\dd v
\leq \int_{ |n(x) \cdot v| \leq 5\delta} \mu_{\Theta} (\xb, v) | n(x) \cdot v | \dd v   \lesssim \delta^2.
\end{split}
\Ee
The rest of proof follows from Lemma \ref{lem:small_largek}.
\end{proof}

\begin{proof}[\textbf{Proof of Proposition \ref{prop:j linfty bound} of magnetic case \eqref{field property mag} }]
Using \eqref{est: G11 mag}, \eqref{est: G12 mag}, Lemma \ref{lem:bound1 mag} and \ref{lem:small_largek mag}, we have
\Be \notag
\begin{split}
\J (x) 
& \leq \| \J (x)\|_{L^{\infty}_x} \times e^{- \frac{t^2}{2b}} + \| \J (x)\|_{L^{\infty}_x} \times e^{- \frac{t^2}{8b}} 
\\ & \ \ \ \ + \sum^{k-1}_{i = 2} \big( 2^i \e^{-2} \| F(x, v) \|_{L^1_{x,v}} 
+ \e^{i-1} \| \J (x)\|_{L^{\infty}_x} \big) + \| \J (x)\|_{L^{\infty}_x} \times e^{-t}
\\& \leq \big( e^{- \frac{t^2}{2b}} + e^{- \frac{t^2}{8b}} +  e^{-t} + 2 \e \big) \times \| \J (x)\|_{L^{\infty}_x} + 2^k \e^{-2} \times \| F(x, v) \|_{L^1_{x,v}}.
\end{split}
\Ee
We can pick sufficiently large but fixed $t$ and sufficiently small $\e$ such that
\be \label{choice:t mag}
e^{- \frac{t^2}{2b}} + e^{- \frac{t^2}{8b}} + e^{-t} + 2 \e \leq 1/2.
\ee
Then we derive
\Be \notag
(1 - e^{- \frac{t^2}{2b}} - e^{- \frac{t^2}{8b}} - e^{-t} - 2 \e) \| \J (x)\|_{L^{\infty}_x} 
\leq 2^k \e^{-2} \times \| F(x, v) \|_{L^1_{x,v}}.
\Ee
Choosing $k = \mathfrak{a} \Big(\big[\frac{t  }{C_{\Omega} \delta}\big] + 1\Big)$ and $t$ in \eqref{choice:t mag}, we conclude that
\Be \label{est: j mag}
\| \J (x)\|_{L^{\infty}_x} 
\lesssim 2^{k+1} \times \| F(x, v) \|_{L^1_{x,v}}.
\Ee

Moreover, $F (x ,v) = F (x_\mathbf{b}, \vb) = \mu_{\Theta} (x_\mathbf{b}, \vb) \J (x_\mathbf{b})$, the $L^{\infty}_x$-estimate in \eqref{est: j mag} implies
\Be \notag
\|\mu^{-1}_{\Theta} (x_\mathbf{b}, \vb) F (x, v)\|_{L^{\infty}_{x,v}} 
= \| \J (x)\|_{L^{\infty}_x}
\lesssim 2^{k+1} \times \| F(x, v) \|_{L^1_{x,v}}.
\Ee
From Lemma \ref{conservative field mag}
and \eqref{def:Theta}, we get
\be
\mu^{-1}_{\Theta} (\xb, \vb) = 2 \pi \Theta^2(\xb) e^{ \frac{|\vb|^2}{2 \Theta(\xb)}}
\gtrsim e^{ \frac{|v|^2/2 + \Phi(x)}{2 b}}.
\ee
Therefore, we prove the proposition.
\end{proof}


\section{The existence of the steady state}
\label{sec: steady state}

In this section, we show the existence of steady state under the non-isothermal diffusive boundary condition.
We first use sequential argument on the penalized transport equation, then apply $L^1$-$L^{\infty}$ bootstrap to derive the solution. 
The main purpose of this section is to prove Proposition \ref{existence on f}, Theorem \ref{existence on F}.

Define
\be \label{isothermal maxwellian}
\mu (v) := \frac{1}{2 \pi}e^{-  \frac{|v|^2}{2}}
\ee 
Now we extend this Maxwellian into the whole phase space, such that for $(x, v) \in \O \times \R^3$,
\be \label{mu extend}
\mu (x, v) := \frac{1}{2 \pi}e^{-  \frac{|v|^2}{2 } - \Phi(x)},
\ee
where $\mu$ satisfies \eqref{equation for F_infty}
(resp. \eqref{def:f}). Moreover, we set $\| \mu \|_{L^1_{x,v}} := \mathfrak{m}_{\mu} > 0$.

\begin{proposition} \label{existence on f}
Recall 
$\mu_{\Theta} (x, v) =   \frac{1}{2 \pi \Theta^2(x)}e^{-  \frac{|v|^2}{2 \Theta(x)}}$ for $(x, v) \in \p\O \times \R^3$ and $0 < a \leq \Theta(x) \leq b < \infty$ with $b^2 = 2 a^2$ for all $x \in \O$. Now we define the following: 
\Be \notag
r (x, v) = \mu_{\Theta} (x, v) - \mu (x, v),
\ee
where $(x, v) \in \p\O \times \R^3$, with
\Be \notag
\int_{\gamma_{-}} r (x, v) |n \cdot v| \dd S_x \dd v = 0.
\Ee
Then there exists a unique solution $f (t, x, v)$ to
\eqref{equation for F_infty} (resp. \eqref{def:f}) and the boundary condition:
\Be \label{diff_f_infty}
f (x, v)  
= \mu_{\Theta} (x, v) \int_{n(x)\cdot v^1>0} f(x, v^1) \{ n(x) \cdot v^1 \} \dd v^1 + r (x, v), \ \ \text{for} \ (x,v) \in \gamma_-,
\Ee 
such that 
\Be \label{zero mass on f_infty}
\iint_{\O \times \R^3} f (x, v) \dd x \dd v = 0.
\Ee
\end{proposition}

For the proof of Proposition \ref{existence of f}, we start with the penalized transport equation with a reduced diffuse reflection boundary condition, with the purpose of setting up a contracting map argument.

\begin{lemma} \label{lem: existence of f^e}
For any $\epsilon > 0$ and for any integer $j > 0$, there exists a unique solution $f (t,x,v)$ to
\Be \label{equation for f_j} 
\epsilon f^j + v \cdot \nabla_{x} f^j + (v \times B - \nabla \Phi ) \cdot \nabla_{v} f^j = 0,
\Ee
\Be \label{diff_f_j}
f^j (x, v)  
= (1 - \frac{1}{j}) \mu_{\Theta} (x, v) \int_{n(x)\cdot v^1 >0} f^j (x, v^1) \{ n(x) \cdot v^1 \} \dd v^1 + r (x, v)
, \ \ \text{for} \ (x,v) \in \gamma_-,
\Ee   
where $B$ and $\Phi (x)$ are defined in \eqref{field property} (resp. \eqref{field property mag}).

Moreover, uniformly in $j$, we have
\Be \label{f_j L1}
\| f^j \|_1 \lesssim_{\epsilon, \O} |r|_1
\Ee
Finally the limit $\fe$ as $j \rightarrow \infty$ of the sequence $\{ f^j \}$ exists and solves uniquely
\Be \label{equation for f_e} 
\epsilon \fe + v \cdot \nabla_{x} \fe + (v \times B - \nabla \Phi ) \cdot \nabla_{v} \fe = 0,
\Ee
\Be \label{diff_f_e}
\fe (x, v)  
= \mu_{\Theta} (x, v) \int_{n(x) \cdot v^1 > 0} \fe (x, v^1) \{ n(x) \cdot v^1 \} \dd v^1 + r (x, v)
, \ \ \text{for} \ (x,v) \in \gamma_-.
\Ee  
where $B$ and $\Phi (x)$ are defined in \eqref{field property} (resp. \eqref{field property mag}).
\end{lemma}

\begin{proof}
For any $\e$, we first consider $B$, $\Phi (x)$ defined in \eqref{field property}, and use the following double iteration in both $j$ and $l$,
\Be \label{equation for f_l} 
\epsilon f^{l+1}_{j} + v \cdot \nabla_{x} f^{l+1}_{j} - \nabla \Phi \cdot \nabla_{v} f^{l+1}_{j} = 0,
\Ee
\Be \label{diff_f_l}
f^{l+1}_{j} (x, v)  
= (1 - \frac{1}{j}) \mu_{\Theta} (x, v) \int_{n(x)\cdot v^1>0} f^{l}_{j} (x, v^1) \{ n(x) \cdot v^1 \} \dd v^1 + r (x, v)
, \ \ \text{for} \ (x,v) \in \gamma_-.
\Ee 
with initial setting $f^0_{j} = 0$.

\textbf{Step 1.}
We first fix $j$ and take $l \rightarrow \infty$. 
Here we claim that $\{ f^l_{j} \}_{l}$ is a Cauchy sequence in $L^1$.
Taking the difference of $f^{l+1}_{j} - f^l_{j}$, and using Green's identity on \eqref{equation for f_l} and \eqref{diff_f_l}, we have

\Be \notag
\begin{split}
\epsilon \| f^{l+1}_{j} - f^l_{j} \|_1 + |f^{l+1}_{j} - f^l_{j} |_{1, +} 
 = |f^{l+1}_{j} - f^l_{j} |_{1, -}
  = |(1 - \frac{1}{j}) P_{\Theta} (f^{l}_{j} - f^{l-1}_{j} ) |_{1, -},
\end{split}
\Ee
where $P_{\Theta} f^l_{j} := \mu_{\Theta} (x, v) \int_{n(x)\cdot v^1>0} f^l_{j} (x, v^1) \{ n(x) \cdot v^1 \} \dd v^1$ with $\mu_{\Theta} (x, v) =   \frac{1}{2 \pi \Theta^2(x)}e^{-  \frac{|v|^2}{2 \Theta(x)}}$.
By direct computation and $| P_{\Theta} f^l_{j} |_{1, -} \leq | f^l_{j} |_{1, +}$, we get
\Be \notag 
\begin{split}
|(1 - \frac{1}{j}) P_{\Theta} ( f^{l}_{j} - f^{l-1}_{j} ) |_{1, -}
\leq (1 - \frac{1}{j}) | f^{l}_{j} - f^{l-1}_{j} |_{1, +}.
\end{split}
\Ee
Thus, we deduce that
\Be \label{gamma-+ estimate}
\epsilon \| f^{l+1}_{j} - f^l_{j} \|_1 + |f^{l+1}_{j} - f^l_{j} |_{1, +}
\leq (1 - \frac{1}{j}) | f^{l}_{j} - f^{l-1}_{j} |_{1, +}.
\Ee
Since $1 - \frac{1}{j} < 1$,  by iteration over $l$, we deduce
\Be \notag
| f^{l+1}_{j} - f^l_{j} |_{1, +} 
\leq (1 - \frac{1}{j})^{l} | f^{1}_{j} - f^{0}_{j} |_{1, +}.
\Ee
Applying this on \eqref{gamma-+ estimate}, we have
\Be \notag
\epsilon \| f^{l+1}_{j} - f^l_{j} \|_1 + |f^{l+1}_{j} - f^l_{j} |_{1, +}
\leq (1 - \frac{1}{j}) | f^{l}_{j} - f^{l-1}_{j} |_{1, +}
\leq (1 - \frac{1}{j})^{l} | f^{1}_{j} - f^{0}_{j} |_{1, +}.
\Ee
Then, for any $l, k \in \Z_{+}$ with $l > k > 0$, we derive
\Be \notag
\epsilon \| f^{l}_{j} - f^{k}_{j} \|_1
\leq \epsilon \sum\limits^{l-1}_{p = k} \| f^{p+1}_{j} - f^{p}_{j} \|_1
\leq \sum\limits^{l-1}_{p = k} (1 - \frac{1}{j})^{p} | f^{1}_{j} - f^{0}_{j} |_{1, +}
 < j (1 - \frac{1}{j})^{k} | f^{1}_{j} - f^{0}_{j} |_{1, +}.
\Ee	
Since $f^0_{j} = 0$ and $f^1_{j} (x, v) = r (x, v)$ for $(x, v) \in \gamma_{-}$, we get $| f^{1}_{j} - f^{0}_{j} |_{1, +} < \infty$. Therefore, we prove the claim.

Now we can take 
$l \rightarrow \infty$ to obtain the $L^1$ limit such that 
$f^l_{j} \rightarrow f^j \in L^1$.
Since \eqref{equation for f_l} is linear and $f^l_{j} \rightarrow f^j$ in $L^1$, we can get $f^j$ is the weak solution for \eqref{equation for f_j}. 
Note that on the boundary, we have $| f^{l+1}_{j} - f^l_{j} |_{1, +} 
\leq (1 - \frac{1}{j})^{l} | f^{1}_{j} - f^{0}_{j} |_{1, +}$. This shows $f^l_{j}$ is also a Cauchy sequence in $L^1_{\gamma_+}$, thus the limit $f^j$ satisfies \eqref{diff_f_j}.

Suppose $g^j$ is also the solution for $j > 0$, and set $\delta^j := f^j - g^j$. From Green's identity, we have,
\Be \notag
\epsilon \| \delta^j \|_1 + |\delta^j|_{1, +} 
\leq (1 - \frac{1}{j}) |P_{\Theta} \delta^j|_{1, -}
\leq (1 - \frac{1}{j}) |\delta^j|_{1, +}.
\Ee
Since $\epsilon > 0$ and $j > 0$, we obtain,
\Be \notag
\| \delta^j \|_1
\leq 0.
\Ee
which leads to the uniqueness on the solution to \eqref{equation for f_j} and \eqref{diff_f_j}.

\textbf{Step 2.}
Now we take $j \rightarrow \infty$ for $f^j$ with $\epsilon > 0$. Using Green's identity and $|P_{\Theta} f^l|_{1, -} \leq |f^l|_{1, +}$, we obtain that,
\Be \label{fj l_1 estimate}
\epsilon \| f^j \|_1 + |f^j|_{1, +} 
\leq (1 - \frac{1}{j}) |f^j|_{1, +} + |r|_{1, -}.
\Ee
Then we derive uniformly in $j$,
\Be \notag
\| f^{j} \|_1
\leq \frac{1}{\epsilon} |r|_{1, -} < \infty,
\Ee
which shows \eqref{f_j L1}.
Then following the characteristic line on \eqref{equation for f_j}, we get
\Be \notag
\| f^j (x, v) \|_{L^{\infty}} 
= \| e^{- \e \tb} f^j (\xb, \vb) \|_{L^{\infty}}  
\leq \| f^j (\xb, \vb) \|_{L^{\infty}}.
\ee
Here we define the boundary outgoing flux for $f^j (x, v)$, denoted by $\J^j (x)$,
\Be \label{diff_j}
\J^j (x) = \int_{n(x)\cdot v >0} f^j (x, v) \{ n(x) \cdot v \} \dd v, \ \text{for} \ x \in \p\O.
\Ee
From \eqref{diff_f_j}, we can immediately get
\Be \label{Fj relation}
f^j (x, v) = (1 - \frac{1}{j}) \mu_{\Theta} (x, v) \J^j (x) + r (x, v), \ \text{for} \ (x, v) \in \gamma_-.
\Ee
From \eqref{equation for f_j}, \eqref{diff_f_j} and the stochastic cycle on $f^j$, we have
\begin{align}
	\J^j (x) \leq 
    & \int_{\mathcal{V}_0}   
    \mathbf{1}_{t^{1} < 0 } \big| f^j (X(0; t, x, v^0), V(0; t, x, v^0)) \{ n(x) \cdot v^0 \} \big| \dd v^0
\label{expand_fj11}
	\\& + \int_{\prod_{j=0}^{1} \mathcal{V}_j}   
    \mathbf{1}_{t^{2} < 0 \leq t^{1}} \Big( \big| f^j (X(0; t^1, x^1, v^1), V(0; t^1, x^1, v^1)) \big| \dd \tilde{\Sigma}_{1} + |r (x^1, \vb^0)|  \dd \tilde{\Sigma}_{0} \Big)
\label{expand_fj12}
    \\& 	+ \sum\limits^{k-1}_{i=2}  \int_{\prod_{j=0}^{i} \mathcal{V}_j}  \mathbf{1}_{t^{i+1} < 0 \leq t^{i }} 
   \Big( \big| f^j (X(0; t^i, x^i, v^i), V(0; t^i, x^i, v^i)) \big| \dd \tilde{\Sigma}_{i}
\notag
\\& \qquad \qquad
+ |r (x^i, \vb^{i-1})| \dd \tilde{\Sigma}_{i-1} + |r (x^{i-1}, \vb^{i-2})| \dd \tilde{\Sigma}_{i-2} + \dots  + |r (x^1, \vb^0)| \dd \tilde{\Sigma}_{0} \Big)
\label{expand_fj1}
    \\& + \int_{\prod_{j=0}^{k } \mathcal{V}_j}  
  \mathbf{1}_{t^{k } \geq 0 } 
    \Big( \big| f^j (x^{k }, v^{k }) \big|  \dd \tilde{\Sigma}_{k}
     + |r (x^k, \vb^{k-1}|) \dd \tilde{\Sigma}_{k-1} + \dots  + |r (x^1, \vb^0)| \dd \tilde{\Sigma}_{0} \Big),
\label{expand_fj2}
\end{align} 
where 
$\dd \tilde{\Sigma}_{i} := \frac{ \dd \sigma_{i}}{\mu_{\Theta} (x^{i+1}, v^{i})} \dd \sigma_{i-1} \cdots \dd \sigma_1 \dd \sigma_0$ and
$\dd \sigma_j = \mu_{\Theta} (x^{j+1}, v^{j}) \{ n(x^j) \cdot v^j \} \dd v^j$ in \eqref{def:sigma measure}.

From Lemma \ref{conservative field}, we know that $|\vb| = |v|$. Moreover, it is easy to check that $r (x, v) = r (x, |v|)$, thus we deduce that,
\be \notag
\begin{split}
\eqref{expand_fj12} 
& \leq \int_{\prod_{j=0}^{1} \mathcal{V}_j}   
    \mathbf{1}_{t^{2} < 0 \leq t^{1}} \big| f^j (X(0; t^1, x^1, v^1), V(0; t^1, x^1, v^1)) \big| \dd \tilde{\Sigma}_{1} 
+ \int_{\mathcal{V}_0}   
    \mathbf{1}_{t^{2} < 0 \leq t^{1}} |r (x^1, v^0)| \dd \tilde{\Sigma}_{0}.
\end{split}
\ee
For \eqref{expand_fj1}, we have,
\be \notag
\begin{split}
\eqref{expand_fj1} 
& \leq \int_{\prod_{j=0}^{i} \mathcal{V}_j} 
\sum\limits^{k-1}_{i=2} \Big\{
 \mathbf{1}_{t^{i+1} < 0 \leq t^{i }} 
   \big| f^j (X(0; t^i, x^i, v^i), V(0; t^i, x^i, v^i)) \big| \Big\} \dd \tilde{\Sigma}_{i}
\\& + \sum\limits^{k-1}_{i=2} \int_{\prod_{j=0}^{i} \mathcal{V}_j}  \mathbf{1}_{t^{i+1} < 0 \leq t^{i }} 
\Big( |r (x^i, v^{i-1})| \dd \tilde{\Sigma}_{i-1} + |r (x^{i-1}, v^{i-2})| \dd \tilde{\Sigma}_{i-2} + \dots  + |r (x^1, v^0)| \dd \tilde{\Sigma}_{0} \Big).
\end{split}
\ee
For the last term \eqref{expand_fj2}, we obtain,
\be \notag
\begin{split}
\eqref{expand_fj2} \leq
\int_{\prod_{j=0}^{k } \mathcal{V}_j}  
  \mathbf{1}_{t^{k } \geq 0 } 
    \Big( \big| f^j (x^{k }, v^{k }) \big|  \dd \tilde{\Sigma}_{k}
     + |r (x^k, v^{k-1}|) \dd \tilde{\Sigma}_{k-1} + \dots  + |r (x^1, v^0)| \dd \tilde{\Sigma}_{0} \Big).
\end{split}
\ee
By direct computation, we get
\be \notag
\begin{split}
\int_{\mathcal{V}_{i-1}} r (x^i, v^{i-1}) \{ n(x^{i-1}) \cdot v^{i-1} \} \dd v^{i-1}
& \leq \int_{\mathcal{V}_{i-1}} \big( \mu_{\Theta} (x^i, v^{i-1}) + \mu (x^i, v^{i-1}) ) \{ n(x^{i-1}) \cdot v^{i-1} \} \dd v^{i-1}
\lesssim 2.
\end{split}
\ee
Then for any $i = 2, \cdots, k$,
\be \label{r measure estimate}
\int_{\prod_{j=0}^{i} \mathcal{V}_j}  
r (x^i, v^{i-1}) \dd \tilde{\Sigma}_{i-1}
\lesssim 2 \prod_{j=0}^{i-2} \int_{\mathcal{V}_i} \dd \sigma_i.
\ee
Using the above estimate, we derive
\begin{align}
\J^j (x) \leq 
& \int_{\mathcal{V}_0}   
    \mathbf{1}_{t^{1} < 0 } \big| f^j (X(0; t, x, v^0), V(0; t, x, v^0)) \{ n(x) \cdot v^0 \} \big| \dd v^0
\notag
	\\& + \int_{\prod_{j=0}^{1} \mathcal{V}_j}   
    \mathbf{1}_{t^{2} < 0 \leq t^{1}} \big| f^j (X(0; t^1, x^1, v^1), V(0; t^1, x^1, v^1)) \big| \dd \tilde{\Sigma}_{1} 
+ 2
\notag
    \\& 	+ \int_{\prod_{j=0}^{i} \mathcal{V}_j} 
\sum\limits^{k-1}_{i=2} \Big\{
 \mathbf{1}_{t^{i+1} < 0 \leq t^{i }} 
   \big| f^j (X(0; t^i, x^i, v^i), V(0; t^i, x^i, v^i)) \big| \Big\} \dd \tilde{\Sigma}_{i}
\notag
\\& + \sum\limits^{k-1}_{i=2} 
\Big( 2 \prod_{j=0}^{i-2} \int_{\mathcal{V}_i} \dd \sigma_i + 2 \prod_{j=0}^{i-3} \int_{\mathcal{V}_i} \dd \sigma_i + \dots  + 2 \Big)
\notag
    \\& + \int_{\prod_{j=0}^{k } \mathcal{V}_j}  
  \mathbf{1}_{t^{k } \geq 0 } 
   \big| f^j (x^{k }, v^{k }) \big|  \dd \tilde{\Sigma}_{k} 
   + 2 \prod_{j=0}^{k-2} \int_{\mathcal{V}_i} \dd \sigma_i + 2 \prod_{j=0}^{k-3} \int_{\mathcal{V}_i} \dd \sigma_i + \dots  + 2.
\notag
\end{align} 
From Remark \ref{sigma measure estiamte}, $\int_{\mathcal{V}_i} \dd \sigma_i$ is uniformly bounded for all $i$ and $x \in \O$. Now we apply the same idea in Lemma \ref{lem:bound1}, \ref{lem:small_largek} and Proposition \ref{prop:j linfty bound}, we can derive
\be \label{Jj l_infty estimate}
\| \J^j \|_{L^{\infty}} 
\lesssim \| f^{j} \|_1 + 2 \leq \frac{1}{\epsilon} |r|_{1} + 2.
\ee
Using $f^j (x ,v) = e^{- \epsilon \tb} f^j (x_\mathbf{b}, \vb)$, $r (x, v) \in L^{\infty}(\O \times \R^3)$ and \eqref{fj l_1 estimate}, \eqref{Fj relation}, then for any $(x, v) \in \O \times \R^3$ and $j$,
\Be \label{fj l_infty estimate}
\begin{split}
| f^j (x, v) |
 \leq \mu_{\Theta} (x_\mathbf{b}, \vb) \| \J^j \|_{L^{\infty}} + | r (\xb, \vb) | \lesssim \Big( \frac{1}{\epsilon} |r|_{1} + 2 \Big) e^{ - \frac{1}{2 b} (|v|^2/2 + \Phi(x))},
\end{split}
\Ee
where the last inequality follows from Lemma \ref{conservative field} and \eqref{def:Theta}.

From the uniform boundness on $\{ f^j\}$ in \eqref{fj l_infty estimate}, then up to a subsequence $\{ f^{j_k} \}$ we obtain the weak $*$ limit in $L^{\infty}$ such that $f^{j_k} \stackrel{\ast}{\rightharpoonup} \fe \in L^{\infty}$. 
Since \eqref{equation for f_j} is linear, we obtain that $f^j$ is the weak solution for \eqref{equation for f_e}.
Note from \eqref{Jj l_infty estimate}, we have uniform $L^\infty (\p\O)$ bound on $\{ \J^j \}$. 
Again up to a subsequence, we derive the weak $*$ limit in $L^\infty (\p\O)$ such that $J^{j_k} \stackrel{\ast}{\rightharpoonup} J^{\e}$.
Since $f^j (x, v) = (1 - \frac{1}{j}) \mu_{\Theta} (\xb, \vb) \J^j (\xb, \vb) + r (\xb, \vb)$, we let $\fe (x, v) = \mu_{\Theta} (\xb, \vb) J^{\e} (\xb, \vb) + r (\xb, \vb)$ for $(x,v) \in \gamma_-$, and it satisfies \eqref{diff_f_e}.

\textbf{Step 3.} Moreover, we set $\delta^j := f^j - \fe$ which solves the following: 
\Be \label{equation for d_j} 
\epsilon \delta^j + v \cdot \nabla_{x} \delta^j - \nabla \Phi \cdot \nabla_{v} \delta^j = 0,
\Ee
\Be \label{diff_d_j}
\delta^j (x, v)  
= (1 - \frac{1}{j}) P_{\Theta} \delta^j - \frac{1}{j} P_{\Theta} \fe.
\Ee 
Again from Green's identity and \eqref{gamma-+ estimate}, we obtain,
\Be \notag
\epsilon \| \delta^j \|_1 + |\delta^j|_{1, +} 
\leq | (1 - \frac{1}{j}) P_{\Theta} \delta^j|_{1, -} + | \frac{1}{j} P_{\Theta} \fe|_{1, -}.
\Ee
From
$|P_{\Theta} \fe|_{1, -} \leq |\fe|_{1, +}$ and $\fe \in L^{\infty}$, using \eqref{fj l_infty estimate}, we get that
\Be \notag
\epsilon \| \delta^j \|_1 + |\delta^j|_{1, +} 
\leq |\delta^j |_{1, +} + \frac{1}{j} ( \| \fe \|_{L^{\infty}} + \| r \|_{L^{\infty}} ).
\Ee
Thus we derive
\Be \notag
\| \delta^j \|_1 
\leq \frac{1}{\epsilon \times j} ( \| \fe \|_{L^{\infty}} + 2) \rightarrow 0 \ \text{as} \ j \rightarrow \infty,
\Ee
which implies
$f^j \rightarrow \fe$ as $j \rightarrow \infty$ in $L^1$ norm with any $\epsilon > 0$.

Suppose $g^{\e}$ is also the solution with the same $\e > 0$, and set $\delta^{\e} := \fe - g^{\e}$. From Green's identity, we have,
\Be \notag
\epsilon \| \delta^{\e} \|_1 + |\delta^{\e}|_{1, +} 
\leq |P_{\Theta} \delta^{\e}|_{1, -}
\leq |\delta^{\e}|_{1, +}.
\Ee
Since $\epsilon > 0$, we obtain,
\Be \notag
\| \delta^{\e} \|_1
\leq 0.
\Ee
which gives the uniqueness on the solution to \eqref{equation for f_e} and \eqref{diff_f_e}.

For the magnet field case, we consider $B$, $\Phi (x)$ defined in \eqref{field property mag}. Similarly, we use the following double iteration in both $j$ and $l$,
\Be \label{equation for f_l mag}  
\epsilon f^{l+1}_{j} + v \cdot \nabla_{x} f^{l+1}_{j} + (v \times B - \nabla \Phi (x)) \cdot \nabla_{v} f^{l+1}_{j} = 0,
\Ee
\Be \label{diff_f_l mag}
f^{l+1}_{j} (x, v)  
= (1 - \frac{1}{j}) \mu_{\Theta} (x, v) \int_{n(x)\cdot v^1>0} f^{l}_{j} (x, v^1) \{ n(x) \cdot v^1 \} \dd v^1 + r (x, v)
, \ \ \text{for} \ (x,v) \in \gamma_-.
\Ee 
with initial setting $f^0_{j} = 0$.

Since $\nabla_{v} \cdot (v \times B - \nabla \Phi (x)) = 0$, the magnet field doesn't influence the integration when we use Green's identity.
Thus, we still first fix $j$ and take $l \rightarrow \infty$, then take $j \rightarrow \infty$ to obtain $\fe$.
The rest of proof follows from the gravitational case.
\end{proof}

Next Lemma states the uniform in $L^1$ bound for $\fe$. It will allow us to take the limit as $\epsilon \rightarrow 0$.

\begin{lemma} \label{existence of f}
For any $\epsilon > 0$ and $\fe$ solving \eqref{equation for f_e}-\eqref{diff_f_e}
with $B$, $\Phi (x)$ defined in \eqref{field property} (resp. \eqref{field property mag}), then uniformly in $\epsilon$ and $(x, v) \in \O \times \R^3$, 
\Be \label{positive on fe}
\mu (x, v) + \fe (x, v) \geq 0.
\Ee
Moreover, uniformly in $\epsilon$, we have
\Be \label{L1 estimate on fe}
\iint_{\O \times \R^3} \fe (x, v) \dd x \dd v = 0 \ \text{and} \ \| \fe \|_1 \leq 2 \|\mu \|_{L^1}.
\Ee
Finally the limit $f$ as $\epsilon \rightarrow 0$ of the sequence $\{ \fe \}$ exists, and it solves \eqref{equation for F_infty} and \eqref{diff_f_infty}
(resp. \eqref{def:f} and \eqref{diff_f_infty}).
Moreover, we obtain for $(x, v) \in \O \times \R^3$,
\Be \label{positivity on f+mu}
\mu (x, v) + f (x, v) \geq 0.
\Ee
\end{lemma}

\begin{proof}

\textbf{Step 1.}
We first show that for any $l, j \in \mathbb{N}$ and $(x, v) \in \p\O \times \R^3$,
\Be \label{mu+fl positive}
\mu (x, v) + f^l_{j} (x, v) \geq 0.
\Ee
where $f^l_{j} (x, v)$ is the solution to \eqref{equation for f_l} and \eqref{diff_f_l} with $\Phi (x)$ defined in \eqref{field property}.

To prove this, we use the induction. For $l = 0$, $f^0_{j} = 0$ and it is easy to check $\mu + f^0_{j} \geq 0$.
Next we assume for $0 < k \in \mathbb{N}$ and $l \leq k$, and we have $\mu + f^l_{j} \geq 0$. 
Now for $l = k+1$, and we recall the sequential setting on $f^l_{j}$.
From \eqref{equation for f_l} and the characteristic line on $e^{\epsilon t} f^l_{j} (x, v)$, for $(x, v) \in \O \times \R^3$, 
\Be \label{fl characteristic}
f^l_{j} (x, v) = e^{- \epsilon \tb} f^l_{j} (\xb, \vb),
\Ee
So we just need to show $\mu (x, v) + f^{k+1}_{j} (x, v) \geq 0$ for $(x, v) \in \p\O \times \R^3$.
Using the boundary condition \eqref{diff_f_l} and the assumption that $f^k_{j} \geq - \mu (x, v)$, for $(x, v) \in \p\O \times \R^3$,
\Be \notag 
\begin{split}
f^{k+1}_{j} (x, v) 
& = (1 - \frac{1}{j}) \mu_{\Theta} (x, v) \int_{n(x)\cdot v^1>0} f^k_{j} (x, v^1) \{ n(x) \cdot v^1 \} \dd v^1 + r (x, v)
\\& \geq \mu_{\Theta} (x, v) \int_{n(x)\cdot v^1>0} - \mu (x, v^1) \{ n(x) \cdot v^1 \} \dd v^1 + r (x, v)
\\& \geq - \mu_{\Theta} (x, v) + r (x, v) = - \mu (x, v),
\end{split}
\Ee
which implies \eqref{mu+fl positive}.

Since $f^l_{j} \rightarrow f^j$ as $l \rightarrow \infty$ in $L^1$ and $f^j (x, v)$ solves \eqref{equation for f_j}-\eqref{diff_f_j}, we deduce  for $(x, v) \in \p\O \times \R^3$,
\Be \notag
\mu (x, v) + f^j (x, v) \geq 0
\Ee
Similarly, since $f^j \rightarrow \fe$ as $j \rightarrow \infty$ in $L^1$ and $\fe (x, v)$ solves \eqref{equation for f_e}-\eqref{diff_f_e}, for $(x, v) \in \p\O \times \R^3$,
\Be \notag
\mu (x, v) + \fe (x, v) \geq 0.
\ee 
By following the characteristic line, we conclude \eqref{positive on fe} for $(x, v) \in \O \times \R^3$. 

\textbf{Step 2.}
Using Green's identity on \eqref{equation for f_e} and \eqref{diff_f_e}, we get,
\Be \notag 
\epsilon \iint_{\O \times \R^3} \fe (x, v) \dd x \dd v + \iint_{\gamma} \fe (x, v) \{ n(x) \cdot v \} \dd v \dd S_x = 0.
\Ee
Using the fact that $\int_{\gamma_{-}} r (x, v) \dd S_x \dd v = 0$, we obtain, for any $\epsilon > 0$,
\Be \notag 
\epsilon \iint_{\O \times \R^3} \fe (x, v) \dd x \dd v = 0.
\Ee
Applying \eqref{positive on fe} and the zero mass on $\fe$, we derive that 
\Be \notag
\|\mu + \fe\|_{L^1} 
 = \iint_{\O \times \R^3} (\mu + \fe) \dd x \dd v = \|\mu\|_{L^1}.
\Ee
Thus we have, uniformly in $\epsilon$,
\Be \notag
\|\fe\|_{L^1} \leq 2 \|\mu\|_{L^1}.
\Ee

Next, we use the same idea in \eqref{Jj l_infty estimate}, \eqref{fj l_infty estimate},  then for any $(x, v) \in \O \times \R^3$ and $\epsilon$, 
\Be \label{fe l_infty estimate}
\begin{split}
| \fe (x, v) |
\lesssim \mu_{\Theta} (x_\mathbf{b}, \vb) \| \fe \|_1 + \| r (\xb, \vb) \|_{L^{\infty}} 
\leq \Big( 2 \|\mu\|_{L^1} + 2 \Big) e^{ - \frac{1}{2 b} (|v|^2/2 + \Phi(x))}.
\end{split}
\Ee
Therefore, we derive the weak $*$ limit on $\{\fe\}$ in $L^{\infty}$ norm up to a subsequence, such that 
\be \notag
\fe \stackrel{\ast}{\rightharpoonup} f \in L^{\infty}.
\ee
Since \eqref{equation for f_e} is linear, we derive that $f$ is the weak solution for \eqref{equation for F_infty}.
Note from \eqref{fe l_infty estimate}, we have uniform $L^\infty (\p\O)$ bound on $\{ \J^{\e} \}$ where $\fe (x, v) = \mu_{\Theta} (\xb, \vb) J^{\e} (\xb, \vb) + r (\xb, \vb)$.
Again up to a subsequence, we derive the weak $*$ limit on $\{ \J^{\e} \}$ in $L^\infty (\p\O)$ such that $J^{\e} \stackrel{\ast}{\rightharpoonup} J$.
We let $f (x, v) = \mu_{\Theta} (\xb, \vb) J (\xb, \vb) + r (\xb, \vb)$ for $(x,v) \in \gamma_-$, and it satisfies \eqref{diff_f_infty}.

Lastly, for every Borel set $A$ in $\O \times \R^3$, the weak $*$ convergence in $L^{\infty}$ norm implies that
\be \notag
\iint_{\O \times \R^3} \fe \times \mathbf{I}_{A}
\rightarrow \iint_{\O \times \R^3} f \times \mathbf{I}_{A},
\ee
where $\mathbf{I}_{A}$ is the sign function on $A$. Since $\mu (x, v) + \fe (x, v) \geq 0$, we derive \eqref{positivity on f+mu}.

For the magnet field case, since the magnet field doesn't influence the integration when we use Green's identity, the proof follows from the gravitational case.
\end{proof}

Now we are ready to prove Proposition \ref{existence on f}.

\begin{proof}[\textbf{Proof of Proposition \ref{existence on f}}]

Since we have shown the existence of the solution $f$ in Lemma \ref{existence of f}, we only need to prove $\iint_{\O \times \R^3} f (x, v) \dd x \dd v = 0$ and the uniqueness.

For any $k < \infty$, it is easy to check that $\mathbf{1}_{ \{|x|, |v| \leq k\} } \in L^1 (x. v)$. Applying $\fe \stackrel{\ast}{\rightharpoonup} f \in L^{\infty}$, we have
\Be \notag
\iint_{\O \times \R^3} \fe \times \mathbf{1}_{\{|x|, |v| \leq k\}} \dd x \dd v \rightarrow \iint_{\O \times \R^3} f \times \mathbf{1}_{\{|x|, |v| \leq k\}} \dd x \dd v.
\Ee
From \eqref{equation for f_e} and the characteristic line on $e^{\epsilon t} \fe (x, v)$, for $(x, v) \in \O \times \R^3$, we obtain,
\Be \label{fe characteristic}
\fe (x, v) = e^{- \epsilon \tb} \fe (\xb, \vb),
\Ee
Using \eqref{COV+} 
and \eqref{fe characteristic}, we have
\Be \label{first fe xv>k estimate}
\begin{split}
& \ \ \ \ \iint_{\O \times \R^3} \fe (x, v) \times \mathbf{1}_{\{|x| > k \ \text{or} \ |v| > k\}} \dd x \dd v
\\& = \int_{\gamma_{-}} \int_0^{t_{+}} \fe  (X(t,t-s,x,v),V(t,t-s,x,v)) \times \mathbf{1}_{\{|X| > k \ \text{or} \ |V| > k\}} |n(x) \cdot v|\dd s \dd v \dd S_x
\\& = \int_{\gamma_{-}} \int_0^{t_{+}} e^{- \epsilon s} \fe (x,v) \times \mathbf{1}_{\{|X| > k \ \text{or} \ |V| > k\}} \ |n(x) \cdot v|\dd s \dd v \dd S_x.
\end{split}
\Ee
From Lemma \ref{conservative field}, we have 
\be \notag
|v|^2 + 2\Phi (x) = |\vb (x, v)|^2.
\ee 
Using \eqref{v3t estimate}, \eqref{fj l_infty estimate} and the uniform $L^{\infty}$ bound on $\fe$, we deduce
\Be \label{second fe xv>k estimate}
\begin{split}
\eqref{first fe xv>k estimate} 
& \lesssim \int_{\gamma_{-}} \mu_{\Theta} (x, v) 
\big( \|\fe\|_{\infty} + 2
\big) \times \mathbf{1}_{\{|v|\gtrsim k \}} \ |n(x) \cdot v|^2 \dd v \dd S_x
\\& \lesssim \int_{n(x) \cdot v < 0} \mu^{1/b} (x, v) \times \mathbf{1}_{\{|v|\gtrsim k \}} \ |n(x) \cdot v|^2 \dd v.
\end{split}
\Ee
Note that we can derive the same estimate on $f$ since $f$ and $\fe$ share the same characteristics trajectories. 
Therefore, for any $\epsilon > 0$, we can find sufficient large $k$ such that,
\Be \notag 
\begin{split}
& \Big| \iint_{\O \times \R^3} \fe \times \mathbf{1}_{\{|x| > k \ \text{or} \ |v| > k\}} \dd x \dd v \Big| < \frac{\epsilon}{3},
\\& \Big| \iint_{\O \times \R^3} f \times \mathbf{1}_{\{|x| > k \ \text{or} \ |v| > k\}} \dd x \dd v \Big| < \frac{\epsilon}{3}.
\end{split}
\Ee
Next, we can find sufficient small $l$ such that, for any $\epsilon < l$,
\Be \notag
\Big| \iint_{\O \times \R^3} \fe \times \mathbf{1}_{\{|x|, |v| \leq k\}} \dd x \dd v - \iint_{\O \times \R^3} f \times \mathbf{1}_{\{|x|, |v| \leq k\}} \dd x \dd v 
\Big| \leq \frac{\epsilon}{3}.
\Ee
Using the above and the zero mass on $\fe$, we derive
\Be \notag
\begin{split}
\big| \iint_{\O \times \R^3} f \dd x \dd v \big|
& = \Big| \iint_{\O \times \R^3} f \dd x \dd v - \iint_{\O \times \R^3} \fe \dd x \dd v \Big|
\\& \leq \Big| \iint_{\O \times \R^3} \fe \times \mathbf{1}_{\{|x|, |v| \leq k\}} \dd x \dd v - \iint_{\O \times \R^3} f \times \mathbf{1}_{\{|x|, |v| \leq k\}} \dd x \dd v 
\Big| 
\\& \ \ \ \ + \Big| \iint_{\O \times \R^3} f \times \mathbf{1}_{\{|x| > k \ \text{or} \ |v| > k\}} \dd x \dd v \Big| + \Big| \iint_{\O \times \R^3} \fe \times \mathbf{1}_{\{|x| > k \ \text{or} \ |v| > k\}} \dd x \dd v \Big| 
\\& \leq \frac{\epsilon}{3} + \frac{\epsilon}{3} + \frac{\epsilon}{3} = \epsilon,
\end{split}
\Ee
which concludes \eqref{zero mass on f_infty}.

Suppose $g$ is also the solution, and set $\delta := f - g$. From Green's identity, we have,
\Be \notag 
|\delta^j|_{1, +} 
= |\delta^j|_{1, -}
\leq |\delta^j|_{1, +}.
\Ee
Moreover, we have, for any $x \in \p\O$,
\Be \notag
\int_{n(x)\cdot v^1>0} | \delta | \{ n(x) \cdot v^1 \} \dd v^1 
= \Big| \int_{n(x)\cdot v^1>0} \delta \{ n(x) \cdot v^1 \} \dd v^1 \Big|.
\Ee
which implies $\delta (x, \cdot)$ keeps non-positivity/non-negativity for any fixed $x$ and 
$v \in \mathcal{V}  := \{v \in \R^3: n(x) \cdot v > 0 \}$. Then we deduce, for $(x, v) \in \gamma_{+}$,
\Be \notag
\delta (x, v) = \mu_{\Theta} (\xb, \vb) \int_{n(\xb)\cdot v^1>0} \delta (\xb, v^1) \{ n(\xb) \cdot v^1 \} \dd v^1 \ \text{with} \ \mu > 0.
\Ee
This shows that $\delta$ doesn't change non-positivity/non-negativity for any $x$ or $v$. Using the fact that both $f$ and $g$ have zero mass, we derive 
\Be \notag
\delta = 0 \ a.e. \ ,
\Ee
which gives the uniqueness on the solution.

For the magnet field case, we can deduce \eqref{first fe xv>k estimate} by \eqref{COV+}. Moreover, using Lemma \ref{conservative field mag}, we have $|v|^2 + 2\Phi (x) = |\vb|^2$ and $\tb (x, v) = |v_3|/5$.
The rest of proof is similar as the gravitational case and so the proof is omitted.
\end{proof}

Now we are well equipped to prove Theorem \ref{existence on F}.

\begin{proof}[\textbf{Proof of Theorem \ref{existence on F}}] 

Recall \eqref{mu extend} and $\| \mu \|_{L^1_{x,v}} = \mathfrak{m}_{\mu} > 0$.
For $(x, v) \in \p\O \times \R^3$, we write 
\Be \notag
F (x, v) = \frac{\mathfrak{m}}{\mathfrak{m}_{\mu}} \big( \mu (x, v) + f (x, v) \big)
\ee
W.l.o.g. we suppose $\mathfrak{m} = \mathfrak{m}_{\mu}$, i.e. 
$F (x, v) = \mu (x, v) + f (x, v)$.
By direct computation, for $(x, v) \in \gamma_{-}$, 
\Be \notag
\begin{split}
F (x, v) 
& = \mu (x, v) + \mu_{\Theta} (x, v) \int_{n(x)\cdot v^1>0} f(x, v^1) \{ n(x) \cdot v^1 \} \dd v^1 + r (x, v)
\\& = \mu (x, v) + \mu_{\Theta} (x, v) \int_{n(x)\cdot v^1>0} F(x, v^1) \{ n(x) \cdot v^1 \} \dd v^1 - \mu_{\Theta} (x, v) + r (x, v)
\\& = \mu_{\Theta} (x, v) \int_{n(x)\cdot v^1>0} F(x, v^1) \{ n(x) \cdot v^1 \} \dd v^1,
\end{split}
\Ee
this shows $F (x, v)$ solves \eqref{equation for F_infty} and \eqref{diff_F}. Moreover, since $\iint_{\O \times \R^3} f (x, v) \dd x \dd v = 0$, we derive
\Be \notag
\iint_{\O \times \R^3} F (x, v) \dd x \dd v  = \mathfrak{m},
\ee
thus we show the existence of the solution.

Next, it is easy to check that for any steady state $F (x, v)$ with unit mass,
$F (x, v) - \mu (x, v)$ need to satisfy \eqref{equation for F_infty} and \eqref{zero mass on f_infty}.
Applying the uniqueness on $f (x, v)$ from Proposition \ref{existence on f}, we conclude the uniqueness on $F (x, v)$.
Finally, we derive \eqref{L_infty estimate on F_infty} from Proposition \ref{prop:j linfty bound}.
\end{proof}

\begin{proof}[\textbf{Proof of Theorem \ref{existence on F} for magnetic case}] 

For the magnet field case, we also set
$F (x, v) = \frac{\mathfrak{m}}{\mathfrak{m}_{\mu}} \big( \mu (x, v) + f (x, v) \big)$, and then for $(x, v) \in \gamma_{-}$, 
\Be \notag
\begin{split}
F (x, v) 
= \mu_{\Theta} (x, v) \int_{n(x)\cdot v^1>0} F(x, v^1) \{ n(x) \cdot v^1 \} \dd v^1,
\end{split}
\Ee
this shows $F (x, v)$ solves \eqref{def:f} and \eqref{diff_F}. 
The rest of proof is similar to the proof of Theorem \ref{existence on F}.
\end{proof}

\begin{remark} 
\label{positivity on F_infty}
From \eqref{positivity on f+mu}, we derive that the steady state is non-negative with positive mass. 
In the rest of paper, we use $F_s (x, v)$ to represent the non-negative solution to
\eqref{equation for F_infty} and \eqref{diff_F} 
(resp. \eqref{def:f} and \eqref{diff_F})
with $\iint_{\O \times \R^3} F_s (x, v) \dd x \dd v = 1$.
\end{remark}


\section{Weighted \texorpdfstring{$L^1$}{L1}-Estimates}
\label{sec: L1 estimate}
Recall 
$f(t, x, v)=
F(t, x, v) - F_s (x, v),$
which solves
\begin{align} 
	\partial_t f + v \cdot \nabla_{x} f + (v \times B - \nabla \phi - g \mathbf{e}_3)\cdot \nabla_{v} f = 0,& \ \  \text{for} \  \    (t, x, v) \in \R_{+} \times \Omega \times \R^3,  \label{eqtn_f} 
	\\
	f (t,x,v) |_{t = 0}  = F_0 (x, v) -   F_s (x, v) = f_0 (x,v),&  \ \  \text{for} \   \   (x, v) \in \Omega \times \R^3,
	\label{init_f}
	\\
	f (t, x, v)  = \mu_{\Theta} (x, v) \int_{n(x^1) \cdot v^1>0} f(t, x, v^1) \{ n(x^1) \cdot v^1 \} \dd v^1 ,& \ \  \text{for} \  \   (t,x, v) \in \R_{+} \times \gamma_-. \label{diff_f} 
\end{align}

The main purpose of this section is to prove Theorem \ref{theorem_1} which concerns $L^1$-estimates on fluctuations.
 \begin{theorem} \label{theorem_1}
 	Assume the same conditions in Theorem \ref{theorem} for both cases of \eqref{field property} and \eqref{field property mag}. Then a unique solution $F(t,x,v) = F_s(x,v) + f(t,x,v)$ satisfies \hide
	Let $F_s(x,v)$ be a stationary solution in Theorem \ref{existence on F}. Consider the initial data $F_0 (x, v) = F_s (x,v) + f_0 (x,v)$ such that $\iint_{\O \times \R^3} f_0 (x,v) \dd x \dd v = 0$. $\| e^{\theta^\prime (|v|^2+ 2\Phi (x))} f_0\|_{L^\infty_{x,v}}< \infty$ for $0 < \theta^\prime < \frac{1}{2b}$ with $b$ defined in \eqref{def:Theta}, and $\| e^{\delta (|v|^2 + 2 \Phi(x))^{1/2} } f_0 \|_{L^1_{x,v}} < \infty$ with $0 < \delta < \infty$. There exists a unique global-in-time solution $F(t,x,v) = F_s (x,v) + f(t,x,v)$ to \eqref{equation for F mag}-\eqref{diff_F} with $F(t,x,v)|_{t=0} = F_0(x,v)$ in $\O \times \R^3$,
	such that
	\Be \label{cons_mass_f} 
	\iint_{\Omega \times \R^3} f (t, x, v) \dd x \dd v = 0, \ \ \text{for all } t\geq 0,
	\Ee
	\Be \label{theorem_infty_1}
	\sup_{t\geq0}\| e^{\theta^\prime (|v|^2+ 2\Phi (x))} f (t)\|_{L^\infty_{x,v}} \lesssim \| e^{\theta^\prime (|v|^2+ 2\Phi (x))} f_0\|_{L^\infty_{x,v}}.
	\Ee
	Moreover, for all $t \geq 0$ and $0 \leq \theta< \theta^\prime$,
	\Be \label{theorem_infty}
	\begin{split}
		\sup_{x \in \bar{\O}}\int_{\R^3} e^{\theta  (|v|^2+ 2\Phi (x))} |f(t,x,v) |\dd v  
		\lesssim_{\theta}  e^{- \Lambda t},
	\end{split} 
	\Ee
	where $\Lambda = \Lambda (\O, \delta, T_0)$ is defined in \eqref{def:M} and $T_0$ in \eqref{cond:T0}.

	\unhide
	\Be \label{est:theorem_1} 
	\| f(t) \|_{L^1_{x,v}} 
	\lesssim  e^{- \Lambda t}
	\{ \| e^{\theta^\prime (|v|^2+ 2\Phi (x))} f_0\|_{L^\infty_{x,v}}+ \| e^{\delta (|v|^2 + 2 \Phi(x))^{1/2} } f_0 \|_{L^1_{x,v}}
	\},
	\Ee
	where $\Lambda=\Lambda (\O, \delta, T_0)$ is defined in \eqref{def:M} and $T_0$ in \eqref{cond:T0}.
	\hide
	.
	There exists a unique solution $F (t,x,v)= f (t,x,v) + \mathfrak{M} F_s (x, v)$ 
	for \eqref{equation for F mag} with \eqref{field property} and \eqref{diff_F} 
	where $F_s (x, v)$ solves \eqref{equation for F_infty} and \eqref{diff_F} with 
	$\iint_{\O \times \R^3} F_s (x, v) \dd x \dd v = 1$ and $\iint_{\O \times \R^3} f_0 (x,v) \dd x \dd v = 0$.
	In addition, if $\| e^{\theta^\prime (|v|^2+ 2\Phi (x))} f_0\|_{L^\infty_{x,v}}<\infty$ for $0<\theta^\prime< \frac{1}{2b}$ with $b$ defined in \eqref{def:Theta}, and $\| e^{\delta (|v|^2 + 2 \Phi(x))^{1/2} } f_0 \|_{L^1_{x,v}} < \infty$ with $0 < \delta < \infty$, then we have
	\unhide
\end{theorem}

%


\subsection{\texorpdfstring{$f (t, x, v)$}{f(t,x,v)} via Stochastic Cycles}

\begin{lemma}
\label{sto_cycle}
 Suppose $f$ solve \eqref{eqtn_f}-\eqref{diff_f} for either \eqref{field property} or \eqref{field property mag}, and $t_* \leq t$, then we have
\begin{align}
    f (t, x, v) = & \mathbf{1}_{t^1 < t_*}
    f (t_*, X(t_*; t, x, v), V(t_*; t, x, v)) \label{expand_h1} 
    \\&  + \mu_{\Theta} (x^1, \vb) \int_{\prod_{j=1}^{i} \mathcal{V}_j}   
     \sum\limits^{k-1}_{i=1} 
     \Big\{   \mathbf{1}_{t^{i+1}<t_* \leq t^{i }} f (t_*, X(t_*; t^i, x^i, v^i), V(t_*; t^i, x^i, v^i))  \Big\}
      \dd  \Sigma_{i}
\label{expand_h2}
    \\& + \mu_{\Theta} (x^1, \vb) \int_{\prod_{j=1}^{k } \mathcal{V}_j}   
    \mathbf{1}_{t^{k } \geq t_* }
    f (t^{k}, x^{k }, v^{k })
     \dd  \Sigma_{k}
, \label{expand_h3}
\end{align} 
where 
$\dd  {\Sigma}_{i} 
:= \frac{ \dd \sigma_{i}}{ \mu_{\Theta} (x^{i+1}, v^{i})} \dd \sigma_{i-1} \cdots  \dd \sigma_1$, with $\dd \sigma_j = \mu_{\Theta} (x^{j+1}, v^{j}) \{ n(x^j) \cdot v^j \} \dd v^j$ in \eqref{def:sigma measure}. Here, $(X,V)$ in \eqref{characteristics} (resp. \eqref{characteristics mag}).
\end{lemma}

\begin{proof}
Following the similar steps in Lemma \ref{sto_cycle_1} and Remark \ref{probability}, we can obtain this Lemma.
\end{proof}

\begin{lemma} \label{lemma:energy}
Consider $b$ in \eqref{def:Theta}, suppose $\varphi (\tau ) \geq0$, $\varphi^\prime \geq0$, and 
\Be \label{cond:varphi} 
\int_1^\infty e^{- \frac{1}{2b} \tau^2}\varphi(\tau) \dd \tau < \infty. 
\Ee
Suppose $f$ solves \eqref{eqtn_f}-\eqref{diff_f} for either \eqref{field property} or \eqref{field property mag}, there exists $C>0$ independent of $t_*$, $t$, such that for all $0 \leq t_* \leq t$, 
\Be \label{energy_varphi}
\begin{split}
& \ \ \ \ \| \varphi(\tf) f(t) \|_{L^1_{x,v}}
  + \int^{t}_{t_*}
  \| \varphi^\prime(\tf) f  \|_{L^1_{x,v}} \dd s
  +  \int_{t_*}^{t} | \varphi(\tf) f|_{L^1_{\gamma_+}} \dd s
\\& \leq \| \varphi(\tf) f(t_*) \|_{L^1_{x,v}} +
 C (t - t_* + 1) \| f(t_*) \|_{L^1_{x,v}} + \frac{1}{4}  \int^{t}_{t_*} |f |_{L^1_{\gamma_+}} \dd s.
\end{split}
\Ee 
\end{lemma}

For the proof of Lemma \ref{lemma:energy}, We shall start it from Lemma \ref{lem:energy estimate on f}, the energy estimate. 

\begin{lemma} 
\label{lem:energy estimate on f}
Suppose $f$ solves \eqref{eqtn_f}-\eqref{diff_f} for either \eqref{field property} or \eqref{field property mag}. For $0 \leq t_* \leq t$ with $\delta \in (0, t-t_*)$,
\begin{align}
& \ \ \ \ \| f(t) \|_{L^1_{x,v}}  \leq \| f(t_*) \|_{L^1_{x,v}}, \label{maximum} 
\\& \int^{t}_{t_*}  |f(s )|_{L^1_{\gamma_+}} \dd s 
 \leq \Big \lceil \frac{t-t_*}{\delta}   \Big \rceil \| f(t_*) \|_{L^1_{x,v}}
+ O(\delta^2) \int^{t}_{t_*}  |f(s )|_{L^1_{\gamma_+}} \dd s, \label{trace}
\end{align}
and if $f_0$ is non-negative, so is $f(t, x, v)$ for all $(t, x, v) \in \R_{+} \times \Omega \times \R^3$.
\end{lemma}

\begin{proof}
From \eqref{eqtn_f}, \eqref{diff_f} and taking integration over $(t_*, t) \times \Omega \times \R^3$, 
we derive that
\Be \notag
\| f(t) \|_{L^1_{x,v}} + \int^t_{t*} \iint_{\gamma_+}|f| \dd s - \int^t_{t*} \iint_{\gamma_-}|f| \dd s = \| f(t_*) \|_{L^1_{x,v}}.
\Ee
Due to the choice of $\mu_{\Theta} (x, v)$ in \eqref{diff_F}, for $\forall t > 0$,
\Be \notag
\int_{\gamma_-} |f (t, x, v)| |n(x) \cdot v| \dd S_x \dd v 
= \Big|\int_{\gamma_+} f (t, x, v) \{n(x) \cdot v\} \dd S_x \dd v 
\Big|.
\Ee
Therefore, we have
\Be \notag
\begin{split}
& \ \ \ \ \int^t_{t*} \iint_{\gamma_+}|f| \dd s - \int^t_{t*} \iint_{\gamma_-}|f| \dd s
\\& = \int^t_{t*} \iint_{\gamma_+}|f| \dd s -\int^t_{t*} \Big|\iint_{\gamma_+}f\Big| \dd s
\geq  \int^t_{t*} \iint_{\gamma_+}|f| \dd s -  \int^t_{t*} \iint_{\gamma_+}|f| \dd s =0,
\end{split}
\Ee 
therefore we prove \eqref{maximum}.

Next we work on \eqref{trace}. For $\delta \in (0, t-t_*)$ and $(x,v) \in \gamma_+$,  
\Be \label{|f|1}
\begin{split}
|f(s,x,v)|
& \leq \underbrace{ \sum\limits^{ \lceil \frac{t-t_*-\delta}{\delta} \rceil }_{k=1} \mathbf{1}_{t_* + k \delta \leq s \leq t_* + (k+1) \delta, \ \delta < \tb(x,v)}
|f(t_* + k \delta, X(t_* + k \delta, s, x, v), V(t_* + k \delta, s, x, v))| }_{\eqref{|f|1}_1} 
\\& \ \ \ \ + \underbrace{ \mathbf{1}_{t_* + \delta \leq s, \ \delta \geq \tb(x,v)}
|f(s- \tb (x,v), \xb (x,v), \vb (x,v))|}_{\eqref{|f|1}_2}
\\& \ \ \ \ + \underbrace{ \mathbf{1}_{s - t_* < \delta, \ s - t_* < \tb(x,v)}
|f(t_*, X(t_*, s, x, v), V(t_*, s, x, v))| }_{\eqref{|f|1}_3} 
\\& \ \ \ \ + \underbrace{ \mathbf{1}_{\tb(x,v) \leq s - t_* < \delta}
|f(s- \tb (x,v), \xb (x,v), \vb (x,v))|}_{\eqref{|f|1}_4},
\end{split} 
\Ee
where $s \in [t_*,t]$. 

First we do estimate on $\eqref{|f|1}_1$.
From \eqref{COV}, \eqref{maximum} and $t_* + \delta \leq s \leq t$ with $\delta < \tb (x, v)$,
\begin{align} 
& \ \ \ \ \int^{t}_{t_*} \int_{\gamma_+} \eqref{|f|1}_1 \dd s \notag
\\& \leq \sum\limits^{ \lceil \frac{t-t_*-\delta}{\delta} \rceil }_{k=1} \int_{\gamma_+}
\int^{ t_* + k \delta + \tb (x, v)}_{ t_* + k \delta}
|f(t_* + k \delta, X(t_* + k \delta, s, x, v), V(t_* + k \delta, s, x, v))| \dd s 
\{n(x) \cdot v\} \dd S_x \dd v \notag 
\\& \leq \sum\limits^{ \lceil \frac{t-t_*-\delta}{\delta} \rceil  }_{k=1} \| f(t_* + k \delta) \|_{L^1_{x,v}}  
\leq \Big \lceil \frac{t-t_*-\delta}{\delta}   \Big \rceil \times \| f(t_*) \|_{L^1_{x,v}}.
\end{align}

Now we consider $\eqref{|f|1}_2$. For $y= \xb(x,v)$ and $s \in [t_*,t]$, we have 
\Be \label{delta tf}
\mathbf{1}_{\delta \geq \tb(x,v)}
= \mathbf{1}_{\delta \geq \tf(y, \vb)}.
\Ee
From \eqref{COV_bdry}, \eqref{delta tf} and using the Fubini's theorem, we derive
\begin{align} 
\int^{t}_{t_*} \int_{\gamma_+} \eqref{|f|1}_2 \dd s
& = \int_{\gamma_+}
\int^t_{ t_* + \delta}
\mathbf{1}_{\delta \geq \tb(x,v)}
|f(s- \tb(x, v), \xb, \vb)| \dd s 
\{n(x) \cdot v\} \dd S_x \dd v \notag \\
& \leq \int_{\p\O} \int_{n(y) \cdot v<0}
\mathbf{1} _{\delta> \tf(y, v)}
\int^t_{t_*}|f(s, y, v)| \dd s 
|n(y) \cdot v| \dd S_{y} \dd v \notag \\
& \leq \int_{\p\O}
\underbrace{\Big( \int_{n(y) \cdot v<0}
\mathbf{1} _{\delta> \tf(y, v)} \mu_{\Theta} (y, v) |n(y) \cdot v|\dd v\Big)}_{\eqref{est:|f|2}_*}
\int^t_{t_*}
\int_{n(y) \cdot v^1>0} 
|f(s,y, v^1)| \{n(y) \cdot v^1\} \dd v^1
 \dd s 
 \dd S_{y} .
 \label{est:|f|2}
\end{align}
From \eqref{estimate on delta}, we derive
\be \notag
\eqref{est:|f|2} \leq O(\delta^2) \int^t_{t_*} \int_{\gamma_+} |f| \dd s.
\ee 
From \eqref{COV}, \eqref{maximum} and $s < t_* + \tb(x,v)$, we have
\Be \notag
\int^{t}_{t_*} \int_{\gamma_+} \eqref{|f|1}_3 \dd s
\leq \int^{t_* + \tb(x,v)}_{t_*} \int_{\gamma_+} \eqref{|f|1}_3 \dd s 
= \| f(t_*) \|_{L^1_{x,v}}
\leq \| f(t_*) \|_{L^1_{x,v}}.
\Ee
Again setting $y= \xb(x,v)$ and $s \in [t_*,t]$, we have 
\Be \label{delta tf 2}
\mathbf{1}_{\tb(x,v) \leq s - t_* < \delta}
\leq \mathbf{1}_{\delta \geq \tf(y, \vb)}.
\Ee
From \eqref{COV_bdry}, \eqref{delta tf 2} and using the Fubini's theorem, we derive
\begin{align} 
\int^{t}_{t_*} \int_{\gamma_+} \eqref{|f|1}_4 \dd s
& = \int_{\gamma_+}
\int^{t_* + \delta}_{ t_* + \tb(x,v)}
\mathbf{1}_{\delta \geq \tb(x,v)}
|f(s- \tb(x, v), \xb, \vb)| \dd s 
\{n(x) \cdot v\} \dd S_x \dd v \notag \\
& \leq \int_{\p\O} \int_{n(y) \cdot v<0}
\mathbf{1} _{\delta> \tf(y, v)}
\int^{t_* + \delta}_{ t_*} |f(s, y, v)| \dd s |n(y) \cdot v| \dd S_{y} \dd v \notag \\
&\leq \int_{\p\O}
\underbrace{\Big( \int_{n(y) \cdot v<0}
\mathbf{1} _{\delta> \tf(y, v)} \mu_{\Theta} (y, v) |n(y) \cdot v|\dd v\Big)}_{\eqref{est:|f|4}_*}
\int^{t_* + \delta}_{ t_*} \int_{\gamma_+} |f| \dd s.
 \label{est:|f|4}
\end{align}
Then we conclude $\eqref{est:|f|4} \leq O(\delta^2) \int^t_{t_*} \int_{\gamma_+} |f| \dd s$, therefore we prove \eqref{trace}.

To prove the positivity property, we write 
$$
f_{-} = \frac{|f| - f}{2}.$$
From \eqref{maximum} and the assumption $f_0 \geq 0$, we have
\Be \label{positivity}
\| f_{-} (t) \|_{L^1_{x,v}} 
= \big\| \frac{|f| (t) - f (t)}{2} \big\|_{L^1_{x,v}} 
= \Big\| \frac{ \big(|f| - f \big) (t)}{2} \Big\|_{L^1_{x,v}}
\leq \big\| \frac{|f_0| - f_{0} }{2} \big\|_{L^1_{x,v}} = 0,
\Ee
then we conclude $f_{-} (t, x, v) = 0$ on $\Omega \times \R^3$.

For the magnet field case, we consider $f$ solves \eqref{eqtn_f}-\eqref{diff_f} for \eqref{field property mag}.
By direct computation, we have
\be \notag
\nabla_{v} \cdot (v \times B - \nabla \Phi (x))
= \nabla_{v} \cdot (B_3 v_2, -B_3 v_1, 0) 
= 0.
\ee
By taking integration over $(t_*, t) \times \Omega \times \R^3$, 
we derive \eqref{maximum}.
The rest of proof is similar as the gravitational case and so the proof is omitted.
\end{proof}

Now we are ready to prove Lemma \ref{lemma:energy}, which will be used frequently in this paper.

\begin{proof}[\textbf{Proof of Lemma \ref{lemma:energy}}] 

Considering the characteristics trajectory
$\big(s, X(s; t, x, v), V(s; t, x, v) \big)$ determined by $(t, x, v)$, then we apply this characteristics on $\tf (x, v)$, 
\Be \notag
\begin{split}
    -1 & = \frac{d}{d s} \tf (s, X(s; t, x, v), V(s; t, x, v))
    \\& = \frac{\partial}{\partial s} \tf (s, X, V) + \frac{\partial}{\partial X} \tf (s, X, V) \cdot \frac{d}{ds} X
    + \frac{\partial}{\partial V} \tf (s, X, V) \cdot \frac{d}{ds} V,
\end{split}
\Ee
By setting $s = t$, we have
\Be \notag
\begin{split}
    & \frac{\partial}{\partial t} \tf (t, x, v) + \frac{\partial}{\partial x} \tf (t, x, v) \cdot v
    + \frac{\partial}{\partial v} \tf (t, x, v) \cdot - \nabla \Phi (x) = -1.
\end{split}
\Ee
Then, in the sense of distribution 
\Be \label{function for varphi f}
[\p_t + v\cdot \nabla_x - \nabla \Phi \cdot \nabla_{v}] \big(\varphi(\tf) |f| \big) 
= \varphi^\prime(\tf) [\p_t + v\cdot \nabla_x - \nabla \Phi \cdot \nabla_{v}](\tf) |f|
= -\varphi^\prime(\tf)| f|. 
\Ee
From \eqref{diff_f}, \eqref{function for varphi f}, $\varphi (\tau ) \geq0$, $\varphi^\prime \geq0$ and taking integration over $(t_*, t) \times \Omega \times \R^3$, 
we derive
\begin{align}
& \ \ \ \ \| \varphi(\tf) f (t) \|_{L^1_{x,v}} + \int^t_{t_*} \| \varphi^\prime(\tf) f (s) \|_{L^1_{x,v}} \dd s
+ \int^t_{t_*} \int_{\gamma_+} \varphi(\tf) |f| \dd s \notag
\\&
\leq \| \varphi(\tf) f (t_*) \|_{L^1_{x,v}}
+ \int^t_{t_*} 
\int_{\p\O}
\int_{n(x) \cdot v<0}
 \varphi(\tf) |f| |n(x) \cdot v|
\dd v
\dd S_x \dd s \notag
\\& = \| \varphi(\tf) f (t_*) \|_{L^1_{x,v}}
+ \int^t_{t_*} 
\int_{\p\O} 
\int_{n(x) \cdot v<0}
 \varphi(\tf) \mu_{\Theta} (x, v) |n(x) \cdot v| \dd v \dd S_x \dd s \notag
\\& \hspace{6cm} \times
\int_{n(x)\cdot v^1>0}| f(s,x,v^1)| 
\{n(x^1) \cdot v^1\}\dd v^1. \label{rho.f}
\end{align}

Now we prove the following claim: 
If \eqref{cond:varphi} holds, then 
\Be \label{claim in lemma 3}
\sup_{x \in \p\O} \int_{n(x) \cdot v<0}
 \varphi(\tf)(x,v) \mu_{\Theta} (x, v) |n(x) \cdot v| \dd v\lesssim 1.
\Ee
We split $\int_{n(x) \cdot v<0} \varphi(\tf)(x,v) \mu_{\Theta} (x, v) |n(x) \cdot v| \dd v$ into two parts: integration over the regimes of $\tf \leq 1$ and $\tf > 1$. 

For $\tf \leq 1$, since $|n(x) \cdot v| \lesssim \tf \leq 1$, we bound 
\Be \label{int:rho_1}
\begin{split}
 \int_{n(x) \cdot v<0} \mathbf{1}_{\tf \leq 1} \
\varphi(\tf) \mu_{\Theta} (x, v) |n(x) \cdot v| \dd v
\lesssim \varphi(1) \int_{\R^3} 
 e^{-|v|^2 / 2b} \dd v \lesssim 1. 
\end{split} 
\Ee

Next we focus on the integration over the regimes of $\tf > 1$. 
From \eqref{v3t estimate} and \eqref{t>1,i-2}, we obtain
\Be \label{int:rho_2} 
\int_{n (x) \cdot v<0} \varphi(\tf) \mu_{\Theta} (x, v) |n (x) \cdot v| \dd v
\lesssim \int_{\p\O} \int_{1}^{\infty}
\varphi(\tf) e^{- |\tf|^2/2b } |n (x) \cdot v| \frac{|n (\xf) \cdot \vf |}{|\tf|^3}  
\dd \tf \dd S_{\xf}. 
\Ee
From \eqref{cond:varphi}, we derive that 
\Be \notag
\eqref{int:rho_2}   \lesssim 
\int_{1}^\infty
\frac{\varphi(\tf)}{|\tf|} e^{- |\tf|^2 / 2b}
\dd \tf
\lesssim \int_{1}^\infty \varphi(\tf) e^{- |\tf|^2 / 2b} \dd \tf\lesssim 1.
\Ee
Combining the above bound with \eqref{int:rho_1}, we prove \eqref{claim in lemma 3}. 
Then picking sufficiently small $\delta$ in \eqref{trace} and using \eqref{claim in lemma 3}, we conclude \eqref{energy_varphi}, through, for $C>1$,   
\Be \notag
\begin{split}
\eqref{rho.f} 
& \lesssim \int^t_{t_*} 
\int_{\gamma_+} |f(s,x,v^1)| \{n(x^1) \cdot v^1\} \dd v^1 \dd S_x
\dd s 
\\& \leq C (t - t_* + 1) \| f(t_*) \|_{L^1_{x,v}} + \frac{1}{4}  \int^{t}_{t_*} |f(s)|_{L^1(\gamma_+)}.
\end{split}
\Ee  

For the magnet field case, from
$\nabla_{v} \cdot (v \times B - \nabla \Phi (x))
= 0$ and $\frac{d}{d s} \tf (s, X(s; t, x, v), V(s; t, x, v)) = -1$, we get
\Be \label{function for varphi f mag}
[\p_t + v\cdot \nabla_x + (v \times B - \nabla \Phi ) \cdot \nabla_{v}] \big(\varphi(\tf) |f| \big) 
= -\varphi^\prime(\tf)| f|. 
\Ee
The rest of proof follows from the gravitational case.
\end{proof}

\subsection{Doeblin condition}

In this section, we will prove two of the key cornerstones, Proposition \ref{prop:Doeblin 1} and \ref{prop:Doeblin mag}, the lower bound with the unreachable defect.
 
\begin{proposition} \label{prop:Doeblin 1}
Suppose $f$ solves \eqref{eqtn_f}-\eqref{diff_f}. Assume $f_0(x,v) \geq 0$. For any $T_0\gg1$ and $N\in \mathbb{N}$ there exists $\mathfrak{m}(x,v)\geq 0$, which only depends on $\O$ and $T_0$ (see \eqref{def:m} for the precise form), such that 
\Be \label{est:Doeblin}
f(NT_0,x,v)\geq  
\mathfrak{m}(x,v) \Big\{
\iint_{\O \times \R^3}f((N-1)T_0,x,v) \dd v \dd x 
-  \iint_{\O \times \R^3} \mathbf{1}_{t_\mathbf{f}(x,v)\geq \frac{T_0}{4}} f((N-1)T_0,x,v)  \dd v \dd x 
\Big\}.
\Ee
\end{proposition} 

\begin{proof}
\textbf{Step 1.} 
From \eqref{positivity}, we have $f(t,x,v) \geq 0$ from assumption $f_0 (x, v) \geq 0$. 
From \eqref{expand_h1}-\eqref{expand_h3} and setting $t = NT_0$, $t_* = (N-1)T_0$, $k=2$, we can derive that 

\begin{align}
f(NT_0, x,v) 
& \geq \mathbf{1}_{ \tb(x,v) \leq \frac{T_0}{4}} \mu_{\Theta} (x^{1}, \vb) \int_{\mathcal{V}_1} 
 \int_{\mathcal{V}_2}  
 \mathbf{1}_{t^2 \geq (N-1)T_0}
 f(t^2, x^2, v^2)  \{n(x^2) \cdot v^2\} \dd v^2 \dd \sigma_1.\label{Doeblin_1}
\end{align} 

Now we apply Proposition \ref{prop:mapV}  on $v^1 \in \mathcal{V}_1$ with \eqref{mapV} and \eqref{jacob:mapV}. 
In order to have the bijective mapping with \eqref{mapV}, we restrict the range of $\vb^1$ as 
\Be \label{def:x^2}
\mathcal{V}_{1,b} := \{ \vb^1 \in \R^3:  x^2 + \int^{t}_{t - \tb} \big( \vb +  \int^{s}_{t - \tb}  - \nabla \Phi ( X(\alpha; t^1,x^1, v^1)) \dd \alpha \big) \dd s = x^1 \}.
\Ee
This implies all characteristic trajectories $X(\alpha; t^1,x^1, v^1)$ between $x^1$ and $x^2$ under $\vb^1 \in \mathcal{V}_{1,b}$ don't cross the periodic boundary.

Therefore, we derive 
\Be \label{Doeblin_2}
\begin{split}
\eqref{Doeblin_1} 
& \gtrsim
\mathbf{1}_{ \tb(x,v) \leq \frac{T_0}{4}}  \mu_{\Theta} (x^{1}, \vb) 
 \int_{0}^{T_0 - \tb (x, v)}
  \int_{\p\O} 
  \underbrace{ \frac{|n(x_{1_\textbf{b}}) \cdot \vb^1|}{|\tb^1|^3}  
   \{n(x^1) \cdot v^1\}
\mu_{\Theta} (x^2, \vb^1)}_{\eqref{Doeblin_2}_*}
\\& \ \ \ \ \ \ \times 
   \int_{n(x^2) \cdot v^2>0}
   f(t^2, x^2, v^2)  \{n(x^2) \cdot v^2\} \dd v^2
    \dd S_{x^2}\dd \tb^1,
\end{split}
\Ee
where $t^2 = NT_0 - \tb(x,v) - \tb^1$.
  
\smallskip 
  
\textbf{Step 2.} In order to bound the integrands of the first line of \eqref{Doeblin_2}, we will further restrict integration regimes. Note that $x^1 = \xb(x,v)$ is given, $x^2$ is free variables and $t^2 \geq (N-1) T_0$.

Now we restrict the integral regimes of the variable $\tb^1$ as 
  
\Be \label{def:T+}
\mathfrak{T}^{T_0} :=
 \Big\{
 \tb^1 \in  [0, \infty):
 T_0- \tb(x,v) - \min\Big(\tb(x^2, v^2) ,
 \frac{T_0}{4}
\Big)
  \leq \tb^1 \leq 
 T_0- \tb(x,v)
 \Big\}.
\Ee
As a consequence of \eqref{def:T+} 
and $\tb(x,v) \leq \frac{T_0}{4}$ in \eqref{Doeblin_2}, we will derive  \eqref{min:tb1} and \eqref{cond:tf},
\Be \label{min:tb1}
\frac{T_0}{2} \leq T_0 - \tb(x,v) - \frac{T_0}{4}
\leq \tb^1 \leq T_0.
\Ee
Secondly, we prove \eqref{cond:tf}. Note that if $\tb^1 \in \mathfrak{T}^{T_0}$, we have
\Be \notag
(N-1)T_0
\leq t^2 =
 NT_0 - \tb(x,v) - \tb^1 \leq 
 (N-1)T_0 
 +\min \{\tb(x^2, v^2) ,
 \frac{T_0}{4}
\}.
\Ee
This implies that, for $y_* = X((N-1)T_0; t^2, x^2, v^2)$ and $v_* = V((N-1)T_0; t^2, x^2, v^2)$, we have 
\Be \label{cond:tf}
\tf(y_*, v_*)= t^2 - (N-1)T_0 = T_0 - \tb(x,v) - \tb^1 \in 
 \Big[0, \frac{T_0}{4}\Big],
\Ee
where we use $\tf (y_*, v_*) \leq \tb (x^2, v^2)$ since $x^2 = \xf(y_*, v_*)$. 
 
\smallskip

\textbf{Step 3.} For \eqref{Doeblin_2}, we apply the restriction of integral regimes in \eqref{def:x^2} and \eqref{def:T+}.
Using \eqref{v3t estimate}, \eqref{min:tb1} and the assumption $T_0 \gg 1$, we derive that 

\Be \notag
\begin{split}
\eqref{Doeblin_2}_*
& \gtrsim \frac{|\tb^1|}{|\tb^1|^3} \tb^1 \mu_{\Theta} (x^2, \vb^1)
\gtrsim \frac{1}{|\tb^1|}  
\mu^{1/b} \big( |\vb^1- n(x^2) \cdot \vb^1| + |n(x^2) \cdot \vb^1| \big)
\\& \gtrsim \frac{1}{|\tb^1|}  
\mu^{1/b} \big(|n(x^2) \cdot \vb^1| \big) 
\mu^{1/b} \big( |\vb^1- n(x^2) \cdot \vb^1|\big) 
\\& \gtrsim \frac{1}{T_0} \mu^{1/b} ( \frac{g \tb^1}{2}) 
\geq e^{- 25 T_0^2/b}.
\end{split}
\Ee
Finally we get 
 \begin{align}
 \eqref{Doeblin_2} 
& \geq  \mathbf{1}_{ \tb(x,v) \leq \frac{T_0}{4}} 
e^{- 25 T_0^2/b}  
  \mu_{\Theta} (x^1, \vb) \int_{\p\O} \dd S_{x^2} \int_{n(x^2) \cdot v^2 > 0}
 \dd v^2 \{n(x^2) \cdot v^2\} \notag \\
&  \ \ \ \ \ \ \times 
	\int_{\mathfrak{T}^{T_0}} \dd \tb^1
   f(
   NT_0 -\tb(x,v) - \tb^1
   ,x^2, v^2)  \notag \\
& \gtrsim  \mathbf{1}_{ \tb(x,v) \leq \frac{T_0}{4}} 
e^{- 25 T_0^2/b} 
  \mu_{\Theta} (x^1, \vb)
  \int_{\p\O} \dd S_{x^2} \int_{n(x^2) \cdot v^2 > 0}\dd v^2 \{n(x^2) \cdot v^2\} 
  \notag
  \\
& \ \ \ \ \times
  \int^{ T_0- \tb(x,v)}_{   
 T_0- \tb(x,v) - \min\big(\tb(x^2, v^2) ,
 \frac{T_0}{4}
\big)} \dd \tb^1  f(
   NT_0 -\tb(x,v) - \tb^1
   ,x^2, v^2).\label{lower1}
 \end{align} 
 
Now we focus on the integrand of \eqref{lower1}. Recall \eqref{def:T+}, we have  
\Be \notag
(NT_0 - \tb(x,v) - \tb^1)- (N-1)T_0= T_0 - \tb(x,v) - \tb^1
\in \Big[0, \min\Big(\tb(x^2, v^2) ,
 \frac{T_0}{4}
\Big)\Big]. 
\Ee
Now Setting $y_* = X((N-1)T_0;t^2, x^2, v^2)$, $v_* = V((N-1)T_0;t^2, x^2, v^2)$ and $\alpha = T_0 - \tb(x,v) - \tb^1$, 
we have
\Be \label{lower2}
\eqref{lower1}
= \int^{\min\big(\tb(x^2, v^2) ,
 \frac{T_0}{4}
\big)}_{0}
 f
\big( (N-1)T_0, y_*, v_* \big) \dd \alpha.
\Ee
From \eqref{cond:tf}, we have $\tf(y_*, v_*) \in \big[0, \frac{T_0}{4}\big]$. Now applying \eqref{COV}, we conclude that 
 \Be \notag
\eqref{Doeblin_2} 
\geq \mathbf{1}_{ \tb(x,v) \leq \frac{T_0}{4}} 
e^{- 25 T_0^2/b} 
  \mu_{\Theta} (x^1, \vb)
  \iint_{\O \times \R^3} \mathbf{1}_{\tf(y,v)  \in [0, \frac{T_0}{4}]} 
  f((N-1)T_0, y,v) \dd v \dd y.  \notag
 \Ee
We conclude \eqref{est:Doeblin} by setting 
 \Be \label{def:m}
 \mathfrak{m} (x,v):=  \mathbf{1}_{ \tb(x,v) \leq \frac{T_0}{4}} 
e^{- 25 T_0^2/b} \mu_{\Theta} (x^1, \vb).
\Ee 
\end{proof} 

An immediate consequence of Proposition \ref{prop:Doeblin 1}, follows.

\begin{proposition} \label{prop:Doeblin}
 Suppose $f$ solves \eqref{eqtn_f}-\eqref{diff_f}, and satisfies \eqref{cons_mass_f}. Then for all $T_0\gg1$ and $N \in \mathbb{N}$, 
\Be \label{L1_coerc}
   \|f(NT_0)\|_{L^1_{x,v}}  \leq  (1-\| \mathfrak{m}\|_{L^1_{x,v}} )   \|f((N-1)T_0)\|_{L^1_{x,v}} 
   + 2 \| \mathfrak{m}\|_{L^1_{x,v}}  \| \mathbf{1}_{\tf\geq \frac{T_0}{4}} f((N-1)T_0)\|_{L^1_{x,v}}.
 \Ee 
Moreover, we have
\Be \label{est:m}
 \| \mathfrak{m} \|_{L^1_{x,v}} := \mathfrak{m}_{T_0} \lesssim e^{-25 T_0^2/b} T_0 |\p\O|. 
\Ee
\end{proposition}

\begin{proof}

We decompose 
\begin{align*}
 f((N-1)T_0,x,v) 
 = f_{ N-1 ,+}(x,v) - f_{ N-1 ,-}(x,v),\end{align*}
where 
\Be \notag
\begin{split}
& f_{N-1,+} (x,v ) = \mathbf{1}_{f((N-1)T_0,x,v)  \geq 0} f \big( (N-1)T_0,x,v \big),
\\& f_{N-1,-} (x,v ) = \mathbf{1}_{f((N-1)T_0,x,v)  < 0} |f\big( (N-1)T_0,x,v \big)|.
\end{split}
\Ee
Let $f_{\pm}(s,x,v)$ solve \eqref{eqtn_f} for $s \in [ (N-1)T_0,NT_0]$ with the initial data $f_{N-1,+}$ and $f_{N-1,-}$ at $s=(N-1)T_0$, respectively.
 
Now we apply Proposition \ref{prop:Doeblin 1} on $f_{\pm}(t,x,v)$ and conclude \eqref{est:Doeblin} for $f= f_+$ and $f=f_-$ respectively. We also note that 
\Be \notag 
\iint_{\O \times \R^3} f((N-1)T_0,x,v) \dd x \dd v=\iint_{\O \times \R^3} f_{N-1,+}(x,v) \dd x \dd v- \iint_{\O \times \R^3} f_{N-1,-}(x,v) \dd x \dd v =0.
\Ee
This implies,
\Be \label{pm equality}
\iint_{\O \times \R^3} f_{N-1, \pm}( x,v)  \dd x \dd v =\frac{1}{2} \iint_{\O \times \R^3} |f((N-1)T_0,x,v)| \dd x \dd v.
\Ee
From \eqref{est:Doeblin},
\Be \notag
f_{N-1, \pm}( x,v) \geq \mathfrak{m}(x,v) \iint f_{N-1,\pm}(x,v) \dd x \dd v -  
  \mathfrak{m}(x,v) \iint_{\O \times \R^3} \mathbf{1}_{\tf(x,v) \geq \frac{T_0}{4}} f_{N-1,\pm} (x,v ) \dd x \dd v
\Ee
Using \eqref{pm equality}, we have
\Be \label{lowerB:f}
f_{N-1, \pm}( x,v) \geq 
\underbrace{ 
\mathfrak{m}(x,v) \Big(
\frac{1}{2} \|f((N-1)T_0)\|_{L^1_{x,v}}  
  - \| \mathbf{1}_{\tf\geq \frac{T_0}{4}} f((N-1)T_0)\|_{L^1_{x,v}}
  \Big)
}_{\mathfrak{l}(x,v)}. 
\Ee 
Then we deduce
\Be \notag
\begin{split}
|f(NT_0,x,v)| 
& = |f_{N-1, +}( x,v) - f_{N-1, -}( x,v) + \mathfrak{l}(x,v) - \mathfrak{l}(x,v)|
\\& \leq  |f_{N-1, +}( x,v) - \mathfrak{l}(x,v)| +|f_{N-1, -}( x,v) - \mathfrak{l}(x,v)|.
\end{split}
\Ee 
From \eqref{lowerB:f},
\Be \label{upperNT:f}
|f(NT_0,x,v)| \leq f_{N-1, +}( x,v) + f_{N-1, -}( x,v) - 2 \mathfrak{l}(x,v).
\Ee
Note that $f_+(NT_0,x,v)+f_-(NT_0,x,v)$ solves \eqref{eqtn_f} with the initial datum $$f_{N-1,+} + f_{N-1,-}
= \big| f \big((N-1)T_0,x,v \big) \big|.$$ 
Using \eqref{pm equality}, \eqref{upperNT:f} and taking the integration on \eqref{lowerB:f} over $\O \times \R^3$, we derive
\Be \notag
\begin{split}
\|f(NT_0)\|_{L^1_{x,v}} 
& \leq \iint_{\O \times \R^3} f_{N-1, +}( x,v)  \dd x \dd v
+ \iint_{\O \times \R^3} f_{N-1, -}( x,v)  \dd x \dd v - \iint_{\O \times \R^3} 2 \mathfrak{l}(x,v)  \dd x \dd v
\\& = (1-\| \mathfrak{m}\|_{L^1_{x,v}} )   \|f((N-1)T_0)\|_{L^1_{x,v}} 
   + 2 \| \mathfrak{m}\|_{L^1_{x,v}}  \| \mathbf{1}_{\tf\geq \frac{T_0}{4}} f((N-1)T_0)\|_{L^1_{x,v}},
\end{split}
\Ee
therefore we prove \eqref{L1_coerc}.

To derive \eqref{est:m}, it suffices to bound $\| \mathbf{1}_{\tb(x,v) \leq \frac{T_0}{4}}  
\mu_{\Theta} (x^1, \vb)\|_{L^1_{x,v}}$. 
From Lemma \ref{conservative field},  \eqref{def:Theta}, \eqref{COV} and the fact that $\tb (t-s, X(t-s,t,x,v),V(t-s,t,x,v)) = \tb(t,x,v)- s$, we have
\Be \notag
\begin{split}
& \ \ \ \ \| \mathbf{1}_{\tb(x,v) \leq \frac{T_0}{4}} \mu_{\Theta} (x^1, \vb) \|_{L^1_{x,v}} 
 = \int_{\gamma_{+}} \int_{ \max\{0,\tb(x,v)- \frac{T_0}{4}\}}  ^{\tb(x,v)} 
\mu_{\Theta} (x^1, v) \{n(x) \cdot v\}\dd s \dd v \dd S_x \\
\\& \lesssim \int_{\gamma_{+}}  
\Big(  \mathbf{1}_{\tb(x,v) \leq \frac{T_0}{4}} \int_0^{\tb(x,v)} \dd s
+\mathbf{1}_{\tb(x,v) \geq \frac{T_0}{4}}\int^{\tb(x,v)}_{\tb(x,v)- \frac{T_0}{4}}  \dd s \Big) 
e^{-|v|^2 / 2b} \{n(x) \cdot v\} \dd v \dd S_x
\\& \leq \frac{T_0}{4} \int_{\p\O } \dd S_x \int_{n(x) \cdot v>0} 
e^{-|v|^2 / 2b} \{n(x) \cdot v\} \dd v \lesssim T_0 |\p\O|.  
\end{split}
\Ee
Combining the above bound with \eqref{def:m}, we conclude \eqref{est:m}.
\end{proof}

\begin{proposition} \label{prop:Doeblin mag}
Suppose $f$ solves \eqref{eqtn_f}-\eqref{diff_f} for \eqref{field property mag} and assume $f_0(x,v) \geq 0$. For any $T_0\gg1$ and $N\in \mathbb{N}$ there exists $\mathfrak{m}(x,v)\geq 0$, which only depends on $\O$ and $T_0$ (see \eqref{def:m mag} for the precise form), such that 
\Be \label{est:Doeblin mag}
f(NT_0,x,v)\geq  
\mathfrak{m}(x,v) \Big\{
\iint_{\O \times \R^3}f((N-1)T_0,x,v) \dd v \dd x 
-  \iint_{\O \times \R^3} \mathbf{1}_{t_\mathbf{f}(x,v)\geq \frac{T_0}{4}} f((N-1)T_0,x,v)  \dd v \dd x 
\Big\}.
\Ee
And if $f$ also satisfies \eqref{cons_mass_f}, then for any $T_0\gg1$ and $N \in \mathbb{N}$
\Be \label{L1_coerc mag}
   \|f(NT_0)\|_{L^1_{x,v}}  \leq  (1-\| \mathfrak{m}\|_{L^1_{x,v}} )   \|f((N-1)T_0)\|_{L^1_{x,v}} 
   + 2 \| \mathfrak{m}\|_{L^1_{x,v}}  \| \mathbf{1}_{\tf\geq \frac{T_0}{4}} f((N-1)T_0)\|_{L^1_{x,v}}.
 \Ee 
Moreover, we have
\Be \label{est:m mag}
 \| \mathfrak{m} \|_{L^1_{x,v}} := \mathfrak{m}_{T_0} \lesssim e^{-25 T_0^2/b} T_0 |\p\O|,
\Ee
where $b$ was defined in \eqref{def:Theta}.
\end{proposition} 

\begin{proof}
\textbf{Step 1.} 
From Lemma \ref{lem:energy estimate on f}, we have $f(t,x,v) \geq 0$ from assumption $f_0 (x, v) \geq 0$. 
From \eqref{expand_h1}-\eqref{expand_h3} and setting $t = NT_0$, $t_* = (N-1)T_0$, $k=2$, we can derive that 
\be \label{Doeblin_1 mag}
f(NT_0, x,v) 
 \geq \mathbf{1}_{ \tb(x,v) \leq \frac{T_0}{4}} \mu_{\Theta} (x^{1}, \vb) \int_{\mathcal{V}_1} 
\int_{\mathcal{V}_2}  
\int_{\mathcal{V}_3} 
\mathbf{1}_{t^3 \geq (N-1)T_0}
 f(t^3, x^3, v^3)  \{n(x^3) \cdot v^3\} \dd v^3 \dd \sigma_2 \dd \sigma_1.
\ee 

Now we apply Proposition \ref{prop:mapV mag} on $v^1 \in \mathcal{V}_1$ and $v^2 \in \mathcal{V}_2$ with \eqref{mapV mag} and \eqref{jacob:mapV mag}. 
In order to have the bijective mapping with \eqref{mapV mag}, we restrict the range of $\vb^1$ and $\vb^2$ as 
\Be \label{def:x^2 mag}
\begin{split}
& \mathcal{V}_{1,b} := \{ \vb^1 \in \R^3:  x^2 + \int^{t}_{t - \tb} \big( \vb^1 +  \int^{s}_{t - \tb}  - \nabla \Phi ( X(\alpha; t^1,x^1, v^1)) \dd \alpha \big) \dd s = x^1 \},
\\& \mathcal{V}_{2,b} := \{ \vb^2 \in \R^3:  x^3 + \int^{t - \tb}_{t - \tb - \tb^1} \big( \vb^2 +  \int^{s}_{t - \tb - \tb^1}  - \nabla \Phi ( X(\alpha; t^2,x^2, v^2)) \dd \alpha \big) \dd s = x^2 \}.
\end{split}
\Ee 
This implies two characteristic trajectories $X(\alpha; t^1,x^1, v^1)$ between $x^1$ and $x^2$ under $\vb^1 \in \mathcal{V}_{1,b}$, and $X(\alpha; t^2,x^2, v^2)$ between $x^2$ and $x^3$ under $\vb^2 \in \mathcal{V}_{2,b}$ don't cross the periodic boundary.

Therefore, we derive 
\Be \label{Doeblin_2 mag}
\begin{split}
\eqref{Doeblin_1 mag} 
& \geq
\mathbf{1}_{ \tb(x,v) \leq \frac{T_0}{4}}  \mu_{\Theta} (x^{1}, \vb) 
 \int_{0}^{T_0 - \tb }
\int_{\p\O} 
\underbrace{ \frac{5 B^2_3}{2 - 2 \cos (B_3 \tb^1)}
   \{n(x^1) \cdot v^1\}
\mu_{\Theta} (x^2, \vb^1)}
_{\eqref{Doeblin_2 mag}_{**}} \dd S_{x^2} \dd \tb^1
\\& \ \ \ \ \ \ \times 
\int_{0}^{T_0 - \tb - \tb^1}
  \int_{\p\O} 
\underbrace{ \frac{5 B^2_3}{2 - 2 \cos (B_3 \tb^2)}
   \{n(x^2) \cdot v^2\}
\mu_{\Theta} (x^3, \vb^2)}_{\eqref{Doeblin_2 mag}_{*}} \dd S_{x^3} \dd \tb^2
\\& \ \ \ \ \ \ \times 
   \int_{n(x^3) \cdot v^3>0}
   f(t^3, x^3, v^3)  \{n(x^3) \cdot v^3\} \dd v^3,
\end{split}
\Ee
where $t^3 = NT_0 - \tb(x,v) - \tb^1 - \tb^2$.
  
\textbf{Step 2.} In order to bound the integrands of the first line of \eqref{Doeblin_2 mag}, we will further restrict integration regimes. Note that $x^1 = \xb(x,v)$ is given, $x^2$ is free variables and $t^2 \geq (N-1) T_0$.

Now we restrict the integral regimes of the variable $\tb^1$ and $\tb^2$ as 
\Be \label{def:T+ mag}
\begin{split}
& \mathfrak{T}^{T_0} 
:= \Big\{
 \tb^1, \tb^2 \geq 0:
 T_0 - \tb(x,v) - \min\Big(\tb(x^3, v^3) ,
 \frac{T_0}{4}
\Big)
  \leq \tb^1 + \tb^2 \leq 
 T_0- \tb(x,v)
 \Big\},
\\& \mathfrak{T}_2^{T_0} 
:= \Big\{
 \tb^2 \in  [0, \infty):
 \frac{\pi}{2B_3}
  \leq \tb^2 \leq 
 \frac{3\pi}{2B_3}
 \Big\},
\end{split}
\Ee
As a consequence of \eqref{def:T+ mag} 
and $\tb(x,v) \leq \frac{T_0}{4}$ in \eqref{Doeblin_2 mag}, we will derive  
\Be \label{min:tb1 mag}
\frac{T_0}{2} \leq T_0 - \tb(x,v) - \frac{T_0}{4}
\leq \tb^1 + \tb^2 \leq T_0,
\Ee
and since $B_3, T_0 \geq 1$, the restriction on $\tb^2$ is reasonable.

Secondly, note that if $\tb^1, \tb^2 \in \mathfrak{T}^{T_0}$, we have
\Be \notag
(N-1)T_0
\leq t^3 =
 NT_0 - \tb(x,v) - \tb^1 - \tb^2 \leq 
 (N-1)T_0 
 +\min \{\tb(x^3, v^3),
 \frac{T_0}{4}
\}.
\Ee
This implies that, for $y_* = X((N-1)T_0; t^3, x^3, v^3)$ and $v_* = V((N-1)T_0; t^3, x^3, v^3)$, we have 
\Be \label{cond:tf mag}
\tf(y_*, v_*)= t^3 - (N-1)T_0 = T_0 - \tb(x,v) - \tb^1 - \tb^2 \in 
 \Big[0, \frac{T_0}{4}\Big],
\Ee
where we use $\tf (y_*, v_*) \leq \tb (x^3, v^3)$ since $x^3 = \xf(y_*, v_*)$. 

\textbf{Step 3.} For \eqref{Doeblin_2 mag}, we apply the restriction of integral regimes in \eqref{def:x^2 mag} and \eqref{def:T+ mag}, and obtain
$\pi/2 \leq B_3 \tb^2 \leq 3\pi/2$ and $\|x^3 - x^2\| \leq \sqrt{2}$.
Then using  \eqref{def:Theta}, \eqref{tb expression mag}, \eqref{vmnb12 first estimate mag} with $m = n = 0$, we derive that 
\Be \notag
\begin{split}
\eqref{Doeblin_2 mag}_*
&  = \frac{5 B^2_3}{2 - 2 \cos (B_3 \tb^2)}
   \{n(x^2) \cdot v^2\}
\mu_{\Theta} (x^3, \vb^2)
\\& \geq \frac{5 B^2_3}{4} \times \frac{5 \pi}{2 B_3} \times e^{- \frac{1}{2b} \big( B^2_3 + (\frac{15 \pi}{2 B_3})^2 \big)}
\gtrsim B_3 e^{- \frac{1}{2b} B^2_3}.
\end{split}
\Ee
Similarly, we get
\Be \notag
\begin{split}
\eqref{Doeblin_2 mag}_{**}
& = \frac{5 B^2_3}{2 - 2 \cos (B_3 \tb^1)}
   \{n(x^1) \cdot v^1\}
\mu_{\Theta} (x^2, \vb^1)
\\& \geq \frac{25 B^2_3 \tb^1}{2 - 2 \cos (B_3 \tb^1)} \times e^{- \frac{1}{2b} \big( \frac{2 B^2_3 }{2 - 2 \cos (B_3 \tb^1)} + (5 \tb^1)^2 \big)}.
\end{split}
\ee
Finally, applying the change of variable $(\tb^1, \tb^2) 
\mapsto (\tb^1, \tb^1 + \tb^2)$, we get 
\begin{align}
 \eqref{Doeblin_2 mag}
& \geq \mathbf{1}_{ \tb(x,v) \leq \frac{T_0}{4}} B_3 e^{- \frac{1}{2b} B^2_3} \mu_{\Theta} (x^{1}, \vb) 
\int_{\p\O} \dd S_{x^3} \int_{n(x^3) \cdot v^3>0} \{n(x^3) \cdot v^3\} \dd v^3 \notag
\\& \ \ \ \ \ \ \times 
\int^{ T_0- \tb(x,v)}_{   
 T_0- \tb(x,v) - \min\big(\tb(x^3, v^3) ,
 \frac{T_0}{4}
\big)} \dd (\tb^1 + \tb^2)  f(NT_0 - \tb(x,v) - \tb^1 - \tb^2, x^3, v^3) \label{lower1 mag}
\\& \ \ \ \ \ \ \times 
 \int_{\tb^1 + \tb^2 - \frac{3\pi}{2B_3}}^{\tb^1 + \tb^2 - \frac{\pi}{2B_3}}
\int_{\p\O} 
\frac{25 B^2_3 \tb^1}{2 - 2 \cos (B_3 \tb^1)} \times e^{- \frac{1}{2b} \big( \frac{2 B^2_3 }{2 - 2 \cos (B_3 \tb^1)} + (5 \tb^1)^2 \big)} \dd S_{x^2} \dd \tb^1.
\label{lower12 mag}
\end{align} 
 
Now we focus on the integrand of \eqref{lower1 mag}. Recall \eqref{def:T+ mag}, we have  
\Be \notag
NT_0 - \tb(x,v) - \tb^1 - \tb^2 - (N-1)T_0= T_0 - \tb(x,v) - \tb^1 - \tb^2
\in \Big[0, \min\Big(\tb(x^3, v^3),
 \frac{T_0}{4}
\Big)\Big]. 
\Ee
Setting $y_* = X((N-1)T_0;t^3, x^3, v^3)$, $v_* = V((N-1)T_0;t^3, x^3, v^3)$ and $\alpha = T_0 - \tb(x,v) - \tb^1 - \tb^2$, 
we have
\Be \label{lower2 mag}
\eqref{lower1 mag}
= \int^{\min\big(\tb(x^3, v^3) ,
 \frac{T_0}{4}
\big)}_{0}
 f
\big( (N-1)T_0, y_*, v_* \big) \dd \alpha.
\Ee
From \eqref{cond:tf mag}, we have $\tf(y_*, v_*) \in \big[0, \frac{T_0}{4}\big]$.
Next, from \eqref{min:tb1 mag} we can deduce that
\Be \notag
\frac{T_0}{2} - \frac{3\pi}{2B_3} 
\leq \tb^1 \leq T_0 - \frac{\pi}{2B_3}.
\ee
Picking $T_0 > 10$ and $B_3 \geq 1$, we get $T_0/5 \leq \tb^1\leq T_0$. Note that the time gap for $\tb^1$ is $\pi/B_3$ and $2 - 2 \cos (B_3 \tb^1)$ is periodic, we derive
\Be \notag
\eqref{lower12 mag} 
\gtrsim e^{- \frac{1}{2b} (5 \tb^1)^2 }
\geq e^{- 25T_0^2/b  }.
\ee
Now applying \eqref{COV}, we conclude that 
 \Be \notag
\eqref{Doeblin_2 mag} 
\gtrsim \mathbf{1}_{ \tb(x,v) \leq \frac{T_0}{4}} e^{- 25T_0^2/b  }
 \mu_{\Theta} (x^1, \vb)
  \iint_{\O \times \R^3} \mathbf{1}_{\tf(y,v)  \in [0, \frac{T_0}{4}]} 
  f((N-1)T_0, y,v) \dd v \dd y.  \notag
 \Ee
We conclude \eqref{est:Doeblin mag} by setting 
 \Be \label{def:m mag}
 \mathfrak{m} (x,v):=  \mathbf{1}_{ \tb(x,v) \leq \frac{T_0}{4}} 
 e^{- 25T_0^2/b  }
\mu_{\Theta} (x^1, \vb).
\Ee 
Since $\mathfrak{m} (x,v)$ is equal to \eqref{def:m}  used in the gravitational case,
the proof for \eqref{L1_coerc mag} and \eqref{est:m mag} follows from Proposition \ref{prop:Doeblin} and so the proof is omitted.
\end{proof} 

\subsection{Proof of weighted \texorpdfstring{$L^1$}{L1}-Estimates}

Now we establish the uniform estimate of the following energy:
\Be \label{|||i}
\vertiii{f} := \|f  \|_{L^1_{x,v}}
    +    \frac{ 4 \mathfrak{m}_{T_0} }{ \varphi (\frac{T_0}{4})}
  \Big(  {1 + \frac{1}{\delta T_0}} \Big) \times
 \| \varphi (\tf) f\|_{L^1_{x,v}},
\Ee
with $\| \mathfrak{m} \|_{L^1_{x,v}} := \mathfrak{m}_{T_0}$ (see \eqref{est:m}) and $\varphi$ defined in \eqref{varphis}.

Here we first introduce the weight function $\varphi$.

\begin{definition} \label{def:varphis}
For $0 < \delta < \infty$, we set
\Be \label{varphis}
\begin{split}  
& \varphi (\tau) := e^{\delta \tau}.
\end{split}
\Ee 
First, we check $\varphi$ satisfies \eqref{cond:varphi}:
\Be \notag
\begin{split} 
& \int_1^\infty e^{- \frac{1}{2b} \tau^2} e^{\delta \tau} \dd \tau
\lesssim \int_1^\infty e^{- \frac{1}{2b} (\tau - b \delta )^2} \dd \tau < \infty.
\end{split}
\Ee
Second, we notice that,
\Be \label{varphi|0}
\varphi (0) = 1.
\Ee 
Finally, we check, for $\tau \geq 0$,
\Be \label{phi'}
\begin{split}
& \varphi^{\prime} (\tau)   
= \delta \varphi (\tau) .
\end{split}   
\Ee 
\end{definition}

\begin{lemma} \label{lem:energy}
Choose $T_0 > 1$ such that 
\Be \label{cond:T0}
\delta T_0 > 1, \ \
4 C (1+T_0) (1 + \delta T_0) (\delta T_0)^{-1}   \Big( \varphi (\frac{T_0}{4})\Big)^{-1}
\leq \frac{1}{2}.
\Ee
Suppose $f$ solves \eqref{eqtn_f}-\eqref{diff_f} for either \eqref{field property} or \eqref{field property mag}, for all $N \in \mathbb{N}$,
\Be \label{energyi}
\begin{split}   
& \ \ \ \ \| f(NT_0) \|_{L^1_{x,v}}
 + \frac{4 \mathfrak{m}_{T_0} }{ \varphi (\frac{T_0}{4})}
 \Big\{
 \| \varphi (\tf) f (NT_0)\|_{L^1_{x,v}}  +   
\frac{1}{\delta T_0} \| \varphi (\tf) f(NT_0) \|_{L^1_{x,v}} 
 \Big\}
\\& \leq 
\big\{ 1- \mathfrak{m}_{T_0} 
+ \mathfrak{m}_{T_0} \frac{ 4 C (1+T_0) (1+ \delta T_0) }{\delta T_0 \varphi (\frac{T_0}{4})}
\big\} 
\times \|f((N-1)T_0)\|_{L^1_{x,v}}  
\\& \hspace{3cm} + \frac{4 \mathfrak{m}_{T_0}}{ \varphi (\frac{T_0}{4})}
 \Big\{ 
\frac{1}{2} \|   \varphi (\tf)f((N-1)T_0)\|_{L^1_{x,v}} 
+ \frac{1 }{\delta T_0 }
\| \varphi (\tf) f((N-1)T_0) \|_{L^1_{x,v}}
 \Big\}.   
\end{split}
\Ee    
\end{lemma}

\begin{proof} 
Applying Lemma \ref{lemma:energy} on $f(t,x,v)$ which solves \eqref{eqtn_f}-\eqref{diff_f}, with $\varphi$ in \eqref{varphis}.
From \eqref{varphi|0} and \eqref{phi'}, we derive that, for $(N-1)T_0 \leq t_* \leq t = NT_0$,
\Be \label{energy:log}
\| \varphi (\tf) f(NT_0) \|_{L^1_{x,v}} +   \frac{3}{4} \int^{NT_0}_{t_*} |f (s)|_{L^1_{\gamma_+}} \dd s
   \leq  \|  \varphi (\tf)f(t_*) \|_{L^1_{x,v}}  + C (1+T_0) \| f(t_*) \|_{L^1_{x,v}}.
\Ee  
For $(N-1)T_0 = t_* \leq t = NT_0$, 
\Be \label{energy:phii}
\begin{split}
& \ \ \ \ \| \varphi (\tf) f(NT_0) \|_{L^1_{x,v}} + \int^{NT_0}_{(N-1)T_0} 
\big\{ \| \varphi^\prime(\tf) f (s) \|_{L^1_{x,v}}
 +  \frac{3}{4}  |f (s)|_{L^1_{\gamma_+}}
\big\} \dd s
\\& \leq \| \varphi (\tf) f((N-1)T_0) \|_{L^1_{x,v}}  + C (1+T_0) \| f((N-1)T_0) \|_{L^1_{x,v}}.
\end{split}
\Ee  
From \eqref{maximum}, \eqref{phi'} and \eqref{energy:log}, we derive that,
\Be \notag
\begin{split}
 \int^{NT_0}_{(N-1)T_0}
 \| \varphi^\prime(\tf) f (t_*) \|_{L^1_{x,v}} \dd t_* 
& = \delta \int^{NT_0}_{(N-1)T_0}
 \| \varphi (\tf) f (t_*) \|_{L^1_{x,v}} \dd t_* 
\\& \geq \delta \int^{NT_0}_{(N-1)T_0}
\big\{ \| \varphi (\tf) f(NT_0) \|_{L^1_{x,v}} 
 - C (1+T_0) \| f(t_*) \|_{L^1_{x,v}}
\big\}  \dd t_*
\\& \geq \delta T_0 \| \varphi (\tf) f(NT_0) \|_{L^1_{x,v}} - \delta C (1+T_0) T_0 \| f((N-1)T_0) \|_{L^1_{x,v}}.
\end{split}
\Ee
Applying the above bound on \eqref{energy:phii}, we have
 \Be \label{energy:phi1}
 \begin{split}
& \ \ \ \ \| \varphi (\tf) f(NT_0) \|_{L^1_{x,v}} 
+ \delta T_0
 \| \varphi (\tf) f  (NT_0)\|_{L^1_{x,v}}
 + \frac{3}{4} \int^{NT_0}_{(N-1)T_0} |f (s)|_{L^1_{\gamma_+}} \dd s
\\& \leq  \| \varphi (\tf) f((N-1)T_0) \|_{L^1_{x,v}}  + C (1+T_0) (1+ \delta T_0) \| f((N-1)T_0) \|_{L^1_{x,v}}.
\end{split} 
\Ee  
From \eqref{varphi|0} and \eqref{phi'}, we obtain
\be \notag
\mathbf{1}_{\tf \geq \frac{T_0}{4}}
  \leq \big( \varphi (\frac{T_0}{4}) \big)^{-1} \varphi (\tf).
\ee
Combining the above bound with \eqref{L1_coerc}, we have
\Be \label{L1_coerc'}
\begin{split}
& \|f(NT_0)\|_{L^1_{x,v}}  
\leq  (1-\mathfrak{m}_{T_0} )   \|f((N-1)T_0)\|_{L^1_{x,v}} 
   +  2 \mathfrak{m}_{T_0} \big( \varphi (\frac{T_0}{4}) \big)^{-1} \| \varphi (\tf)f((N-1)T_0)\|_{L^1_{x,v}}.
\end{split}
\Ee  
For $T_0>0$ in \eqref{cond:T0} and considering $\eqref{L1_coerc'} + \frac{4 \mathfrak{m}_{T_0} }{\delta T_0 \varphi (\frac{T_0}{4})} \eqref{energy:phi1}$, we deduce \eqref{energyi}.

For the magnet field case, since we have proved Proposition \ref{prop:Doeblin mag} and Lemma \ref{lem:energy}, the rest of proof follows from the gravitational case.
\end{proof}

Now we are well equipped to prove Theorem \ref{theorem_1} and \ref{theorem_1}.

\begin{proof}[\textbf{Proof of Theorem \ref{theorem_1}}]

We first note that $\tf (x, v) \leq \tf (\xb, \vb) \lesssim |n(\xb) \cdot \vb| \leq |\vb|$ and $|v|^2 + 2 \Phi(x) = |\vb|^2$ shown in Lemma \ref{conservative field}. Thus, we deduce
\be \notag
\tf \lesssim (|v|^2 + 2 \Phi(x))^{1/2}.
\ee
Since $\varphi (\tau) = e^{\delta \tau}$ with $0 < \delta < \infty$, we derive that
\be \label{est:varphis tf}
| \varphi (\tf) |
\lesssim \big| \varphi \big( (|v|^2 + 2 \Phi(x))^{1/2} \big) \big| 
= e^{\delta (|v|^2 + 2 \Phi(x))^{1/2}}.
\ee

Now fix $T_0$ in \eqref{cond:T0} and recall norms of $\vertiii{\cdot}$ in \eqref{|||i},
\Be \notag
\vertiii{f} = \|f  \|_{L^1_{x,v}}
    +    \frac{ 4 \mathfrak{m}_{T_0} }{ \varphi (\frac{T_0}{4})}
  \Big(  {1 + \frac{1}{\delta T_0}} \Big) \times
 \| \varphi (\tf) f\|_{L^1_{x,v}},
\Ee
From \eqref{energyi}, and setting 
$\mathfrak{R} :=   
\max \big\{ 
1 - \mathfrak{m}_{T_0} (1- \frac{ 4 C (1+T_0) (1+ \delta T_0) }{\delta T_0 \varphi (\frac{T_0}{4})}), \
\frac{1 + \delta T_0 / 2}{1 + \delta T_0}
\big\}$, 
we derive for all $N \in \mathbb{N}$,
\Be \label{est|||}
\begin{split} 
\vertiii{ f(NT_0 )} 
\leq \mathfrak{R} \times \vertiii{ f((N-1)T_0 )}.
\end{split}
\Ee
Moreover, from $T_0$ in \eqref{cond:T0}, we have 
\be \label{est:R}
\begin{split}
& \frac{1 + \delta T_0 / 2}{1 + \delta T_0} 
= 1 - \frac{\delta T_0}{2 (1 + \delta T_0)}
< 1, 
\\& 1 - \mathfrak{m}_{T_0} (1- \frac{ 4 C (1+T_0) (1+ \delta T_0) }{\delta T_0 \varphi (\frac{T_0}{4})})
\leq 1 - \frac{1}{2} \mathfrak{m}_{T_0} < 1.
\end{split}
\ee
For $t\geq 0$, we can choose $N_* \in \mathbb{N}$ such that $t \in [N_* T_0, (N_* + 1) T_0]$.
Then applying \eqref{energy_varphi}, \eqref{maximum}  first, and using \eqref{est|||} successively, we conclude that 
\Be \label{est:|||1}
\begin{split}
\vertiii{f(t)}
& \lesssim \vertiii{f(N_*T_0)} 
\leq \mathfrak{R} \vertiii{f((N_*-1)T_0)} 
\leq \cdots  \leq  \mathfrak{R}^{N_*}\vertiii{f(0)}. 
\end{split}
\Ee
Note that $N_* T_0 \leq t \leq (N_* + 1) T_0$ and $\mathfrak{R} < 1$ in \eqref{est:R}, we deduce
\be
\mathfrak{R}^{ N_*} 
\leq \mathfrak{R}^{ \frac{t}{T_0} -1 } 
\lesssim \big( \mathfrak{R}^{\frac{1}{T_0}} \big)^t. 
\ee
Since $T_0$ is fixed, we have 
\Be \label{est:|||1'}
 \vertiii{f(t)}
\leq \eqref{est:|||1} 
\lesssim 
\big( \mathfrak{R}^{\frac{1}{T_0}} \big)^t
\times \vertiii{ f(0) }.
\Ee

Now we choose the following constant $\Lambda$, such that
\Be \label{def:M}
\Lambda = - \frac{1}{T_0} \ln (\mathfrak{R}).
\Ee
Therefore, we derive that
\Be \notag
e^{- \Lambda} = \mathfrak{R}^{\frac{1}{T_0}}.
\Ee
From \eqref{COV} and the assumption
$\| e^{\theta^\prime (|v|^2+ 2\Phi (x))} f_0\|_{L^\infty_{x,v}}<\infty$, we have 
\Be \notag
\vertiii{ f(0) }  
\lesssim \| e^{\theta^\prime (|v|^2+ 2\Phi (x))} f_0 \|_{L^\infty_{x,v}} + \| \varphi (\tf) f_0 \|_1.
\Ee 
Using \eqref{est:|||1'}, \eqref{def:M}, \eqref{est:varphis tf}, and the assumption $\| e^{\delta (|v|^2 + 2 \Phi(x))^{1/2} } f_0 \|_{L^1_{x,v}} <\infty$, we finally prove Theorem \ref{theorem_1} via
\Be \notag
\| f(t) \|_{L^1_{x,v}} \leq \vertiii{f (t)} 
\lesssim e^{- \Lambda t}
 \{ \| e^{\theta^\prime (|v|^2+ 2\Phi (x))} f_0 \|_{L^\infty_{x,v}}+ \| e^{\delta (|v|^2 + 2 \Phi(x))^{1/2} } f_0 \|_{L^1_{x,v}}
 \}.
\Ee 
\end{proof}

\begin{proof}[\textbf{Proof of Theorem \ref{theorem_1}}]

Recall from Lemma \ref{lem:energy estimate on f} and a null flux condition of the boundary, we conclude \eqref{cons_mass_f} by taking integration over $(t_*, t) \times \Omega \times \R^3$.

Next, to show \eqref{est:theorem_1} , we fix $T_0$ in \eqref{cond:T0} and recall norm of $\vertiii{\cdot}$ in \eqref{|||i}.
From \eqref{energyi}, for all $N \in \mathbb{N}$, we obtain
\Be \notag
     \vertiii{f(NT_0)} \leq  \vertiii{f((N-1)T_0)} \leq \cdots \leq \vertiii{f(0)}.   
\Ee 
The rest of proof is similar to the proof of Theorem \ref{theorem_1}.
\end{proof}


\section{Continuity on \texorpdfstring{$\J (x)$}{J(x)}}
\label{sec: continuity}

In this section, we first show the flux function $\J (x)$ is continuous. Then we use the continuity to obtain the uniform upper bound of the residual measures in the stochastic cycles. We outline the strategy here: We begin by showing boundary outgoing flux $\J (x)$ for any steady state is continuous on the boundary $\p\O$.
Assume any sequence $x^i \rightarrow x \in \p\O$, we expand $\J (x^i)$ and $\J (x)$ into the integration forms:
\be \notag
\J (x) = \int_{n(x) \cdot v >0} F(x, v) \{ n(x) \cdot v \} \dd v, \ \
\J (x^i) = \int_{n(x^i) \cdot v >0} F(x^i, v) \{ n(x^i) \cdot v \} \dd v.
\ee
We introduce the following set:
\Be
\Gamma^{\delta} = \{ (x, v) \in \gamma: |n (x) \cdot v| \leq \delta \ \text{or} \ |v| \geq \frac{1}{\delta} \},
\Ee
and write $\J (x)$ as 
\Be \notag
\J (x) 
= \int_{v \in \gamma_{+} \cap \Gamma^{\delta}} F(x, v) \{ n(x) \cdot v \} \dd v
+ \int_{v \in \gamma_{+} \backslash \Gamma^{\delta}} F(x, v) \{ n(x) \cdot v \} \dd v.
\Ee

We can control the first part by
$\int_{\Gamma^{\delta}} \mu_{\Theta} (\xb, v) \{ n(\xb) \cdot v \} \dd v \lesssim \delta^2$ with $0 < \delta \ll 1$ (see Lemma \ref{gamma set}).
For the second part, we first apply the the change of variable $v \mapsto (\tb, S_{\xb})$ on the integration forms.
Then we compare and obtain the pointwise convergence on the integrands 
where the key is that $v$ and $\vb$ are continuous.
Applying the dominated convergence theorem, we derive 
$\J (x) \rightarrow \J (x^{i})$ and prove the continuity in Lemma \ref{lem: j continue}.

Then we consider the boundary outgoing flux $\J (x)$ on non-negative steady state $F_s (x, v)$ with unit mass.
Since the boundary is compact and $\J (x)$ is continuous, $\J (x)$ achieves an infimum $L \geq 0$ on $\p\O$.
Applying the change of variable $v \mapsto (\tb, S_{\xb})$ on the integration form of $\J (x)$, we can show that if $L = 0$, $\J (x)$ has to be zero almost everywhere. This will contradict with the assumption that $F_s (x, v)$ has unit mass. Thus, we deduce that $L > 0$.

Now we consider the residual measure as part of the stochastic cycle on $\J (x)$ and obtain
\Be 
\J (x) 
\geq \int_{\prod_{j=0}^{i-1} \mathcal{V}_j}   
\mathbf{1}_{0 < t_{i } \leq t} \
\dd \sigma_{i-1} \cdots \dd \sigma_1 \dd \sigma_0 \times \J (x^i).
\Ee
Using the $L^\infty$-Estimates for $\J (x)$ on the left side, we write $\J (x) \leq U$.
On the right hand, we apply the lower bound on $\J (x^i)$ as $\J (x^i) \geq L$.
Therefore, we establish the uniform estimate on residual measure
\Be 
\begin{split}
	\int_{\prod_{j=0}^{i-1} \mathcal{V}_j}   
	\mathbf{1}_{0 < t_{i } \leq t} \
	\dd \sigma_{i-1} \cdots \dd \sigma_1 \dd \sigma_0 \leq L^{-1} U.
\end{split}
\Ee


For the magnet field case, we also start with showing boundary outgoing flux $\J (x)$ for any steady state is continuous on the boundary $\p\O$. 
For any sequence $x^i \rightarrow x \in \p\O$, we expand $\J (x^i)$ and $\J (x)$ into the integration forms and apply the change of variable $v \mapsto (\tb, S_{\xb})$ on them.

The main difficulty still comes from the Jacobian term $\frac{5 B^2_3}{2 - 2 \cos (B_3 \tib)}$ in the integrands. Thus, we need to introduce the following set:
\Be
\Gamma^{\delta} = \{ (x, v) \in \gamma: v \in \mathcal{V}^{B_{\delta}} (x) \ \text{or} \ |v| \geq \frac{1}{\delta} \}.
\Ee
Now we split $(x, v)$ into two parts $\Gamma^{\delta}$ and $\gamma_{+} \backslash \Gamma^{\delta}$. 
Though the Jacobian term can be large when $(x, v) \in \Gamma^{\delta}$, we use 
$\int_{\Gamma^{\delta}} \mu_{\Theta} (\xb, v) \{ n(\xb) \cdot v \} \dd v \lesssim \delta$ (see Lemma \ref{gamma set mag}) to control $\int_{v \in \gamma_{+} \cap \Gamma^{\delta}} F(x, v) \{ n(x) \cdot v \} \dd v$. 
For $(x, v) \in \gamma_{+} \backslash \Gamma^{\delta}$, we compute and derive the pointwise convergence on the integrands in forms of $\J (x)$ and $\J (x^i)$.
Then we apply the dominated convergence theorem to deduce the continuity in Lemma \ref{lem: j continue mag}.

\subsection{Gravitational case}
\label{subsec: continuity}


For the sake of simplicity, we have abused the notations temporarily in section \ref{subsec: continuity} and \ref{subsec: continuity mag}:
all superscript represent some sequences instead of the notations in stochastic cycles.

Suppose $F (x, v)$ solves \eqref{equation for F_infty} and \eqref{diff_F} and $\J (x)$ defined in \eqref{diff_J}, for $x \in \p\O$,
\Be \label{jx_tbxb}
\begin{split}
\J (x) 
& = \int_{n(x) \cdot v >0} F (x, v) \{ n(x) \cdot v \} \dd v
\\& = \sum_{m, n \in \Z} \int_{v \in v^{m, n}} F (x, v) \{ n(x) \cdot v \} \dd v
\\& = \sum_{m, n \in \Z} \int_{\p\O} \int^{\infty}_0 
\underbrace{\mu_{\Theta} (\xb, v) \J (\xb) \{ n(x) \cdot v \} \Big\{ \frac{|n (\xb ) \cdot \vb |}{(\tb)^{3}} d_0 \Big\}}_{\J^{m,n} (x)} \dd \tb \dd \xb,
\end{split}
\Ee
where $(d_0)^{-1} = 
	\det 
    \big[
    -{Id}_{3 \times 3} - \frac{1}{t_{\mathbf{b}}^1} \int^{t^1 - t_{\mathbf{b}}^1}_{t^1}
    \int^{s}_{t^1} - \nabla_{v^1} X \cdot \nabla \nabla \phi \dd \alpha \dd s
    \big]$, and $v^{m, n}$ defined in \eqref{def:vmn}.
    
Now we build the sequence $\{x^{i}\}_{i \in \N}$ with $\| x^{i} - x \| \rightarrow 0$ as $i \rightarrow \infty$. Similar, we have 
\Be \label{jxi_tbxb}
\J (x^{i}) = \sum_{m, n \in \Z} \int_{\p\O} \int^{\infty}_0 
\underbrace{\mu_{\Theta} (\xib, v^{i}) \J (\xib) \{ n(x^i) \cdot v^i \} \Big\{ \frac{|n (\xib) \cdot \vib |}{(\tib)^{3}} d_{i, 0} \Big\}}_{\J^{m,n} (x^{i})}
\dd \tib \dd \xib,
\Ee
where $\tib = \tb (x^{i}, v^{i})$, $\vib = \vb (x^{i}, v^{i})$ and $\xib = \xb (x^{i}, v^{i})$.

\begin{lemma} \label{viv_vibvb approach estimate}

Consider $(x, v), (x^i, v^i) \in \p\O \times \R^3$, and their backward stochastic cycles with $(X,V)$ in \eqref{characteristics}:
\be \notag
(\tb, \xb + (m, n, 0), \vb), \ \ (\tb^i, \xb^i + (m, n, 0), \vb^i).
\ee
Suppose $m, n \in \Z$ are fixed, $\tib = \tb$, $\xib =\xb$, and $| x^{i} - x | \rightarrow 0$ as $i \rightarrow \infty$, then
\begin{enumerate}
\item $v^{i} \rightarrow v$ as $i \rightarrow \infty$,
\item $\vib (x^i, v^i) \rightarrow \vb (x, v)$ as $i \rightarrow \infty$,
\item $\J^{m,n} (x^{i}) \rightarrow \J^{m,n} (x)$ as $i \rightarrow \infty$.
\end{enumerate}
\end{lemma}

\begin{proof}

From fixed $\xb, \tb$ and $t$ with $\tb = \tib$ and $\xb = \xib$, we have
\Be \label{xxb}
\xb + \int^{t}_{t - \tb} (v - \int^{t}_{s} - \nabla \Phi(X(\alpha, t, x, v)) \dd  \alpha) \dd s = x.
\Ee
\Be \label{xixib}
\xb + \int^{t}_{t - \tb} (v^{i} - \int^{t}_{s} - \nabla \Phi(X(\alpha, t, x^{i}, v^{i})) \dd  \alpha) \dd s = x^{i}. 
\Ee
Subtracting \eqref{xxb} from \eqref{xixib}, we deduce
\Be \label{xvxivi}
\tb (v - v^{i}) + 
\underbrace{\int^{t}_{t - \tb} \int^{t}_{s} \nabla \Phi(X(\alpha, t, x, v)) \dd  \alpha \dd s - \int^{t}_{t - \tb} \int^{t}_{s} \nabla \Phi(X(\alpha, t, x^{i}, v^{i})) \dd \alpha \dd s}_{\eqref{xvxivi}^*} = x - x^{i}. 
\Ee
To control $\eqref{xvxivi}^*$, we need to estimate $\big| \nabla_{x}
    \int^{t}_{t - \tb} \int^{s}_{t} \nabla \Phi(X(\alpha, t, x, v)) \dd \alpha \dd s \big|$. 
From \eqref{characteristics}, we have 
\Be \notag
\begin{split}
\frac{\dd}{\dd s}|\nabla_x X(s;t, x, v)| 
& \lesssim |\nabla_x V(s;t, x, v)|,
\\ \frac{\dd}{\dd s}|\nabla_x V(s;t, x, v)| 
& \lesssim |\nabla \nabla_X \phi (X(s;t, x, v))||\nabla_x X(s;t, x, v)|
    \\& \lesssim \varrho_2 e^{- \varrho_1 X_3 (s;t, x, v)} |\nabla_x X(s;t, x, v)|.
\end{split}
\Ee
Since $
\nabla_x X(t;t, x, v) = \nabla_x x = Id_{3 \times 3}$, $\nabla_x V(t;t, x, v) = \nabla_x v = 0$ and $t - \tb \leq s \leq t$, we get
\Be \label{Xx first} 
\begin{split} 
|\nabla_x X(s;t, x, v)| 
& \lesssim 1 + \int^{t}_{s} |\nabla_x V(s_1;t, x, v)| \dd s_1
\\& \lesssim 1 + \int^{t}_{s} \int^{t}_{s_1}
    \varrho_2 e^{- \varrho_1 X_3 (s_2;t, x, v)} |\nabla_x X(s_2;t, x, v)| \dd s_2 \dd s_1.
\end{split}
\Ee
Using the Fubini's theorem, we derive
\Be \notag
\begin{split}
\eqref{Xx first} 
& \lesssim 1 + \int^{t}_{s} \int^{s_2}_{s}
    \varrho_2 e^{- \varrho_1 X_3 (s_2;t, x, v)} |\nabla_x X(s_2;t, x, v)| \dd s_1 \dd s_2
\\& \lesssim 1 + \int^{t}_{s} \varrho_2 
(s_2-s) e^{- \varrho_1 X_3 (s_2;t, x, v)}
|\nabla_x X(s_2;t, x, v)| \dd s_2.
\end{split}
\Ee
Using the Gronwall’s inequality, for $t-\tb \leq s \leq t$,
\Be \label{first estimate on X_x}
\begin{split} 
    & \ \ \ \ |\nabla_x X(s;t, x, v)| 
    \\& \lesssim \exp 
    \big(
    \int^{t}_{s} \varrho_2 
(s_2-s) e^{- \varrho_1 X_3 (s_2;t, x, v)} \dd s_2 
	\big)
	\\& \leq \exp 
    \Big(
    \int^{t}_{t-\tb} \varrho_2 
	\big(
	s_2- (t-\tb)
	\big)
 	e^{- \varrho_1 X_3 (s_2;t, x, v)} \dd s_2 
	\Big).
\end{split}
\Ee
Using \eqref{estimate on X_v}, we derive
\Be \label{estimate on X_x}
|\nabla_x X(s;t, x, v)| \leq e^{\varrho_2/\varrho_1}.
\Ee
From direct computation, we have,
\Be \notag 
\nabla_{x} 
\big[ \nabla\Phi(X(\alpha, t, x, v) \big]
= \nabla_{x} X(\alpha, t, x, v) \cdot \nabla \nabla\Phi(X) =  \nabla_{x} X(\alpha, t, x, v) \cdot \nabla \nabla \phi,
\Ee    
and from \eqref{estimate on X_x} and Fubini's theorem, we have 
\Be \label{d_1 estimate}
\begin{split}
    & \ \ \ \ 
    \Big|
    \int^{t - \tb}_{t}
\int^{s}_{t} - \nabla_{x} X \cdot \nabla^2 \phi (X(\alpha; t, x, v)) \dd \alpha \dd s
	\Big|
\\& \leq \int^{t}_{t - \tb}
\int^{t}_{s} |\nabla^2 \phi(X)| |\nabla_{x} X(\alpha; t, x, v)| \dd \alpha \dd s
\\& \leq \int^{t}_{t - \tb} \int^{t}_{s} \varrho_2 e^{- \varrho_1 X_3} e^{\varrho_2/\varrho_1} \dd \alpha \dd s
\\& \leq \int^{t}_{t - \tb} 
\varrho_2 
\big(
\alpha - (t - \tb)
\big)
e^{- \varrho_1 X_3} e^{\varrho_2/\varrho_1} \dd \alpha 
\leq e^{\varrho_2/\varrho_1} \frac{\varrho_2}{\varrho_1},
\end{split}
\Ee
where the last inequality holds from \eqref{estimate on X_v}. 
Applying Mean value theorem, we get
\Be \notag
\begin{split}
| \eqref{xvxivi}^* | 
& \leq \| \nabla_v \int^{t}_{t - \tb} \int^{t}_{s} \nabla \Phi(X(\alpha, t, x, v)) \dd  \alpha \dd s \|_{\infty} \times |v - v^{i}|
\\& + \| \nabla_x \int^{t}_{t - \tb} \int^{t}_{s} \nabla \Phi(X(\alpha, t, x^{i}, v^{i})) \dd \alpha \dd s \|_{\infty} \times |x - x^{i}|.
\end{split}
\Ee
Using \eqref{d_1 estimate}, we deduce
\Be \notag
| \eqref{xvxivi}^* | 
\leq e^{\varrho_2/\varrho_1} \frac{\varrho_2}{\varrho_1} \tb \times |v - v^{i}|
+ e^{\varrho_2/\varrho_1} \frac{\varrho_2}{\varrho_1} \times |x - x^{i}|.
\Ee
Combining the above with \eqref{xvxivi}, we derive
\Be \notag 
| \tb (v - v^{i}) | 
\leq e^{\varrho_2/\varrho_1} \frac{\varrho_2}{\varrho_1} \tb \times |v - v^{i}|
+ e^{\varrho_2/\varrho_1} \frac{\varrho_2}{\varrho_1} \times |x - x^{i}| + | x - x^{i}|. 
\Ee
Since $e^{\varrho_2/\varrho_1} \frac{\varrho_2}{\varrho_1} \ll 1$, we conclude that
\Be \label{vvi}
|v - v^{i}| 
\lesssim \frac{1}{\tb} | x - x^{i} |, 
\Ee
which implies, for fixed $\tb \neq 0$,
$|v - v^{i}| \rightarrow 0$ as $i \rightarrow \infty$.

Now, we begin to prove the second part. From fixed $\tb$, we have
\Be \label{vvb}
\vb + \int^{t}_{t - \tb} - \nabla \Phi(X(s, t, x, v)) \dd s = v.
\Ee
\Be \label{vivib}
\vib + \int^{t}_{t - \tb} - \nabla \Phi(X(s, t, x^{i}, v^{i})) \dd s = v^{i}. 
\Ee
Subtracting \eqref{vvb} from \eqref{vivib}, we deduce
\Be \label{vvivbvib}
(\vb - \vib) -
\underbrace{\int^{t}_{t - \tb} \nabla \Phi(X(s, t, x, v)) \dd s + \int^{t}_{t - \tb} \nabla \Phi(X(s, t, x^{i}, v^{i})) \dd s}_{\eqref{vvivbvib}^*} = v - v^{i}. 
\Ee
To control $\eqref{vvivbvib}^*$, we need to estimate $\big| \nabla_{x, v} \int^{t}_{t - \tb} \nabla \Phi(X(s, t, x, v)) \dd s \big|$. 
From direct computation, we have
\Be \notag 
\big| \nabla_{x, v} \int^{t}_{t - \tb} \nabla \Phi(X (s, t, x, v)) \dd s \big|
= \big| \int^{t}_{t - \tb} \nabla \nabla \phi \cdot  \nabla_{x, v} X (s, t, x, v) \dd s \big|.
\Ee 
Using \eqref{x3 estimate}, \eqref{estimate on X_v} and \eqref{estimate on X_x}, we get
\Be \label{d_2 estimate}
\begin{split}
& \big| \int^{t}_{t - \tb} \nabla \nabla \phi \cdot  \nabla_{x} X (s, t, x, v) \dd s \big|
\lesssim \big| \int^{t}_{t - \tb} \varrho_2 e^{- \varrho_1 X_3} e^{\varrho_2/\varrho_1} \dd s \big|
\lesssim \varrho_2 e^{\varrho_2/\varrho_1},
\\& \big| \int^{t}_{t - \tb} \nabla \nabla \phi \cdot  \nabla_{v} X (s, t, x, v) \dd s \big|
\lesssim \big| \int^{t}_{t - \tb} \varrho_2 e^{- \varrho_1 X_3} |t-s| e^{\varrho_2/\varrho_1} \dd s \big|
\lesssim e^{\varrho_2/\varrho_1} \frac{\varrho_2}{\varrho_1}.
\end{split}
\Ee
Applying Mean value theorem, we get
\Be \notag
\begin{split}
| \eqref{vvivbvib}^* | 
& \leq \| \nabla_v \int^{t}_{t - \tb} \nabla \Phi(X(s, t, x, v)) \dd s \|_{\infty} \times |v - v^{i}|
\\& + \| \nabla_x \int^{t}_{t - \tb} \nabla \Phi(X(s, t, x, v)) \dd s \|_{\infty} \times |x - x^{i}|.
\end{split}
\Ee
Using \eqref{d_2 estimate}, we deduce
\Be \notag
| \eqref{xvxivi}^* | 
\leq e^{\varrho_2/\varrho_1} \frac{\varrho_2}{\varrho_1} \times |v - v^{i}|
+ \varrho_2 e^{\varrho_2/\varrho_1} \times |x - x^{i}|.
\Ee
Combining the above with \eqref{vvivbvib}, we derive
\Be \notag 
| \vb - \vib | 
\leq e^{\varrho_2/\varrho_1} \frac{\varrho_2}{\varrho_1} \times |v - v^{i}|
+ \varrho_2 e^{\varrho_2/\varrho_1} \times |x - x^{i}| + | v - v^{i} |. 
\Ee
Using $\varrho_2 \ll 1$, $\varrho_1 > 1$ and $(x^{i}, v^{i}) \rightarrow (x, v)$ as $i \rightarrow \infty$, we conclude, for fixed $\tb \neq 0$,
\Be \notag
| \vb - \vib | \rightarrow 0 \ \text{as} \ i \rightarrow \infty.
\Ee

Finally, from the implicit function theorem, $X(s,t,x,v)$ is a smooth function. 
Since $m, n$ are fixed, $\tib = \tb$ and $\xib =\xb$,
we deduce that $d_{i, 0} \rightarrow d_0$ as $i \rightarrow \infty$ from the first two parts.
Consider \eqref{jx_tbxb} and \eqref{jxi_tbxb}, we derive
\Be \label{ptwise converge on jmn}
\J^{m,n} (x^{i}) \rightarrow \J^{m,n} (x) \ \text{as} \ |x^{i} - x | \rightarrow 0.
\Ee
\end{proof}

\begin{lemma}
\label{gamma set}
Define
\Be \notag
\Gamma^{\delta} = \{ (x, v) \in \gamma: |n (x) \cdot v| \leq \delta \ \text{or} \ |v| \geq \frac{1}{\delta} \}.
\Ee
Consider $(X,V)$ in \eqref{characteristics},
then for $0 < \delta \ll 1$ and $x \in \p\O$,
\Be \label{gamma int estimate}
\int_{\Gamma^{\delta}} \mu_{\Theta} (\xb, v) \{ n(\xb) \cdot v \} \dd v \lesssim \delta^2.
\Ee
\end{lemma} 

\begin{proof}

From \eqref{estimate on delta}, we can obtain 
\be \label{gamma int estimate 1}
\int_{ |n(x) \cdot v| \lesssim \delta} \mu_{\Theta} (\xb, v) | n(\xb) \cdot v | \dd v  \lesssim \delta^2.
\ee
On the other hand, from $|n(\xb) \cdot v| \leq |v|$ and  \eqref{def:Theta}, we have
\Be \label{gamma int estimate 2}
\begin{split} 
 \int_{|v| \geq \frac{1}{\delta}}
\mu_{\Theta} (\xb, v) |n(\xb) \cdot v| \dd v 
& < \int_{|v| \geq \frac{1}{\delta}} \mu_{\Theta} (\xb, v) | v | \dd v  
\\& \lesssim \int^{\infty}_{\frac{1}{\delta}} e^{- \frac{r^2}{2b}} r \dd r 
\leq e^{- \frac{1}{2b \delta^2}} \lesssim \delta^2.
\end{split}
\Ee
Combining \eqref{gamma int estimate 1} and \eqref{gamma int estimate 2}, we conclude \eqref{gamma int estimate}.
\end{proof}

Now we are ready to prove the continuity on $\J (x)$.

\begin{lemma} \label{lem: j continue}
Suppose $F (x, v)$ solves \eqref{equation for F_infty} and \eqref{diff_F}, then $\J (x) = \int_{n(x)\cdot v >0} F (x, v) \{ n(x) \cdot v \} \dd v$ is continuous on $\p\O$.
\end{lemma}

\begin{proof}

For $(x, v) \in \gamma_{+} \backslash \Gamma^{\delta}$, we have $|n(x) \cdot v| \geq \delta$ and  $|v| \leq \frac{1}{\delta}$.
Using \eqref{v3t estimate}, we have
\begin{align}
& \tb \gtrsim |v_3| = |n(x) \cdot v| \geq \delta, \label{Gamma1 estimate}
\\& \tb \lesssim |v_3| = |n(x) \cdot v| \leq |v|\leq \frac{1}{\delta}. \label{Gamma2 estimate}
\end{align}
Using the above and \eqref{first cov first part}, for $x, \xb \in \partial \Omega$ and $m, n \in \Z$, we have
\Be \notag 
\begin{split}
|x - \xb + (m, n) | 
& = |\int^{t}_{t - \tb}  V(s; t, x, v) \dd s| 
\\& \leq \tb \times \| V(s; t, x, v) \|_{\infty} \lesssim \delta^{-2}.
\end{split}
\Ee
This implies 
\Be \label{mn bound}
|m| \lesssim \delta^{-2}, |n| \lesssim \delta^{-2}.
\Ee
For any $0 < \epsilon$, we can find sufficient small and fixed $\delta = \delta (\epsilon)$ such that $\delta^2 \leq \frac{\epsilon}{3}$.
Recall \eqref{jx_tbxb}, we split $\J (x)$ into two following parts:
\Be \label{j_split}
\begin{split}
\J (x) 
& = \int_{n(x)\cdot v >0} F(x, v) \{ n(x) \cdot v \} \dd v
\\& = \underbrace{\int_{v \in \gamma_{+} \cap \Gamma^{\delta}} F(x, v) \{ n(x) \cdot v \} \dd v}_{\eqref{j_split}_1}
+ \underbrace{\int_{v \in \gamma_{+} \backslash \Gamma^{\delta}} F(x, v) \{ n(x) \cdot v \} \dd v}_{\eqref{j_split}_2}.
\end{split}
\Ee
Using the uniform boundness on $\J (x)$, Lemma \ref{conservative field} and Lemma \ref{gamma set}, we get
\Be \label{j_split_1 estimate}
\eqref{j_split}_1 
\lesssim \| \J \|_{L^\infty} \int_{v \in \Gamma^{\delta}} \mu_{\Theta} (\xb, v) \{ n(\xb) \cdot v \} \dd v \lesssim \delta^2 \leq \frac{\epsilon}{3}.
\Ee
Using \eqref{Gamma1 estimate}, \eqref{Gamma2 estimate} and \eqref{mn bound}, we find that for $v \in \gamma_{+} \backslash \Gamma^{\delta}$, 
\Be \notag
\delta \lesssim \tb \lesssim \frac{1}{\delta} \ \text{and} \ 0 \leq |m|, |n| \lesssim \delta^{-2},
\Ee
which implies
\Be \label{j_split^2 first estimate}
\begin{split}
\eqref{j_split}_2
& \leq \sum_{|m|, |n| \lesssim \delta^{-2}} \int_{\p\O} \int^{\frac{1}{\delta}}_{\delta} \mu_{\Theta} (\xb, v) \J (\xb) \{ n(x) \cdot v \} \Big\{ \frac{|n (\xb ) \cdot \vb |}{(\tb)^{3}} d_0 \Big\} \dd \tb \dd \xb,
\end{split}
\Ee
Thus, we have
\Be \notag
\J (x) = \sum_{|m|, |n| \lesssim \delta^{-2}} \int_{\p\O} \int^{\frac{1}{\delta}}_{\delta} \mu_{\Theta} (\xb, v) \J (\xb) \{ n(x) \cdot v \} \Big\{ \frac{|n (\xb ) \cdot \vb |}{(\tb)^{3}} d_0 \Big\} \dd \tb \dd \xb + \Delta_{\J} (x),
\Ee
where $|\Delta_{\J} (x)| = \eqref{j_split}_1 \leq \frac{\epsilon}{3}$. 

We can do the similar estimate on $\J (x^{i})$, and due to the uniform boundness on $\J (x)$, and the pointwise convergence shown in \eqref{ptwise converge on jmn} as $\delta \leq \tb$. We apply the dominated convergence theorem in finitely many terms in \eqref{j_split^2 first estimate}, and derive that
\Be \notag
\sum_{|m|, |n| \lesssim \delta^{-2}} \int_{\p\O} \int^{\frac{1}{\delta}}_{\delta} 
\J^{m,n} (x^{i})
\dd \tib \dd \xib
\rightarrow 
\sum_{|m|, |n| \lesssim \delta^{-2}} \int_{\p\O} \int^{\frac{1}{\delta}}_{\delta} 
\J^{m,n} (x)
\dd \tb \dd \xb
 \ \text{as} \ x^{i} \rightarrow x.
\Ee
Then there exists sufficient large $N$ such that for any $i \geq N$, we have
\Be \label{j_split^2 second estimate}
\Big|
\sum_{|m|, |n| \lesssim \delta^{-2}} \int_{\p\O} \int^{\frac{1}{\delta}}_{\delta} 
\J^{m,n} (x^{i})
\dd \tib \dd \xib
-
\sum_{|m|, |n| \lesssim \delta^{-2}} \int_{\p\O} \int^{\frac{1}{\delta}}_{\delta} 
\J^{m,n} (x)
\dd \tb \dd \xb
\Big|
\leq \frac{\epsilon}{3}.
\Ee

Now we compute the difference between $\J (x)$ and $\J (x^{i})$ with $\delta$ and $i$ defined above,
\Be \label{j_x_j_xi estimate}
\begin{split}
& \ \ \ \ |\J (x) - \J (x^{i})|
\\& = \Big|
\sum_{|m|, |n| \lesssim \delta^{-2}} \int_{\p\O} \int^{\frac{1}{\delta}}_{\delta} 
\J^{m,n} (x)
\dd \tb \dd \xb + \Delta_{\J} (x)
-
\sum_{|m|, |n| \lesssim \delta^{-2}} \int_{\p\O} \int^{\frac{1}{\delta}}_{\delta} 
\J^{m,n} (x^{i})
\dd \tib \dd \xib - \Delta_{\J} (x^{i})
\Big|
\\& \leq 
\Big|
\sum_{|m|, |n| \lesssim \delta^{-2}} \int_{\p\O} \int^{\frac{1}{\delta}}_{\delta} 
\J^{m,n} (x^{i})
\dd \tib \dd \xib
-
\sum_{|m|, |n| \lesssim \delta^{-2}} \int_{\p\O} \int^{\frac{1}{\delta}}_{\delta} 
\J^{m,n} (x)
\dd \tb \dd \xb
\Big|
+ |\Delta_{\J} (x)| + |\Delta_{\J} (x^{i})|
\\& \leq \frac{\epsilon}{3} + \frac{\epsilon}{3} + \frac{\epsilon}{3} = \epsilon.
\end{split}
\Ee
Therefore, we conclude that $\J (x)$ is continuous on $\p\O$.
\end{proof}

\begin{remark}
Suppose $F (x, v)$ solves \eqref{equation for F_infty} and \eqref{diff_F}.
Since $F(x, v) = \mu_{\Theta} (\xb, \vb) \J (x)$ and $\xb, \vb$ are continuous for $n(x) \cdot v \neq 0$. Therefore, $F(x, v)$ is continuous if $\mu_{\Theta} (x, v)$ is a continuous function and $n(x) \cdot v \neq 0$.
A similar argument can be made for the case with magnet field after we conclude Lemma \ref{lem: j continue mag}.
\end{remark}

\subsection{With Magnet field}
\label{subsec: continuity mag}


Now consider $F (x, v)$ solves \eqref{def:f} and \eqref{diff_F}, with $\J (x)$ defined in \eqref{diff_J}, for $x \in \p\O$,
\Be \label{jx_tbxb mag}
\begin{split}
\J (x) 
& = \int_{n(x) \cdot v >0} F(x, v) \{ n(x) \cdot v \} \dd v
\\& = \sum_{m, n \in \Z} \int_{v \in v^{m, n}} F (x, v) \{ n(x) \cdot v \} \dd v
\\& = \sum_{m, n \in \Z} \int_{\p\O} \int^{\infty}_0 
\underbrace{\mu_{\Theta} (\xb, v) \J (\xb) \{ n(x) \cdot v \} \big\{ \frac{5 B^2_3}{2 - 2 \cos (B_3 \tb)} \big\}}_{\J^{m,n} (x)} \dd \tb \dd \xb,
\end{split}
\Ee
where $v^{m, n}$ defined in \eqref{def:vmn mag}.
    
Now we build the sequence $\{x^{i}\}_{i \in \N}$ with $| x^{i} - x | \rightarrow 0$ as $i \rightarrow \infty$. Similar, we have 
\Be \label{jxi_tbxb mag}
\J (x^{i}) = \sum_{m, n \in \Z} \int_{\p\O} \int^{\infty}_0 
\underbrace{\mu_{\Theta} (\xib, v^{i}) \J (\xib) \{ n(x^i) \cdot v^i \} \big\{ \frac{5 B^2_3}{2 - 2 \cos (B_3 \tib)} \big\}}_{\J^{m,n} (x^{i})}
\dd \tib \dd \xib,
\Ee
where $\tib = \tb (x^{i}, v^{i})$, $\vib = \vb (x^{i}, v^{i})$ and $\xib = \xb (x^{i}, v^{i})$.

\begin{lemma} 
\label{viv_vibvb approach estimate mag}

Consider $(x, v), (x^i, v^i) \in \p\O \times \R^3$, and their backward stochastic cycles with $(X,V)$ in \eqref{characteristics mag}:
\be \notag
(\tb, \xb + (m, n, 0), \vb), \ \ (\tb^i, \xb^i + (m, n, 0), \vb^i).
\ee
Suppose $m, n \in \Z$ are fixed, $\tib = \tb$, $\xib =\xb$ and $| x^{i} - x | \rightarrow 0$ as $i \rightarrow \infty$, then
\be \label{converge on vmn mag}
| v^{i} | \rightarrow | v | \ \text{as} \ i \rightarrow \infty.
\ee
Moreover, we have
\Be \label{ptwise converge on jmn mag}
\J^{m,n} (x^{i}) \rightarrow \J^{m,n} (x) \ \text{as} \ i \rightarrow \infty.
\Ee
\end{lemma}

\begin{proof}

From \eqref{tb expression mag} and the assumption $\tib = \tb$, we obtain that
\Be \notag 
5 \tb (x, v) = | v_3 | = | v^i_3 |.
\ee
Using \eqref{vmnb12 first estimate mag} and the assumption $\xib = \xb$, we get
\Be \notag 
\begin{split}
& \big| v_{\parallel} \big|^2 
 = \frac{B^2_3 \times | x - \xb + (m, n, 0) |^2 }{2 - 2 \cos (B_3 \tb)},
\ \ \big| v^{i}_{\parallel} \big|^2 
= \frac{B^2_3 \times | x^i - \xb + (m, n, 0) |^2 }{2 - 2 \cos (B_3 \tb)},
\end{split}
\Ee
where $v_{\parallel} = (v_1, v_2)$ and $v^i_{\parallel} = (v^i_1, v^i_2)$.
From $| x^{i} - x | \rightarrow 0$, we derive $\big| v_{\parallel} \big| \rightarrow \big| v^i_{\parallel} \big|$, and obtain \eqref{converge on vmn mag}.

Next, from the implicit function theorem, the characteristic trajectory $X(s,t,x,v)$ is smooth.
Using \eqref{tb expression mag}, \eqref{converge on vmn mag}, and for fixed $m, n$ with $0 < \tib = \tb$ and $\xib =\xb$, we conclude \eqref{ptwise converge on jmn mag}.
\end{proof}

Following Lemma \ref{lem: Be mag} and Lemma \ref{gamma set}, we derive the following Lemma:

\begin{lemma}
\label{gamma set mag}
Define
\Be \notag
\Gamma^{\delta} = \{ (x, v) \in \gamma: v \in \mathcal{V}^{B_{\delta}} (x) \ \text{or} \ |v| \geq \frac{1}{\delta} \}.
\Ee
Consider $(X,V)$ in \eqref{characteristics mag},
then for $0 < \delta \ll 1$ and $x \in \p\O$,
\Be \notag
\int_{\Gamma^{\delta}} \mu_{\Theta} (\xb, v) \{ n(\xb) \cdot v \} \dd v \lesssim \delta.
\Ee
\end{lemma}

Now we are able to prove the continuity on $\J (x)$ under magnet field.

\begin{lemma} \label{lem: j continue mag}
Suppose $F (x, v)$ solves \eqref{def:f} and \eqref{diff_F}, with $\J (x)$ defined in \eqref{diff_J}, then $\J (x)$ is continuous on $\p\O$.
\end{lemma}

\begin{proof}

For $(x, v) \in \gamma_{+} \backslash \Gamma^{\delta}$ with $0 < \delta \ll 1$, we have $v \in \mathcal{V}^{G_{\delta}} (x)$. This shows that 
\Be \label{Gamma1 estimate mag}
\frac{\delta}{B_3} + \frac{2 k \pi}{B_3}
  \leq \tb (x, v) \leq 
 \frac{2 \pi - \delta}{B_3} + \frac{2 k \pi}{B_3} \ \text{for} \ k \in \Z_{+}.
\Ee
On the other hand, for $|v| \leq \frac{1}{\delta}$, we have 
\Be \label{Gamma2 estimate mag}
\tb \lesssim |n(x) \cdot v| \leq |v|\leq \frac{1}{\delta}.
\Ee
Using the above estimate, for $x, \xb \in \partial \Omega$ and $m, n \in \Z$, we have
\Be \notag 
\begin{split}
|x - \xb + (m, n) | 
& = |\int^{t}_{t - \tb}  V(s; t, x, v) \dd s| 
\\& \leq \tb \times \| V(s; t, x, v) \|_{\infty} \lesssim \delta^{-2}.
\end{split}
\Ee
This implies 
\Be \label{mn bound mag}
|m| \lesssim \delta^{-2}, |n| \lesssim \delta^{-2}.
\Ee
For any $\epsilon > 0$, we can find sufficient small and fixed $\delta = \delta (\epsilon)$ such that $\delta \leq \frac{\epsilon}{3}$.
Recall \eqref{jx_tbxb mag}, we split $\J (x)$ into two parts:
\Be \label{j_split mag}
\begin{split}
\J (x) 
& = \int_{n(x)\cdot v >0} F(x, v) \{ n(x) \cdot v \} \dd v
\\& = \underbrace{\int_{v \in \gamma_{+} \cap \Gamma^{\delta}} F(x, v) \{ n(x) \cdot v \} \dd v}_{\eqref{j_split mag}_1}
+ \underbrace{\int_{v \in \gamma_{+} \backslash \Gamma^{\delta}} F(x, v) \{ n(x) \cdot v \} \dd v}_{\eqref{j_split mag}_2}.
\end{split}
\Ee
Using the uniform boundness on $\J (x)$, Lemma \ref{conservative field mag} and Lemma \ref{gamma set mag}, we get
\Be \label{j_split_1 estimate mag}
\eqref{j_split mag}_1 
\lesssim \| \J \|_{L^\infty} \int_{v \in \Gamma^{\delta}} \mu_{\Theta} (\xb, v) \{ n(\xb) \cdot v \} \dd v \lesssim \delta \leq \frac{\epsilon}{3}.
\Ee
Using \eqref{Gamma1 estimate mag}, \eqref{Gamma2 estimate mag} and \eqref{mn bound mag}, we find that for $v \in \gamma_{+} \backslash \Gamma^{\delta}$, 
\Be \notag
v \in B_{\delta}, \ \tb \lesssim \frac{1}{\delta} \ \text{and} \ 0 \leq |m|, |n| \lesssim \delta^{-2},
\Ee
which implies
\Be \label{j_split^2 first estimate mag}
\begin{split}
\eqref{j_split mag}_2
& \leq \sum_{|m|, |n| \lesssim \delta^{-2}} \int_{\p\O} \int^{\frac{1}{\delta}}_0 \mathbf{1}_{v \in B_{\delta}} \times \mu_{\Theta} (\xb, v) \J (\xb) \{ n(x) \cdot v \} \big\{ \frac{5 B^2_3}{2 - 2 \cos (B_3 \tb)} \big\} \dd \tb \dd \xb,
\end{split}
\Ee
Thus, we have
\Be \notag
\J (x) = \sum_{|m|, |n| \lesssim \delta^{-2}} \int_{\p\O} \int^{\frac{1}{\delta}}_0 \mathbf{1}_{v \in B_{\delta}} \times \J^{m,n} (x) \dd \tb \dd \xb + \Delta_{\J} (x),
\Ee
where $\Delta_{\J} (x) = \eqref{j_split mag}_1 \leq \frac{\epsilon}{3}$. 

Recall \eqref{jxi_tbxb mag}, given a sequence $\{x^{i}\}_{i \in \N}$ with $| x^{i} - x | \rightarrow 0$ as $i \rightarrow \infty$, we can do the similar estimate on $\J (x^{i})$. 
Due to the uniform boundness and the pointwise convergence shown in \eqref{ptwise converge on jmn mag}. We apply the dominated convergence theorem in finitely many terms in \eqref{j_split^2 first estimate mag}, and derive that
\Be \notag
\sum_{|m|, |n| \lesssim \delta^{-2}}\int^{\frac{1}{\delta}}_0 \mathbf{1}_{v \in B_{\delta}} \times 
\J^{m,n} (x^{i})
\dd \tib \dd \xib
\rightarrow 
\sum_{|m|, |n| \lesssim \delta^{-2}} \int^{\frac{1}{\delta}}_0 \mathbf{1}_{v \in B_{\delta}} \times 
\J^{m,n} (x)
\dd \tb \dd \xb
\ \ \text{as} \ x^{i} \rightarrow x.
\Ee
Then there exists sufficient large $N$ such that for any $i \geq N$, we have
\Be \label{j_split^2 second estimate mag}
\Big|
\sum_{|m|, |n| \lesssim \delta^{-2}} \int^{\frac{1}{\delta}}_0 \mathbf{1}_{v \in B_{\delta}} \times
\J^{m,n} (x^{i})
\dd \tib \dd \xib
-
\sum_{|m|, |n| \lesssim \delta^{-2}} \int^{\frac{1}{\delta}}_0 \mathbf{1}_{v \in B_{\delta}} \times 
\J^{m,n} (x)
\dd \tb \dd \xb
\Big|
\leq \frac{\epsilon}{3}.
\Ee

Now we compute the difference between $\J (x)$ and $\J (x^{i})$ with $\delta$ and $i$ defined above,
\Be \label{j_x_j_xi estimate mag}
\begin{split}
& \ \ \ \ |\J (x) - \J (x^{i})|
\\& = \Big|
\sum_{|m|, |n| \lesssim \delta^{-2}}\int^{\frac{1}{\delta}}_0 \mathbf{1}_{v \in B_{\delta}} \times 
\J^{m,n} (x)
\dd \tb \dd \xb + \Delta_{\J} (x)
-
\sum_{|m|, |n| \lesssim \delta^{-2}} \int^{\frac{1}{\delta}}_0 \mathbf{1}_{v \in B_{\delta}} \times
\J^{m,n} (x^{i})
\dd \tib \dd \xib - \Delta_{\J} (x^{i})
\Big|
\\& \leq \eqref{j_split^2 second estimate mag} + |\Delta_{\J} (x)| + |\Delta_{\J} (x^{i})| 
\leq
\frac{\epsilon}{3} + \frac{\epsilon}{3} + \frac{\epsilon}{3} = \epsilon.
\end{split}
\Ee
Therefore, we conclude that $\J (x)$ is continuous on $\p\O$.
\end{proof}

\subsection{Uniform \texorpdfstring{$L^1$}{L1} bound of Residual Measure}

Recall that in \eqref{est_a:forcing} and \eqref{est_1:forcing mag}, we bound the residual measure terms as 
\Be \notag
\int_{\mathcal{V}_0} \dd \sigma_0 \cdots \int_{\mathcal{V}_{i-1}} \dd \sigma_{i-1} \leq 2^{i}, \ \text{with} \ i = 2, \cdots , k-1,
\Ee
where $\dd \sigma_j = \mu_{\Theta} (x^{j+1}, v^{j}) \{ n(x^j) \cdot v^j \} \dd v^j$ in \eqref{def:sigma measure}.
To derive the decay of exponential moments, we show the following uniform $L^1$ bound of the residual measure:

\begin{proposition} \label{prop: est:measure}
Consider 
$(X,V)$ in \eqref{characteristics} (resp. \eqref{characteristics mag}),
then for $i = 1, \cdots , k-1$,
there exists $C > 0$ such that
\be \label{est:measure}
\int_{\prod_{j=0}^{i-1} \mathcal{V}_j}   
\mathbf{1}_{0 < t_{i } \leq t} \
\dd \sigma_{i-1} \cdots \dd \sigma_1 \dd \sigma_0 \leq C,
\ee
where $\dd \sigma_j = \mu_{\Theta} (x^{j+1}, v^{j}) \{ n(x^j) \cdot v^j \} \dd v^j$ in \eqref{def:sigma measure}, and
$C = C (\O, \Theta)$ is independent of both $i$ and $t$.
\end{proposition}

\begin{proof}

Recall from Lemma \ref{lem: j continue}, we have the continuity on the following:
	\Be \notag
	\J (x) = \int_{n(x)\cdot v >0} F_s (x, v) \{ n(x) \cdot v \} \dd v, \ \text{for} \ x \in \p\O,
	\Ee
where we consider $F_s (x, v) \geq 0$ solving \eqref{equation for F_infty} and \eqref{diff_F} with $\iint_{\O \times \R^3} F_s (x, v) \dd x \dd v = 1$.
	
From the $L^{\infty}_x$-estimate on $\J (x)$ in Proposition \ref{prop:j linfty bound}, there exists $U > 0$ such that 
	\be \notag
	\| \J (x)\|_{L^{\infty}_x} \leq U.
	\ee 
	Since $\p\O$ is compact and $\J (x)$ is continuous, $\J (x)$ must have the infimum, and we define 
	\Be \notag
	L := \inf\limits_{x \in \p\O} \J (x).
	\Ee
	
	It is easy to see that $L \geq 0$. Moreover, from \eqref{diff_J}, \eqref{FJ relation} and Lemma \ref{conservative field}, we have
	\Be \notag
	\J (x) = \int_{n(x)\cdot v >0} \mu_{\Theta} (\xb, v) \J (\xb) \{ n(x) \cdot v \} \dd v, \ \text{for} \ x \in \p\O.
	\Ee
	Now we claim that $L > 0$.
	Suppose not, if $\J (x) = 0$ for some $x \in \p\O$, we consider the map:
	\be \notag
	v^{0, 0}
	\in \mathcal{V}
	\mapsto (\xb, \tb),
	\ee
	with $1 \leq \tb \leq 10$.
	Then the integrand has to be zero almost everywhere since it is non-negative. This implies
	$\J (x)$ has to be zero almost everywhere, which contradicts with $\iint_{\O \times \R^3} F_s (x, v) \dd x \dd v > 0$.
	Therefore, we derive that $L > 0$ and show the claim.
	
	From the stochastic cycles in Lemma \ref{sto_cycle}, for $x \in \p\O$, we have
	\Be \notag 
	\begin{split}
		U \geq \J (x) 
		& \geq \int_{\prod_{j=0}^{i-1} \mathcal{V}_j}   
		\mathbf{1}_{0 < t_{i } \leq t} \
		\dd \sigma_{i-1} \cdots \dd \sigma_1 \dd \sigma_0 
		\int_{\mathcal{V}_i} F_s (x^{i}, v^{i}) 
		\{ n(x^i) \cdot v^i \} \dd v^i
		\\& \geq \int_{\prod_{j=0}^{i-1} \mathcal{V}_j}   
		\mathbf{1}_{0 < t_{i } \leq t} \
		\dd \sigma_{i-1} \cdots \dd \sigma_1 \dd \sigma_0 \times L.
	\end{split}
	\Ee
	Therefore, for $i = 2, \cdots , k-1$, we derive
	\Be \notag
	\int_{\prod_{j=0}^{i-1} \mathcal{V}_j}   
	\mathbf{1}_{0 < t_{i } \leq t} \
	\dd \sigma_{i-1} \cdots \dd \sigma_1 \dd \sigma_0 \leq L^{-1} U.
	\Ee
	
For the magnet field case, since we have proved $\J (x) $ is continuous in Lemma \ref{lem: j continue mag}. Following the steps in gravitational case, we can conclude \eqref{est:measure}.
\end{proof}

From Proposition \ref{prop: est:measure} and
following the similar idea in Lemma \ref{lem:bound1}, we obtain:

\begin{lemma} \label{lem:bound1_2}
Suppose $f(t,x,v)$ solves \eqref{eqtn_f}-\eqref{diff_f} and $(X,V)$ in \eqref{characteristics}, for $0 \leq t^i \leq t$, $i = 2, \cdots , k-1$, 
\Be \label{est:bound1_2}
	\int_{\prod_{j=0}^{i} \mathcal{V}_j}      
	\mathbf{1}_{t^{i+1} < 0 \leq t^{i }}
	\int^{t^{i}}_{0} \varrho^\prime(s) f(s, X(s; t^i, x^i, v^i), V(s; t^i, x^i, v^i)) \dd s 
	\dd \tilde{\Sigma}_{i}
	\lesssim  \int^t_0 \| \varrho^\prime (s) f(s) \|_{L^1_{x,v}} \dd s.
\Ee
where 
$\dd \tilde{\Sigma}_{i} := \frac{ \dd \sigma_{i}}{\mu_{\Theta} (x^{i+1}, v^{i})} \dd \sigma_{i-1} \cdots \dd \sigma_1 \dd \sigma_0$ with $\dd \sigma_j = \mu_{\Theta} (x^{j+1}, v^{j}) \{ n(x^j) \cdot v^j \} \dd v^j$ in \eqref{def:sigma measure}.
\end{lemma}

\begin{proof}
For \eqref{est:bound1_2}, it suffices to prove this upper bound for $i = 2,...,k-1$, 
\Be \label{forcing_2}
\int_{\mathcal{V}_0} \dd \sigma_0 \cdots \int_{\mathcal{V}_{i-1}} \dd \sigma_{i-1} 
\underbrace{ \int_{\mathcal{V}_i} 
\mathbf{1}_{t^{i+1} < 0 \leq t^i}
\int^{t^{i}}_{0} \varrho^\prime(s) | f(s, X(s; t^i, x^i, v^i), V(s; t^i, x^i, v^i)) | \dd s
\{n(x^i) \cdot v^i\} \dd v^i}_{\eqref{forcing_2}^*}.
\Ee  
Following Lemma \ref{lem:bound1}, \eqref{est_a:forcing} and \eqref{est_b:forcing}, we have
\begin{align}
\eqref{forcing_2} \lesssim 
& \int_{\mathcal{V}_0} \dd \sigma_0 \cdots \int_{\mathcal{V}_{i-3}} \dd \sigma_{i-3}  
\int^{t^{i-2}}_0 
\frac{\dd \tb^{i-1}}{\langle \tb^{i-1} \rangle^{5}}
\int_0^{\min\{t^{i-2}- \tb^{i-1},  \tb^{i-1}  \}} \dd \tb^{i-2} \int_{\p\O} \dd S_{x^i} \eqref{forcing_2}^* \label{est_a:forcing_2}
\\& + \int_{\mathcal{V}_0} \dd \sigma_0 \cdots \int_{\mathcal{V}_{i-3}} \dd \sigma_{i-3} 
\int^{t^{i-2}}_0 \frac{\dd \tb^{i-2}}{\langle  \tb^{i-2}\rangle^{5}}
\int_0^{ \min\{t^{i-2}- \tb^{i-2},\tb^{i-2}  \}} \dd \tb^{i-1} \int_{\p\O} \dd S_{x^i} \eqref{forcing_2}^*. \label{est_b:forcing_2}
\end{align} 

For \eqref{est_a:forcing_2}, we employ the change of variables 
$$
(x^{i}, \tb^{i-2}, v^{i}) 
\mapsto (y, w) = (X(s; t^{i-2} -\tb^{i-2} - \tb^{i-1}, x^{i}, v^{i}), V(s; t^{i-2} -\tb^{i-2} - \tb^{i-1}, x^{i}, v^{i})) \in \O \times\R^3, $$
with $
|n(x^i) \cdot v^i| \dd S_{x^i} \dd \tb^{i-2} \dd v^i \lesssim \dd y \dd w$. Applying this change of variables, Proposition \ref{prop: est:measure} and $0 \leq t^i \leq t$, we derive
\Be \notag
\begin{split}
\eqref{est_a:forcing_2}
& \leq \int_{\mathcal{V}_0} \dd \sigma_0 \cdots \int_{\mathcal{V}_{i-3}} \dd \sigma_{i-3}
\int^{t^{i-2}}_0 \dd \tb^{i-1}\langle \tb^{i-1} \rangle^{-5} 
\int^{t^{i}}_{0} \varrho^\prime(s)
\iint_{\O \times\R^3 } 
|f (s, y, w)|  \dd y \dd w \dd s
\\& \lesssim \int^t_0 \| \varrho^\prime (s) f(s) \|_{L^1_{x,v}} \dd s.
\end{split}
\Ee

A bound of \eqref{est_b:forcing_2} can be derived, using the change of variables 
$$
(x^{i}, \tb^{i-1}, v^{i}) 
\mapsto (y, w) = (X(s; t^{i-2} -\tb^{i-2} - \tb^{i-1}, x^{i}, v^{i}), V(s; t_{i-2} -\tb^{i-2} - \tb^{i-1}, x^{i}, v^{i})) \in \O \times\R^3, $$
with $
|n(x^i) \cdot v^i| \dd S_{x^i} \dd \tb^{i-1} \dd v^i \lesssim \dd y \dd w$.
\end{proof}


For the magnetic case, we need Proposition \ref{prop: est:measure} and Lemma \ref{lem:bound1 mag} to derive the following:

\begin{lemma} \label{lem:bound1_2 mag}
Suppose $f(t,x,v)$ solves \eqref{eqtn_f}-\eqref{diff_f} for \eqref{field property mag}  and $(X,V)$ in \eqref{characteristics mag},
for $0 \leq t^i \leq t$, $i = 2, \cdots , k-1$ 
and $0 < \e \ll 1$, 
\Be 
\begin{split}
& \ \ \ \  \int_{\prod_{j=0}^{i} \mathcal{V}_j} \mathbf{1}_{t^{i+1} < 0 \leq t^{i }}
 \int^{t^{i}}_{0} \varrho^\prime(s) f(s, X(s; t^i, x^i, v^i), V(s; t^i, x^i, v^i)) \dd s 
 \dd \tilde{\Sigma}_{i} 
\\&  \lesssim  \e^{i-1} \int_{\mathcal{V}_i} \mathbf{1}_{t^{i+1} < 0 \leq t^{i}} \ \varrho (t^i) f(t^i, x^i, v^i) \{ n(x^i) \cdot v^i \}  \dd v^i
 +  (i-1) \e^{-2} \int^t_0 \| \varrho^\prime (s) f(s) \|_{L^1_{x,v}} \dd s.
\end{split}
\Ee
where 
$\dd \tilde{\Sigma}_{i} := \frac{ \dd \sigma_{i}}{\mu_{\Theta} (x^{i+1}, v^{i})} \dd \sigma_{i-1} \cdots \dd \sigma_1 \dd \sigma_0$ with $\dd \sigma_j = \mu_{\Theta} (x^{j+1}, v^{j}) \{ n(x^j) \cdot v^j \} \dd v^j$ in \eqref{def:sigma measure}.
\end{lemma}

\begin{proof}
It suffices to prove this upper bound for $i = 2,...,k-1$, 
\Be \label{forcing_2 mag}
\int_{\prod_{j=0}^{i-1} \mathcal{V}_j}
\int_{\mathcal{V}_i} 
\mathbf{1}_{t^{i+1} < 0 \leq t^i}
 \int^{t^{i}}_{0}
|\varrho^\prime(s) f(s, X(s; t^i, x^i, v^i), V(s; t^i, x^i, v^i))| \dd s \dd \tilde{\Sigma}_{i} .
\Ee  
Again recall Definition \ref{def:BG mag}, we obtain for $j = 0, \cdots, i-1$ and $0 < \e \ll 1$,
\be \notag
v^{j} \in \mathcal{V}^{B_{\e}}_j, \ \text{or} \ v^{j} \in \mathcal{V}^{G_{\e}}_j.
\ee
Now we split $\{v^{j}\}^{i-1}_{j=0}$ into two cases: 

\textbf{Case 1:} For all $0 \leq j \leq i-2$, $v^{j} \in \mathcal{V}^{B_{\e}}_j$.
From Lemma \ref{lem: Be mag} and \ref{lem:bound1 mag}, we derive that
\Be \notag
\begin{split}
\eqref{forcing_2 mag}
& \lesssim \e^{i-1} 
\int_{\mathcal{V}_i} 
\mathbf{1}_{t^{i+1} < 0 \leq t^i}
 \int^{t^{i}}_{0}
|\varrho^\prime(s) f(s, X(s; t^i, x^i, v^i), V(s; t^i, x^i, v^i))| \dd s \{n(x^i) \cdot v^i\} \dd v^i 
\\& = \e^{i-1} \int_{\mathcal{V}_i} \mathbf{1}_{t^{i+1} < 0 \leq t^{i}}
\int^{t^{i}}_{0}
\varrho^\prime(s) \dd s 
|f(t^i, x^i, v^i)| \{ n(x^i) \cdot v^i \} \dd v^i 
\\& \lesssim \e^{i-1} \int_{\mathcal{V}_i} \mathbf{1}_{t^{i+1} < 0 \leq t^{i}} \ \varrho (t^i) f(t^i, x^i, v^i) \{ n(x^i) \cdot v^i \}  \dd v^i.
\end{split}
\ee

\textbf{Case 2:} There exists $0 \leq j \leq i-2$ such that $v^{j} \in \mathcal{V}^{G_{\e}}_j$.
Using Proposition \ref{prop:mapV mag}, we get 
\begin{align}
\eqref{forcing_2 mag} 
& \lesssim \int_{\mathcal{V}_0} \dd \sigma_0 \cdots \int_{\mathcal{V}_{j-1}} \dd \sigma_{j-1} \times 
\int_{v^j \in \mathcal{V}^{G_{\e}}_j} \frac{5 B^2_3 \sum\limits_{m, n \in \mathbb{Z}} \mu_{\Theta} (x^{j+1}, v^{m, n}_{j, \mathbf{b}})}{2 - 2 \cos (B_3 \tb^{j})} \{ n(x^j) \cdot v^j \} \dd S_{x^{j+1}} \dd \tb^{j}  \notag
\\& \ \ \ \ \times \cdots \times \int_0^{t^{i-2}} \int_{\p\O} \frac{5 B^2_3 \sum\limits_{m, n \in \mathbb{Z}} \mu_{\Theta} (x^{i-1}, v^{m, n}_{i-2, \mathbf{b}})}{2 - 2 \cos (B_3 \tb^{i-2})} \{ n(x^{i-2}) \cdot v^{i-2} \}  \dd \tb^{i-2} \dd S_{x^{i-1}} \notag
\\& \ \ \ \ \times \int^{t^{i-1}}_0 \int_{\p\O}  
\frac{5 B^2_3 \sum\limits_{m, n \in \mathbb{Z}} \mu_{\Theta} (x^{i}, v^{m, n}_{i-1, \mathbf{b}})}{2 - 2 \cos (B_3 \tb^{i-1})} \{ n(x^{i-1}) \cdot v^{i-1} \} \dd \tb^{i-1} \dd S_{x^{i}} \notag
\\& \ \ \ \ \times
\int_{\mathcal{V}_i} 
\mathbf{1}_{t^{i+1} < 0 \leq t^i} \
 \int^{t^{i}}_{0}
|\varrho^\prime(s) f(s, X(s; t^i, x^i, v^i), V(s; t^i, x^i, v^i))| \dd s  \{n(x^i) \cdot v^i\} \dd v^i. \label{est1:forcing_2 mag}
\end{align}
Following the similar steps in Lemma \ref{lem:bound1 mag}, we derive that
\Be \label{est_1:forcing_2 mag}
\begin{split} 
\eqref{forcing_2 mag} 
\lesssim 
& \int_{\mathcal{V}_0} \dd \sigma_0 \cdots \int_{\mathcal{V}_{j-1}} \dd \sigma_{j-1} \times 
\e^{-2} 
 \times \int_{v^j \in \mathcal{V}^{G_{\e}}_j} \dd \tb^{j} \int_{\p\O} \dd S_{x^{i}} \eqref{est1:forcing_2 mag}. 
\end{split}
\ee

For \eqref{est_1:forcing_2 mag}, we employ the change of variables 
\be \notag
(x^{i}, \tb^{j}, v^{i}) 
\mapsto (y, w) = (X(s; t^{j} -\tb^{j} \cdots - \tb^{i-1}, x^{i}, v^{i}), V(s; t^{j} -\tb^{j} \cdots - \tb^{i-1}, x^{i}, v^{i})) \in \O \times\R^3, 
\ee
with 
$|n(x^i) \cdot v^i| \dd S_{x^i} \dd \tb^{j} \dd v^i \lesssim \dd y \dd w$. Applying the above and Proposition \ref{prop: est:measure}, we derive
\Be \notag
\begin{split}
\eqref{est_1:forcing_2 mag}
& \lesssim \e^{-2} \times
 \int^{t^{i}}_{0}
\varrho^\prime(s) \dd s
 \iint_{\O \times\R^3 } 
|f (s, y, w)|  \dd w \dd y
\\& \lesssim \e^{-2} \times \int^t_0 \| \varrho^\prime (s) f(s) \|_{L^1_{x,v}} \dd s.
\end{split}
\Ee
Applying the above on all parts when $j = 0, \cdots, i-2$ in second cases, we deduce
\Be \notag
\eqref{forcing_2 mag} 
\lesssim \sum\limits^{i-2}_{j = 0}  \e^{-2} \times \int^t_0 \| \varrho^\prime (s) f(s) \|_{L^1_{x,v}} \dd s 
\leq (i-1) \e^{-2} \times \int^t_0 \| \varrho^\prime (s) f(s) \|_{L^1_{x,v}} \dd s.
\ee
Collecting the estimates from two cases, we prove this Lemma.
\end{proof}


\section{Estimates on Exponential Moments}
\label{sec: exponential moments}

Using Proposition \ref{prop: est:measure}, now we are able to show the asymptotic behaviour of the exponential moments. The main purpose of this section to prove Theorem \ref{theorem}.

We first set $w (x, v) := e^{\theta (|v|^2+ 2\Phi (x))}$, $w^\prime (x, v) := e^{\theta^\prime (|v|^2+ 2\Phi (x))}$ for $0< \theta < \theta^\prime< \frac{1}{2b}$ with $b$ defined in \eqref{def:Theta}.
Then we include a time dependent weight function $\varrho (t)$ and consider the stochastic cycle representation of $\varrho (t) w^\prime (x, v) f(t, x, v)$:
\begin{align}
   & \varrho (t) w^\prime (x, v) f(t, x, v) 
   = \mathbf{1}_{t^1 < 0} \
   \varrho (t) w^\prime (X(0; t, x, v), V(0; t, x, v)) f(0, X, V) 
\label{expand_k1}
\\& + w^\prime \mu_{\Theta} (x^1, \vb) \int_{\prod_{j=1}^{i} \mathcal{V}_j}   
     \sum\limits^{k-1}_{i=1} 
     \Big\{ \mathbf{1}_{t^{i+1} < 0 \leq t^{i}} 
     \varrho(0) w^\prime (X(0; t^i, x^i, v^i), V(0; t^i, x^i, v^i)) f (0, X, V) \Big\}
      \dd \tilde{\Sigma}_{i}
\label{expand_k2}
    \\& + w^\prime \mu_{\Theta} (x^1, \vb) \int_{\prod_{j=1}^{i} \mathcal{V}_j}   
     \sum\limits^{k-1}_{i=1} 
      \mathbf{1}_{0 \leq t^{i}}
     \Big\{ \int^{t^{i}}_{ \max(0, t^{i+1})}  
     \varrho^\prime (s) w^\prime (X(s; t^i, x^i, v^i), V(s; t^i, x^i, v^i))
\notag     
\\& \hspace{8cm} \times f(s, X(s; t^i, x^i, v^i), V(s; t^i, x^i, v^i)) \dd s 
      \Big\} \dd \tilde{\Sigma}_{i}
\label{expand_k3}
    \\& + w^\prime \mu_{\Theta} (x^1, \vb) \int_{\prod_{j=1}^{k } \mathcal{V}_j}   
    \mathbf{1}_{t^{k} \geq 0} \
    \varrho (t^k) w^\prime f (t^{k}, x^{k}, v^{k})
     \dd \tilde{\Sigma}_{k}, 
\label{expand_k4}
\end{align} 
where 
$\dd \tilde{\Sigma}_{i} 
:= \frac{ \dd \sigma_{i}}{\mu_{\Theta} (x^{i+1}, v^{i}) w^\prime (x^i, v^{i})} \dd \sigma_{i-1} \cdots \dd \sigma_1$.
Here, $(X,V)$ in \eqref{characteristics} (resp. \eqref{characteristics mag}).

\subsection{Gravitational case}

\begin{proof}[\textbf{Proof of Theorem \ref{theorem}}]

We start to prove \eqref{theorem_infty_1}, and pick $\varrho (t) = t + 1$ to utilize the $L^1$-decay of Theorem \ref{theorem_1}. Then we work on the stochastic cycle representation of $\varrho (t) w^\prime (x, v) f(t, x, v)$ in \eqref{expand_k1}-\eqref{expand_k4}.

For the contribution of \eqref{expand_k1}, since $t^1 < 0$, and $w^\prime, f$ are constant along the characteristic trajectory. Thus we deduce that
\be \label{est:expand_k1}
w^\prime (x, v) f(t, x, v) 
= w^\prime (X(0; t, x, v), V(0; t, x, v)) f(0, X, V)
\leq \| w^\prime f(0) \|_{L^\infty_{x,v}}.
\ee

Now we bound the contribution of \eqref{expand_k2}. Since
$0 < n(x) \cdot v \lesssim w^\prime (x, v) < \mu^{-1} (x, v)$ and applying the uniform bound on residual measure in Proposition \ref{prop: est:measure}, we derive
\Be \label{est:expand_k2}
\begin{split}
\frac{1}{\varrho (t)} | \eqref{expand_k2} |
& \lesssim \frac{k}{\varrho (t)}  \bigg( \sup_{i }    \int_{\prod_{j=1}^{k} \mathcal{V}_j}   
     \mathbf{1}_{t^{i+1} < 0 \leq t^{i }} 
      \dd \tilde{\Sigma}_{i}\bigg)
      \varrho(0)
   \| w^\prime f(0) \|_{L^\infty_{x,v}}   
\\&  \lesssim  \frac{k}{\varrho (t)}  \bigg(   \int_{n(x^{j}) \cdot v^{j} >0}
        \frac{ |n(x^{j}) \cdot v^{j}| }{w^\prime (x^j, v^{j})}\dd v^j \bigg)   \| w^\prime f(0) \|_{L^\infty_{x,v}}
\\& \lesssim  \frac{k}{\varrho (t)}  \| w^\prime f(0) \|_{L^\infty_{x,v}}.
\end{split}
\Ee

Applying Lemma \ref{lem:bound1_2}, Proposition \ref{prop: est:measure} and Theorem \ref{theorem_1}, we bound the contribution of \eqref{expand_k3}. Since
$0 < n(x) \cdot v \lesssim w^\prime (x, v) < \mu_{\Theta} ^{-1} (x, v)$ and $\varrho^\prime = 1$, we have
\be \label{est:expand_k3}
\begin{split}
\frac{1}{\varrho (t)} |\eqref{expand_k3}|     
& \lesssim \frac{k}{\varrho (t)} \sup_{i }  \int_{\prod_{j=1}^{i} \mathcal{V}_j}  \mathbf{1}_{0 \leq t^{i}}
\int^{t^{i}}_{ \max(0, t^{i+1})} 
w^\prime (X(s; t^i, x^i, v^i), V(s; t^i, x^i, v^i)) f(s, X, V) \dd s \dd \tilde{\Sigma}_{i}    
\\& \lesssim \frac{k}{\varrho (t)} \int^t_0 \| f(s) \|_{L^1_{x,v}} \dd s   
\lesssim \frac{k}{\varrho (t)} \times \| w^\prime f (0) \|_{L^\infty_{x,v}}.
\end{split}
\ee
where the last inequality holds from $w^\prime (x, v) = w^\prime (\xb, \vb)$ and $|\vb| \gtrsim \tf$.

Lastly we bound the contribution of \eqref{expand_k4}. From Lemma \ref{lem:small_largek}, Proposition \ref{prop: est:measure} and $0 < n(x^2) \cdot v^2 \lesssim w (x^2, v^2) < w^\prime (x^2, v^2)$, we get
\Be \label{est:expand_k4}
\begin{split}
 \frac{1}{\varrho (t)} |\eqref{expand_k4}| 
& \lesssim \frac{\varrho (t^k)}{\varrho (t)} \sup_{(x,v) \in \bar{\O} \times \R^3}  \Big(\int_{\prod_{j=1}^{k -1} \mathcal{V}_j}   
    \mathbf{1}_{t_{k }(t,x,v,v^1,\cdots, v^{k-1}) \geq 0 }
\dd \sigma_1 \cdots \dd \sigma_{k-1}\Big)
 \| w^\prime f(t_k) \|_{L^\infty_{x,v}}
\\&  \lesssim e^{-t} \sup_{t \geq s \geq 0} \| w^\prime f(s) \|_{L^\infty_{x,v}}.
\end{split}
\Ee

Collecting estimates from \eqref{est:expand_k1}-\eqref{est:expand_k4} and using $k \lesssim t$, we derive 
\Be \label{k estimate}
(1 -  e^{-t}) \sup_{t \geq 0} \| w^\prime f(t) \|_{L^\infty_{x,v}}
\lesssim 
(1 + \frac{k}{\varrho (t)}) \times
\| w^\prime f (0) \|_{L^\infty_{x,v}}. 
\Ee
Therefore, we prove \eqref{theorem_infty_1}.

Next, to prove \eqref{theorem_infty} and show the decay of exponential moments and again utilize the $L^1$-decay, we set a new weight function
\Be \label{varrho}
\varrho(t):=  e^{\Lambda t}.
\Ee
Clearly we have $\varrho^\prime(t) \lesssim   e^{\Lambda t}$.

From Lemma \ref{sto_cycle}, we derive the form of $\int_{\R^3} w(x, v) |f (t, x, v)| \dd v$. First we split $|v| \geq t/2$ and $t_1 \leq 3t/4$ case to get \eqref{v<t/2 term in wf}-\eqref{first team in f_exp}. Next, for $t_1 \geq 3t/4$ case, we follow along the stochastic cycles twice with $k=2$ and $t_* = t/2$ and get \eqref{second team in f_exp} and \eqref{f_exp2}.
\begin{align}
\int_{\R^3} w(x, v) & |f(t,x,v)| \dd v 
\leq 
\int_{|v| \geq t/2} w(x, v) |f(t,x,v)| \dd v
\label{v<t/2 term in wf}
\\& + \int_{|v| \leq t/2} \mathbf{1}_{t^1 \leq 3t/4}  w(x, v) |f(3t/4, X(3t/4; t, x, v), V(3t/4; t, x, v))| \dd v
\label{first team in f_exp}
\\& + \int_{\R^3}  \mathbf{1}_{t^1 \geq 3t/4} w (x, v) \mu_{\Theta} (x^1, \vb)  \int_{\prod^2_{j=1} \mathcal{V}_j} \mathbf{1}_{t^2 < t/2 < t^1} w (x^1, v^1) |f(t^1, x^1, v^1)| \dd \Sigma_{1}^2 \dd v 
\label{second team in f_exp} 
\\& + \int_{\R^3} \mathbf{1}_{t^1 \geq 3t/4}  
   w(x, v) \mu_{\Theta} (x^1, \vb) 
  \Big| \int_{\prod^2_{j=1} \mathcal{V}_j} \mathbf{1}_{t^2 \geq t/2} w (x^2, v^2)
f(t^2, x^2, v^2) \dd \Sigma_{2}^2 \Big| \dd v, \label{f_exp2}
\end{align} 
where $\dd  {\Sigma}^{2}_{1} = \dd \sigma_{2}  \frac{ \dd \sigma_{1}}{ \mu_{\Theta} (x^2, v^1)w(x^1, v^1)} $ and $\dd  {\Sigma}^{2}_{2} = \frac{ \dd \sigma_{2}}{ \mu_{\Theta} (x^3, v^2) w(x^2, v^2)} \dd \sigma_{1}$,
and 
$\dd \sigma_j = \mu_{\Theta} (x^{j+1}, v^j) \{ n(x^j) \cdot v^j \} \dd v^j$ in \eqref{def:sigma measure}.

For \eqref{v<t/2 term in wf},
from $|v| \geq t/2$, the $L^\infty$-boundedness and $0 < w < w^\prime$, we derive that
\Be \label{v<t/2 part in wf}
\begin{split}
\int_{|v| \geq t/2} w(x, v) |f(t,x,v)| \dd v
& \leq \int_{|v| \geq t/2} \frac{w (x, v)}{w^\prime (x, v)} \dd v \| w^\prime f(0) \|_{L^\infty_{x,v}}
\\& \lesssim \frac{1}{\sqrt{\theta^\prime - \theta}} e^{- \frac{(\theta^\prime - \theta) t^2}{4}} \| w^\prime f(0) \|_{L^\infty_{x,v}}.
\end{split}
\Ee

For \eqref{first team in f_exp}, from $|v| \leq t/2$ and $t^1 \leq 3t/4$, we have
\Be \label{v=3t/4}
|V(3t/4; t, x, v)| \gtrsim |v_3 + g (t - 3t/4)| \geq 2t \geq |v| + t.
\Ee
Recall $\Phi (x)$ defined in \eqref{field property}, we derive
\Be \label{positive on Phi}
\begin{split}
\Phi (x) 
& = g x_3 + \int^{x_3}_{0} \Big\{ \int^{s}_{0}\nabla_{x_3} \nabla_{x_3} \phi (x_1, x_2, \alpha) \dd \alpha + \nabla_{x_3} \phi (x_1, x_2, 0) \Big\} \dd s + \phi(x_1, x_2, 0) 
\\& \geq g x_3 - \frac{1}{\varrho_1} \varrho_2 x_3 - \varrho_3 x_3 \gtrsim g x_3 \geq 0.
\end{split}
\Ee
From Lemma \ref{conservative field} and $\Phi(x)|_{x_3 = 0} = 0$, we have
\Be \label{vvb relation}
\frac{|V(s;t,x,v)|^2}{2} + \Phi(X(s;t,x,v)) 
= \frac{|v|^2}{2} = \frac{|\vb|^2}{2}.
\Ee
Then, from the $L^\infty$-boundedness, \eqref{v=3t/4} and $0 < w < w^\prime$, we deduce that 
\Be \label{first part in wf}
\begin{split}
\eqref{first team in f_exp} 
& \lesssim \int_{|v| \leq t/2} \frac{w (x, v)}{w^\prime (X(3t/4; t, x, v), V(3t/4; t, x, v))} \dd v \| w^\prime f(0) \|_{L^\infty_{x,v}}
\\& \leq \int_{|v| \leq t/2} 
e^{(\theta - \theta^\prime) (|V(3t/4; t, x, v)|^2+ \Phi (X))} \dd v \| w^\prime f(0) \|_{L^\infty_{x,v}}
\\& \leq \int_{|v| \leq t/2} 
e^{(\theta - \theta^\prime) (|v|^2 + t^2)} \dd v \| w^\prime f(0) \|_{L^\infty_{x,v}}
\lesssim \frac{1}{\sqrt{\theta^\prime - \theta}} e^{- (\theta^\prime - \theta) t^2} \| w^\prime f(0) \|_{L^\infty_{x,v}}.
\end{split}
\Ee
Next, we bound $\int_{\R^3} w(x, v) \mu_{\Theta} (x^1, \vb) \dd v$ shown in \eqref{second team in f_exp} and \eqref{f_exp2}.
Note that from \eqref{positive on Phi} and \eqref{vvb relation}, we have $|v| = |\vb| \geq |V(s;t,x,v)|$. Thus, we derive
\Be \label{vvb estimate}
\int_{\R^3} w(x, v) \mu_{\Theta} (x^1, \vb) \dd v 
\lesssim \int_{\R^3} \frac{w (x, v)}{w^\prime (x, v)} \dd v \lesssim 1.
\Ee

For \eqref{second team in f_exp}, from \eqref{vvb estimate} and $\int_{\mathcal{V}_2} \dd \sigma_2$ is bounded, we have 

\Be \label{first step in second part in wf}
\eqref{second team in f_exp} 
\lesssim \int_{\mathcal{V}_1} \mathbf{1}_{t^2 < t/2 < t^1} |f(t^1, x^1, v^1)| \{ n(x^1) \cdot v^1 \} \dd v^1.
\Ee
Since $t^1 \geq 3t/4$ and $t^2 < t/2$, we get $t/4 \leq t^1 - t^2$. From \eqref{v3t estimate}, we have $|v^1| \geq |n(x^1) \cdot v^1| \gtrsim g t/8$. Again, from the $L^\infty$-boundedness, $t \gg 1$ and $0 < n(x^1) \cdot v^1 \lesssim w (x^1, v^1) < w^\prime (x^1, v^1)$,
we derive
\Be \label{second part in wf}
\begin{split}
\eqref{first step in second part in wf} 
& \lesssim \int_{|v^1| \gtrsim g t/8} w (x^1, v^1) |f(t^1, x^1, v^1)| \dd v^1 
\\& \lesssim \int_{|v^1| \gtrsim g t/8} \frac{w (x^1, v^1)}{w^\prime (x^1, v^1)} \dd v^1 \| w^\prime f(0) \|_{L^\infty_{x,v}}
\lesssim \frac{1}{\sqrt{\theta^\prime - \theta}} e^{- \frac{(\theta^\prime - \theta) t^2}{4}} \| w^\prime f(0) \|_{L^\infty_{x,v}}.
\end{split}
\Ee

Now we only need to bound \eqref{f_exp2}. Since $\int_{\R^3} w(x, v) \mu_{\Theta} (x^1, \vb) \dd v \lesssim 1$ and $\int_{\mathcal{V}_1} \dd \sigma_1$ is bounded, it suffices to prove the decay of
\Be \label{decay of f_exp2}
\sup_{v \in \R^3, v^1 \in \mathcal{V}_1}
\Big| \int_{\mathcal{V}_2} 
\mathbf{1}_{t^2 \geq t/2} f(t^2,x^2,v^2) \{ n(x^{2}) \cdot v^{2} \} \dd v^{2} \Big|. 
\Ee 

Here we define $g (t, x, v) := \varrho (t) w (x, v) f (t, x, v)$, and note that
\Be \notag
\frac{1}{\varrho (t^2)} \int_{\mathcal{V}_2} \frac{|n(x^2) \cdot v^2|}{w (x^2, v^2)} g (t^2, x^2, v^2) \dd v^2 = \int_{\mathcal{V}_2} f(t^2,x^2,v^2) \{ n(x^2) \cdot v^2 \} \dd v^2.
\Ee 
Therefore, it suffices to show the decay of $\big| \frac{1}{\varrho (t^2)} \int_{\mathcal{V}_2} \mathbf{1}_{t^2 \geq t/2} \frac{|n(x^2) \cdot v^2|}{w (x^2, v^2)} g (t^2, x^2, v^2) \dd v^2 \big|$.

Applying Lemma \ref{sto_cycle} with $w(x, v) = e^{\theta (|v|^2+ 2\Phi (x))}$, and $\varrho(t)$ in \eqref{varrho}, and choosing $t_*=0$, $k \geq \mathfrak{C}t$ as in Lemma \ref{lem:small_largek}, we obtain the following stochastic cycle representation of 
$g (t^2, x^2, v^2) = \varrho (t^2) w (x^2, v^2) f (t^2, x^2, v^2)$:
\begin{align}
    g (t^2, x^2, v^2) 
   = & \mathbf{1}_{t^3 < 0} \
   \varrho (0) w(x^2, v^2) f (0, X(0; t^2, x^2, v^2), V(0; t^2, x^2, v^2))
\label{expand_g1}
    \\& + w(x^2, v^2) \int^{t^2}_{\max(0, t^3)} \varrho^\prime(s) f(s, X(s; t^2, x^2, v^2), V(s; t^2, x^2, v^2)) \dd s \label{expand_g2}
    \\& +  w \mu_{\Theta} (x^3, v^2_{\mathbf{b}}) \int_{\prod_{j=3}^{i} \mathcal{V}_j}   
     \sum\limits^{k-1}_{i=3} 
     \Big\{ \mathbf{1}_{t^{i+1} < 0 \leq t^{i}}  \varrho (0) w(x^i, v^{i}) \notag
     \\& \hspace{5cm} \times f (0, X(0; t^i, x^i, v^i), V(0; t^i, x^i, v^i)) \Big\}
      \dd \tilde{\Sigma}_{i}
\label{expand_g3}
    \\& + w \mu_{\Theta} (x^3, v^2_{\mathbf{b}}) \int_{\prod_{j=3}^{i} \mathcal{V}_j}   
     \sum\limits^{k-1}_{i=3} 
      \mathbf{1}_{0 \leq t^{i}}
     \Big\{ \int^{t^{i}}_{ \max(0, t^{i+1})} \varrho^\prime(s) w (x^i, v^{i}) 
     \notag
     \\& \hspace{5cm} \times f(s, X(s; t^i, x^i, v^i), V(s; t^i, x^i, v^i)) \dd s 
      \Big\} \dd \tilde{\Sigma}_{i}
\label{expand_g4}
    \\& + w \mu_{\Theta} (x^3, v^2_{\mathbf{b}}) \int_{\prod_{j=3}^{k } \mathcal{V}_j}   
    \mathbf{1}_{t^{k} \geq 0} \
    g (t^{k}, x^{k}, v^{k})
     \dd \tilde{\Sigma}_{k}, \label{expand_g5} 
\end{align} 
where 
$\dd \tilde{\Sigma}_{i} 
:= \frac{ \dd \sigma_{i}}{\mu_{\Theta} (x^{i+1}, v^{i}) w(x^i, v^{i})} \dd \sigma_{i-1} \cdots \dd \sigma_3$ with $3 \leq i \leq k$.
Here, we regard $t^2, x^2, v^2$ as free parameters and from Lemma \ref{conservative field}, we have 
$\mu_{\Theta} (x^3, v^2_{\mathbf{b}}) = \mu_{\Theta} (x^3, v^2)$.

Next we estimate the contribution of \eqref{expand_g1}-\eqref{expand_g5} in 
$\frac{1}{\varrho (t^2)} \int_{\mathcal{V}_2} \frac{ |n(x^2) \cdot v^2| }{w (x^2, v^2)} g (t^2, x^2, v^2) \dd v^2$
term by term. 

For the contribution of \eqref{expand_g1},
from $t^2 \geq t/2$, $t^3 \leq 0$ and \eqref{v3t estimate}, we have $n(x^2) \cdot v^2 \gtrsim \frac{gt}{4}$. From the $L^\infty$-boundedness, $t \gg 1$ and $0 < n(x^2) \cdot v^2 \lesssim w (x^2, v^2) < w^\prime (x^2, v^2)$,  
we deduce that 
\Be \label{est:expand_g1}
\begin{split}
\frac{1}{\varrho (t^2)} \int_{\mathcal{V}_2} \frac{|n(x^2) \cdot v^2|}{w (x^2, v^2)} |\eqref{expand_g1}| \dd v^2
& = \frac{1}{\varrho (t^2)} \int_{\mathcal{V}_2} |n(x^2) \cdot v^2| | \varrho (0) f (t^2, x^2, v^2) | \dd v^2 
\\& \lesssim \frac{1}{\varrho (t^2)} \int_{n(x^2) \cdot v^2 \gtrsim \frac{gt}{4}} \frac{|n(x^2) \cdot v^2|}{w (x^2, v^2)} \varrho(0) \dd v^2 \| w f(0) \|_{L^\infty_{x,v}}
\\& \lesssim \frac{1}{\varrho (t)} \varrho(0) \| w f(0) \|_{L^\infty_{x,v}}
\lesssim  \frac{1}{\varrho (t)} \| w^\prime f(0) \|_{L^\infty_{x,v}}.
\end{split}
\Ee

Now we bound the contribution of \eqref{expand_g2}. Recall Theorem \ref{theorem_1} with $\varrho^\prime(t) \lesssim e^{\Lambda t}$ and Lemma  \ref{lem:bound1_2}, we get 
\Be \label{est:expand_g2}
\begin{split}
\frac{1}{\varrho (t^2)} \int_{\mathcal{V}_2} \frac{|n(x^2) \cdot v^2|}{w (x^2, v^2)} |\eqref{expand_g2} | \dd v^2
& \lesssim \frac{1}{\varrho (t)} \int^t_0 \| \varrho^\prime (s) f(s) \|_{L^1_{x,v}} \dd s    
\\& \lesssim \frac{1}{\varrho (t)} \int^t_0 \|   e^{\Lambda s} f(s) \|_{L^1_{x,v}} \dd s
\\& \lesssim \frac{t}{\varrho (t)} \times \{ \| w^\prime f (0) \|_{L^\infty_{x,v}}
+ \| \varphi \big( (|v|^2 + 2 \Phi(x))^{1/2} \big) f_0 \|_{L^1_{x,v}}
 \}. 
\end{split}
\Ee

Next, we bound the contribution of \eqref{expand_g3}. From the $L^\infty$-boundedness,
$0 < n(x^2) \cdot v^2 \lesssim w (x^2, v^2) < w^\prime (x^2, v^2) < \mu_{\Theta}^{-1} (x^3, v^2)$  and $f (0, X(0; t^i, x^i, v^i), V(0; t^i, x^i, v^i)) = f (t^i, x^i, v^i)$, we derive 
\Be \label{est:expand_g3}
\begin{split}
\frac{1}{\varrho (t^2)} \int_{\R^3} \frac{|n(x^2) \cdot v^2|}{w (x^2, v^2)} |\eqref{expand_g3}| \dd v^2
& \lesssim \frac{k}{\varrho (t)}  \bigg( \sup_{i }    \int_{\prod_{j=3}^{i} \mathcal{V}_j}   
     \mathbf{1}_{t^{i+1} < 0 \leq t^{i }} 
      \dd \tilde{\Sigma}_{i} \bigg)
      \varrho(0)
   \| w f(0) \|_{L^\infty_{x,v}}   
\\&  \lesssim  \frac{k}{\varrho (t)}  \bigg(   \int_{n(x^{i}) \cdot v^{i} >0}
        \frac{ |n(x^{i}) \cdot v^{i}| }{w (x^i, v^{i})}\dd v^i \bigg)   \| w f(0) \|_{L^\infty_{x,v}}
\\& \lesssim  \frac{k}{\varrho (t)}  \| w^\prime f(0) \|_{L^\infty_{x,v}},
\end{split}
\Ee
where the second line follows from Proposition \ref{prop: est:measure}.

Again using Lemma \ref{lem:bound1_2}, Proposition \ref{prop: est:measure} and Theorem \ref{theorem_1}, we bound the contribution of \eqref{expand_g4}. From  $0 < n(x^2) \cdot v^2 \lesssim \mu_{\Theta}^{-1} (x^3, v^2)$ and $\varrho^\prime(t) \lesssim  e^{\Lambda t}$, we have
\Be \label{est:expand_g4}
\begin{split}
& \ \ \ \ \frac{1}{\varrho (t^2)} \int_{\R^3} \frac{|n(x^2) \cdot v^2|}{w (x^2, v^2)} |\eqref{expand_g4}| \dd v^2   \\
& \lesssim \frac{k}{\varrho (t)} \times  \sup_{i }  \int_{\prod_{j=3}^{i} \mathcal{V}_j}  \mathbf{1}_{0 \leq t^{i}}
\int^{t^{i}}_{ \max(0, t^{i+1})}w (x^i, v^{i}) \varrho^\prime(s) f(s, X(s; t^i, x^i, v^i), V(s; t^i, x^i, v^i)) \dd s \dd \tilde{\Sigma}_{i}    
\\& \lesssim \frac{k}{\varrho (t)} \int^t_0 \| \varrho^\prime (s) f(s) \|_{L^1_{x,v}} \dd s     
\lesssim \frac{k}{\varrho (t)} \int^t_0 
\| e^{\Lambda s} f(s) \|_{L^1_{x,v}} \dd s
\\& \lesssim \frac{k t}{\varrho (t)} \times \{ \| w^\prime f (0) \|_{L^\infty_{x,v}}
+ \| \varphi \big( (|v|^2 + 2 \Phi(x))^{1/2} \big) f_0 \|_{L^1_{x,v}}
 \}. 
\end{split}
\Ee

Lastly we bound the contribution of \eqref{expand_g5}. From Lemma \ref{lem:small_largek}, Proposition \ref{prop: est:measure} and $0 < n(x^2) \cdot v^2 \lesssim w (x^2, v^2) < w^\prime (x^2, v^2)$, we get
\Be \label{est:expand_g5}
\begin{split}
& \ \ \ \ \frac{1}{\varrho (t^2)} \int_{\R^3} \frac{|n(x^2) \cdot v^2|}{w (x^2, v^2)} |\eqref{expand_g5}| \dd v^2\\
& \lesssim \frac{\varrho (t^k)}{\varrho (t^2)} \sup_{(x,v) \in \bar{\O} \times \R^3}  \Big(\int_{\prod_{j=3}^{k -1} \mathcal{V}_j}   
    \mathbf{1}_{t^{k }(t^2,x^2,v^2,\cdots, v^{k-1}) \geq 0 }
\dd \sigma_3 \cdots \dd \sigma_{k-1}\Big)
\sup_{t_k \geq 0} \| w f(t^k) \|_{L^\infty_{x,v}}
\\&  \lesssim e^{-t} \| w f(0) \|_{L^\infty_{x,v}}
 \lesssim e^{-t} \| w^\prime f(0) \|_{L^\infty_{x,v}}.
\end{split}
\Ee

Collecting estimates from \eqref{est:expand_g1}-\eqref{est:expand_g5} and using $k \lesssim t$, we derive 
\Be \label{g estimate}
\big| \frac{1}{\varrho (t^2)} \int_{\mathcal{V}_2} \mathbf{1}_{t^2 \geq t/2} \frac{|n(x^2) \cdot v^2|}{w (x^2, v^2)} g (t^2, x^2, v^2) \dd v^2 \big|
\leq 
\max\{  \frac{1}{\varrho (t)}, \frac{(k+1) t}{\varrho (t)}  , e^{-t}
\}   
\lesssim \frac{\langle t\rangle^2}{\varrho (t)}.
\Ee
Using $\varrho(t)= e^{\Lambda t}$, $0 < w (x, v) < \mu_{\Theta}^{-1} (x, v)$ and \eqref{g estimate}, we conclude
\Be \notag
\eqref{f_exp2} 
\lesssim \frac{\langle t\rangle^2}{\varrho (t)}
\lesssim e^{- \Lambda t}.
\Ee
The above estimate, together with \eqref{v<t/2 part in wf}, \eqref{first part in wf} and \eqref{second part in wf}, we prove \eqref{theorem_infty}.
\end{proof}

\subsection{With Magnet field}

\begin{proof}[\textbf{Proof of Theorem \ref{theorem} of the case \eqref{field property mag}}]

We first show \eqref{theorem_infty_1}. Setting $\varrho (t) = t^6 + 1$ with $\varrho^\prime (t) \lesssim t^5 + 1$, we work on the stochastic cycle representation of $\varrho (t) w^\prime (x, v) f(t, x, v)$ in \eqref{expand_k1}-\eqref{expand_k4}.

For the contribution of \eqref{expand_k1}, since $t^1 < 0$, and $w^\prime, f$ are constant along the characteristic trajectory, we deduce
\be \label{est:expand_k1 mag}
\frac{1}{\varrho (t)} | \eqref{expand_k1} |= w^\prime (x, v) f(t, x, v) 
= w^\prime (X(0; t, x, v), V(0; t, x, v)) f(0, X, V)
\leq \| w^\prime f(0) \|_{L^\infty_{x,v}}.
\ee

Now we bound the contribution of \eqref{expand_k2}. Using
$0 < n(x) \cdot v \lesssim w^\prime (x, v) < \mu^{-1} (x, v)$ and Proposition \ref{prop: est:measure}, we derive
\Be \label{est:expand_k2 mag}
\begin{split}
\frac{1}{\varrho (t)} | \eqref{expand_k2} |
\lesssim \frac{k}{\varrho (t)}  \bigg( \sup_{i }    \int_{\prod_{j=1}^{k} \mathcal{V}_j}   
     \mathbf{1}_{t^{i+1} < 0 \leq t^{i }} 
      \dd \tilde{\Sigma}_{i}\bigg)
      \varrho(0)
   \| w^\prime f(0) \|_{L^\infty_{x,v}}   
\lesssim  \frac{k}{\varrho (t)}  \| w^\prime f(0) \|_{L^\infty_{x,v}}.
\end{split}
\Ee

Next we bound the contribution of \eqref{expand_k3}. 
Following \eqref{est:expand_k3} and Lemma \ref{lem:bound1_2 mag} with $\e = t^{-2}$, we have
\be \label{est:expand_k3 mag}
\begin{split}
& \ \ \ \ \frac{1}{\varrho (t)} | \eqref{expand_k3} |     
\\& \lesssim \frac{k}{\varrho (t)} \sup_{1 \leq i \leq k-1}  \int_{\prod_{j=1}^{i} \mathcal{V}_j}  \mathbf{1}_{0 \leq t^{i}}
\int^{t^{i}}_{ \max(0, t^{i+1})} 
 \varrho^\prime (s) w^\prime (X(s; t^i, x^i, v^i), V(s; t^i, x^i, v^i)) f(s, X, V) \dd s \dd \tilde{\Sigma}_{i}    
\\& \lesssim \frac{k}{\varrho (t)} 
 \big( \e^{i-1} \int_{\mathcal{V}_i} \mathbf{1}_{0 \leq t^{i}} \ \varrho (t^i) f(t^i, x^i, v^i) \{ n(x^i) \cdot v^i \}  \dd v^i
+  (i-1) \e^{-2} \int^t_0 \|  \varrho^\prime (s) f(s) \|_{L^1_{x,v}} \dd s \big)   
\\& \lesssim \frac{k}{\varrho (t)} \times 
\big( \e \sup_{t \geq 0} \| w^\prime f(t) \|_{L^\infty_{x,v}} \int_{\mathcal{V}_i} \mathbf{1}_{0 \leq t^{i}} \frac{\varrho (t^i)}{w^\prime (x^i, v^i)} \{ n(x^i) \cdot v^i \}  \dd v^i
+ k \e^{-2}
\| w^\prime f (0) \|_{L^\infty_{x,v}} \big) 
\\& \lesssim \frac{k}{t^2} \times 
\sup_{t \geq 0} \| w^\prime f(t) \|_{L^\infty_{x,v}} 
+   \frac{k^2 t^4}{\varrho (t)} \times 
\| w^\prime f (0) \|_{L^\infty_{x,v}}.
\end{split}
\ee

Lastly we bound the contribution of \eqref{expand_k4}. 
Following from \eqref{est:expand_k4}, Lemma \ref{lem:small_largek mag} and Proposition \ref{prop: est:measure}, we get
\Be \label{est:expand_k4 mag}
\begin{split}
 \frac{1}{\varrho (t)} |\eqref{expand_k4}| 
& \lesssim \frac{\varrho (t^k)}{\varrho (t)} \sup_{(x,v) \in \bar{\O} \times \R^3}  \Big(\int_{\prod_{j=1}^{k -1} \mathcal{V}_j}   
    \mathbf{1}_{t_{k }(t,x,v,v^1,\cdots, v^{k-1}) \geq 0 }
\dd \sigma_1 \cdots \dd \sigma_{k-1}\Big)
 \| w^\prime f(t_k) \|_{L^\infty_{x,v}}
\\&  \lesssim e^{-t} \sup_{t \geq s \geq 0} \| w^\prime f(s) \|_{L^\infty_{x,v}}.
\end{split}
\Ee

Collecting estimates from \eqref{est:expand_k1 mag}-\eqref{est:expand_k4 mag}, and picking $k \lesssim t$ and $0 < \e$, we derive 
\Be \label{k estimate mag}
(1 - \frac{k}{t^2} - e^{-t}) \sup_{t \geq 0} \| w^\prime f(t) \|_{L^\infty_{x,v}}
\lesssim 
(1 + \frac{k}{\varrho (t)} + \frac{k^2 t^4}{\varrho (t)}) \times
\| w^\prime f (0) \|_{L^\infty_{x,v}}. 
\Ee
Therefore, we prove \eqref{theorem_infty_1}.

Next, we prove \eqref{theorem_infty}. To show the decay of exponential moments and again utilize the $L^1$-decay, we set the weight function $\varrho(t):=  e^{\Lambda t}$ such that $\varrho^\prime(t) \lesssim e^{\Lambda t}$.

From Lemma \ref{sto_cycle}, we derive the form of $\int_{\R^3} w(x, v) |f (t, x, v)| \dd v$. First we split out the case when $|v| \geq t/2$ and $t_1 \leq 3t/4$, and get \eqref{v<t/2 term in wf mag}-\eqref{first team in f_exp mag}. Next, for $t_1 \geq 3t/4$, we follow along the stochastic cycles twice with $k=2$ and $t_* = t/2$ and get \eqref{second team in f_exp mag} and \eqref{f_exp2 mag}.
\begin{align}
\int_{\R^3} w(x, v) |f(t,x,v)| \dd v 
& \leq 
\int_{|v| \geq t/2} w(x, v) |f(t,x,v)| \dd v
\label{v<t/2 term in wf mag}
\\& + \int_{|v| \leq t/2} \mathbf{1}_{t^1 \leq 3t/4}  w(x, v) |f(3t/4, X(3t/4; t, x, v), V(3t/4; t, x, v))| \dd v
\label{first team in f_exp mag}
\\& + \int_{\R^3}  \mathbf{1}_{t^1 \geq 3t/4} w (x, v) \mu_{\Theta} (x^1, \vb)  \int_{\prod^2_{j=1} \mathcal{V}_j} \mathbf{1}_{t^2 < t/2 < t^1} w (x^1, v^1) |f(t^1, x^1, v^1)| \dd \Sigma_{1}^2 \dd v 
\label{second team in f_exp mag} 
\\& + \int_{\R^3} \mathbf{1}_{t^1 \geq 3t/4} w(x, v) \mu_{\Theta} (x^1, \vb) 
  \Big| \int_{\prod^2_{j=1} \mathcal{V}_j} \mathbf{1}_{t^2 \geq t/2} w (x^2, v^2)
f(t^2, x^2, v^2) \dd \Sigma_{2}^2 \Big| \dd v, \label{f_exp2 mag}
\end{align} 
where $\dd  {\Sigma}^{2}_{1} = \dd \sigma_{2}  \frac{ \dd \sigma_{1}}{ \mu_{\Theta} (x^2, v^1)w(x^1, v^1)}$ and $\dd  {\Sigma}^{2}_{2} = \frac{ \dd \sigma_{2}}{ \mu_{\Theta} (x^3, v^2) w(x^2, v^2)} \dd \sigma_{1}$ with
$\dd \sigma_j = \mu_{\Theta} (x^{j+1}, v^j) \{ n(x^j) \cdot v^j \} \dd v^j$ for $j =1, 2$.

For \eqref{v<t/2 term in wf mag},
following from $|v| \geq t/2$ and \eqref{v<t/2 part in wf} we derive that
\Be \label{v<t/2 part in wf mag}
\begin{split}
\int_{|v| \geq t/2} w(x, v) |f(t,x,v)| \dd v
& \lesssim \frac{1}{\sqrt{\theta^\prime - \theta}} e^{- \frac{(\theta^\prime - \theta) t^2}{4}} \| w^\prime f(0) \|_{L^\infty_{x,v}}.
\end{split}
\Ee

For \eqref{first team in f_exp mag}, from $|v| \leq t/2$ and $t^1 \leq 3t/4$, we have
\Be \label{v=3t/4 mag}
|V(3t/4; t, x, v)| \gtrsim |v_3 + g (t - 3t/4)| \geq 2t \geq |v| + t.
\Ee
Then, from the $L^\infty$-boundedness, \eqref{v=3t/4 mag} and $0 < w < w^\prime$, we deduce that 
\Be \label{first part in wf mag}
\begin{split}
\eqref{first team in f_exp mag} 
& \lesssim \int_{|v| \leq t/2}  \mathbf{1}_{t^1 \leq 3t/4} \frac{w (x, v)}{w^\prime (X(3t/4; t, x, v), V(3t/4; t, x, v))} \dd v \| w^\prime f(0) \|_{L^\infty_{x,v}}
\\& \leq \int_{|v| \leq t/2}  \mathbf{1}_{t^1 \leq 3t/4}
e^{(\theta - \theta^\prime) (|V(3t/4; t, x, v)|^2+ \Phi (X))} \dd v \| w^\prime f(0) \|_{L^\infty_{x,v}}
\\& \leq \int_{|v| \leq t/2} 
e^{(\theta - \theta^\prime) (|v|^2 + t^2)} \dd v \| w^\prime f(0) \|_{L^\infty_{x,v}}
\lesssim \frac{1}{\sqrt{\theta^\prime - \theta}} e^{- (\theta^\prime - \theta) t^2} \| w^\prime f(0) \|_{L^\infty_{x,v}}.
\end{split}
\Ee
Next, we bound $\int_{\R^3} w(x, v) \mu_{\Theta} (x^1, \vb) \dd v$ shown in \eqref{second team in f_exp mag} and \eqref{f_exp2 mag}.
From $|v| = |\vb| \geq |V(s;t,x,v)|$, we derive
\Be \label{vvb estimate mag}
\int_{\R^3} w(x, v) \mu_{\Theta} (x^1, \vb) \dd v 
\lesssim \int_{\R^3} \frac{w (x, v)}{w^\prime (x, v)} \dd v \lesssim 1.
\Ee

For \eqref{second team in f_exp mag}, from \eqref{vvb estimate mag}, \eqref{second part in wf} and $\int_{\mathcal{V}_2} \dd \sigma_2$ is bounded, we have 
\Be \label{second part in wf mag}
\begin{split}
\eqref{second team in f_exp mag} 
& \lesssim \int_{|v^1| \geq 5 t/4} w (x^1, v^1) |f(t^1, x^1, v^1)| \dd v^1 
\\& \lesssim \int_{|v^1| \geq 5 t/4} w (x^1, v^1) |f(t^1, x^1, v^1)| \dd v^1 
\\& \lesssim \int_{|v^1| \geq 5 t/4} \frac{w (x^1, v^1)}{w^\prime (x^1, v^1)} \dd v^1 \| w^\prime f(0) \|_{L^\infty_{x,v}}
\lesssim \frac{1}{\sqrt{\theta^\prime - \theta}} e^{- \frac{(\theta^\prime - \theta) t^2}{4}} \| w^\prime f(0) \|_{L^\infty_{x,v}}.
\end{split}
\Ee

Now we only need to bound \eqref{f_exp2 mag}. Since $\int_{\R^3} w(x, v) \mu_{\Theta} (x^1, \vb) \dd v \lesssim 1$ and $\int_{\mathcal{V}_1} \dd \sigma_1$ is bounded, it suffices to prove the decay of
\Be \label{decay of f_exp2 mag}
\sup_{v \in \R^3, v^1 \in \mathcal{V}_1}
\Big| \int_{\mathcal{V}_2} 
\mathbf{1}_{t^2 \geq t/2} f(t^2,x^2,v^2) \{ n(x^{2}) \cdot v^{2} \} \dd v^{2} \Big|. 
\Ee 

Here we define $g (t, x, v) := \varrho (t) w (x, v) f (t, x, v)$, and note that
\Be \notag
\frac{1}{\varrho (t^2)} \int_{\mathcal{V}_2} \frac{|n(x^2) \cdot v^2|}{w (x^2, v^2)} g (t^2, x^2, v^2) \dd v^2 = \int_{\mathcal{V}_2} f(t^2,x^2,v^2) \{ n(x^2) \cdot v^2 \} \dd v^2.
\Ee 
Therefore, it suffices to show the decay of $\big| \frac{1}{\varrho (t^2)} \int_{\mathcal{V}_2} \mathbf{1}_{t^2 \geq t/2} \frac{|n(x^2) \cdot v^2|}{w (x^2, v^2)} g (t^2, x^2, v^2) \dd v^2 \big|$.

Applying Lemma \ref{sto_cycle} with $w(x, v)$ and $\varrho(t)$, 
we choose $t_*=0$, $k \geq \mathfrak{C}t$, and obtain the following stochastic cycle representation of 
$g (t^2, x^2, v^2) = \varrho (t^2) w (x^2, v^2) f (t^2, x^2, v^2)$: 
\begin{align}
    g (t^2, x^2, v^2) 
   = & \mathbf{1}_{t^3 < 0} \
   \varrho (0) w(x^2, v^2) f (0, X(0; t^2, x^2, v^2), V(0; t^2, x^2, v^2))
\label{expand_g1 mag}
    \\& + w(x^2, v^2) \int^{t^2}_{\max(0, t^3)} \varrho^\prime(s) f(s, X(s; t^2, x^2, v^2), V(s; t^2, x^2, v^2)) \dd s \label{expand_g2 mag}
    \\& +  w \mu_{\Theta} (x^3, v^2_{\mathbf{b}}) \int_{\prod_{j=3}^{i} \mathcal{V}_j}   
     \sum\limits^{k-1}_{i=3} 
     \Big\{ \mathbf{1}_{t^{i+1} < 0 \leq t^{i}}  \varrho (0) w(x^i, v^{i}) \notag
     \\& \hspace{5cm} \times f (0, X(0; t^i, x^i, v^i), V(0; t^i, x^i, v^i)) \Big\}
      \dd \tilde{\Sigma}_{i}
\label{expand_g3 mag}
    \\& + w \mu_{\Theta} (x^3, v^2_{\mathbf{b}}) \int_{\prod_{j=3}^{i} \mathcal{V}_j}   
     \sum\limits^{k-1}_{i=3} 
      \mathbf{1}_{0 \leq t^{i}}
     \Big\{ \int^{t^{i}}_{ \max(0, t^{i+1})} \varrho^\prime(s) w (x^i, v^{i}) 
     \notag
     \\& \hspace{5cm} \times f(s, X(s; t^i, x^i, v^i), V(s; t^i, x^i, v^i)) \dd s 
      \Big\} \dd \tilde{\Sigma}_{i}
\label{expand_g4 mag}
    \\& + w \mu_{\Theta} (x^3, v^2_{\mathbf{b}}) \int_{\prod_{j=3}^{k } \mathcal{V}_j}   
    \mathbf{1}_{t^{k} \geq 0} \
    g (t^{k}, x^{k}, v^{k})
     \dd \tilde{\Sigma}_{k}, \label{expand_g5 mag} 
\end{align} 
where 
$\dd \tilde{\Sigma}_{i} 
:= \frac{ \dd \sigma_{i}}{\mu_{\Theta} (x^{i+1}, v^{i}) w(x^i, v^{i})} \dd \sigma_{i-1} \cdots \dd \sigma_3$ with $3 \leq i \leq k$.
Here, we regard $t^2, x^2, v^2$ as free parameters and from Lemma \ref{conservative field mag}, we have 
$\mu_{\Theta} (x^3, v^2_{\mathbf{b}}) = \mu_{\Theta} (x^3, v^2)$.

Next we estimate the contribution of \eqref{expand_g1 mag}-\eqref{expand_g5 mag} in 
$\frac{1}{\varrho (t^2)} \int_{\mathcal{V}_2} \frac{ |n(x^2) \cdot v^2| }{w (x^2, v^2)} g (t^2, x^2, v^2) \dd v^2$
term by term. 

For the contribution of \eqref{expand_g1 mag},
from $t^2 \geq t/2$, $t^3 \leq 0$, we have $n(x^2) \cdot v^2 \gtrsim \frac{gt}{4}$. From the $L^\infty$-boundedness, $t \gg 1$ and $0 < n(x^2) \cdot v^2 \lesssim w (x^2, v^2) < w^\prime (x^2, v^2)$, we deduce that 
\Be \label{est:expand_g1 mag}
\begin{split}
\frac{1}{\varrho (t^2)} \int_{\mathcal{V}_2} \frac{|n(x^2) \cdot v^2|}{w (x^2, v^2)} |\eqref{expand_g1 mag}| \dd v^2
& = \frac{1}{\varrho (t^2)} \int_{\mathcal{V}_2} |n(x^2) \cdot v^2| | \varrho (0) f (t^2, x^2, v^2) | \dd v^2 
\\& \lesssim \frac{\varrho(0)}{\varrho (t)} \| w f(0) \|_{L^\infty_{x,v}}
\lesssim  \frac{1}{\varrho (t)} \| w^\prime f(0) \|_{L^\infty_{x,v}}.
\end{split}
\Ee 
  
Now we bound the contribution of \eqref{expand_g2 mag}. Following \eqref{est:expand_g2} and Lemma \ref{lem:bound1_2 mag} with $\e = \varrho^{- \frac{1}{3}} (t^2)$, we get 
\Be \label{est:expand_g2 mag}
\begin{split}
& \ \ \ \ \frac{1}{\varrho (t^2)} \int_{\mathcal{V}_2} \frac{|n(x^2) \cdot v^2|}{w (x^2, v^2)} |\eqref{expand_g2 mag} | \dd v^2
\\& \lesssim \frac{1}{\varrho (t^2)} \int_{\mathcal{V}_2} \frac{|n(x^2) \cdot v^2|}{w (x^2, v^2)} w(x^2, v^2) \int^{t^2}_{\max(0, t^3)} \varrho^\prime(s)  f(t^2, x^2, v^2) \dd s \dd v^2
\\& \lesssim \frac{1}{\varrho (t^2)} \times
 \big( \e \int_{\mathcal{V}_i} \mathbf{1}_{0 \leq t^{2}} \ \varrho (t^2) f(t^2, x^2, v^2) \{ n(x^2) \cdot v^2 \}  \dd v^2
+  \e^{-2} \int^t_0 \|  \varrho^\prime (s) f(s) \|_{L^1_{x,v}} \dd s \big) 
\\& \lesssim \frac{1}{\varrho (t)} \times  \| w^\prime f (0) \|_{L^\infty_{x,v}}.
\end{split}
\Ee

Next, we bound the contribution of \eqref{expand_g3 mag}. 
Following from \eqref{est:expand_g3} and Proposition \ref{prop: est:measure}, we derive 
\Be \label{est:expand_g3 mag}
\begin{split}
\frac{1}{\varrho (t^2)} \int_{\R^3} \frac{|n(x^2) \cdot v^2|}{w (x^2, v^2)} |\eqref{expand_g3 mag}| \dd v^2
& \lesssim \frac{k}{\varrho (t)}  \bigg( \sup_{i }    \int_{\prod_{j=3}^{i} \mathcal{V}_j}   
     \mathbf{1}_{t^{i+1} < 0 \leq t^{i }} 
      \dd \tilde{\Sigma}_{i} \bigg)
      \varrho(0)
   \| w f(0) \|_{L^\infty_{x,v}}   
\\& \lesssim  \frac{k}{\varrho (t)}  \| w^\prime f(0) \|_{L^\infty_{x,v}}.
\end{split}
\Ee

Using Lemma \ref{lem:bound1_2 mag} with $\e = \varrho^{- \frac{1}{3}} (t^2)$, Proposition \ref{prop: est:measure} and Theorem \ref{theorem_1}, we bound the contribution of \eqref{expand_g4 mag} as
\Be \label{est:expand_g4 mag}
\begin{split}
& \ \ \ \ \frac{1}{\varrho (t^2)} \int_{\R^3} \frac{|n(x^2) \cdot v^2|}{w (x^2, v^2)} |\eqref{expand_g4 mag}| \dd v^2   \\& \lesssim \frac{k}{\varrho (t)} \times  \sup_{i }  \int_{\prod_{j=3}^{i} \mathcal{V}_j}  \mathbf{1}_{0 \leq t^{i}}
\int^{t^{i}}_{ \max(0, t^{i+1})}w (x^i, v^{i}) \varrho^\prime(s) f(s, X(s; t^i, x^i, v^i), V(s; t^i, x^i, v^i)) \dd s \dd \tilde{\Sigma}_{i} 
\\& \lesssim \frac{k}{\varrho (t)} \times
 \big( \e^{i-1} \int_{\mathcal{V}_i} \mathbf{1}_{0 \leq t^{i}} \ \varrho (t^i) f(t^i, x^i, v^i) \{ n(x^i) \cdot v^i \}  \dd v^i
+  (i-1) \e^{-2} \int^t_0 \|  \varrho^\prime (s) f(s) \|_{L^1_{x,v}} \dd s \big)   
\\& \lesssim \frac{k}{\varrho (t)} \times
\big( \e^{i-1} \| w^\prime f(t^i) \|_{L^\infty_{x,v}} \int_{\mathcal{V}_i} \mathbf{1}_{0 \leq t^{i}} \frac{\varrho (t^i)}{w^\prime (x^i, v^i)} \{ n(x^i) \cdot v^i \}  \dd v^i
+ k \e^{-2}
\int^t_0 \|  \varrho^\prime (s) f(s) \|_{L^1_{x,v}} \dd s \big)    
\\& \lesssim \frac{k}{\varrho (t)} \times \big( \e \| w^\prime f (0) \|_{L^\infty_{x,v}} + k \e^{-2}
\int^t_0 \| e^{\Lambda s} f(s) \|_{L^1_{x,v}} \dd s \big)      
\\& \lesssim \frac{k^2 t}{\varrho (t)} \times \{ \| w^\prime f (0) \|_{L^\infty_{x,v}}
+ \| \varphi \big( (|v|^2 + 2 \Phi(x))^{1/2} \big) f_0 \|_{L^1_{x,v}}
 \}. 
\end{split}
\Ee

Lastly we bound the contribution of \eqref{expand_g5 mag}. 
Following \eqref{est:expand_g5} and Lemma \ref{lem:small_largek mag}, we get
\Be \label{est:expand_g5 mag}
\begin{split}
& \ \ \ \ \frac{1}{\varrho (t^2)} \int_{\R^3} \frac{|n(x^2) \cdot v^2|}{w (x^2, v^2)} |\eqref{expand_g5 mag}| \dd v^2\\
& \lesssim \frac{\varrho (t^k)}{\varrho (t^2)} \sup_{(x,v) \in \bar{\O} \times \R^3}  \Big(\int_{\prod_{j=3}^{k -1} \mathcal{V}_j}   
    \mathbf{1}_{t^{k }(t^2,x^2,v^2,\cdots, v^{k-1}) \geq 0 }
\dd \sigma_3 \cdots \dd \sigma_{k-1}\Big)
\sup_{t_k \geq 0} \| w f(t^k) \|_{L^\infty_{x,v}}
\\&  
 \lesssim e^{-t} \| w^\prime f(0) \|_{L^\infty_{x,v}}.
\end{split}
\Ee

Collecting estimates from \eqref{est:expand_g1 mag}-\eqref{est:expand_g5 mag} and using $k \lesssim t$, we derive 
\Be \label{g estimate mag}
\big| \frac{1}{\varrho (t^2)} \int_{\mathcal{V}_2} \mathbf{1}_{t^2 \geq t/2} \frac{|n(x^2) \cdot v^2|}{w (x^2, v^2)} g (t^2, x^2, v^2) \dd v^2 \big|
\leq 
\max\{  \frac{1}{\varrho (t)}, \frac{(k^2 + 1) t}{\varrho (t)}  , e^{-t}
\}   
 \lesssim \frac{\langle t\rangle^3}{\varrho (t)}.
\Ee
Using $\varrho(t)= e^{\Lambda t}$, $0 < w (x, v) < \mu_{\Theta}^{-1} (x, v)$ and \eqref{g estimate mag}, we conclude $
\eqref{f_exp2 mag} 
\lesssim \frac{\langle t\rangle^3}{\varrho (t)}
\lesssim e^{- \Lambda t}.$ The above estimate, together with \eqref{v<t/2 part in wf mag}, \eqref{first part in wf mag} and \eqref{second part in wf mag}, we prove \eqref{theorem_infty}.
\end{proof}





\end{document}